\documentclass{article}
\usepackage[utf8]{inputenc}
\usepackage[T1]{fontenc} 

\usepackage{bbold}
\newcommand{\rot}{{\mathbb r}}

\newcommand{\hide}[1]{}

\usepackage{titlesec}

\titleformat{\subsection}[runin]
  {\normalfont\normalsize\bfseries}
  {\thesubsection}
  {1em}
  {}
  [.]

\titlespacing{\subsection}
  {0pt}
  {3.25ex plus 1ex minus .2ex}
  {1em}

\titleformat{\subsubsection}[runin]
  {\normalfont\normalsize\itshape}
  {\thesubsubsection}
  {1em}
  {}
  [.]

\titlespacing{\subsubsection}
  {0pt}
  {3.25ex plus 1ex minus .2ex}
  {1em}

\usepackage{amsmath, amssymb, amscd, amsthm, comment, amsxtra,mathrsfs, enumerate}
\usepackage{accents}
\usepackage{graphicx}
\usepackage{subfigure}
\usepackage{enumitem}

\usepackage{euscript}
\usepackage[format=plain,labelfont=bf,up,width=.99\textwidth]{caption}
\usepackage[toc,page]{appendix}
\usepackage[all]{xy}
\usepackage[textwidth=16.5cm,textheight=21cm]{geometry}
\usepackage{tikz-cd}
\usepackage{xfrac}

\usepackage[colorinlistoftodos,prependcaption,textsize=tiny]{todonotes}

\usepackage{arcs}
\usepackage[pdftex, linktocpage=true]{hyperref}

\usepackage{accents}

\usepackage{tikz}
\usetikzlibrary{arrows.meta}

\newtheorem{claim}{Claim}
\usepackage{etoolbox}
\AtEndEnvironment{proof}{\setcounter{claim}{0}}

\newtheorem{MainTheorem}{Theorem}
\newtheorem{MainCorollary}[MainTheorem]{Corollary}
\newtheorem{MainRemark}[MainTheorem]{Remark}
\newtheorem{MainLemma}[MainTheorem]{Lemma}

\newtheorem{KeyConjecture}{Conjecture}

\newtheorem{KeyTheorem}[KeyConjecture]{Theorem}

\DeclareFontFamily{U}{mathb}{\hyphenchar\font45}
\DeclareFontShape{U}{mathb}{m}{n}{ <5> <6> <7> <8> <9> <10> gen * mathb <10.95> mathb10 <12> <14.4> <17.28> <20.74> <24.88> mathb12 }{}
\DeclareSymbolFont{mathb}{U}{mathb}{m}{n}
\DeclareFontSubstitution{U}{mathb}{m}{n}
\DeclareMathSymbol{\selfmap}{3}{mathb}{"FD}

\usepackage{marginnote}

\usepackage{ifthen,changepage}
\usepackage{xargs}

\newcommandx{\leftmarginnote}[2][2=0pt]
           {\checkoddpage
            \ifoddpage
              {\reversemarginpar\marginnote{#1}[#2]}
            \else
              {\marginnote{#1}[#2]}
            \fi}
\newcommandx{\rightmarginnote}[2][2=0pt]
           {\checkoddpage
            \ifoddpage
              {\marginnote{#1}[#2]}
            \else
              {\reversemarginpar\marginnote{#1}[#2]}
            \fi}

\newcounter{reminder}
\newcommand{\reminder}[1]{\rightmarginnote{\stepcounter{reminder}\textcolor{red}{Note~\arabic{reminder}}}$\langle$~{\sf #1 }~$\rangle$}

\makeatletter
\newcommand{\caseitem}[2]{%
  \item[Case #1]
  \phantomsection
  \protected@edef\@currentlabel{#1}
  \label{#2}
}
\makeatother

\everymath{\displaystyle}

\newtheorem{thm}{Theorem}[section]
\newtheorem{lem}[thm]{Lemma}
\newtheorem{prop}[thm]{Proposition}

\newtheorem{cor}[thm]{Corollary}

\newtheorem*{claim*}{Claim}
\newtheorem*{thm*}{Theorem}
\newtheorem{defn}[thm]{Definition}

\newtheorem{rem}[thm]{Remark}

\newtheorem{snakelmm}[thm]{Snake Induction Lemma}
\newtheorem{weldlmm}[thm]{Welding Lemma}
\newtheorem{caliblem}[thm]{Calibration Lemma}

\newtheorem{amplmm}[thm]{Amplification Lemma}

\newtheorem{sneaklmm}[thm]{Sneaking Lemma}

\newtheorem{assumption}{Assumption}

\newtheorem*{conj*}{Conjecture}
\newtheorem*{rem*}{Remark}
\newtheorem*{lem*}{Lemma}

\newtheorem*{def*}{Definition}
\newtheorem*{fact*}{Fact}
\newtheorem*{assumption*}{Assumption}
\newtheorem*{assum*}{Assumption}
\newtheorem*{remnota*}{Remark on the notation}

\numberwithin{equation}{section}

\newlist{steps}{enumerate}{1}
\setlist[steps, 1]{label = \textbf{Step \arabic*}:}

\usepackage{mathtools}

\makeatletter

\newcommand{\setword}[2]{%
  \phantomsection
  #1\def\@currentlabel{\unexpanded{#1}}\label{#2}%
}

\newcommand{\grnd}{{\mathrm{grnd}}}

\newcommand{\qq}{{\mathfrak q}}
\newcommand{\pp}{{\mathfrak p}}
\newcommand{\ovlTheta}{{\overline \Theta}}
\newcommand{\bnd}{{\mathrm{bnd}}}
\newcommand{\egm}{{\mathrm{egm}}}
\renewcommand{\sec}{{\mathrm{sec}}}
\newcommand{\FPsec}{{\bf x}}

\newcommand{\Tow}{{\mathrm{Tow}_{\mathrm{sec}}}}
\newcommand{\bTow}{{\mathrm{\bf Tow}_{\mathrm{sec}}}}
\newcommand{\ovlTow}{{\overline {\mathrm{\bf Tow}}_{\mathrm{sec}}}}
\newcommand{\wtheta}{{\hat \theta}}

\newcommand{\I}{P}

\newcommand{\cW}{{\cal W}}

\newcommand{\cF}{{\cal F}}

\newcommand{\RR}{{\Bbb R}}

\newcommand{\ZZ}{{\Bbb Z}}

\newcommand{\wC}{{\widehat{\mathbb C}}}

\def\cal{\mathcal}

\def\B{{\mathcal B}}
\def\C{{\mathbb C}}
\def\D{{\mathfrak D}}
\def\F{{\mathcal F}}
\def\G{{\mathcal G}}
\def\H{{\mathcal H}}
\def\I{{\mathcal I}}
\def\K{{\mathcal K}}
\def\N{{\mathbb N}}

\def\Q{{\mathbb Q}}
\def\R{{{\mathcal R}}}

\def\T{{\mathcal T}}
\def\X{{\mathcal X}}
\def\Y{{\mathcal Y}}

\def\Z{{\mathcal Z}}

 \def\epsilon{{\varepsilon}}

\def\l{{\mathfrak l}}
\def\q{{\mathfrak q}}
\def\p{{\mathfrak p}}

\newcommand{\W}{\mathcal{W}}
\newcommand{\Lnew}{\mathcal{L}^{\mathrm{new}}}
\newcommand{\LNew}{\mathcal{L}^{\mathrm{New}}}
\newcommand{\LNEW}{\mathcal{L}^{\mathrm{NEW}}}
\newcommand{\Rnew}{\mathcal{R}^{\mathrm{new}}}
\newcommand{\RNew}{\mathcal{R}^{\mathrm{New}}}

\newcommand{\fillR}{\accentset{\bullet}{\mathcal{R}}}
\newcommand{\fillRnew}{\accentset{\bullet}{\mathcal{R}}^{\mathrm{new}}}
\newcommand{\fillX}{\accentset{\bullet}{\mathcal{X}}}
\newcommand{\fillY}{\accentset{\bullet}{\mathcal{Y}}}
\newcommand{\partialhL}{\partial^{h,0}}
\newcommand{\partialhR}{\partial^{h,1}}
\newcommand{\dist}{\mathrm{dist}}
\newcommand{\hZ}{\widehat{Z}}
\newcommand{\Bepsilon}{\boldsymbol{\epsilon}}
\newcommand{\Bdelta}{\boldsymbol{\delta}}
\newcommand{\Bkappa}{\boldsymbol{\kappa}}

\newcommand{{\BDelta}}{\boldsymbol{\Delta}}

\newcommand{\Blambda}{\boldsymbol{\mathrm{\lambda}}}
\newcommand{\Bt}{\boldsymbol{\mathrm{t}}}
\newcommand{\BK}{\boldsymbol{\mathrm{K}}}
\newcommand{\bfrho}{\boldsymbol{\rho}}

\newcommand{\Ainn}{A^{\mathrm{inn}}}
\newcommand{\Aout}{A^{\mathrm{out}}}

\newcommand{\Iout}{I^{\mathrm{out}}}
\newcommand{\Jinn}{J^{\mathrm{inn}}}

\newcommand{\Sinn}{S^{\mathrm{inn}}}
\newcommand{\Sout}{S^{\mathrm{out}}}
\newcommand{\partialinn}{\partial^{\mathrm{inn}}}
\newcommand{\partialout}{\partial^{\mathrm{out}}}

\newcommand{\RomA}{\mathrm{\uppercase\expandafter{\romannumeral1\relax}}}
\newcommand{\RomB}{\mathrm{\uppercase\expandafter{\romannumeral2\relax}}}
\newcommand{\RomC}{\mathrm{\uppercase\expandafter{\romannumeral3\relax}}}

\newcommand{\bfB}{\mathbf B}
\newcommand{\bfC}{\mathbf C}

\newcommand{\bfH}{\mathbf H}
\newcommand{\bfI}{\mathbf I}

\newcommand{\bfK}{\mathbf K}
\newcommand{\bfM}{\mathbf M}
\newcommand{\bfN}{\mathbf N}

\newcommand{\bfS}{\mathbf S}

\newcommand{\bfb}{\mathbf b}

\newcommand{\bfk}{\mathbf k}

\newcommand{\bfs}{\mathbf s}
\newcommand{\bft}{\mathbf t}

\newcommand{\bfx}{\mathbf x}
\newcommand{\bfy}{\mathbf y}

\renewcommand{\Re}{\operatorname{Re}}
\renewcommand{\Im}{\operatorname{Im}}

\newcommand{\wZ}{{\widehat Z}}

\newcommand{\Width}{{\mathcal W}}

\newcommand{\LL}{{\mathcal L}}
\newcommand{\Fam}{{\mathcal F}}

\newcommand{\Pol}{{\mathcal P}_{\mathfrak{g}}}
\newcommand{\Zak}{\mathrm{Zak}_{\mathfrak{g}}}
\newcommand{\CubZak}{\mathrm{Zak}_{\mathfrak{3}}}
\newcommand{\CubPol}{{\mathcal P}_{\mathfrak{3}}}

\def\b{{b}}

\title{Multicritical pseudo-Siegel bounds and \\the optimality of the Brjuno condition \\for all neutral polynomials.}
\author{Dzmitry Dudko, Yusheng Luo, Runze Zhang}

\begin{document}

\maketitle

\begin{abstract}
We establish uniform multicritical pseudo-Siegel bounds for polynomials of Zakeri type, where all free critical points are on the boundary of an invariant Siegel disk of bounded type. This generalizes the unicritical result~\cite{DL:wZ}. As an application, we show the optimality of the Brjuno condition under the Zakeri restriction: as the rotation number becomes non-Brjuno, the Siegel disk evolves into the Cremer Hedgehog. Combined with~\cite{BCR26}, this implies the optimality of the Brjuno condition for all neutral polynomials, resolving a conjecture by Douady from the 1980s.
\end{abstract}

\setcounter{tocdepth}{1}
\tableofcontents

\section{Introduction} In many natural families, one expects rotational dynamics to break down according to universal rules. For analytic circle diffeomorphisms, the Brjuno arithmetic condition \eqref{eq:Br:function} on rotation numbers governs when linearization is stable under small perturbations; see~\S\ref{ss:intro:BrjunoCondit} for motivation and Item~\ref{cond:Br} for a precise statement.

The Brjuno condition has implications far beyond circle maps. One of the main outcomes of the current paper is Theorem~\ref{thm:locSiegelDisk}, which confirms the necessity (i.e., optimality) and sufficiency of the Brjuno linearization condition for Siegel disks arising in polynomial dynamics. This has been conjectured by Douady in the 1980s, see~\cite[Remarque 2 page 162]{Do87}.  We remark that the sufficiency part has been known since the work of Brjuno \cite{Brjuno1971-1972} and R\"ussmann \cite{Russmann1967}.

In the early 1990s~\cite{Yoc}, Yoccoz established the quadratic case of Theorem~\ref{thm:locSiegelDisk} by verifying the optimality of the Brjuno condition. However, outside one-parameter families such as the quadratic family, the optimality part of Theorem~\ref{thm:locSiegelDisk} had remained widely open; see \cite{Geyer2019, BCR26} for some related progress. We note that the optimality of the Brjuno condition does not hold in the setting of local germs; see~\S\ref{sss:rem.natural} for more details.

Buff, Ch\'eritat, and Roesch~\cite{BCR26} have recently reduced the optimality problem for the Brjuno condition to polynomials satisfying the \emph{Zakeri} restriction. Roughly speaking, Zakeri-type polynomials are those whose Siegel disks are maximal among polynomials with a fixed rotation number. Main Theorem~\ref{mainthm:C} provides the final step towards Theorem~\ref{thm:locSiegelDisk} by establishing the optimality of the Brjuno condition for all polynomials satisfying the Zakeri restriction. We learned of the reduction argument implemented in~\cite{BCR26} during the summer of 2025; see Remark~\ref{rem:D}.

Even though Theorem~\ref{mainthm:C} is a local statement, its proof relies on understanding global impacts of rotational dynamics. A general principle is that, in natural families such as polynomials $f\colon \C\selfmap$, rotation around a neutral fixed point should not stop ``accidentally''; rather, it is expected to persist until it reaches and absorbs a critical orbit. Conjecture~\ref{mainconj:MH} in~\S\ref{ss:intro:globRD} asserts that there is always a rotating Mother Hedgehog around any neutral fixed point. Conjecture~\ref{conj:globSiegelDisk} describes the structure of the Mother Hedgehog depending on the arithmetic properties of the rotation number.

Over the last few decades, much of the effort in complex dynamics has been devoted to developing near-neutral renormalization theory for quadratic polynomials yielding some of the most significant results in the field; see~\S\ref{ss:intro:globRD} for more details. For instance, the quadratic cases of Conjecture~\ref{mainconj:MH} and in part Conjecture~\ref{conj:globSiegelDisk} received confirmation. In~\cite{DL:wZ}, Lyubich and the first author introduced and established \emph{pseudo-Siegel} bounds in the quadratic case. For the first time, \emph{a priori} bounds (precompactness of first return maps) became available for \emph{all} neutral quadratic polynomials.

Our main results extend many of the results known in the quadratic case to the multicritical setting of Zakeri polynomials. Theorem~\ref{mainthm:A} establishes \emph{multicritical} pseudo-Siegel bounds that are uniform over all Zakeri polynomials of a fixed degree. In particular, this confirms Conjecture~\ref{mainconj:MH} under the Zakeri condition. Theorem~\ref{mainthm:B} constructs the multicritical renorm-attractor for a sectorial renormalization scheme. By going up along renormalization towers, Lemma~\ref{MainLemma} provides a desired control on small periodic points and yields Theorem~\ref{mainthm:C}. This analysis also gives an opening towards Conjecture~\ref{conj:globSiegelDisk}, again under the Zakeri condition.

Introduction is organized as follows. Subsections~\S\ref{ss:intro:BrjunoCondit} and \S\ref{ss:intro:globRD} state Theorem~\ref{thm:locSiegelDisk}, Conjectures~\ref{mainconj:MH},~\ref{conj:globSiegelDisk} and also provide motivation and historical background for neutral $1$-D dynamics. Main results are stated in~\S\ref{ss:intro:main.results},~\S\ref{ss:intro.secto.tower}, and ~\S\ref{ss:outline.SmallPerPts} after introducing appropriate notations in~\S\ref{ss:dfn_notat}. Outline of the proofs, with emphasis on the differences to the quadratic case, is presented in~\S\ref{ss:intro:prf:outline}; see also Figure~\ref{fig:placeholder}. For conceptual remarks on renormalization, parameter rigidity, and the nature of near-rotational dynamics, see~\S\ref{ss:intro:comments}.

\subsection{Local linearization: Brjuno Condition and its history} \label{ss:intro:BrjunoCondit}

In the late nineteenth century, Poincaré realized that the seemingly recurrent motions of three celestial bodies can become unstable over long timescales. This insight marked the discovery of chaotic behavior and, over time, raised fundamental questions about when rotational dynamics persists and when it breaks down. Such questions are closely tied to the arithmetic properties of rotation numbers and lead to what is now known as a \emph{small-divisor problem}.

The first successful solution of a small-divisor problem was given by Siegel in 1942, in his work on analytic linearization near an irrationally neutral fixed point for one-dimensional holomorphic germs. As a prototypical model, Siegel’s work helped motivate the Kolmogorov-Arnold-Moser (KAM) theory that emerged in the 1950s and 1960s, which showed that “most” quasi-periodic rotational motions in any dimension persist under sufficiently small perturbations, provided suitable Diophantine conditions are satisfied. The precise nature of how rotational behavior breaks down later became a central topic in low-dimensional dynamics, especially after the discovery of dynamical renormalization in the 1970s. This led to many conjectures concerning the universality of the transition from rotation to chaos.

A foundation to the theory of circle maps was designed by Denjoy (1932), who justified a topological linearization for $C^2$-diffeomorphisms of the circle with irrational rotation numbers. Higher-smoothness linearization theory for circle diffeomorphisms has again led to a small-divisor problem and is commonly referred to as a one-dimensional KAM-type theory.

The theory of analytic circle diffeomorphisms, together with the local theory of neutral holomorphic germs, received an essentially complete treatment in the second half of the twentieth century through the work of Arnold (in the KAM framework), Cherry, Brjuno, \'Ecalle, Voronin, Douady, Herman, Yoccoz, Perez-Marco. In particular, the following characterization was established in the circle case:
\begin{itemize}
    \item[\setword{{\bf(B)}}{cond:Br}] Any small analytic perturbation $f\colon \mathbb R/\mathbb Z\selfmap$ of the irrational rotation 
    $x\mapsto x+\theta$ of the same rotation number $\rho(f)=\theta=[a_0,a_1,a_2,\dots]$ is analytically linearizable if and only if $\theta$ is \emph{Brjuno}:
    \begin{equation}
        \label{eq:Br:function}  \mathbb{B}(\theta)\coloneqq \sum_n\frac{\log q_{n+1}}{q_n}<\infty,\qquad\quad \frac{p_n}{q_n}=[a_0,a_1,\dots ,a_n] \text{ are best approximations of $\theta$}.
    \end{equation}
\end{itemize}
Equivalently, the Brjuno rotations are precisely those for which linearizable maps form an open set among maps with a fixed rotation number $\theta$.

 It has long been expected that Characterization~\ref{cond:Br} extends far beyond the circle setting. One of the main outcomes of the current paper confirms the Brjuno linearization condition for Siegel disks arising in polynomial dynamics. This has been conjectured by Douady in the 1980s, see~\cite[Remarque 2 page 162]{Do87}.

\begin{KeyTheorem}[Local Siegel disks of polynomials. Consequence of Theorem~\ref{mainthm:C} and~{\cite[Main Theorem]{BCR26}}]\label{thm:locSiegelDisk}\leavevmode\\ Consider a polynomial $f\colon \mathbb C\selfmap$ of degree $\deg f\ge 2$, 
and assume that $0$ is an irrational neutral fixed point:
\begin{equation}
    \label{eq:conj:locSiegelDisk}
    f(0)=0, \qquad f'(0)=e^{2\pi i\theta}, \qquad \theta\in \mathbb R\setminus \mathbb Q.
\end{equation}
Then $f$ has a linearizable Siegel disk around $0$ if and only if $\theta$ satisfies the Brjuno condition~\eqref{eq:Br:function}.
\end{KeyTheorem}

\noindent By passing to iterates, the case of periodic neutral points of polynomials is reduced to that of fixed points as in~\eqref{eq:conj:locSiegelDisk}. Similar characterizations as in Theorem~\ref{thm:locSiegelDisk} are expected, for instance, for rational maps (conjectured by Douady in~\cite{Do87}) and (less certain) for dissipative H\'enon maps in $\C^2$, which may be regarded as invertible perturbations of one-dimensional polynomials.

Observe that $f$ as in~\eqref{eq:conj:locSiegelDisk} induces the infinitesimal linear rotation $x\mapsto x+ \theta\colon \mathbb R/\mathbb Z\selfmap$ of the tangent circle at $0=f(0)$. Accordingly, the dynamics on sufficiently small circles around $0$ is a small perturbation of this rotation. Therefore, the sufficiency part of Theorem~\ref{thm:locSiegelDisk} naturally extends the setting in Item~\ref{cond:Br} to perturbations of circle dynamics that need not be real-symmetric.

The \emph{sufficiency} of the Brjuno condition (as in Theorem~\ref{thm:locSiegelDisk}) for the linearization of holomorphic germs was established independently in 1960s by Brjuno \cite{Brjuno1971-1972} and R\"ussmann \cite{Russmann1967}, and later by Yoccoz~\cite{Yoc} with a new proof using renormalization. To achieve this, Yoccoz developed the framework of local \emph{sector renormalization}. See \cite{R99} for a similar result concerning rotating complex annuli, and \cite{GY} for the hyperbolicity of the associated renormalization operator.

The \emph{optimality} of the Brjuno condition as in Theorem~\ref{thm:locSiegelDisk} is much more subtle; in particular, it does not hold for local analytic germs; see~\S\ref{sss:rem.natural} for trivial and more sophisticated examples. 
Implicitly, the theorem asserts that the pathology does not occur for natural global families such as polynomials.

 In \cite{Yoc}, Yoccoz justified Theorem~\ref{thm:locSiegelDisk} for quadratic polynomials by constructing, for every non-Brjuno rotation number $\theta$, a local analytic germ with \emph{small periodic orbits}. This analytic germ is then perturbed into the space of quadratic-like maps and straightened into a neutral quadratic polynomial $f$ as in~\eqref{eq:conj:locSiegelDisk}. The resulting $f$ inherits small periodic orbits and hence is not linearizable around $0$. Since a neutral quadratic polynomial $f$ is uniquely determined by its rotation number $\theta$, all neutral quadratic polynomials are accounted for,  justifying Theorem~\ref{thm:locSiegelDisk} in $\deg=2$. 

Outside one-parameter families such as the quadratic family, however, the optimality part of Theorem~\ref{thm:locSiegelDisk} had remained widely open. We refer to~\cite{BCR26} for more discussion regarding local linearization.

\subsection{Semi-global rotational dynamics and neutral renormalization}\label{ss:intro:globRD}  Consider a polynomial $f\colon {\mathbb C}\selfmap$ of degree $\deg f\ge 2$ with a neutral fixed point at $0$ as in~\eqref{eq:conj:locSiegelDisk}. A \emph{hedgehog} $0 \in H\Subset \C$ of $f$ at $0$ is a compact, connected, full (i.e., $\C\setminus H$ is connected), infinite set such that $f\colon H\selfmap$ is a self-homeomorphism. We say that $H$ \emph{rotates} under $f$.

The existence of ``local'' hedgehogs in arbitrarily small neighborhoods of $0$ was established by Perez-Marco in the 1990s~\cite{PM} in the setting of local analytic germs. This newly introduced topological object greatly clarified the local structure of non-linearizable Cremer maps.

In the case of polynomials (and also rational maps), it has been suggested that there exists a maximal \emph{Mother Hedgehog} (cf,~\cite{Chi}), which reaches and absorbs at least one critical orbit:

\begin{KeyConjecture}[Mother Hedgehog $H_f$]\label{mainconj:MH} In the setting of Theorem~\ref{thm:locSiegelDisk}, there exists a unique hedgehog $H_f$ around $0=f(0)$ such that $f\colon H_f\selfmap$ is a self-homeomorphism and $\partial H_f$ contains at least one critical point.
\end{KeyConjecture}

Since the Mother Hedgehog $H_f$ absorbs a critical orbit, we refer to the dynamics $f\colon H_f\selfmap$ as \emph{semi-global rotational dynamics}.

In the theory of analytic circle diffeomorphisms, Herman rotation numbers govern global linearization:

\begin{itemize}
    \item[\setword{{\bf(H)}}{cond:Herr}] Every analytic diffeomorphism 
    $f\colon \mathbb R/\mathbb Z\selfmap$ with rotation number $\theta$ is linearizable 
    if and only if $\theta$ is \emph{Herman}.
\end{itemize}

An arithmetic characterization of the Herman condition was given by Yoccoz in~\cite[\S2.5]{Yoc02}. 
Just like the Brjuno condition~\ref{cond:Br}, the Herman condition has implications beyond circle maps. In the setting of polynomials (and also rational maps), the following has been conjectured:

\begin{KeyConjecture}[Siegel disks are Jordan domains]\label{conj:globSiegelDisk}
In the setting of Theorem~\ref{thm:locSiegelDisk}, let 
$Z_f\subset {\mathbb C}$ be the maximal domain of linearization around $0$. 
If $Z_f\not=\emptyset$, then $Z_f$ is a Jordan domain: $\overline Z_f$ is a closed topological disk  in ${\mathbb C}$. Moreover, the boundary $\partial Z_f$ contains a critical point and $\overline Z_f=H_f$ if and only if $\theta$ is Herman.

If $\theta$ is not Herman, then $H_f$ is either a hairy Siegel disk or a Cremer Cantor bouquet of curves. 
\end{KeyConjecture}

\noindent In other words, as we move away from $0$, there is a well-defined last rotating curve  $\partial Z_f$ (cf.~the last KAM curve). In the Herman case, the linearization can break down only upon reaching a critical point. In the non-Herman case, the linearization must stop strictly before any critical point is reached. In this situation, the rotation on $\overline Z_f$ is expected to extend along a Cantor bouquet of arcs until a critical point is reached. These arcs, together with $\overline Z_f$, form the Mother Hedgehog $H_f$ stated in Conjecture~\ref{mainconj:MH}.

Over the last few decades, much of the effort in complex dynamics has been devoted to developing near-neutral renormalization theory for quadratic polynomials $p_\theta\colon z\mapsto e^{2\pi i \theta}z+z^2$; see, for instance, \cite{Sh98,McM1,IS, BC, AL-posmeas, CC15, Che18,  DLS,DL23,ABC,Chera25,ShY,DLL26}. The neutral quadratic renormalization is closely related to the quadratic-like theory; see the recent survey~\cite{D:ICM} for an overview of the renormalization theory of quadratic polynomials. 

The neutral quadratic family $\{p_\theta\}_{\theta\in \RR/\ZZ}$ parametrizes the boundary of the main cardioid of the Mandelbrot set and is responsible for some of the most delicate features of quadratic polynomials. By controlling the critical orbit that forms the boundary of Mother Hedgehog $H_{p_\theta}$, quadratic near-neutral renormalization revealed many subtleties related to Conjectures~\ref{mainconj:MH} and~\ref{conj:globSiegelDisk}. 
For instance, $H_{p_\theta}$ may \emph{not} be \emph{locally connected} -- one of the most prominent geometric irregularities linked to the parabolic type of near-neutral renormalization. Deep insight into the failure of local connectivity came from the near-Parabolic renormalization theory, building on the \emph{perturbative a priori} bounds established by Inou and Shishikura~\cite{IS} in the mid-2000s.

Conjectures~\ref{mainconj:MH} and~\ref{conj:globSiegelDisk} were confirmed for neutral quadratic polynomials with sufficiently high-type rotation numbers~\cite{Chera25,ShY}, where the Inou-Shishikura bounds~\cite{IS} are applicable.

In~\cite{DL:wZ} Lyubich and the first author introduced almost-invariant pseudo-Siegel disks $\wZ^m$ enveloping $H_{p_\theta}$ and hiding its geometric irregularities. By employing the Near-Degenerate regime, pseudo-Siegel \emph{a priori} bounds were established for all quadratics $p_\theta$ (i.e., appropriate precompactness of $\wZ^m$ and of the corresponding first return maps). With the development of pseudo-Siegel bounds, Conjecture~\ref{mainconj:MH} has recently been confirmed for all neutral quadratic polynomials~\cite{DL:wZ}. This also gave an opening to the quadratic case of Conjecture~\ref{conj:globSiegelDisk}, by reducing it to the comparability of the model in~\cite{Chera25} with the bounds in~\cite{DL:wZ}.

The main results of the paper justify Conjecture~\ref{mainconj:MH} in the multicritical setting of all Zakeri polynomials by establishing the associated pseudo-Siegel bounds, see Theorem~\ref{mainthm:A}. Using control of the critical orbits provided by this theorem, we then justify the optimality of the (local) Brjuno linearization condition for all Zakeri polynomials in Theorem~\ref{mainthm:C}, which leads to Theorem~\ref{thm:locSiegelDisk}.

\subsection{Notations and definitions}\label{ss:dfn_notat}
Let us denote by 
 \begin{itemize}
 \item $[a_0,a_1,a_2,\dots]$ the continued fraction expression of a number between $0$ and $1$:
\[[a_0,a_1,a_2,\dots]=\cfrac{1}{{a}_0+\cfrac{1}{{a}_1+\cfrac{1}{{a}_2+\cdots}}} \] 
 \item $\Theta=\{\theta=[a_0,a_1,a_2,\dots] \mid a_i \in \N_{\ge 1}\}\simeq  [0,1]\setminus \Q\simeq (\RR/\ZZ)\setminus \Q$ the set of irrational rotation numbers;
  \item $\pp_n/\qq_n \approx \theta$ be the best approximations of $\theta\in\Theta$ by rational numbers normalized so that $\qq_0=1$, see~\eqref{eq:p_n q_n}.
  \item $\ovlTheta =\{\theta= [a_0,a_1,a_2,\dots] \mid a_i \in \N_{\ge 1}\cup \{\infty\}\}\supset \Theta$ the formal compactification of $\Theta$;

 \item $\Theta_\bnd=\{\theta=[a_0,a_1,a_2,\dots] \mid a_i \le M_\theta\}\subset \Theta$ the set of bounded-type (irrational) rotation numbers; 
  \item $\Theta_\egm=\{\theta=[a_0,a_1,a_2,\dots] \mid a_i =1 \ \text{for }i\gg1\}\subset \Theta_\bnd$ the set of eventually golden-mean rotation numbers;
  \item for $\theta\in \Theta$ and $\mathfrak{g}\ge 2$,  the space of neutral polynomials with a fixed rotation number at $0$:
  \begin{equation}
    \Pol(\theta) =\Big\{ f\colon z\mapsto e^{2\pi i \theta}z +a_2z^2+\dots + a_{\mathfrak{g}-1}z^{\mathfrak{g}-1}+ z^{\mathfrak{g}}\Big\}.
\end{equation}
   \end{itemize}

 Every polynomial with an irrationally neutral periodic point is represented in some $\Pol(\theta)$, after passing to a finite iterate.

In the bounded-type case $\theta\in \Theta_\bnd$, every map $f\in \Pol(\theta)$ has a Siegel disk $Z_f\ni 0$ such that $\overline Z_f$ is a $K$-qc disk whose dilatation $K$ is bounded in terms of $\sup a_n$. Moreover, $\partial Z_f$ contains at least one critical point of $f$.

A polynomial $f\in \Pol(\theta)$ with $\mathfrak{g}\ge 3$ and $\theta\in\Theta_\bnd$ is called \emph{Zakeri} if $\partial Z_f$ contains all critical points of $f\colon \C\selfmap$. We denote by 
\begin{equation}
    \label{ed:dfn:Zak:bnd} \Zak (\theta)\coloneqq \big\{f\in \Pol(\theta)\ :\  \partial Z_f\supset \{\text{critical points of }f\colon \C\selfmap\}\big\}\ \subset\  \Pol(\theta)
\end{equation}
the set of Zakeri polynomials. If $\mathfrak{g}=3$, then $\Zak(\theta)$ is a simple arc within a one-dimensional parameter space $\Pol(\theta)$, see Figure~\ref{fig.siegelpara-tourne}. 
\begin{figure}[t]
\centering
\includegraphics[width=0.7\textwidth]{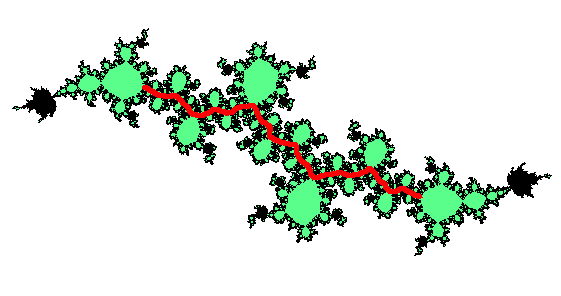} 
\caption{The parameter space $\cal{P}_3(\theta)$ with $\theta$ golden mean. The Zakeri locus $\mathrm{Zak}_3(\theta)$ is a Jordan arc, marked in red.}
\label{fig.siegelpara-tourne} 
\end{figure}

More generally, we define the set of Zakeri polynomials as 
\begin{equation}
    \label{eq:dfn:zak} \Zak \coloneqq \ \overline{\bigcup_{\theta\in \Theta_\bnd} \Zak(\theta)}.
\end{equation}

It will follow from our results (Theorem~\ref{mainthm:A}) that 
\begin{equation}
    \label{eq:prop:zak:theta} \Zak(\theta)=  \Zak \cap \Pol(\theta)\qquad \text{ if }\theta\in \Theta_\bnd.
\end{equation}
Therefore, we define
\begin{equation}
    \label{eq:dfn:zak:theta} \Zak(\theta) \coloneqq \Zak \cap \Pol(\theta)\qquad \text{ if }\theta\in \Theta\setminus \Theta_\bnd.
\end{equation}
For a rational rotation number $\phi\in \Q/\mathbb{Z}$, the definition of Zakeri set requires taking into account of parabolic enrichment, and is beyond the scope of this paper.

\subsection{Multicritical pseudo-Siegel bounds} \label{ss:intro:main.results} Consider $f_n \in \Zak(\Theta_\egm)$. Then every $f_n$ has a closed Siegel disk $\overline Z_{f_n}$ containing all of the critical points on its boundary. By passing to a subsequence, we can assume that $f_n\to f_{\infty}\in \Zak$. Since $\overline Z_{f_n}$ is $f_n$-forward invariant, by passing again to a subsequence, we can consider the Hausdorff limit  
\begin{equation}
\label{eq:lim.H_infty}
\big[f_n \colon \overline Z_{f_n}\selfmap\big]\ \rightarrow \  \big[f_\infty \colon H_{f_\infty}\selfmap\big],\qquad f_n\in \Zak(\Theta_\egm), \quad f_\infty\in \Zak .
\end{equation}
Then $H_{f_\infty}$ is $f_\infty$-forward invariant, containing all of the critical points of $f_\infty\colon\C\selfmap$.

Our first main result extends pseudo-Siegel bounds from quadratic polynomials~\cite{DL:wZ} to multicritical Zakeri-type polynomials. Theorem~\ref{mainthm:A} envelopes Siegel disks of maps in $\Zak(\Theta_\egm)$ by almost-invariant pseudo-Siegel disks $\wZ^m$ with uniform bounds. Consequently, approximation~\eqref{eq:lim.H_infty} is applicable to study $\Zak$ as the compactification $\Zak(\Theta_\egm)$. Without Theorem~\ref{mainthm:A}, it could happen that $H_{f_\infty}$ in~\eqref{eq:lim.H_infty} is the filled Julia set of $f_{\infty}$, this would make approximation~\eqref{eq:lim.H_infty} meaningless.

\begin{MainTheorem}[=Theorems ~\ref{thm.apriori-bound} and ~\ref{thm.reg.threshold}]\label{mainthm:A} For every $f \in \Zak(\theta)$ with $\theta\in\Theta_\egm$, there is a nested sequence of pseudo-Siegel disks
\begin{equation}
\label{eq:mainthm:A}
    \wZ^{-1}\supset \wZ^0\supset \dots \supset \wZ^m \supset \dots \supset \overline Z_f
\end{equation}
such that every $\wZ^m$ is almost invariant under $f^{\qq_{m+1}}$ and is explicitly constructed from $\overline Z$ by filling-in all parabolic fjords of level $\ge m$. All $\wZ^m$ have uniform bounds on renormalization levels $\ge m$.

In particular, $\wZ^{-1}$ is a uniformly qc closed disk so that $f\mid \wZ^{-1}$ is injective. Consequently, $\wZ^{-1}$ exists for all polynomials in $\Zak$ and Conjecture~\ref{mainconj:MH} holds for $\Zak(\Theta)$. 
\end{MainTheorem}

Theorem~\ref{thm.apriori-bound} states uniform bounds in the a priori setting with an explicit near-degenerate threshold $\bfK$. In particular, regularizations $\wZ^{m+1}\leadsto \wZ^{m}$ are constructed by selecting an outermost rectangle $\sqrt{\bfK}$, see~\eqref{eq:sqrt.bfK}. In the a posteriori setting, i.e., by letting $\bfK=O_{\bfK}(1)$, Theorem~\ref{thm.reg.threshold} derives an absolute combinatorial threshold for the regularization $\wZ^{m+1}\leadsto \wZ^{m}$. Consequently, $\partial \wZ^m$ has appropriate tilings so that $\wZ^{-1}$ is uniformly qc.
The last assertion of Theorem~\ref{mainthm:A} is immediate since the qc compactness of $\wZ^{-1}$ controls the Hausdorff limit in~\eqref{eq:lim.H_infty}.

\begin{figure}[t]
\begin{tikzpicture}
    \node at (0,0) {\includegraphics[width=15cm]{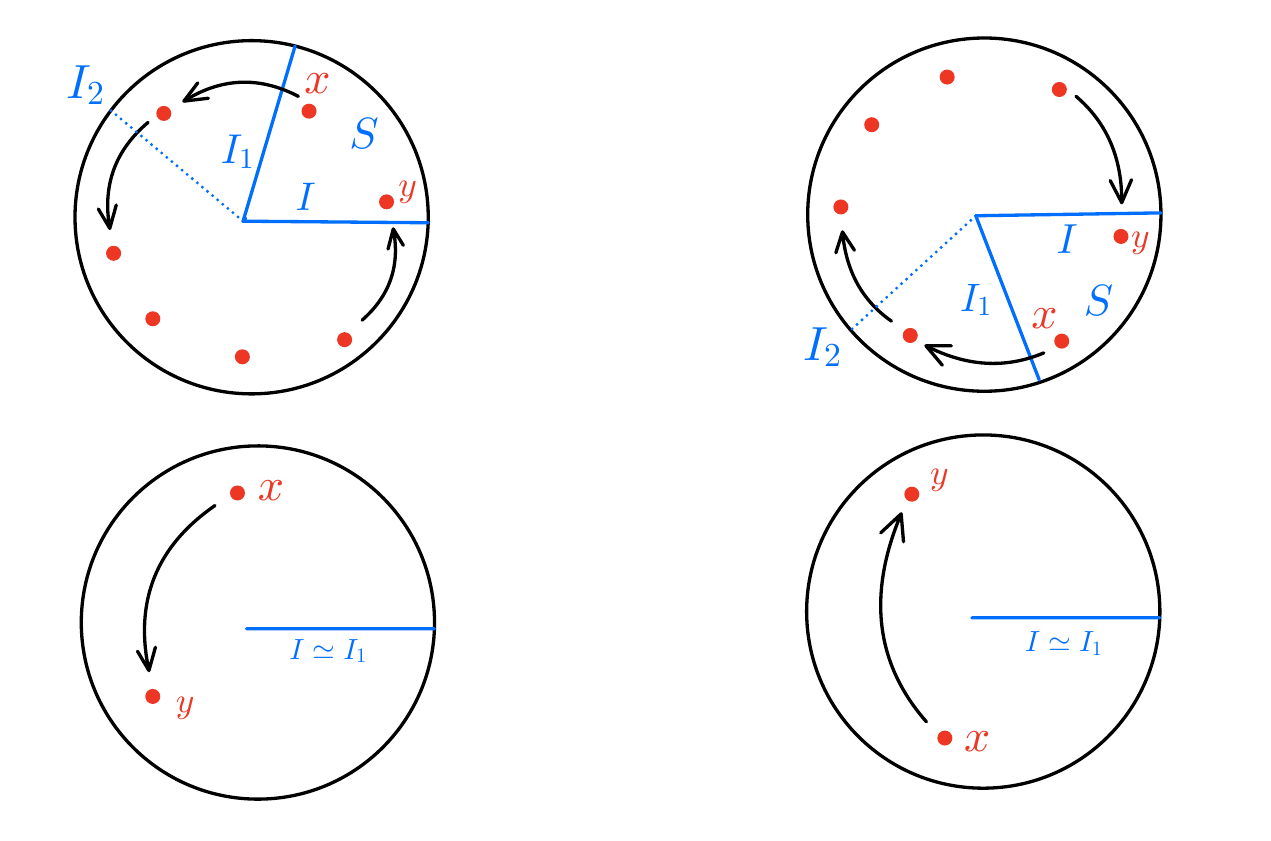}};
    \node at (-0.2,3.7) {The First Return Map };
    \node at (-0.2,3) {of $\rot_n$ to $S$ is};
    \node at (-0.2,2.3) {$ (\rot_n^{\bar a_n},\rot_n^{\bar a_n+1})$};
    \node at (-0.2,1.6) {$ \bar a_n\ge 2$};

    \node at (-0.2,0.5) {$\rot_{n+1}\coloneqq \RR_\sec \rot_n$};
    \node at (-0.2,-0.2) {$\simeq (\rot_n^{\bar a_n}\mid S,\ \rot_n^{\bar a_n+1}\mid S)/_{I\sim I_1}$};

    \node at(-5.1,-2) {$\rot_{n+1}$};
    \node at(3.4,-2) {$\rot_{n+1}$};

    \node at(-5.8,2.8) {$\rot_{n}$};

\node at(5.3,3.1) {$\rot_{n}$};

\end{tikzpicture}
   \caption{The Sector Renormalization $\rot_{n+1}=\RR_\sec \rot_n$ is the first return map of $\rot_n\colon \overline{\mathbb D}\selfmap$ to the smallest sector $S$ between $I$ and $I_1\coloneqq \rot_n(I)$, where $I\simeq I_1$ are glued by $\rot_n$. There are two combinatorial cases for $\rot_n\leadsto \rot_{n+1}$ depending on the orientation of $I, I_1, I_2$.}
   \label{Fig:sect.enorm}
\end{figure}

\subsection{Sectorial renorm-towers} \label{ss:intro.secto.tower}Finer applications of Theorem~\ref{mainthm:A} follow from developing the sectorial renormalization theory involving $\wZ^{m}$. In brief, pseudo-Siegel bounds provide geometric control near $\partial \wZ^{m}$, while sectorial renormalization extends this control (in the a posteriori setting) to the interior of $\wZ^m$.

Consider $f\in \Zak(\Theta_\egm)$. Its restriction to the closed Siegel disk $f\colon \overline Z_f\selfmap$ is conjugate via a Riemann map $\phi\colon \overline Z_f\to \overline {\mathbb D},\ \phi(0)=0$ to a rigid rotation of the unit disk:
\begin{equation}
    \label{eq:rig.tor} \rot_0\equiv\rot_{\theta} \colon z\mapsto e^{2\pi i \theta} z \colon \overline {\mathbb D}\selfmap, \qquad\qquad e^{2\pi i\theta}=f'(0).
 \end{equation}

The \emph{sector renormalization} $\rot_1=\R_\sec\rot_0\colon \overline {\mathbb D}\selfmap$ is the first return map of $\rot_0\colon \overline {\mathbb D}\selfmap$ to the smallest sector $S$ between $I$ and $I_1 := \rot_0(I)$, where $I\simeq I_1$ are glued by $\mathbb r_0$. Similarly are defined consecutive sectorial renormalizations $\rot_{n+1}=\R_\sec \rot_n$, see Figure~\ref{Fig:sect.enorm}. Since $f\colon \overline Z_f\selfmap$ is identified with $\rot\colon \overline{\mathbb D}\selfmap$, we construct sectorial renormalization for maps in $\Zak(\Theta_\egm)$:

\begin{equation}
\label{eq:Rsec.egm}
    {\mathcal R}^n_\sec\colon \big[f\colon \overline Z_f\selfmap \big] \ \mapsto \ \big[f_n\colon \overline Z_{f_n}\selfmap \big],\qquad f\in \Zak(\Theta_\egm),
\end{equation}
where $\overline Z_{f_n}$ is a compactified Riemann surface conformally equivalent to $\overline {\mathbb D}$. The following theorem asserts that the $f_n$ in ~\eqref{eq:Rsec.egm} admit extensions to a definite neighborhood around $\wZ^{-1}_{f_n}\supset \overline Z_{f_n}$ so that standard sectorial \emph{a priori} bounds hold. This defines a precompact set of renorm-towers $\bTow(\Theta_\egm)$ admitting a full compactification $\ovlTow$.

\begin{MainTheorem}[=Theorem~\ref{thm:sector.bounds}, Renorm-Attractor of $\Zak$]\label{mainthm:B} For $f\in\Zak(\Theta_\egm)$, the sector renormalization $f\mapsto f_n$ in~\eqref{eq:Rsec.egm} extends to a certain  neighborhood of $\wZ^{[n-1]}_{f}\supset \overline Z_f$ such that this neighborhood projects to a definite neighborhood of $\wZ^{-1}_{f_n}$. This constructs a sectorial renorm-tower 
\begin{equation}
\label{eq:1.mainthm.B}
\Tow (f) =(f_n\colon U_{f_n}\to V_{f_n})_{n\ge 0}\ \equiv \ (f_n\colon \mathbb D\dashrightarrow  \mathbb D)_{n\ge 0}, 
\end{equation}
where
\[  U_{f_n}\cup V_{f_n} \subset W_{f_n} \ \simeq \ \mathbb D, \qquad 0=f_n(0)\text{ is a Siegel fixed point}\]
satisfies uniform sectorial a priori bounds stated as Items~\ref{item:1:thm:sector.bounds} -- \ref{item:4:thm:sector.bounds} in Theorem~\ref{thm:sector.bounds}. 

The set of all renorm-towers 
\[\bTow(\Theta_\egm)  \coloneqq \{\Tow (f)\ : \ f\in \Zak(\Theta_\egm)\}\]
is precompact with respect to sequential limits and shifts $(f_n)_{n\ge 0} \mapsto (f_n)_{n\ge t}, t\ge 0$. The compactification of $\bTow(\Theta_\egm)$ and its shifts constructs a compact set of renorm-towers $\ovlTow$; see~\eqref{eq:Tow.egm}. All parabolic limits in $\ovlTow$ have multiplicity at most $\mathfrak{g}$.     
\end{MainTheorem}

The proof follows the quadratic case~\cite{DL:sect.bounds}, with appropriate adjustments to the multicritical setting.

\subsection{Small periodic points}\label{ss:outline.SmallPerPts} Consider a sectorial tower \[ \Tow(f)=(f_n\colon \mathbb D\dashrightarrow  \mathbb D)_{n\ge 0}\qquad\text{ as in~\eqref{eq:1.mainthm.B},\quad where $f\equiv f_0\in \Zak(\Theta_\egm)$.}\] 

Let 
\begin{itemize}
    \item $\theta_n\in \RR/\ZZ$ be the rotation number of $f_n$ at $0$, i.e., $f'_n(0)=e^{2\pi i \theta_n }$;
    \item $\wtheta_n =\dist_{\RR/\ZZ}(\theta_n, 0)\in (0,1/2)$ be the $\RR/\ZZ$-norm of $\theta_n$,
\end{itemize}

In \cite{Yoc}, Yoccoz introduced the following version of $\mathbb{B}(\theta)$,~\eqref{eq:Br:function}:
\begin{equation}\label{eq:Br:function:v2}
    \bfB_N(\theta) = \sum_{n=0}^{N} \wtheta_0\cdot \wtheta_1{\cdot} \dots{\cdot}   \wtheta_{n-1}\cdot (-\ln \wtheta_n),\qquad  \bfB(\theta)\equiv \bfB_\infty(\theta) \asymp \mathbb{B}(\theta).
\end{equation}
The summation~\eqref{eq:Br:function:v2} predicts the location of small periodic points as follows. For more details, we refer to the outline at the beginning of Section~\ref{s:anti.renorm.coord}.

For $n\ge 0$, let $\FPsec_n\not=0$ be the repelling fixed point of $f_n\colon \mathbb D\dashrightarrow \mathbb D$ closest to $0$. If no such $\FPsec_n$ exists, set $\FPsec_n\coloneqq 1$. The following estimate characterizes how close $\FPsec_n$ is to $0$: 
\begin{equation} \label{eq:intro.beta_n} - \ln |\FPsec_n| \ \asymp \ - \ln \wtheta_n +O(1), 
\end{equation} 
where the comparison constants in ``$\asymp$'' depends only on the degree $\mathfrak{g}\ge 3$ and is independent of the choice of $f_n$.  Estimate~\eqref{eq:intro.beta_n} is a consequence of Theorem~\ref{mainthm:B};  the estimate follows from the compactness of the set of all sectorial towers $\ovlTow$ and the fact that parabolic limits in $\ovlTow$ have bounded multiplicity.

The finite sum $\bfB_N(\theta)$ predicts the location of a lift of $\FPsec_{N}$ in the dynamical plane of the original $f_0$ under the antirenormalizations $f_N\leadsto f_{N-1}\leadsto \dots \leadsto f_0$. If $\bfB_N(\theta)$ is big, then the resulting lift of $\FPsec_N$ is a small periodic point in the dynamical plane of $f_0$. Small periodic points are \emph{obstructions to the linearization}.

In Section~\ref{s:anti.renorm.coord}, we will justify the existence of small periodic points predicted by the Brjuno functions. Here is a key estimate.

\begin{MainLemma}[Corollary \ref{cor.yoccoz.estimate}]
\label{MainLemma}
    There is an absolute $\mathbf{L}\gg 1$ depending only on the degree $\mathfrak{g}\ge 2$ such that the following holds. Assume that a point $z$ is sufficiently close to $0$ in the dynamical plane of $f_{n+1}$:
    \begin{equation} \label{eq:mL.C:cond}
      -\ln |z| \ \ge\  \mathbf{L} / \wtheta_n.
     \end{equation}
    Then $z$ admits a lift $\widetilde z$ to the dynamical plane of $f_n$ satisfying 
    \begin{equation}
    \label{eq:intro.main.estim}
        - \ln |\widetilde z| \ = \ \wtheta_n \cdot\big(-\ln |z|\big)\  +\  b \cdot\big(-\ln\wtheta_n\big) + O(1),\qquad \text{where $b\asymp 1$.}
    \end{equation}     
\end{MainLemma}

In other words, the antirenormalization near $0$ is \emph{combinatorially modeled} by 
\begin{equation}
\label{eq:intro:CombMod}
    \omega \ \mapsto \ \wtheta_n \omega + \big(-\ln \wtheta_n\big),\qquad\qquad \omega \approx  - \ln |\widetilde z|,\quad\text{ cf,~\eqref{eq:intro.main.estim}}.
\end{equation}
Iterates of~\eqref{eq:intro:CombMod} define the Brjuno summations~\eqref{eq:Br:function:v2}. 

We also note that Condition~\eqref{eq:mL.C:cond} arises naturally when estimating small periodic points. Indeed, if this condition is violated, so that $-\ln |z| = O( \mathbf{L} / \wtheta_n)$, then the term $\wtheta_n \cdot(-\ln |z|)$ in~\eqref{eq:intro.main.estim} would be $O(1)$, and the entire estimate~\eqref{eq:intro.main.estim} would reduce to the form~\eqref{eq:intro.beta_n}. In other words, if a periodic point is not sufficiently close to $0$ in the sense of~\eqref{eq:mL.C:cond}, then its lift would not be much closer to $0$ than a fixed point is.

Let us recall (see~\S\ref{ss:intro:BrjunoCondit}) that the sufficiency of the Brjuno condition for the linearization was established in \cite{Brjuno1971-1972, Russmann1967, Yoc} in the general setting of local analytic germs. Estimate~\eqref{eq:intro.main.estim} allows us to iteratively lift certain small fixed points of $f_n$ (characterized by~\eqref{eq:intro.beta_n}) to construct small periodic points of $f_0$. This is the key mechanism in the following theorem.

\begin{MainTheorem}[Theorem~\ref{thm.Brjuno.Zaktype}]\label{mainthm:C}
The following quantified version of Theorem~\ref{thm:locSiegelDisk} holds for all $f\in\Zak(\Theta_\egm)$. If $\rho_f$ is the conformal radius of the maximal linearization domain around $0$, then 
\begin{equation}
\label{eq:rho_f:mainthm.B}
-\ln \rho_f\asymp \mathbf{B}(\theta),\qquad\qquad \text{see~\eqref{eq:Br:function} }
\end{equation}
where the comparison is uniform over all $f\in \Zak(\Theta_\egm)$ with a fixed degree $\mathfrak g$.
\end{MainTheorem}
\begin{proof}[Proof of Theorem \ref{thm:locSiegelDisk} modulo Theorem \ref{mainthm:C}]
The Main Theorem in \cite{BCR26} asserts that for all bounded type $\theta$, one has
\begin{equation}\label{eq.mainBCR}
    \sup_{f\in\mathcal{P}_{\mathfrak{g}}(\theta)}\rho(f) = \sup_{f\in\mathrm{Zak}_{\mathfrak{g}}(\theta)}\rho(f).
\end{equation}
Let $\theta_*$ be a non-Brjuno number and $\theta_{*,n}$ be its EGM approximation. Suppose by contradiction that some $f_*\in \mathcal{P}_{\mathfrak{g}}(\theta_{*})$ is linearizable at $z=0$. Let $f_{*,n}\in \mathcal{P}_{\mathfrak{g}}(\theta_{*,n})$ be the algebraic approximation of $f_*$ ($f_{*,n}$ preserves all the coefficients in $f_*$ except for the multiplier). By Stability of rotational domain (see \cite[Lemma 17]{ABC}), $\liminf_n \rho(f_{*,n})\geq \rho(f_*)>0$. On the other hand, (\ref{eq:rho_f:mainthm.B}) of Theorem \ref{mainthm:C} and (\ref{eq.mainBCR}) implies that $\rho(f_{*,n})\preceq e^{-C_g\mathbf{B}(\theta_{*,n})}$ for some constant $C_g$. But $\lim_n\mathbf{B}(\theta_{*,n})= +\infty$ since $\bfB(\theta_*) = +\infty$. This is a contradiction.
\end{proof}

\begin{MainRemark}[Cubic Case of Theorem~\ref{thm:locSiegelDisk} and a refinement of \eqref{eq:rho_f:mainthm.B}]\label{rem:D}\emph{ The initial scope of this project, initiated in Fall 2024, was to justify the Brjuno condition for cubic polynomials and to establish \emph{a priori} bounds for studying the finer geometry (rigidity-type questions) of the cubic connectedness locus; see~\S\ref{ss:intro:comments} for more details on this direction. The cubic case of Theorem~\ref{thm:locSiegelDisk} follows from Theorem~\ref{mainthm:C}, the existence of the holomorphic motion in the complement of Zakeri loci (see Figure~\ref{fig.siegelpara-tourne}), and \cite[Lemma 15]{ABC}. In Summer 2025, during communication initiated by the third author and Arnaud Ch\'eritat, a strategy for reducing the full version of Theorem~\ref{thm:locSiegelDisk} to the Zakeri case (Theorem~\ref{mainthm:C}) was suggested. This reduction strategy, now implemented in~\cite{BCR26}, was in part inspired by an earlier work of Buff and Ch\'eritat in the 2000s, where they conjectured the following refinement of~\eqref{eq:rho_f:mainthm.B}, see~\cite{BC11,BCR26}:
\begin{equation}
    \label{eq:rho_f:mainrem.E} \frac{1}{\mathfrak g-1}\mathbf{B}(\theta)-O_{\mathfrak g}(1) \ \le\  -\ln \ \rho_f\ \le \ \mathbf{B}(\theta)+O_{\mathfrak g}(1) - \ln \dist(0, \mathrm{Crit}(f)).
\end{equation}
In~\eqref{eq:rho_f:mainrem.E}, the second upper estimate (i.e., on the minimal radius of a Siegel disk) follows from the work of Yoccoz. In the quadratic case $\mathfrak g=2$, the lower estimate (i.e., on the maximal radius of a Siegel disk) was justified in~\cite{BC11}. The ``worst-case'' scenario with the factor $\frac{1}{\mathfrak g-1}$ for~\eqref{eq:rho_f:mainrem.E} for any $\mathfrak g\ge 3$ was conjecturally discussed during the research talk~\cite{Cheri09}. We believe that~\eqref{eq:rho_f:mainrem.E} is now achievable by refining~\eqref{eq:rho_f:mainthm.B} (using Corollary~\ref{cor:parab_limits}) and by consequently estimating $b$ in~\eqref{eq:intro.main.estim}. The reduction argument of~\cite{BCR26} would then promote~\eqref{eq:rho_f:mainrem.E} from the Zakeri case to all neutral polynomials.}
\end{MainRemark}

\begin{figure}
    \centering
   \begin{center}
  \begin{tikzpicture}[
  term/.style={
    align=center,
    text width=43mm,
    minimum height=10mm,
    inner sep=1mm
  },
  block/.style={
    term,
    draw,
    line width=0.45pt,
    inner ysep=3mm
  },
  implication/.style={
  double,
  double distance=1pt,   
  line width=0.7pt,      
  -{Implies[length=3mm,width=4.5mm]}, 
  shorten <=0.5mm,
  shorten >=0.5mm
}
]
  \node[term] (construction-next) at (0,0)
    {Construction of $\widehat{Z}^{m+1}$};

  \node[term] (log-rule) at (5.35,0)
    {Log-Rule of $\widehat{Z}^{m+1}$\\and fjords};

  \node[term] (snake) at (11.7,0)
    {$(\text{Snake})_{m+1}$};

  \node[block] (lemmas) at (11.7,-2.8)
    {$(\text{Calibration})_{m+1}$\\[2mm]
     $(\text{Amplification})_{m+1}$\\[2mm]
     $(\text{Welding})_{m+1}$};

  \node[term] (bound) at (5.35,-2.8)
    {$(\text{A priori bound})_{m+1}$};

  \node[term] (construction-current) at (0,-2.8)
    {Construction of $\widehat{Z}^{m}$};

  \draw[implication]
    (construction-next.east)
    -- node[midway, above=1.5mm, font=\footnotesize]
       {\S\ref{s:wZ:dfn}}
    (log-rule.west);
  \draw[implication]
    (log-rule.east) -- node[midway, above=1.5mm, font=\footnotesize]
       {\S\ref{s:Snake+}} (snake.west);
  \draw[implication, shorten >=3mm]
  (snake.south) -- node[midway, right=1.5mm, font=\footnotesize]
       {\S\ref{s:wca}} (lemmas.north);
  \draw[implication, shorten <=3mm]
  (lemmas.west) -- node[midway, above=1.5mm, font=\footnotesize]
       {\cite[\S 10]{DL:wZ}} (bound.east);
  \draw[implication]
    (bound.west) -- node[midway, above=1.5mm, font=\footnotesize]
       {\S \ref{s:wZ:bounds}} (construction-current.east);
\end{tikzpicture}
\end{center}
    \caption{Road map for establishing \emph{a priori} bounds; see a summary in~\S\ref{ss:intro:prf:outline}.}
    \label{fig:placeholder}
\end{figure}
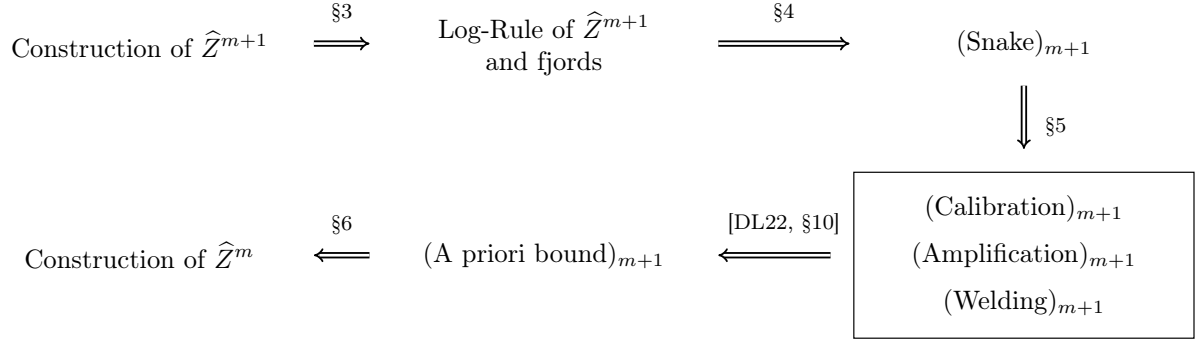
\subsection{Outline of the paper}\label{ss:intro:prf:outline}
The proof has two main parts. In \S\ref{sec:prep}  -- \S\ref{s:wZ:bounds}, we establish uniform multicritical pseudo-Siegel bounds and establish the first main result Theorem~\ref{mainthm:A}. In \S\ref{s:SectorRenorm} and \S\ref{s:anti.renorm.coord}, we construct sectorial renormalization towers (Theorem~\ref{mainthm:B}) and estimate the antirenormalization changes of variables near the neutral fixed point (Lemma~\ref{MainLemma}). Iterating these estimates produces small periodic points at scales prescribed by the Brjuno function, yielding Theorem~\ref{mainthm:C}.

We first describe the a priori setting; see Figure~\ref{fig:placeholder}. In \S\ref{sec:prep}, we introduce multicritical combinatorics and establish the Log-Rule in parabolic fjords, and in \S\ref{s:wZ:dfn}, we specify the geometric and dynamical assumptions on pseudo-Siegel disks in the multicritical case.
A key difference from the quadratic case is that a given level may contain several orbits of preparabolic intervals. The associated fjords are added successively, through intermediate pseudo-Siegel disks, and their dams cannot, a priori, be chosen at the same combinatorial depth; see Figure~\ref{fig.outerprotection}. Consequently, the tunnels used in the quadratic setting are unavailable. Instead, we formulate the Log-Rule in the shortcut metric. These conditions allow us to compare the inner and outer geometries while keeping track of the fjords already added.

In \S\ref{s:Snake+} and \S\ref{s:wca}, we develop the Snake, Welding, Calibration and Amplification lemmas in this multicritical setting. The Snake Lemmas convert sufficiently wide families submerging into a pseudo-Siegel disk into outer degeneration on a deeper scale. Together with the Log-Rule, they provide key tools for the inductive construction in \S\ref{s:wZ:bounds}. Once the a priori bounds are established, the a posteriori combinatorial threshold follows as in the quadratic case; see \S\ref{ss:a_posteriori:threshold}. In particular, the top pseudo-Siegel disk $\hZ^{-1}$ is uniformly quasiconformal. A fundamental difference from the quadratic case lies in the Localization Lemma~\ref{lem.localization}, which is a foundation for the Snake Induction Lemma \ref{lem.snake}: as multi-critical pseudo-Siegel disk might have fjords of drastically different depth, we will need to work with the shortcut metric on $\partial\hZ^{m}$ (see Definition \ref{def.shortcut-dist}); then we convert the results stated for the shortcut metric to the usual combinatorial metric on $\partial Z$ (see \S \ref{ss:ShortCut to Comb}).

In \S\ref{s:SectorRenorm}, we use these uniform bounds to construct sectorial renormalization towers satisfying Theorem~\ref{mainthm:B}.
A further difference in the multicritical case is the absence of the “No Spiraling Property” for the perturbed Fatou coordinates arising in the near-parabolic regime; cf.,~\cite[Item (3) in Lemma 17]{BC}. The parabolic geometry is more complicated when several critical orbits are involved; this is a key technical subtlety preventing us from directly employing various arguments of~\cite{DL:sect.bounds}. Therefore, at this stage, we justify the \emph{iterative antirenormalization changes} of variables only within a definite small neighborhood of $0$; see Lemma~\ref{lem.keyestimate} for details. This is sufficient for the estimates in \S\ref{s:anti.renorm.coord}: small fixed points at higher renormalization levels lift to periodic points near the original neutral fixed point, giving the conformal-radius estimate in \eqref{eq:rho_f:mainthm.B}, and hence Theorem~\ref{mainthm:C}. Combined with~{\cite[Main Theorem]{BCR26}}, this implies Theorem~\ref{thm:locSiegelDisk}. 

We expect that further developments will provide a finer geometric control for the multicritical renorm-attractor constructed by Theorem~\ref{mainthm:B}, see~\S\ref{sss:renorm.attr}, leading to various applications.

\subsection{Discussion and further developments}\label{ss:intro:comments} Theorem~\ref{mainthm:B} constructs (and establishes the compactness of) the multicritical renorm-attractor for sectorial renormalization of Zakeri polynomials. This initiates a multicritical (and non-perturbative) neutral renormalization theory in complex dynamics. It is natural to expect that the renorm-attractor constructed in Theorem~\ref{mainthm:B} in fact attracts sectorial renormalizations of all neutral rational maps. Moreover, the renorm-attractor is \emph{conjecturally rigid and hyperbolic}. We briefly discuss these problems in~\S\ref{sss:renorm.attr}. 

Rigidity-type properties of the multicritical renorm-attractor would imply that the Zakeri loci $\CubZak(\theta)$ illustrated in Figure~\ref{fig.siegelpara-tourne} are well-defined parameter Jordan arcs for all $\theta$ and, moreover, that these parameter Zakeri arcs depend continuously on $\theta$ in the ``parabolic-compactification'' of irrational rotation numbers. The hyperbolicity of the renorm-attractor may imply (less certain) that these parameter curves $\CubZak(\theta)$ are in fact uniformly qc. We briefly discuss these and related parameter rigidity questions in~\S\ref{sss:ParamRigid}.

Linearization Theorem~\ref{thm:locSiegelDisk} highlights a fundamental difference between the linearization properties of local and global analytic dynamics; see~\S\ref{sss:rem.natural} for comments regarding the nature of near-rotational dynamics.

\subsubsection{Rigidity and Hyperbolicity of the multicritical neutral renorm-attractor}\label{sss:renorm.attr} As mentioned in~\S\ref{ss:intro:prf:outline}, the iterative control of spiraling is one of the central subtleties emerging in the multicritical case. In short, under antirenormalization $f_{n+m}\leadsto f_n$, the pseudo-Siegel disk $\wZ^{-1}_{f_{n+m}}$ becomes $\wZ^{[m-1]}_{f_n}$;  see Item~\ref{item:1:thm:sector.bounds} of Theorem~\ref{thm:sector.bounds}. In the bounded-type case, i.e., when there is no regularization on the respective scales, we have $\wZ^{[m-1]}_{f_n}=\wZ^{-1}_{f_n}$, and the antirenormalization coordinates for $\wZ^{-1}_{f_{n+m}}\ \leadsto \ \wZ^{[m-1]}_{f_n}=\wZ^{-1}_{f_n}$ can be qc modeled by a power map $z\mapsto z^{1/D}$. This can be used to obtain the desired bound on spiraling. The opposite case is the near-parabolic one, in which $\wZ^{-1}_{f_n}$ is the welding of $\wZ^{0}_{f_n}$ with parabolic fjords; this is where~\cite{DL:sect.bounds} uses the ``No Spiraling Property'' for the perturbed Fatou coordinates; cf.,~\cite[Item (3) in Lemma 17]{BC}

We believe, with further developments, the desired bound on spiraling can also be obtained in the near-parabolic case; e.g., by analyzing the welding of $\wZ^{0}_{f_n}$ with parabolic fjords. This will make the arguments of~\cite{DLL26} available and should lead to the \emph{rigidity} of the multicritical renorm-attractor.

We anticipate that the \emph{hyperbolicity program} for the multicritical renorm-attractor is substantially more complicated than that for the quadratic case. A fundamental challenge is to deal with multidimensional unstable manifolds, particularly when $\mathfrak g\ge 4$. (By freezing the multiplier, the unstable manifold would have dimension $1$ if $\mathfrak g=3$.)

To justify the conjecture that the renorm-attractor from Theorem~\ref{mainthm:B} also attracts sectorial renormalizations of other neutral rational maps, one needs to establish \emph{a priori} bounds for the respective classes of rational maps. We believe that an appropriate combination of ~\cite{DL:HypComp} and the current paper will provide bounds for the ``Zakeri loci'' arising on the boundaries of any hyperbolic components of rational maps. The general case would require a substantial improvement in the present near-Degenerate Regime.

\subsubsection{Parameter Rigidity for Cubic Polynomials and beyond}\label{sss:ParamRigid} Observe that geometric bounds from Theorem~\ref{mainthm:A} are closed with respect to sequential convergences. If, in addition, cubic Zakeri polynomials $\CubZak$ are uniquely determined by combinatorics (rotation numbers and relative combinatorial positions of two critical points on the Mother Hedgehog) and such combinatorial data depends continuously on a map in $\CubZak$, then Zakeri loci $\CubZak(\theta)$ illustrated in Figure~\ref{fig.siegelpara-tourne} are parameter Jordan arcs depending continuously on $\theta$ in the parabolic-compactification $\overline \Theta$ of irrational rotation numbers $\Theta$; cf.,~\cite{DLL26}. The hyperbolicity of cubic neutral renormalization would allow one to study parameter zooms at Zakeri arcs and investigate their qc properties.

It is well-known that the Combinatorial Rigidity Conjecture for quadratic polynomials, which is equivalent to the MLC Conjecture (Local Connectivity of the Mandelbrot set), requires refinements in degrees $\ge 3$ due to capture mechanisms. Roughly, a critical point can travel through a non-recurrent, non-locally connected part of the Julia set; this would embed that non-locally connected part into the parameter plane. See a cubic example by Inou from his research talk~\cite{In:exm} for one of the most sophisticated mechanisms. 

To address this parameter behavior, the following folklore conjecture has been suggested: if two polynomials are topologically conjugate, then they are qc conjugate with no invariant line field on the Julia sets. In other words, polynomial Julia sets are Topologically Rigid. If this conjecture holds, then it is natural to expect that ``positions'' of critical points in non-recurrent, non-locally connected parts of the Julia set provide complete invariants for a refined combinatorial-type rigidity. Zakeri slices $\CubPol(\theta)$ with Cremer rotation numbers $\theta$ provide a natural setting to investigate the above-mentioned conjectures. Namely, one expects that $\CubPol(\theta)$ has a Zakeri curve $\CubZak(\theta)$, on which the polynomials are combinatorially rigid, and $\CubZak(\theta)$ is surrounded by an appropriately modified parameter copy of a non-locally connected quadratic Julia set; see Figure~\ref{fig.siegelpara-tourne} for an illustration in the bounded-type case. We refer to~\cite{BOST22,YZ24} where such structure was justified for bounded-type $\theta$, and we refer to~\cite{PT06,Ro07,BOPT14,Z24} for related studies of the cubic parameter space. Also, we do expect parameter self-bumps constructed in~\cite{L24} in degrees $\mathfrak g\ge 4$ to occur for the closure of the cubic main hyperbolic component, as it seems to be well-constrained by slicing it with $\CubPol(\theta)$.

Much more generally, one may ask if the Zakeri-type hierarchy always exists for appropriately specified parameter \emph{molecules} (joint domain of neutral and satellite renormalization) of any degree polynomials, while outside of parameter molecules, a quadratic-type Combinatorial Rigidity holds. The latter is motivated by an analogy with the Primitive Case of the MLC Conjecture, where classical quadratic-like \emph{a priori} bounds are expected to hold.

\subsubsection{On the Global and Local analytic rotational dynamics and their difference} \label{sss:rem.natural} 
As noted in~\S\ref{ss:intro:BrjunoCondit}, the Brjuno optimality condition does not hold for the linearizability of local germs. For example, for any irrational $\theta$, a Riemann map $\Psi\colon \mathbb D \to \C$ can be used to construct an ``artificial'' Siegel disk by conjugating the rigid rotation $z\mapsto e^{2\pi i\theta}z$ on $\mathbb D$ to holomorphic dynamics in $\C$. A more sophisticated example of an exotic Siegel disk was constructed in~\cite{Cher11} using Katok-Herman arguments from smooth dynamics. Roughly speaking, the construction uses vector fields to approximate rotational dynamics, producing a Siegel disk that is compactly contained in the domain of definition and whose boundary is a nonlocally connected indecomposable continuum. Rotation numbers in the construction are naturally non-Brjuno. This example clearly violates the local-germ analogues of Theorem~\ref{thm:locSiegelDisk} and Conjectures~\ref{mainconj:MH} and~\ref{conj:globSiegelDisk} and can be viewed as ``\emph{pathological}'' from the viewpoint of Polynomial Dynamics.

In this respect, one can say that \emph{global properties} of polynomials \emph{prevent} the above ``\emph{local pathologies}''. We emphasize that many of the near-degenerate arguments employed in this paper rely on the global analyticity of polynomials. We also note that ``semi-global'' analytic germs arising in renormalization (e.g., in Theorem~\ref{mainthm:B}) should be free of such pathologies due to the presence of a critical point on the boundary of the Mother Hedgehog.

An interesting test question is to know if a ``genuinely exotic'' Siegel disk can exist for a transcendental entire map $g\colon \C\selfmap$. In particular, whether the global properties of an entire map $g\colon \C\selfmap$, without any additional restriction on $g$, suffice to ensure the optimality of the Brjuno linearization condition and to rule out compact non-Jordan transcendental Siegel disks. (It is known that unbounded transcendental Siegel disks can have somewhat complicated access to infinity due to asymptotic passes, as in the case of exponential maps \cite{Rempe2004}. It is still natural to expect some rigidity-type structures in the exponential case.)

Invertible H\'enon maps in $\C^2$ represent an important frontier in understanding the nature of near-rotational dynamics. Strongly dissipative H\'enon maps can be viewed as invertible perturbations of 1-D polynomials; this allows one to employ the hyperbolicity of 1-D renormalization to study in detail certain strongly dissipative H\'enon maps of Siegel type; see, for instance,~\cite{GRY21,YY21}. Recently, using a completely different KAM-type approach, various rotation domains have been constructed by Krikorian~\cite{K25}; in particular, it resolves a long-standing conjecture by rigorously justifying the existence of Herman rings for conservative H\'enon maps, which were numerically observed by Ushiki; see pictures on~\cite{U}.

\subsection{Acknowledgments} We are grateful to Xavier Buff, Arnaud~Ch\'eritat, and Pascale Roesch for communications regarding~\cite{BCR26}. This led to the full version of Theorem~\ref{thm:locSiegelDisk}, rather than its Cubic Case alone; see Remark~\ref{rem:D}.

Theorem~\ref{mainthm:A} was announced in Fall 2025 at the Stony Brook Dynamics Seminar and also informally among colleagues.  

The first author was partially supported by the Simons Fellows grant and by the NSF grants DMS 2055532 and DMS 2246485. 
The second author was partially supported by NSF Grant DMS-2349929 and Simons Foundation.

\section{Preparation: multicritical combinatorics and parabolic fjords}\label{sec:prep}
In ~\S\ref{ss:not:MultiCrit}, we introduce combinatorial terminology for dynamics on Siegel disks, with multiple critical points on their boundary. The length $\l_n$ of a level $n$ combinatorial interval is uniquely determined by the rotation number $\theta$, see~\eqref{eq:dfn:l_n}. If the ratio $\frac{\l_n}{\l_{n+1}}$ is sufficiently big, then a parabolic fjord is developed on level $n$. The justification of this fact will be one of the key outcomes of the paper, see~\S\ref{ss:a_posteriori:threshold}.

After selecting a \emph{preferred} critical point $c_1$, the level $m$ (relative) diffeo-tilings $\mathfrak {D}_m$ are defined in the same way as in the quadratic case, see~\S\ref{sss:dfn:D_m}: the tiling $\mathfrak {D}_m$ is induced by the partial orbit of $c_1$ up to $\mathfrak q_{m+1}$ iterates. Each interval $I\in \mathfrak D_m$ has length either $\mathfrak{l}_m$ or $\mathfrak{l}_{m}+\mathfrak{l}_{m+1}$.

The partial orbits of \emph{all} critical orbits up to $\mathfrak q_{m+1}$ iterates form the full diffeo-tilings $\mathfrak{P}_m\subset\mathfrak{D}_m^{\mathrm{full}}$ and they decompose intervals of $\mathfrak D_m$ into preparabolic intervals $\mathfrak{P}_m$ and ``leftover intervals''. Every left over interval has length $<5\l_{m+1}$, see details in~\S\ref{sss:dfn:preparab}.

For every $k\le \mathfrak q_{m+1}$, we have a natural induced map $f^k_*\colon \mathfrak{P}_m\selfmap$, see~\eqref{eq:dfn:preparab}, on preparabolic intervals so that $f^k(I_j)$ is ``almost'' $I_{f^k_*(j)}$ for every $I_j\in \mathfrak{P}_m$.

In~\S\ref{sec.fjord}, we select a preparabolic interval $T\in \mathfrak{P}_m$ and establish the Log-Rule for the parabolic part $T_{\mathrm{par}}$ of $T$, see Theorem~\ref{thm.fjord}. The argument is similar to the quadratic case and relies on justifying the shifts of parabolic rectangles based on $T_{\mathrm{par}}$ under the first return map $f^{\mathfrak q_{m+1}}$.

As in the quadratic case, Lemma~\ref{lem.central-exists} states that either $T_{\mathrm{par}}$ supports a wide central rectangle after submerging into depth $\Width(\R)$ or there is an exponential boost of the degeneration~\eqref{eq:lem.central-exists}.

\subsection{Multicritical circle map}\label{ss:not:MultiCrit}
Let $\theta = [a_0, a_1,a_2,...]$ be the continued fraction expansion. 
Note that $\theta$ is of {\em bounded-type} if $a_n$ are bounded; and is {\em eventually-golden-mean} if $a_n = 1$ for all sufficiently large $n$.

Let $Z$ be a fixed Siegel disk for a rational map $f$ with a bounded-type rotation number $\theta$.
Let $h: Z \longrightarrow \mathbb{D}$ be a Riemann mapping with $h(\alpha) = 0$, where $\alpha$ is the fixed point in $Z$.
Note that $h$ extends continuously to
$$
h: {Z} \longrightarrow \overline{\mathbb{D}}, \, h(\alpha) = 0
$$
which conjugates $f|_{\partial Z}$ with the rigid rotation $R_{\theta}$ on $S^1$.
The {\em combinatorial length} of an interval $I \subseteq \partial Z$ is defined as
\begin{equation}\label{eq.combinatorial-distance}
    \vert I \vert := \vert h(I)\vert_{\RR/\ZZ} \in (0,1).
\end{equation}
Similarly, the {\em combinatorial distance} between $x, y \in \partial Z$ is defined as
$$
\dist(x,y):= \dist_{\RR/\ZZ}(h(x), h(y)) \in [0,1/2].
$$
Let $x \in \partial Z$ and $t\in \R/\Z$, we set
$$
x \boxplus t = h^{-1}(h(x) + t),
$$
i.e., $x \boxplus t$ is $x$ rotated by angle $t$.
Note that $f(x) = x \boxplus \theta$ for all $x\in \partial Z$.

Consider the sequence of best approximations of $\theta$
\begin{equation}
\label{eq:p_n q_n}
\frac{\mathfrak{p}_n}{\mathfrak{q}_n} := 
      [a_0,a_1,a_2,..., a_n],
\end{equation}
and set $\mathfrak{q}_0 :=1$. Then $f^{\mathfrak{q}_0}, f^{\mathfrak{q}_1},...$ is the sequence of combinatorially closest returns of $f|{\partial Z_f}$, and we denote
\begin{equation}
    \label{eq:dfn:l_n}
\l_n:= \dist(f^{\mathfrak{q}_n}(x), x)= \dist_{\RR/\ZZ}(0, \ \mathfrak q _n \theta), \qquad\; x\in \partial Z_f.
\end{equation}
By convention, we set $\l_{-1} = 1$.

\subsubsection{Relative diffeo-tiling $\mathfrak{D}_m$ and full diffeo-tiling $\mathfrak{D}_m^{\mathrm{full}}$}\label{sss:dfn:D_m}
Let $c_1,..., c_s$ be the list of critical points on the boundary $\partial Z$. We call $c_1$ the {\em (preferred) critical point}.
We denote by 
$$
\mathrm{CP}_m =  \bigcup_{k=0}^{\q_{m+1}-1}f^{-k}(c_1)\cap \partial Z.
$$ 
We define the {\em (preferred) relative diffeo-tiling} $\mathfrak{D}_m$ of level $m$ as the partition of $\partial Z$ induced by $\mathrm{CP}_m$.
Note that there are $\mathfrak{q}_{m+1}$ intervals in $\mathfrak{D}_m$, and each interval has length either $\mathfrak{l}_m$ or $\mathfrak{l}_m + \mathfrak{l}_{m+1}$.
We cyclically label intervals in $\mathfrak{D}_m$ by $(T_i)_{i=0}^{\mathfrak q_{m+1}-1}$.

Let $\mathrm{CP}_m^{\mathrm{full}}$ be the set of critical points of $f^{\mathfrak{q}_{m+1}}$ on $\partial Z$.
We define the {\em full diffeo-tiling} $\mathfrak{D}_m^{\mathrm{full}}$ as the partition of $\partial Z$ induced by $\mathrm{CP}_m^{\mathrm{full}}$.
Since $\mathrm{CP}_m \subseteq \mathrm{CP}_m^{\mathrm{full}}$, $\mathfrak{D}_m^{\mathrm{full}}$ is a refinement of $\mathfrak{D}_m$.

\subsubsection{Preparabolic intervals $\mathfrak{P}_m$} \label{sss:dfn:preparab} Let us define the notion of {\em pre-parabolic intervals} $\mathfrak{P}_m\subset\mathfrak{D}_m^{\mathrm{full}}$ together with the induced by $f^k$ the map $f^k_*\colon \mathfrak{P}_m\selfmap$.

Consider 
\[\mathfrak{D}_m^{\mathrm{full}, \ge 5} \coloneqq \{I\in \mathfrak{D}_m^{\mathrm{full}}\ :\ \vert I \vert \geq 5\mathfrak l_{m+1}\}.\] 
For every $I\in \mathfrak{D}_m^{\mathrm{full}, \ge 5} $ and every $k\le \qq_{m+1}$, we define $f_*^k(I) \in \mathfrak{D}_m^{\mathrm{full}}$ as the unique interval with
$$
\vert(f_*^k(I) \setminus f^k(I)) \cup (f^k(I) \setminus f_*^k(I))\vert \le 2 \l_{m+1}.
$$
In other words, $f_*^k(I)$ is the unique interval that has a small symmetric difference with $f^k(I)$.

Note that $f^{\q_{m+1}}_*(I) = I$.
We call 
\begin{equation}
\label{eq:dfn:orb:prepar}    \mathcal{O}_m(I):= \{f_*^k(I): 0\le k \le \q_{m+1}-1\}\subset \mathfrak{D}_m^{\mathrm{full}} 
\end{equation} the level $m$ (almost invariant) cycle of $I$.
Level $m$ {\em pre-parabolic intervals} are defined by taking the union over all $\mathcal{O}_m(I)$:
\begin{equation}
\label{eq:dfn:preparab}    
\mathfrak{P}_m :=\bigcup_{I \in \mathfrak{D}_m^{\mathrm{full}, \ge 5}} \mathcal{O}_m(I)\  \subseteq\  \mathfrak{D}_m^{\mathrm{full}}, \qquad\qquad f^{k}_*\colon \mathfrak{P}_m\selfmap,\quad k\le \qq_{m+1}
\end{equation}
We call any interval $J \in \mathfrak{P}_m$ a level $m$ {\em pre-parabolic interval}.
Note that each level $m$ pre-parabolic interval $J \in \mathfrak{P}_m$ has combinatorial length $\vert J\vert \geq 3\l_{m+1}$, and each $J \in \mathfrak{D}_m^{\mathrm{full}} \setminus \mathfrak{P}_m$ has combinatorial length $\vert J\vert < 5\l_{m+1}$.

We label the pre-parabolic intervals contained in $T_j \in \mathfrak{D}_m$ cyclically by 
$$
T_j \supseteq I_{j, 1} \cup I_{j,2},..., I_{j, s_m},
$$
where $I_{j,k} \in \mathfrak{P}_m$ and $s_m \leq s$.
With this convention, we have $f^a_*(I_{j,k}) = I_{j+\mathfrak p_{m+1}a, k}$.

\subsection{Parabolic Fjords}\label{sec.fjord}

Let $T = [v,w] \in \mathfrak{P}_m \subseteq \D_m^\mathrm{full}$ ($v<w$) be a pre-parabolic interval as in \S~\ref{sss:dfn:preparab}. By construction, $(v,w)$ contains no point in $\mathrm{CP}_m^{\mathrm{full}}$. Without loss of generality, let us assume that $v' :=f^{\q_{m+1}}(v)\in(v,w)$. Denote $T' :=[v',w]$.

In this section, we discuss the geometry of parabolic fjords. Let us fix a pre-parabolic interval $T=[v,w]\in \mathfrak{D}_m^\mathrm{full}$. We assume that $T$ consists of many level $m+1$ combinatorial intervals. \cite[Figure 7]{DLL} illustrates the geometry of wide families in the associated parabolic fjord. Roughly, we first select the outermost parabolic external rectangle of fixed but big width; we assume that such a rectangle exists and connects $[\tilde x, x]$ and $[y,\tilde y]$. Moreover, we can assume that this rectangle is balanced, see \S \ref{subsubsec:preparation}. Below the rectangle, the Log-Rule (see Theorem~\ref{thm.fjord}) is applicable.

We remark that the interval $[\tilde x,\tilde y]$ is \emph{a priori} specified geometrically. Therefore, the Log-Rule Theorem depends on the following two parameters: \[M=\frac{\dist(\tilde x, x)}{\l_{m+1}}=\frac{\dist(\tilde y, y)}{\l_{m+1}} \qquad \text{ and }\qquad N=\frac{\dist(\tilde y, w)}{\l_{m+1}}. \] 
If $N\gg_{\lambda}\ M$, then there is a non-external family emerging from $E\subset [\tilde y, w]$, see Corollary~\ref{cor.div-expboost}. This leads to an \emph{exponential boost} (an application of the Log-Rule, see~\S\ref{subsubsec:expboost}) stated in Lemma~\ref{lem.central-exists}. In the a posteriori setting, $[\tilde x, \tilde y]$, $N$, and $M$ can be assumed specified combinatorially, see~\S\ref{ss:a_posteriori:threshold}.

\subsubsection{Preparations}\label{subsubsec:preparation}

\begin{defn}[Parabolic rectangles]
    A rectangle $\R$ based on $T$ is called parabolic if
\begin{equation}\label{eq.def.parabo-rect}
\dist_T(\partialhL\R,\partialhR\R)\geq 6\min\{|\partialhL\R|,|\partialhR\R|\}+3\l_{m+1}
\end{equation} 
    A parabolic rectangle is called balanced if $|\partialhL\R|=|\partialhR\R|$.
\end{defn}

\begin{defn}[External and diving family]
    Define $\K_m := f^{-\q_{m+1}}({Z})$. The family $\F^+(I,J)$ is naturally decomposed into 
    $$\F^+(I,J) = \F^+_{\mathrm{ext,m}}(I,J)\sqcup\F^+_\mathrm{div,m}(I,J),$$
    where $\F^+_{\mathrm{ext,m}}(I,J)$ and  $\F^+_{\mathrm{div,m}}(I,J)$ consist, respectively, of curves that are disjoint from $\K_m$ and curves that intersect $\K_m$. The corresponding widths are denoted by $\W^+_{\mathrm{ext}},W^+_{\mathrm{div}}$ respectively.
\end{defn}

We say that a rectangle $\mathcal{R}$ is external (resp. diving) if the vertical lamination of $\mathcal{R}$ is contained in $\F^+_{\mathrm{ext,m}}(\partial^{h,0}\R,\partial^{h,1}\R)$ (resp. $\F^+_{\mathrm{div,m}}(\partial^{h,0}\R,\partial^{h,1}\R)$).

\begin{lem}[{\cite[Lemma 2.7]{DL:wZ}}]\label{lem.extdiv-decompo}
Let $I\subset T$ and $J\subset\partial Z$ be subintervals.
  \begin{enumerate}[label=\upshape(\Alph*)]
  \item\label{itemA.lem.extdiv-decompo} We have
  \begin{equation}
      \W^+(I,J) \geq \W^+_{\mathrm{ext},m}(I,J)+\W^+_{\mathrm{div},m}(I,J)-O(1).
  \end{equation}
  
  \item\label{itemB.lem.extdiv-decompo} There are disjoint laminations $\G_{\mathrm{ext}}\subset \F^+_{\mathrm{ext},m}(I,J)$, $\G_{\mathrm{div}}\subset \F^+_{\mathrm{div},m}(I,J)$ contained in $\C$ such that
  \begin{equation}\label{eq.ext-div}
      \W(\G_{\mathrm{ext}})\geq \W^+_{\mathrm{ext},m}(I,J)-O(1),\quad  \W(\G_{\mathrm{div}})\geq \W^+_{\mathrm{div},m}(I,J)-O(1).
  \end{equation}
  Moreover, $\G_{\mathrm{ext}}$ consists of two rectangles and every curve in $\G_{\mathrm{ext}}$ lands at $T$; $\G_{\mathrm{div}}$ consists of two or three rectangles.
  \end{enumerate}
\end{lem}

\begin{lem}[{\cite[Lemma 4.2]{DL:wZ}}]\label{lem.ext-nonwinding}
    Let $\R$ be an external parabolic rectangle based on $T$ with $\W(\R)>1$. Let $\Rnew$ be the rectangle obtained from $\R$ by removing the outermost $1$-buffer. Then $\Rnew$, $f^{\q_{m+1}}(\Rnew)$ are non-winding and $f^{\q_{m+1}}|\fillRnew$ is injective.
\end{lem}

\begin{cor}\label{cor.based-on-T'}
    Under the same hypothesis and notations of Lemma \ref{lem.ext-nonwinding}, $\Rnew$ is based on $T'$ after removing an outermost $3$-buffer.
\end{cor}
\begin{proof}
    For simplicity we may assume that $\Rnew$ is based on $T$. Let $\X$ be the outermost $3$-buffer of $\Rnew$. If $\partialhL \X\subset [v,v']$, then by (\ref{eq.def.parabo-rect}) and Lemma \ref{lem.ext-nonwinding}, $\X$ is linked to its univalent image $f^{\q_{m+1}}(\X)$. Thus $\W(\X)\leq 2$ by Lemma \ref{lemappen.link}, a contradiction.
\end{proof}

\begin{lem}[{\cite[Lemma 4.5]{DL:wZ}}]\label{lem.upper-logrule}
    Let $\R$ be a parabolic non-winding rectangle based on $T'$ with $\W(\R)>2$. Then $|\partialhL\R|\geq\l_{m+1}$. If $\R$ is based on $T'$, then
    \begin{equation}
        \W(\R)\preceq \log\frac{|\partialhL\R|}{\dist(v,\partialhL\R)}+1.
    \end{equation}
\end{lem}

\begin{lem}[{\cite[Lemma 4.10]{DL:wZ}}]\label{lem.balanced}
    Let $\R$ be a parabolic non-winding rectangle based on $T'$. Let $\R'$ be the subrectangle after removing an outermost $O(1)$-buffer. If $\W(\R)\gg1$, then $\R$ contains a parabolic non-winding balanced geodesic subrectangle $\Rnew$ with $\W(\Rnew)=\W(\R)-O(1)$. Moreover, $\partialhL\Rnew$ and $\partialhR\Rnew$ are the outermost subintervals of $\partialhL\R'$ and $\partialhR\R'$ respectively.
\end{lem}

\subsubsection{Push-forwards in fjords}\label{subsec.pushforward}

Let us select a simple arc $\delta \subset \widehat{\C}$ connecting a point $\delta(0) \in \partial Z \setminus (v,w)$ to $\delta(1) = \infty$ such that $\delta$ is disjoint from $\K_m\setminus \delta(0)$. Then ${\BDelta} := \widehat{C}\setminus({Z}\cup\delta)$ is an open topological disk. For a rectangle $\R$ based on $T$, we will define the {\it push-forward} $\R_k$ of $\R$ under $f^{\q_{m+1}k}$ assuming that $\dist(\partialhR\R,w)>k\l_{m+1}$. The resulting $\R_k$ will be a lamination in ${\BDelta}$.

Let us orient all vertical curves in $\R$ from $\partialhR\R$ to $\partialhL\R$ (not from $\partialhL\R$ to $\partialhR\R$!):
$$ \gamma(0) \in \partialhR\R,\,\,\,\, \gamma(1) \in \partialhL\R \quad\text{ for }\quad [\gamma: [0,1] \longrightarrow \C] \in \R.$$
 Let ${\BDelta}_{-k}$ be the component of $f^{-k\q_{m+1}}({\BDelta})$ attached to $[v,w \boxminus k\theta_{m+1}] \subset T$. For a vertical curve
 $\gamma: [0,1] \longrightarrow \C$ in $\R$, let $t^\gamma_k$
 $(k > 0)$ be the first moment such that $\gamma(t^\gamma_k)\in\partial{\BDelta}_{-k}$. We define
$$ \R_k' := \{\gamma|_{[0,t^\gamma_k]}|\, \gamma \in\R\},\quad \R_k := f^{\q_{m+1}k}(\R_k').$$
In other words $R_k'$ is the restriction of $\R$ to ${\BDelta}_{-k}$. Notice that $f^{\q_{m+1}k}$ sends ${\BDelta}_{-k}$ injectively onto ${\BDelta}$. Hence $\R_k$ is also a lamination of curves.

We say that the curve $\gamma|_{[0,t^\gamma_k]}$ in $\R_k'$ and its image $f^{\q_{m+1}k}(\gamma|_{[0,t^\gamma_k]})$ in $\R_k$ is of
\begin{itemize}
    \item Type $\RomA$ if $t^\gamma_k =1$;
    \item Type $\RomB$ if $t^\gamma_k <1$ but $f^{\q_{m+1}k}(\gamma(t^\gamma_k))\in T$;
    \item Type $\RomC$ otherwise.
\end{itemize}
  We denote by $\R^{\RomA}_k,\R^{\RomB}_k$ and $\R^{\RomC}_k$ the sublaminations of the three types in $\R_k$ and $\R^{'\RomA}_k,\R^{'\RomB}_k$ and $\R^{'\RomC}_k$ the sublaminations in $\R_k'$. Since $\R_k'$ overflows $\R$, we have 
  $$\W(\R)\leq\W(\R_k') = \W(\R_k).$$

\subsubsection{Log-Rule in parabolic fjords}

\begin{thm}\label{thm.fjord}
    Assume that there is an external parabolic rectangle $\Z$ $\mathrm{rel}$ $\K_m$ based on $T'$ with $\W(\Z)\gg 1$. Then there is a subinterval
    $$T_{\mathrm{par}} = [x,y]\subset T',\quad x\in\lfloor v',\partialhL\Z\rfloor,\quad y\in\lfloor\partialhR\Z, w\rfloor$$
    with the following properties.
     \begin{enumerate}[label=\upshape(\roman*)]
    \item\label{thm.fjord:first} If $\R$ is any external parabolic rectangle based on $T$ with $\W(\R)\gg1$, then $\R$ contains a balanced parabolic subrectangle
    $$\Rnew\subset\R,\quad \partial^h\Rnew\subset T_{\mathrm{par}},\quad \W(\Rnew) = \W(\R)-O(1)$$
    such that $x<\partialhL\Rnew<\partialhR\Rnew<y$, $\dist(x,\partialhL\Rnew) = \dist(\partialhR\Rnew,y)$ and
\begin{equation}
    \W(\Rnew)\asymp\log\frac{|\partialhL\Rnew|}{\dist(\partialhL\Rnew,v)} = \log\frac{|\partialhR\Rnew|}{\dist(\partialhL\Rnew,v)}
\end{equation}

    \item\label{thm.fjord:second} If $I,J\subset T_{\mathrm{par}}$, $x<I<J<y$ are two intervals with 
    $$\dist(x,I)\asymp\dist(J,y)\preceq|I|\asymp|J|\preceq\dist(I,J)$$
    and $|I|,|J|,\dist(I,J)\geq\l_{m+1}$, then
    \begin{equation}
        \W^+(I,J)-O(1) = \W^+_{\mathrm{ext},m}(I,J)\asymp\log^+\frac{\min\{|I|,|J|\}}{\dist(v,I)}+1
    \end{equation}
    \item\label{thm.fjord:third} If $I,J\subset T_{\mathrm{par}}$, $x<I<J<y$ are two intervals with 
    $$|\lfloor I,J\rfloor|<\frac{1}{2}|T_{\mathrm{par}}|\text{  and  }\min\{|I|,|J|\}\succeq\dist(I,J)>3\l_{m+1}$$
    and $|I|,|J|\geq\l_{m+1}$, then
    \begin{equation}
        \W^+(I,J)-O(1) = \W^+_{\mathrm{ext},m}(I,J)\asymp\log^+\frac{\min\{|I|,|J|\}}{\dist(I,J)}+1.
    \end{equation}
  \end{enumerate}
\end{thm}
\begin{proof}
    For $\zeta\in T'$ with $\dist(v',\zeta)<|T'|/10$, define $\zeta_{10}\in T'$ so that $\dist(v',\zeta_{10}) = 10\cdot\dist(v',\zeta)$. By Corollary \ref{cor.based-on-T'}, we may assume there exists a sufficiently wide external parabolic rectangle $\Z$ based on $T'$. By Lemma \ref{lem.balanced}, we may assume that $\Z$ is balanced. Let $\B\subset\Z$ be the innermost subrectangle such that 
    \begin{equation}\label{eq.innerbuffer}
        |\partialhL\B| = |\partialhR\B| = \frac{5}{13}|\partialhL\Z|.
    \end{equation}
    Let $k$ be the smallest integer such that $k\l_{m+1}\geq|\partialhL\B|$. Then the pullback of $\B$ under $f^{k\q_{m+1}}$ along $\partial Z$ will cross $\B$. Thus $\W(\B)\leq 1$. Thus by removing $\B$ from $\Z$ and applying again Lemma \ref{lem.balanced}, we may assume that $\Z$ is balanced. Moreover by (\ref{eq.def.parabo-rect}) and (\ref{eq.innerbuffer}), we have $\dist(\partialhL\Z,\partialhR\Z)\geq 11|\partialhL\Z|$.
     By Lemma \ref{lem.upper-logrule}, $|
    \partialhL\Z|$ is large comparing to $\dist(v,\partialhL\Z)$. Hence we may assume furthermore that 
    \begin{equation}\label{eq.distpartialZ-geq-10partialZ}
    \dist(\partialhL\Z,\partialhR\Z)\geq 10(|\partialhL\Z|+\dist(v,\partialhL\Z)).
    \end{equation}
    Let us set $C:=\W(\Z)\gg1$. The existence of such $\Z$ implies that we can define the shortest interval
     \begin{equation}
         S:=[v',\zeta]\subset T \text{  such that  }\W^+_{\mathrm{ext}}(S,[\zeta_{10},w]) = C.
     \end{equation}
     By Lemma \ref{lem.extdiv-decompo} \ref{itemB.lem.extdiv-decompo}, Lemma \ref{lem.ext-nonwinding} and \ref{lem.balanced}, we can select two disjoint parabolic balanced non-winding geodesic rectangles based on $T'$ satisfying
     \begin{equation}\label{eq.W(XY)}
         \X,\Y\subset\F^+_{\mathrm{ext}}(S,[\zeta_{10},w]),\quad \X\subset\fillY\setminus\Y,\quad \W(\X)= C-O(1),\,\,\W(\Y)\geq 10.
     \end{equation}
     In particular we get $|S|>\l_{m+1}$ by Lemma \ref{lem.upper-logrule}. Set
     \begin{equation}
         \tilde{T}_{\mathrm{par}} =[\tilde{x},\tilde{y}] := \lfloor\partial^h\X\rfloor\subset T',\quad T_{\mathrm{par}} = [x,y] := \overline{\tilde{T}_{\mathrm{par}}\setminus\partial^h\X}.
     \end{equation}
     Notice that every vertical curve in $\F^+_{\mathrm{ext}}(S,[\tilde{y},w])$ either crosses the outermost $1$-buffer of $\X$, or is contained in the union of this buffer and $\R^+_{\mathrm{ext}}(S,[\zeta_{10},w])\setminus\X$ ($\R^+_{\mathrm{ext}}(S,[\zeta_{10},w])$ is the canonical rectangle of $\F^+_{\mathrm{ext}}(S,[\zeta_{10},w])$), we have
     \begin{equation}\label{eq.ext-102}
         \W^+_{\mathrm{ext}}([v',x],[\tilde{y},w])\leq \W^+_{\mathrm{ext}}(S,[\tilde{y},w])=O(1).
     \end{equation}

Therefore by (\ref{eq.distpartialZ-geq-10partialZ}), we have $x\in\lfloor v',\partialhL\Z\rfloor$ and $y\in\lfloor\partialhR\Z, w\rfloor$. By Lemma \ref{lem.upper-logrule}, $|\partialhL\X|\succeq \dist(v',\partialhL\X)$. Thus
     \begin{equation}
    \dist(v',z)\asymp\dist(\tilde{x},z) \text{ for all }z\in T_{\mathrm{par}}
     \end{equation}
     
    \begin{claim}\label{claim1.thm.fjord}
      For an interval $I\subset T_{\mathrm{par}}$ with
      $$\dist(I,\{x,y\})\succeq |I|,\quad |I|\geq \l_{m+1},$$
      we have $\W^+_3(I)\preceq 1$.
    \end{claim}
  
     \begin{claim}\label{claim2.thm.fjord}
      For intervals $I,J\subset T_{\mathrm{par}}$, $I<J$  satisfying      
      $$|\lfloor I,J\rfloor|<\frac{1}{2}|T_{\mathrm{par}}|,\quad
      \min\{|I|,|J|\}\asymp \dist(I,J),\quad|I|,|J|,\dist(I,J)\geq \l_{m+1},$$
      we have $\W^+(I,J)\asymp 1$.
    \end{claim}
 Now \ref{thm.fjord:third} follows from Splitting argument (split $|I|$ or $|J|$) and Claim \ref{claim2.thm.fjord}.

 \begin{claim}\label{claim3.thm.fjord}
      For intervals $I,J\subset T_{\mathrm{par}}$, $I<J$  satisfying      
      $$\dist(\tilde{x},I)\asymp\dist(\tilde{y},J)\asymp|I|\asymp|J|\preceq\dist(I,J),\quad |I|,|J|,\dist(I,J)\geq\l_{m+1},$$
      we have $\W^+(I,J)\asymp 1$.
    \end{claim}
 Now \ref{thm.fjord:second} follows from Splitting argument (split $|I|$) and Claim \ref{claim3.thm.fjord}.

\begin{claim}\label{claim4.thm.fjord}
    Assume that there is an $s\in\mathbb{N}$ such that $\dist(y,\tilde{y})<s\l_{m+1}<\dist(\tilde{y},w)-\l_{m+1}$. Define $\Gamma_s := f^{s\q_{m+1}}([y,\tilde{y}])$. Then 
    \begin{equation}\label{eq.div-greater-100}
        \W^+_{\mathrm{div}}(\Gamma_s,\partial Z\setminus[x,w]) =C-O(1)\gg1.
    \end{equation}
\end{claim}

Now we prove \ref{thm.fjord:first}. By Corollary \ref{cor.based-on-T'}, we may assume that $\R$ is based on $T'$. By Lemma \ref{lem.balanced}, there exists a balanced $\R_1\subset \R$ such that $\W(\R_1) = \W(\R)-O(1)$. 
Now \ref{thm.fjord:first} follows from the above claims.

To complete the proof of the theorem, let us prove the above four claims. Claims \ref{claim1.thm.fjord}, \ref{claim2.thm.fjord} and \ref{claim3.thm.fjord} can be deduced easily from splitting argument.
 \begin{proof}[Proof of Claim \ref{claim4.thm.fjord}] Let $\X_s$ be the push-forward of $\X$ under $f^{s\q_{m+1}}$ as in \S \ref{subsec.pushforward}. By
 construction, $\X_s$ is a lamination containing three types of curves. Moreover, every curve in $\X_s$
 emerges from $\Gamma_s\subset[\tilde{y},w]$ and lands at $(\partial Z\setminus[x\boxplus\theta_{m+1}s,w])\cup\delta$ (recall that $\delta\subset\widehat{\C}$ is a curve connecting $\partial Z\setminus T$ to $\infty$ and intersects $\K_m$ only at $\delta(0)$). Let $\phi:\widehat{\C}\setminus\K_m\longrightarrow\widehat{\C}\setminus\overline{\mathbb{D}}$ be the uniformization mapping. Let us choose $\delta$ such that 
 
 \begin{equation}\label{eq.choice-delta}
     \phi(\delta)\text{ is the straight line starting from the midpoint of }  \phi(\partial\K_m\setminus T).
 \end{equation} We will show below that 
 \begin{equation}\label{eq.Xs-less300}
     \W(\X_s\setminus\F^+_{\mathrm{div},m}(\Gamma_s, \partial Z\setminus[x,w])) =O(1);
 \end{equation} 
 this will imply that $\W(\X_s \cap \F^+_{\mathrm{div},m}(\Gamma_s, \partial Z\setminus[x,w])) = C-O(1)$ (since $\W(\X_s)\geq\W(\X) =C-O(1)$); i.e., (\ref{eq.div-greater-100}) holds. Consider two cases:
 \begin{itemize}
     \item If $y>x\boxplus\theta_{m+1}s$, then 
     $$ \partial Z\setminus[x\boxplus\theta_{m+1}s,w] \subset (\partial Z\setminus T)\cup[v,v']\cup[v',x]\cup[x,y].$$
     The required statement (\ref{eq.Xs-less300}) follows from the following estimates:
     \begin{itemize}
         \item $\W^+_{\mathrm{ext},m}(\Gamma_s,\partial Z\setminus T) = O(1)$ by Lemma \ref{lem.extdiv-decompo} \ref{itemB.lem.extdiv-decompo};
         \item $\W^+_{\mathrm{ext},m}(\Gamma_s,[v',x]) = O(1)$ by (\ref{eq.ext-102});
         \item $\W^+_{\mathrm{ext},m}(\Gamma_s,[v,v']) =O(1)$ by Lemma \ref{lem.ext-nonwinding} and (\ref{eq.ext-102}) (after removing the outermost $1$-buffer of $\F^+_{\mathrm{ext},m}(\Gamma_s,[v,v'])$, one can push forward the new family by $f^{\q_{m+1}}$);
         \item $\W^+_{\mathrm{ext},m}(\Gamma_s,[x,y]) \leq 1/\W(\X) < 1$, since these curves must cross $\X$;
         \item the curves in $\X_s$ intersecting $\delta$ have width at most $O(1)$ by the choice of $\delta$ (\ref{eq.choice-delta}).
     \end{itemize}
     \item If $y<x\boxplus\theta_{m+1}s$, then 
     $$ \partial Z\setminus[x\boxplus\theta_{m+1}s,w] \subset (\partial Z\setminus T)\cup[v,v']\cup[v',x]\cup[x,y]\cup [y,x\boxplus\theta_{m+1}s].$$
     The required statement (\ref{eq.Xs-less300}) follows from the above five estimates and an additional estimate below:
     \begin{itemize}
         \item since $\X$ is parabolic (Equation \eqref{eq.def.parabo-rect}), $\dist(x\boxplus\theta_{m+1}s,\Gamma_s)>|\Gamma_s|$. Let $k>0$ be the smallest integer such that $\theta_{m+1}k>|[y,\tilde{y}]| =|\Gamma_s|$. Then one can pull back $\F^+(\Gamma_s,[y,x\boxplus\theta_{m+1}s])$ by $f^{\q_{m+1}k}$ and the resulting family will link to the original family. Hence $\W^+(\Gamma_s,[y,x\boxplus\theta_{m+1}s])\leq 2$ by Lemma \ref{lemappen.link}. 
     \end{itemize}
 \end{itemize}
    \end{proof}
\end{proof}
\begin{cor}\label{cor.div-expboost}
    For every $\lambda \gg 1$, the following holds. If $\dist(y,w)\gg_\lambda\dist(v,x)$, then there is an interval $J \subset [y,w]$ such that
 $$\W^+_{\lambda,\mathrm{div},m}(J) \succeq_\lambda
\frac{\dist(y,w)}{\dist(v,x)}.
 $$
\end{cor}
\begin{proof}
Since $\dist(y,w)\gg_\lambda\dist(v,x)$, one can find $n\asymp\frac{1}{\lambda}\cdot\frac{\dist(y,w)}{\dist(v,x)}$ disjoint intervals $\Gamma_{s_i}$ ($\Gamma_{s_1}<\Gamma_{s_2}<...<\Gamma_{s_n}$) defined in Claim \ref{claim4.thm.fjord}, such that (\ref{eq.div-greater-100}) holds for each $\Gamma_{s_i}$ and 
    $$J :=\lfloor\Gamma_{s_1},\Gamma_{s_n}\rfloor,\quad \lambda J\subset [y,w].$$
    By definition of $J$, we have $$\bigcup_i\F^+_{\mathrm{div}}(\Gamma_{s_i},\partial Z\setminus[x,w])\subset\F^+_{\lambda,\mathrm{div},m}(J).$$ 
    By Lemma \ref{lem.extdiv-decompo} \ref{itemB.lem.extdiv-decompo}, for each $i$, the width of $\F^+_{\mathrm{div}}(\Gamma_{s_i},\partial Z\setminus[x,w])$ is concentrated on at most three rectangles. Therefore we can pick one of the three rectangles $\R_i$ such that $\W(\R_i)\geq\frac{1}{3}\W^+_{\mathrm{div}}(\Gamma_{s_i},\partial Z\setminus[x,w])\gg1$. These rectangles have small overlaps (they are pairwisely disjoint after removing $O(1)$-buffers). Therefore $\W^+_{\lambda,\mathrm{div},m}(J) \succeq n$, which completes the proof.
\end{proof}

\subsubsection{Central rectangles and exponential boost}\label{subsubsec:expboost}
\begin{defn}
A parabolic rectangle $\R$ based on $T'$ is {\it central} if 
$$0.9<\frac{\dist(v,\partialhL\R)}{\dist(\partialhR\R,w)}<1.1.$$
\end{defn}

\begin{lem}[Exponential boost if $T_{\mathrm{par}}$ has no central subrectangle]\label{lem.central-exists}
    Let $\R$ be a parabolic non-winding rectangle based on $T'$ with $K:=\W(\R)\gg_\lambda1$. Then
    \begin{itemize}
        \item either $\R$ contains a parabolic non-winding central balanced geodesic subrectangle $\Rnew$ with $\W(\Rnew)\geq\W(\R)/2$;
        \item or there exists an interval $I\subset T'$ with $|I|>\l_{m+1}$, such that 
       \begin{equation}
          \label{eq:lem.central-exists}
             \log\W^+_{\lambda,\mathrm{div},m}(I)\succeq_\lambda\W(\R).
       \end{equation}
    \end{itemize}
\end{lem}
\begin{proof}
    By Theorem \ref{thm.fjord} \ref{thm.fjord:first}, we may suppose that $\R$ is balanced and based on $T_{\mathrm{par}} = [x,y]$ with 
    $$\dist(x,\partialhL\R) = \dist(\partialhR\R,y).$$
    Let $\Rnew$ be the subrectangle of $\R$ by removing the outermost $K/3$-buffer $\mathcal{N}$ of $\R$. Using the same argument as Lemma \ref{lem.balanced}, we may assume that $\mathcal{N}$ is balanced (still being an outermost buffer of $\R$): $|\partialhL\mathcal{N}| = |\partialhR\mathcal{N}|$. Therefore we also have
    $$ \dist(x,\partialhL\Rnew) = \dist(\partialhR\Rnew,y).$$
    By Lemma \ref{lem.balanced}, we may suppose that $\Rnew$ is geodesic and balanced ($\partial^h\Rnew$ still intersects $\partial^h\mathcal{N}$ at exactly two points). Notice that
    \begin{equation}\label{eq.dist(y,w)}
        \dist(y,w) = \dist(\partialhR\Rnew,w)-\dist(\partialhR\Rnew,y) = \dist(\partialhR\Rnew,w)-\dist(x,\partialhL\Rnew)
    \end{equation}
    \begin{equation}\label{eq.dist(v,x)}
        \dist(x,\partialhL\Rnew) \leq  \dist(v,\partialhL\Rnew).
    \end{equation}
    Suppose that $\Rnew$ is not central. Consider two cases:
    \begin{itemize}
        \item if $\dist(v,\partialhL\Rnew)\leq0.9\cdot \dist(\partialhR\Rnew,w)$, then (\ref{eq.dist(y,w)}) and (\ref{eq.dist(v,x)}) imply that
      $$\dist(y,w) \succeq \dist(v,\partialhL\Rnew)\geq|\partialhL\mathcal{N}|.$$
      Hence by Theorem \ref{thm.fjord} \ref{thm.fjord:second}, 
      $$\log\frac{\dist(y,w)}{\dist(v,x)} \succeq \log\frac{|\partialhL\mathcal{N}|}{\dist(v,x)} \geq\log\frac{|\partialhL\mathcal{N}|}{\dist(v,\partialhL\mathcal{N})}\succeq K.$$
       Now Corollary \ref{cor.div-expboost} gives the required interval $I$. 
      \item if $\dist(v,\partialhL\Rnew)\geq1.1\cdot \dist(\partialhR\Rnew,w)$, then under a similar estimate as the above case, we get 
      $$\log\frac{\dist(v,x)}{\dist(y,w)} \succeq K.$$
      This is impossible, otherwise one can pullback $\X$ by $f^{k\q_{m+1}}$ for some $k\geq 1$ (recall that $\partial^h\X = [\tilde{x},x]\cup[y,\tilde{y}]$), so that the resulting preimage will cross intersect $\X$. This contradicts $\W(\X)\gg 1$ (see (\ref{eq.W(XY)})).
    \end{itemize}
   
\end{proof}

\section{Multicritical pseudo Siegel disk}\label{s:wZ:dfn} In this section, we introduce multicritical pseudo-Siegel disks $\wZ^m_i$. The inductive justification of their existence so that appropriate \emph{a priori} bounds hold will be carried out in~\S\ref{s:wZ:bounds}.

Pseudo-Siegel disks $\wZ^m_i$ are inductively constructed by adding parabolic fjords $\mathfrak F^m_{i,j}$ to the original Siegel disk, see Assumption~\ref{assump.2} and Figure~\ref{fig.outerprotection}. Every fjord $\mathfrak F^m_{i,j}$ is determined by a dam $\beta_{i,j}^m$ connecting two points on a preparabolic interval $T_{i,j}^m$. We say that $\mathfrak F^m_{i,j}$ is \emph{based} on $T_{i,j}^m$.

We specify two types of multicritical pseudo-Siegel disks: \emph{full} $\wZ^m$ and \emph{intermediate} $\wZ^{m}_j$ with the relation 
\begin{equation}
    \wZ^{m+1}=\wZ^{m}_0\subseteq \wZ^{m}_1\subseteq \dots \subseteq \wZ^{m}_s =\wZ^m,\qquad\quad\text{ Assumption~\ref{assump.1}}
\end{equation}
so that either $\wZ^m_{i}=\wZ^m_{i-1}$ or $\wZ^m_{i}$ is obtained from $\wZ^m_{i-1}$ by adding parabolic fjords based on every interval in the orbit $\mathcal{O}_m(T^m_{i,j})$, see~\eqref{eq:dfn:orb:prepar}, of the preparabolic interval $T^m_{i,j}$. After accounting for all orbits of preparabolic intervals on a given level $m$, we obtain a full pseudo-Siegel disk $\wZ^m$.

In~\S\ref{subsec.assump.basic}, we axiomatize the invariance properties of $\wZ^m_i$. Relevant objects are illustrated in Figure~\ref{fig.protection}. A key distinction from the quadratic case is the absence of tunnels in $\wZ_i^m$. Because a given level $m$ may contain several orbits of preparabolic intervals, the dams $\beta_{i,j}^m$ cannot, \emph{a priori}, be chosen so as to be submerged at the same combinatorial depth. (This property will be justified in the \emph{a posteriori} setting; see~\S\ref{ss:a_posteriori:threshold}.) Accordingly, the assumptions in~\S\ref{subsec.assump.basic} are adapted from the quadratic case, with suitable modifications reflecting the absence of tunnels.

The inner geometry of $\wZ^m_s$ is described by the Log-Rule with respect to the shortcut metric on $\partial \wZ^m_{s}$. The (level $m$) shortcut metric (Definition~\ref{def.shortcut-dist}) is obtained from the combinatorial metric by declaring that every dam $\beta^m_{i,j}$ has length $\mathfrak{l}_{m+1}$. In Assumption~\ref{assump.7}, we postulate the Log-Rule on the intermediate scales (between $\mathfrak{l}_{m+1}$ and $\mathfrak{l}_m$). On scales shallower than $\mathfrak{l}_m$, the shortcut and combinatorial metrics are equivalent. Theorem~\ref{thm.large-log-rule} establishes the uniform Log-Rule on sufficiently shallow scales.

Similar to the quadratic case, the outer geometry of $\overline Z$ and of $\wZ^m_s$ are equivalent rel grounded intervals \S\ref{ss:outer.geometry}; i.e., the $\Width^+$ parameter can be converted between Siegel pseudo-Siegel disks

The Log-Rule implies the localization property~\S\ref{ss:LocalLmm} with respect to the shortcut metric. By splitting intervals, we can convert $\Width^+_{\lambda, m,\bfS}$-degeneration rel the shortcut metric to the $\Width^+_{\lambda}$-degeneration rel the combinatorial metric, see~\S\ref{ss:ShortCut to Comb}.

\subsection{Basic assumptions}\label{subsec.assump.basic}
In this subsection, we first define {\em pseudo-Siegel disk}.
The assumptions mirror those in the quadratic case.
 
\begin{assumption}[Nested condition]\label{assump.1}
Let $\theta = [a_0,a_1,...,a_n,1,1,1...]$ with $n = n_\theta$. Then 
\begin{itemize}[label = {-}]
    \item For all $m\geq n_\theta$, $\hZ^m_0 = \hZ^m_1 = ... = \hZ^m_{s-1} \equiv {{Z}}$.
    \item For all $m\geq -1$ and $1\leq i\leq s$, $\hZ^m_{i-1}\subset \hZ^m_{i}$, $\hZ^{m+1}_s = \hZ^{m}_0$.
\end{itemize}

\end{assumption}

\begin{assumption}[Fjord condition]\label{assump.2}
 See Figure \ref{fig.outerprotection}. For each $1\leq i\leq s$, either $\hZ^{m}_{i}=\hZ^{m}_{i-1}$, or $\overline{\hZ^{m}_{i}\setminus{\hZ^{m}_{i-1}}}$ is a finite union of disjoint closed Jordan disks (each one of these closed Jordan disks is called a level-$m$ fjord):
 $$\bigcup_{j=0}^{\q_{m+1}-1} \mathfrak{F}^m_{i,j}$$ 
    such that $\partial \mathfrak{F}^m_{i,j} = I^m_{i,j}\sqcup \beta^m_{i,j}$, where 
    \begin{itemize}[label = {-}]
        \item $I^m_{i,j}\subset\partial\hZ^m_i$ is a sub-interval with two end points $x^m_{i,j},y^m_{i,j}\in\mathrm{CP}_{m+1}\setminus\mathrm{CP}_{m}^\mathrm{}$ such that 
        $$x^m_{i,j}<y^m_{i,j}<x^m_{i+1,j},\quad \dist(x^m_{1,j},y^m_{s,j})<{\l_m}.$$
        Moreover, there exists a unique pair of adjacent points $a^m_{i,j}<b^m_{i,j}$ in $\mathrm{CP}_m^{\mathrm{full}}$, such that 
        \begin{equation}\label{eq.dam-reclaim-distance}
            a^m_{i,j}<x^m_{i,j}<y^m_{i,j}<b^m_{i,j},\quad |\lfloor x^m_{i,j},y^m_{i,j}\rfloor|\geq \frac{4}{5}\cdot|\lfloor a^m_{i,j},b^m_{i,j}\rfloor|
        \end{equation}
        \item $x^m_{i,j},y^m_{i,j}$ satisfy the central condition:
        \begin{equation}\label{eq.central.condition}
\dist(a^m_{i,j},x^m_{i,j})\asymp\dist(y^m_{i,j},b^m_{i,j}).
        \end{equation}
          \item $\beta^m_{i,j}$ is a Jordan arc in $\widehat{\C}\setminus{Z}$ connecting $x^m_{i,j},y^m_{i,j}$; $\beta^m_{i,j}$ is called an $i$-th level-$m$ dam (level-$m$ dam for short).
      \end{itemize}
\end{assumption}

In the following Assumptions~\ref{assump.3} and ~\ref{assump.4}, we first state a stricter assumption, which requires all auxiliary objects to be hyperbolic geodesics; see Remark~\ref{rem:geodvsnongeod} for the relaxation of this requirement.
\begin{assumption}[$\boldsymbol{\delta}$-protection annuli]\label{assump.3} See Figure \ref{fig.protection}. 
For each pair of points $x^m_{i,j},y^m_{i,j}$ given by Assumption \ref{assump.2}, there exist four annuli $$\Ainn(x^m_{i,j}),\Aout(x^m_{i,j}),\Ainn(y^m_{i,j}),\Aout(y^m_{i,j})$$ surrounding $x^m_{i,j},y^m_{i,j}$ respectively. Moreover, let $z = x^m_{i,j}$ or $y^m_{i,j}$, then
\begin{enumerate}[label=\upshape(\arabic*)]
        \item $\mathrm{mod}(\Ainn(z)),\mathrm{mod}(\Aout(z))\geq\Bdelta$.
        \item $\partialout\Ainn(z) = \partialinn\Aout(z)$ (in particular, $\Aout(z)$ surrounds $\Ainn(z)$).
        \item\label{ass:3.3} Each one of $\partialinn\Ainn(z),\partialout\Ainn(z),\partialinn\Aout(z),\partialout\Aout(z)$ is the closure of the union of one geodesic in $\hZ^m_0$ and one geodesic in $\widehat{\C}\setminus\overline{Z}$.  In particular, they intersect $\partial\hZ^m_0$ at exactly two points. Moreover, these two points belong to $\mathrm{CP}_{m+1}$ and are between $a^m_{i,j}<b^m_{i,j}$ given by Assumption \ref{assump.2}.
  \end{enumerate}
\end{assumption}

\begin{assumption}[$\Bdelta$-protection around dams]\label{assump.4} See Figure \ref{fig.protection}.
     For $z = x^m_{i,j}$ or $y^m_{i,j}$, $\epsilon = \mathrm{inn}$ or $\mathrm{out}$,  let $J^{\epsilon}(z)\subset\partial\hZ^m_0$ the interval surrounded by the filled-in $\tilde{A}^{\epsilon}(z)$ of the annulus $A^{\epsilon}(z)$ given by Assumption \ref{assump.3}. One has the concatenation
     $$J^{\epsilon}(z) =N^{\epsilon}(z)\# I^{\epsilon}(z)\# M^{\epsilon}(z)$$ such that $N^{\epsilon}(z)$, $I^{\epsilon}(z)$ and $M^{\epsilon}(z)$ are the three components of $J^{\epsilon}(z)\setminus\partial A^{\epsilon}(z)$. In particular, $z\in I^{\epsilon}(z)$ and $\Jinn(z)=\Iout(z)$. Moreover, we require that for all $z,z' = x^m_{i,j}$ or $y^m_{i,j}$ and $\epsilon,\epsilon'=\mathrm{inn}$ or $\mathrm{out}$,
     \begin{equation}\label{eq.comparable.rectangle.base}
         |N^{\epsilon}(z)|\asymp |I^{\epsilon}(z)|\asymp |M^{\epsilon}(z)|, \quad |N^{\epsilon}(z)|\asymp |N^{\epsilon'}(z')|
     \end{equation}
    
     Let $\Z^{\epsilon}(\beta^m_{i,j}),\X^{\epsilon}(\beta^m_{i,j}),\Y^{\epsilon}(\beta^m_{i,j})\subset\widehat{\C}\setminus{Z}$ be the geodesic rectangles (their vertical boundaries are geodesics in $\widehat{\C}\setminus{Z}$) such that 
     \begin{itemize}[label = {-}]
         \item  $\partial^h\Z^\epsilon(\beta^m_{i,j}) = N^\epsilon(x^m_{i,j})\cup N^\epsilon(y^m_{i,j})$, $\partial^h\X^\epsilon(\beta^m_{i,j}) = I^\epsilon(x^m_{i,j})\cup I^\epsilon(y^m_{i,j})$ and $\partial^h\Y^\epsilon(\beta^m_{i,j}) = M^\epsilon(x^m_{i,j})\cup M^\epsilon(y^m_{i,j})$. In particular, $\beta^m_{i,j}\subset\X^\epsilon(\beta^m_{i,j})$.
         \item $\W(\Z^\epsilon(\beta^m_{i,j})),\W(\X^\epsilon(\beta^m_{i,j})),\W(\Y^\epsilon(\beta^m_{i,j}))\asymp 1$.
         \item For $\epsilon = \mathrm{inn}$ or $\mathrm{out}$, define the Jordan disk $$O^\epsilon(\beta^m_{i,j}) := A^\epsilon(x^m_{i,j})\cup A^\epsilon(y^m_{i,j})\cup \Z^\epsilon(\beta^m_{i,j})\cup \X^\epsilon(\beta^m_{i,j})\cup \Y^\epsilon(\beta^m_{i,j}).$$ 
         Denote by $A^\epsilon(\beta^m_{i,j})$ the annulus such that $\partialinn A^\epsilon(\beta^m_{i,j}) = \partial\X^\epsilon(\beta^m_{i,j})$, and $\partialout A^\epsilon(\beta^m_{i,j}) = \partial O^\epsilon(\beta^m_{i,j})$.
         It follows immediately from Lemma \ref{lem.combination.annulus} that
\begin{equation}\label{eq.Ainn-out-dam}
    \mathrm{mod}(A^\epsilon(\beta^m_{i,j}))\succeq_{\Bdelta}1
    \end{equation}
     \end{itemize}

For every $\beta^m_{i,j}$, we define its thickening by
\begin{equation}\label{eq.fat.dam}
    \widehat{\beta}^m_{i,j} := \beta^m_{i,j}\cup (J^{\mathrm{inn}}(x^m_{i,j})\cup J^{\mathrm{inn}}(y^m_{i,j}))\cap\partial\hZ^m_{i}.
\end{equation}
\end{assumption}

\begin{rem}\label{rem:geodvsnongeod}
    In Assumptions~\ref{assump.3} and ~\ref{assump.4}, the auxiliary objects (annuli and rectangles) are bounded by hyperbolic geodesics of $\widehat\C\setminus\overline{Z}$ or $\wZ^m$. We therefore call the resulting pseudo-Siegel disk a {\em geodesic pseudo-Siegel disk}.
    
    More generally, a pseudo-Siegel disk is obtained by relaxing the conditions in Assumptions~\ref{assump.3} and ~\ref{assump.4}: the auxiliary objects may instead be bounded by curves having the same intersection pattern as their geodesic counterparts.
\end{rem}

\begin{figure}[h]
\centering 
\includegraphics[width=0.8\textwidth]{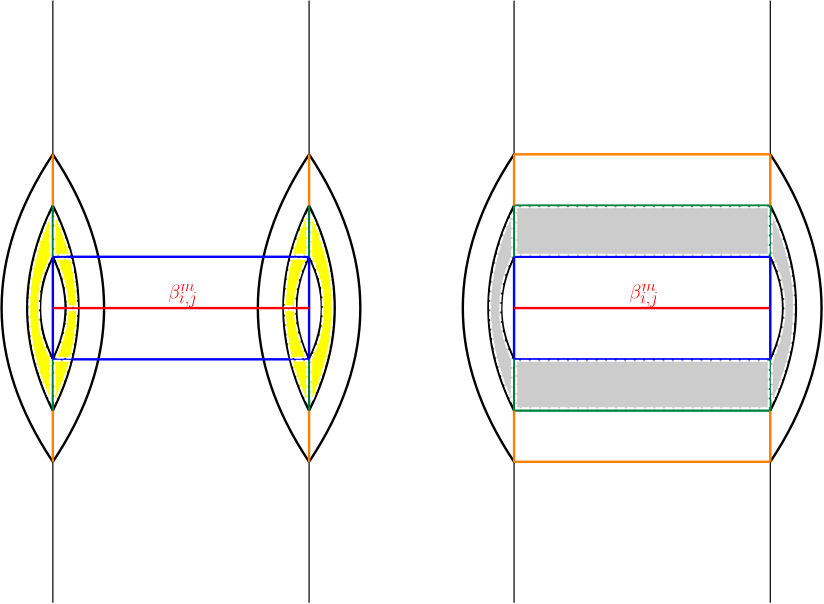} 
\caption{Left: The two nested annuli $A^{\mathrm{inn}}(x^m_{i,j}),A^{\mathrm{out}}(x^m_{i,j})$ and $A^{\mathrm{inn}}(y^m_{i,j}), A^{\mathrm{out}}(y^m_{i,j})$; the inner annulus is colored in yellow. Right: The two nested annuli around the dam: $A^{\mathrm{inn}}(\beta^m_{i,j})$ and $A^{\mathrm{out}}(\beta^m_{i,j})$; the inner annulus is colored in grey; the blue arcs represent $\partial\mathcal{X}^{\mathrm{inn}}(\beta^m_{i,j})$; the green arcs represent $\partial\mathcal{Y}^{\mathrm{inn}}(\beta^m_{i,j})$ and $\partial\mathcal{Z}^{\mathrm{inn}}(\beta^m_{i,j})$; the orange arcs represent $\partial\mathcal{Y}^{\mathrm{out}}(\beta^m_{i,j})$ and $\partial\mathcal{Z}^{\mathrm{out}}(\beta^m_{i,j})$.}
\label{fig.protection} 
\end{figure}

\begin{assumption}[Almost-invariance]\label{assump.5}
    $f$ has an inverse branch $f^*$ defined on $\hZ^m_i$ such that $f^*(Z) = Z$. Moreover, there exists $k \geq 1$ and for each integer $0\leq s\leq k\q_{m+1}$, one has the following:
    \begin{itemize}[label = {-}]
        \item $f^s|_{\hZ^m_i}$ and $(f^*)^s$ are injective.
        \item $f^{s} (\beta^m_{i,j}),(f^*)^s(\beta^m_{i,j})\subset\X^{\mathrm{inn}}(\beta^m_{i,j\pm s\p_{m+1}})$ for each dam $\beta^m_{i,j}$ ($\X^{\mathrm{inn}}$ is given by Assumption \ref{assump.4}).
       
        \item For each $n,t,l$ and $\epsilon = \mathrm{inn}$ or $\mathrm{out}$, let $\tilde{A}^{\epsilon}(\beta^n_{t,l})$ be the filled-in of $A^{\epsilon}(\beta^n_{t,l})$. Then there exists a subfamily of Jordan disks $\{\tilde{A}^{\mathrm{inn}}(\beta^n_{t,l})\}_{n,t,l}$ with $n\geq m$, such that $\tilde{A}^{\mathrm{inn}}(\beta^n_{t,l})\cap(\widehat{\C}\setminus\hZ^m_i)$ are mutually disjoint and
      \begin{equation}\label{eq.tilde-Ainn}
          f^{s}(\hZ^m_i),(f^*)^s(\hZ^m_i)\subset\hZ^m_i\cup\left(\bigcup_{n,t,l}\tilde{A}^{\mathrm{inn}}(\beta^n_{t,l})\right).
      \end{equation}
    \end{itemize}
\end{assumption}

\begin{assumption}[${\BDelta}$-outer protection]\label{assump.6} See Figure \ref{fig.outerprotection}.
    There is a constant ${\BDelta}\gg1$ such that for every pair of adjacent points $a^m_{i,j}<b^m_{i,j}\in\mathrm{CP}_m^\mathrm{full}$ given by Assumption \ref{assump.2}, there exist three rectangles $\mathfrak{X}^m_{i,j}, \mathfrak{X}^{m,\mathrm{inn}}_{i,j},\mathfrak{X}^{m,\mathrm{out}}_{i,j}$ such that
    \begin{itemize}[label = {-}]
        \item $\partial^h\mathfrak{X}^m_{i,j}\subset\partial\hZ^m_0$; $a^m_{i,j}<\partial^h\mathfrak{X}^m_{i,j}<b^m_{i,j}$.
         \item $\mathfrak{X}^m_{i,j}\cap (\Aout(x^m_{i,j})\cup\Aout(y^m_{i,j})\cup\Z^{\mathrm{out}}(\beta^m_{i,j})) = \emptyset$; $\mathfrak{F}^m_{i,j}$ is contained in the bounded component of $\C\setminus(\mathfrak{X}^m_{i,j}\cup\hZ^{m+1})$.
        \item $\mathfrak{X}^{m,\mathrm{inn}}_{i,j},\mathfrak{X}^{m,\mathrm{out}}_{i,j}\subset\mathfrak{X}^m_{i,j}$; $\partial^h\mathfrak{X}^{m,\mathrm{inn}}_{i,j},\partial^h\mathfrak{X}^{m,\mathrm{out}}_{i,j}\subset\partial^h\mathfrak{X}^m_{i,j}$.
       
        \item $\mathfrak{X}^{m,\mathrm{inn}}_{i,j}\cap\mathfrak{X}^{m,\mathrm{out}}_{i,j} = \emptyset; \lfloor\partial^h\mathfrak{X}^{m,\mathrm{inn}}_{i,j}\rfloor\subset\lfloor\partial^h\mathfrak{X}^{m,\mathrm{out}}_{i,j}\rfloor$.

         \item all end points of $\partial^h\mathfrak{X}^{m,\mathrm{inn}}_{i,j},\partial^h\mathfrak{X}^{m,\mathrm{out}}_{i,j},\partial^h\mathfrak{X}^m_{i,j}$ belong to $\mathrm{CP}_{m+1}\setminus\mathrm{CP}_m$.
       
        \item elements in the following set are mutually disjoint:
        $$\{\Gamma\,|\,\Gamma = \partial^v\mathfrak{X}^{m'}_{i',j'},\partial^v\mathfrak{X}^{m',\mathrm{inn}}_{i',j'}, \partial^v\mathfrak{X}^{m'\mathrm{out}}_{i',j'},\beta^{m'}_{i',j'} \text{ with }  m'\geq m, 0\leq i'\leq i(m'), 0\leq j'\leq\q_{m'+1}-1\}.$$
        \item $\W(\mathfrak{X}^m_{i,j})\geq{\BDelta}$, $\W(\mathfrak{X}^{m,\mathrm{inn}}_{i,j}),\W(\mathfrak{X}^{m,\mathrm{out}}_{i,j})\geq{\BDelta}/3$.
    \end{itemize} 
\end{assumption}

\begin{figure}[t]
\centering 
\includegraphics[width=0.7\textwidth]{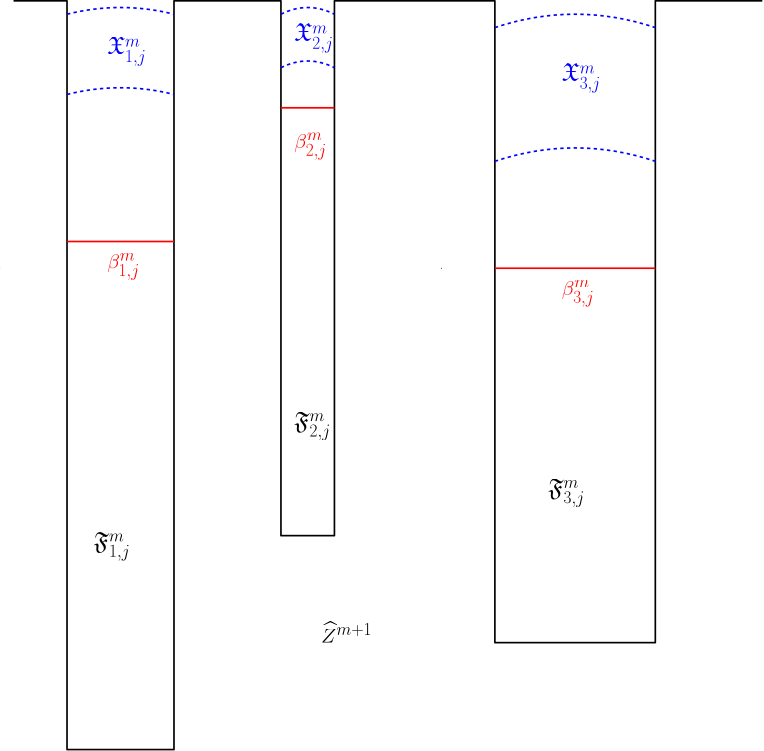} 
\caption{A schematic picture of $\hZ^m$ near a combinatorial interval of length $\l_m$. The number of fjords is 3.}
\label{fig.outerprotection} 
\end{figure}

\vspace{0.2cm}

\subsection{Log-Rule assumption}\label{subsec.assump.logrule}

In order to state Assumption \ref{assump.7}, we need to introduce some terminologies.

Let us define the level-$m$ {\it full pseudo-Siegel disk} by $\hZ^m := \hZ^{m}_{s} = \hZ^{m-1}_{0}$. We will also need the following notion of {\it intermediate pseudo-Siegel disk}:
\begin{defn}\label{def.intermediate pseudo-Siegel-disk}
    For any $\bfS\subset\{1,...,s\}$, the intermediate pseudo-Siegel disk $\hZ^{m}_{\bfS}$ associated to $\bfS$ of the full pseudo-Siegel disk $\hZ^m = \hZ^{m}_s$ is defined by adding all the fjords labeled by $\bfS$ to $\hZ^{m}_0$: 
    $$\hZ^{m}_{\bfS} := \hZ^{m}_0\cup\left(\bigcup_{j=0}^{\q_{m+1}-1}\bigcup_{s\in\bfS}{\mathfrak{F}^m_{s,j}}\right).$$
    In particular, $\hZ^{m}_{\emptyset} = \hZ^{m}_0 = \hZ^{m+1}$, $\hZ^{m}_{\{1,...,i\}} = \hZ^{m}_i$, $\hZ^{m}_{\{1,...,s\}} = \hZ^{m}_{s} = \hZ^{m}$.
\end{defn}

\begin{defn}[Regular intervals]\label{defn.regularinterval} A regular point of $\partial\hZ^{m}_{\bfS}$ is a point in $\partial\hZ^{m}_{\bfS} \cap\partial Z$. A regular interval $I \subset \partial\hZ^{m}_{\bfS}$ is an interval with regular endpoints. An interval $I\subset\partial Z$ is regular $\mathrm{rel}$ $\hZ^{m}_{\bfS}$ if its end points are regular.
\end{defn}
\begin{defn}[Projecting from $\partial Z$ to $\partial \hZ^{m}_{\bfS}$]\label{def.project-intervals}
     For an interval $I \subset\partial Z$, the projection $I^{m,\bfS}$ onto $\partial\hZ^m_{\bfS}$ is the shortest regular interval whose projection onto $\partial Z$ contains $I$. For the projection onto $\partial\hZ^m$, we will simply denote it by $I^m$.
\end{defn}   

\begin{defn}[Grounded intervals]\label{defn.grnd-interval}
    For $z = x^m_{i,j}$ or $y^m_{i,j}$, let $\tilde{A}^{\mathrm{inn}}(z),\tilde{A}^{\mathrm{out}}(z)$ the Jordan disks by filling in $\Ainn(z),\Aout(z)$ respectively (recall Assumption \ref{assump.3}). Define 
    $$\Sinn(z) :=\tilde{A}^{\mathrm{inn}}(z)\cap\partial\hZ^m,\quad \Sout(z):=\tilde{A}^{\mathrm{out}}(z)\cap\partial\hZ^m.$$
    For $\epsilon\in\{\mathrm{inn},\mathrm{out}\}$, we set 
    \begin{equation}\label{eq.Sinn}
    S^\epsilon(\hZ^m_{\bfS}) :=  S^\epsilon(\hZ^m_0)\cup\left(\bigcup_{j=0}^{\q_{m+1}-1}\bigcup_{s\in\bfS} S^\epsilon(x^m_{s,j})\cup S^\epsilon(y^m_{s,j})\right),\quad  S^\epsilon(\hZ^m_0) = \emptyset\quad\text{for }m\gg_\theta 1.
    \end{equation}
    A regular interval $I\subset\partial\hZ^m_{\bfS}$ is called grounded (resp. well-grounded) $\mathrm{rel}$ $\hZ^m_{\bfS}$ if neither its end points belong to $\Sinn(\hZ^m_{\bfS})$ (resp. $\Sout(\hZ^m_{\bfS})$). An interval $I\subset\partial Z$ is grounded (well-grounded) $\mathrm{rel}$ $\hZ^m_{\bfS}$ if its projection $I^{m,\bfS}$ is.
\end{defn}

\begin{defn}[Projective distance on $\partial\hZ^m_{\bfS}$]\label{def.projective.dist} The projection of a regular interval $I\subset \partial \hZ^m_{\bfS}$ onto $\partial Z$ is the interval $I^\bullet \subset \partial Z$ with the same endpoints and the same orientation as $I$. We define the projective length of $I$ by $|I| := |I^\bullet|$. We say that $I$ is a level $m$-combinatorial interval if $|I| = \l_m$.

\end{defn}

\begin{defn}[Shortcut distance on $\partial\hZ^m_{\bfS}$]\label{def.shortcut-dist} Let $I\subset \partial \hZ^m_{\bfS}$ be a regular interval. Then $I$ can be uniquely expressed as the concatenation 
$$I = I_1\#\beta^m_1\#I_2\#\beta^m_2\#...\#I_l\#\beta^m_l\#I_{l+1},$$
where $I_1,...,I_{l+1}$ are regular intervals in $\partial\hZ^m_0$; $\beta^m_1,...,\beta^m_l$ are level-$m$ dams. The shortcut length of $I$ is defined inductively by 
$$|I|_{m,\bfS} := l\cdot\l_{m+1}+\sum_{j=1}^{l+1}|I_{j}|_{m,\emptyset}.$$
The shortcut length for $\partial Z$ is identical to the combinatorial length $|\cdot|$ induced by (\ref{eq.combinatorial-distance}). We will denote the shortcut length for $\partial\hZ^m_i$ (resp. $\partial\hZ^m$) by $|\cdot|_{m,i}$ (resp. $|\cdot|_m$).  
\end{defn}

\begin{rem}\label{rem.dam-distance}
    Let $\beta^m_{i,j}$ be a level-$m$ dam. From the definition, its shortcut length is $|\beta^m_{i,j}|_{m,\bfS} = \l_{m+1}$; its projective length satisfies $\l_{m+1}<|\beta^m_{i,j}|=|\lfloor x^m_{i,j},y^m_{i,j}\rfloor|<\l_m$.

    For any regular interval $I\subset\partial\hZ^m_{\bfS}$, the shortcut distance is smaller than the projective distance: $$|I|_{m,\bfS}\leq|I|.$$
\end{rem}

\begin{rem}\label{rem.forward-backward-image}
    By Assumption \ref{assump.5}, $\hZ^m_{\bfS}$ has iterated forward and backward images, denoted by
    \begin{equation}\label{eqn:pushforward}
    f_*^s(\wZ^m_{\bfS}) = \begin{cases}
        f^s(\hZ^m_{\bfS}) & \text{if $s \geq 0$}\\
        (f^*)^{|s|}(\hZ^m_{\bfS}) & \text{if $s< 0$}
    \end{cases},
    \end{equation}
    with $|s|\leq k\q_{m+1}$. 
    Therefore, the notions of regular, (well-)grounded intervals and projection of intervals are similarly defined on $f^s(\hZ^m_{\bfS})$, $(f^*)^s(\hZ^m_{\bfS})$.
\end{rem}

\begin{assumption}[Log-Rule on intermediate scale]\label{assump.7}
Let $I,J\subset \partial\hZ^m_{\bfS}$ be a grounded pair such that $\dist(I,J)\geq\l_{m+1}$ and $\min\{|I|,|J|\}\geq\l_{m+1}$.

 \begin{enumerate}[label=\upshape(\roman*)]
 \item\label{thm.inter-log-rule:first} If $\min\{|I|_{m,\bfS},|J|_{m,\bfS}\}\geq\dist_{m,\bfS}(I,J)$, then
 $$\W^-_{\hZ^m_{\bfS}}(I,J)\asymp_{\Bdelta}\log\frac{\min\{|I|_{m,\bfS},|J|_{m,\bfS}\}}{\dist_{m,\bfS}(I,J)}+1.$$

 \item\label{thm.inter-log-rule:second} If $\min\{|I|_{m,\bfS},|J|_{m,\bfS}\}\leq\dist_{m,\bfS}(I,J)$, then
 $$\W^-_{\hZ^m_{\bfS}}(I,J)\asymp_{\Bdelta}\left(\log\frac{\dist_{m,\bfS}(I,J)}{\min\{|I|_{m,\bfS},|J|_{m,\bfS}\}}+1\right)^{-1}$$
 \end{enumerate}
\end{assumption}

\begin{lem}\label{lem.shortcut-projective-comparable}
    Let $I,J\subset\partial\hZ^m_{\bfS}$ be regular intervals. If $|I|\geq|J|\geq\l_m$, then $$\frac{1}{2}\cdot\frac{|I|_{m,\bfS}}{|J|_{m,\bfS}}\leq \frac{|I|}{|J|}\leq2\cdot\frac{|I|_{m,\bfS}}{|J|_{m,\bfS}}.$$
\end{lem}
\begin{proof}
    From the definition of $|\cdot|_{m,\bfS}$ and $|\cdot|$, if $|I|=|J|$ or $|I|_{m,\bfS}=|J|_{m,\bfS}$, then the maximal numbers of level $m$-combinatorial intervals contained in $I$ and $J$ are comparable (the constant can be made 
    $\leq2$). It is also easy to verify that different level $m$-combinatorial intervals have the same projective/shortcut length. The lemma then follows. 
\end{proof}

Finally, we introduce for convenience a notation of geodesic topological disks surrounding intervals of $\partial\hZ^m_{\bfS}$:
\begin{defn}[Geodesic disk $O_{J,\hZ^m_{\bfS}}$]\label{def.geodesic-annulus}
    For a topological disk $D\subset\widehat{\C}$ and two points $x,y\in\partial D$, we denote by $[x,y]_D$ the hyperbolic geodesic in $D$ joining $x$ and $y$. Let $\hZ^{m}_{\bfS}$ be an intermediate pseudo-Siegel disk. For an interval $J\subset \hZ^{m}_{\bfS}$ with end points $x,y$, let us denote by $O_{J,\hZ^{m}_{\bfS}}$ the topological disk bounded by $[x,y]_{\hZ^{m}_{\bfS}}$ and $[x,y]_{\widehat{\C}\setminus{Z}}$. When $\hZ^m_{\bfS} = \hZ^{m}$ or $\hZ^{m+1}$, we will write $O_{J,\hZ^m}$ or $O_{J,{\hZ^{m+1}}}$ respectively.
\end{defn}

\subsection{Shift of rectangles in pseudo-Siegel disk}
\begin{lem}\label{lem.coarse-bound}
    There exists a constant $C_{\Bdelta}>0$ only depending on $\Bdelta$, such that for any pair of grounded intervals $I,J\subset\partial\hZ^m_{\bfS}$ satisfying $\dist(I,J)\geq 10\min\{|I|,|J|\}\geq10\l_{m}$, one has
    $\W^-_{\hZ^m_{\bfS}}(I,J)\leq C_{\Bdelta}$.
\end{lem}
\begin{proof}
    We denote by $\R$ the canonical rectangle of $\F^-_{\hZ^m_{\bfS}}(I,J)$. By Assumption \ref{assump.5}, $f^s|_{\hZ^m_{\bfS}}$ is injective for any $0\leq s\leq\q_{m+1}$. We assume that $|I|\leq|J|$ and denote by $L :=f^s(I)$. Note that although $L$ is not necessarily a sub-interval of $\partial\hZ^m_{\bfS}$, its end points belong to $\partial\hZ^m_{\bfS}\cap\partial Z$. Denote by $\tilde{L}\subset\partial\hZ^m_{\bfS}$ the interval with the same end points as $L$ and intersecting $L$ at infinitely many points. Let us pick an $0\leq s\leq\q_{m+1}$ such that $$\dist(\tilde{L},I),\dist(\tilde{L},J)\geq3\l_m.$$
    Let $\Rnew$ be the subrectangle of $\R$ after removing two $5$-buffers. Denote by $\R_L := f^s(\Rnew)$ and we have $\W(\R_L) = \W(\Rnew)$ (this is a conformal image by Assumption \ref{assump.5}). Let $\Rnew_L$ be the subrectangle of $\R_L$ after removing two $5$-buffers. By Lemma \ref{lemappen.disjoint-boundaires}, $\Rnew_L$ and $\Rnew$ have disjoint vertical boundaries.

    To fix the idea, let us assume that $I<\tilde{L}<J$. By Assumption \ref{assump.5}, the four end points of $\partial^{v,1}\R\cup\partial^{v,1}\Rnew$ are surrounded by at most eight annuli $A^{\mathrm{out}}_1,A^{\mathrm{out}}_2,B^{\mathrm{out}}_1,B^{\mathrm{out}}_2$ and $A^{\mathrm{inn}}_1,A^{\mathrm{inn}}_2,B^{\mathrm{inn}}_1,B^{\mathrm{inn}}_2$ ($A$ and $B$ might coincide) with modulus $\succeq_{\Bdelta}1$ and $A^{\mathrm{out}}_l,B^{\mathrm{out}}_l$ surround $A^{\mathrm{inn}}_l,B^{\mathrm{inn}}_l$ respectively, $l=1,2$. Let $\X$ be the right-hand side buffer of $\R\setminus\Rnew$ (notice that $\partial^v\X = \partial^{v,1}\R\cup\partial^{v,1}\Rnew$). By (\ref{eq.tilde-Ainn}), $$\X\cup \tilde{A}^{\mathrm{inn}}_1\cup \tilde{A}^{\mathrm{inn}}_2\cup\tilde{B}^{\mathrm{inn}}_1\cup \tilde{B}^{\mathrm{inn}}_2$$ separates $f^s({\hZ^m_{\bfS}})$ into two connected components ($\tilde{A}^{\mathrm{inn}}_1,\tilde{A}^{\mathrm{inn}}_2,\tilde{B}^{\mathrm{inn}}_1,\tilde{B}^{\mathrm{inn}}_2$ denote the filled-in of the annuli). By construction, $\partialhL\Rnew_L$ and $\partialhR\Rnew_L$ belong to different components of $f^s({\hZ^m_{\bfS}})$. By Lemma \ref{lemappen.annulus-buffer}, vertical curves in $\Rnew_L$ crossing $A^{\mathrm{out}}_1\cup A^{\mathrm{out}}_2\cup B^{\mathrm{out}}_1\cup B^{\mathrm{out}}_2$ forms two $O_{\Bdelta}(1)$-buffers. Let $\RNew_L$ be the sub-rectangle after removing these two buffers from $\Rnew_L$. Then all vertical curves in $\RNew_L$ must cross $\X$. Hence $\W(\RNew_L) = O(1)$ and $\W(\R) = O_{\Bdelta}(1)$.
\end{proof}

Using Log-Rule on intermediate scale (Assumption \ref{assump.7}) and Shift Argument, the Log-Rule on large scale can be made independent of $\Bdelta$ if intervals are sufficiently separated (as a compensation, the separation depends on $\Bdelta$):
\begin{thm}[Uniform Log-Rule on large scale]\label{thm.large-log-rule}
There exists $T_{\Bdelta}\gg1$ only depending on $\Bdelta$, such that for any grounded pair of intervals $I,J\subset \partial\hZ^m$ satisfying 
$\min\{|I|,|J|\}\geq T_{\Bdelta}\l_{m}$ and $\dist(I,J)\geq T_{\Bdelta}\cdot\l_{m}$, the Log-Rule is independent on $\Bdelta$:

 \begin{enumerate}[label=\upshape(\roman*)]
 \item\label{thm.large-log-rule:first} If $\min\{|I|,|J|\}\geq\dist(I,J)$, then
 $$\W^-_{\hZ^m}(I,J)\asymp\log\frac{\min\{|I|,|J|\}}{\dist(I,J)}+1.$$

 \item\label{thm.large-log-rule:second} If $\min\{|I|,|J|\}\leq\dist(I,J)$, then
 $$\W^-_{\hZ^m}(I,J)\asymp\left(\log\frac{\dist(I,J)}{\min\{|I|,|J|\}}+1\right)^{-1}.$$
 \end{enumerate}
\end{thm}
\begin{proof}
    The theorem follows from the following claim:
    \begin{claim*}
    There exists $T_{\Bdelta}\gg1$ only depending on $\Bdelta$ and a universal constant $C>0$, such that for any pair of grounded intervals $I,J\subset\partial\hZ^m$ satisfying $\dist(I,J)\geq10\min\{|I|,|J|\}\geq 10T_{\Bdelta}\l_{m}$, one has
    $\W^-_{\hZ^m}(I,J)\leq C$.
    \end{claim*}
    \begin{proof}[Proof of the claim]
    The proof is similar to that of Lemma \ref{lem.coarse-bound}. We denote by $\R$ the canonical rectangle of $\F^-_{\hZ^m}(I,J)$. By Assumption \ref{assump.5}, $f^s|_{\hZ^m}$ is injective for any $0\leq s\leq\q_{m+1}$. We assume that $|I|\leq|J|$ and denote by $L :=f^s(I)$. Denote by $\tilde{L}\subset\partial\hZ^m$ the interval with the same end points as $L$ and intersecting $L$ at infinitely many points. Let us pick an $0\leq s\leq\q_{m+1}$ such that $$\dist(\tilde{L},I),\dist(\tilde{L},J)\geq T_{\Bdelta}\l_m.$$
    Let $\Rnew$ be the subrectangle of $\R$ after removing two $5$-buffers. Denote by $\R_L := f^s(\Rnew)$ and we have $\W(\R_L) = \W(\Rnew)$ (this is a conformal image by the second point of Assumption \ref{assump.5}. Let $\Rnew_L$ be the subrectangle of $\R_L$ after removing two $5$-buffers. By Lemma \ref{lemappen.disjoint-boundaires}, $\Rnew_L$ and $\Rnew$ have disjoint vertical boundaries.

    To fix the idea, let us assume that $I<\tilde{L}<J$. Let $z$ be any one of the four end points of $\partial^{v,1}\R\cup\partial^{v,1}\Rnew$. By the third point of Assumption \ref{assump.5}, $z$ either belongs to $\partial Z$, or is surrounded by an annulus $A^{\mathrm{inn}}$ with modulus $\geq\Bdelta$. Since $\dist(I,J),\min\{|I|,|J|\}\geq T_{\Bdelta}\l_m$, we can find finitely many mutually disjoint rectangles $\{\R_i\}_{i=1}^{n}$ with $n = n(T_{\Bdelta})$ based on grounded intervals $G_i,H_i\subset\partial\hZ^m$, such that 
    \begin{itemize}
        \item $f^s(G_1)<f^s(G_2)<...<f^s(G_n)<z<f^s(H_1)<f^s(H_2)<...<f^s(H_n)$ on $f^s(\hZ^m)$;
        \item all the $f^s(G_i),f^s(H_i)$ are disjoint from $L$ and $A^{\mathrm{inn}}$;
        \item $\W(\mathcal{R}_i)\asymp_{\Bdelta}1$ (this follows from Assumption \ref{assump.7}); therefore 
        \begin{equation}\label{eq.W(GH)}
            \W(f^s(\bigcup_{i=1}^n \mathcal{R}_i))\succeq 1\quad\text{for}\quad T_{\Bdelta}\gg1.
        \end{equation}
    \end{itemize}
    Let $\X$ be the right-hand side buffer of $\R\setminus\Rnew$ (hence $\partial^v\X = \partial^{v,1}\R\cup\partial^{v,1}\Rnew$) and let $A^{\mathrm{inn}}_l$ be the (possibly) four annuli surrounding the four end points of $\X$. By (\ref{eq.tilde-Ainn}), $$\X\cup \tilde{A}^{\mathrm{inn}}_1\cup \tilde{A}^{\mathrm{inn}}_2\cup\tilde{A}^{\mathrm{inn}}_3\cup \tilde{A}^{\mathrm{inn}}_4$$ separates $f^s(\hZ^m)$ into two connected components ($\tilde{A}^{\mathrm{inn}}_1,\tilde{A}^{\mathrm{inn}}_2,\tilde{B}^{\mathrm{inn}}_1,\tilde{B}^{\mathrm{inn}}_2$ denote the filled-in of the annuli). By construction, $\partialhL\Rnew_L$ and $\partialhR\Rnew_L$ belong to different components of $f^s(\hZ^m)$. By Lemma \ref{lemappen.annulus-buffer}, vertical curves in $\Rnew_L$ crossing $A^{\mathrm{inn}}_1\cup A^{\mathrm{inn}}_2\cup A^{\mathrm{inn}}_3\cup A^{\mathrm{inn}}_4$ forms two buffers. By (\ref{eq.W(GH)}), the width of these buffers is at most $O(1)$. Let $\RNew_L$ be the sub-rectangle after removing these two buffers from $\Rnew_L$. Then all vertical curves in $\RNew_L$ must cross $\X$. Hence $\W(\RNew_L) = O(1)$ and $\W(\R) = O(1)$.
    \end{proof}
\end{proof}

\subsection{Localization lemma}\label{ss:LocalLmm}

\begin{lem}[Localization Lemma]\label{lem.localization}
For every $\lambda>1$ the following holds. If $I,J \subset \partial\hZ^m_{\bfS}$ is a grounded pair with $|\lfloor I,J\rfloor| \leq \frac{1}{2}$, then there exists an innermost sub-pair $I^{\mathrm{new}},J^{\mathrm{new}}$ satisfying the following
\begin{itemize}
\item $|I^{\mathrm{new}}|_{m,\bfS}, |J^{\mathrm{new}}|_{m,\bfS} \leq \frac{1}{\lambda}\min\{|I|_{m,\bfS},|J|_{m,\bfS}\}$;
    \item $I^{\mathrm{new}},J^{\mathrm{new}}$ are either grounded $\mathrm{rel}$ $\hZ^m_{\bfS}$, or grounded $\mathrm{rel}$ $\hZ^n$ for some $n> m$ with $I^{\mathrm{new}}\subset I\cap\partial\hZ^n,\,\,J^{\mathrm{new}}\subset J\cap\partial\hZ^n$,
\end{itemize}
 such that, up to $O_{\Bdelta}(\log\lambda)$, the width of $\F^-_{\hZ^m_{\bfS}}(I,J)$ is contained in $\F^-_{\hZ^m_{\bfS}}(I^{\mathrm{new}}, J^{\mathrm{new}})$:
\begin{equation}\label{eq.localisation}
     \W^-_{\hZ^m_{\bfS}}(I\setminus I^{\mathrm{new}},J)+\W^-_{\hZ^m_{\bfS}}(I,J\setminus J^{\mathrm{new}}) = O_{\Bdelta}(\log\lambda).
\end{equation}
\end{lem}
\begin{proof}
To fix the idea, let us assume that $I<J$. We will develop the following algorithm to find $I^{\mathrm{new}}$ (and similarly for $J^{\mathrm{new}}$). The algorithm will terminate at finitely many steps.

Let $I'\subset I$ and $J'\subset J$ be the innermost sub-intervals such that 
        $$|I'|_{m,\bfS} =|J'|_{m,\bfS} = \frac{1}{\lambda}\min\{|I|_{m,\bfS},|J|_{m,\bfS}\}.$$
\begin{itemize}
    \item If there is a level-$m$ dam between $I,J$, then the Log-Rule in Assumption \ref{assump.7} can be applied directly and the existence of $I^{\mathrm{new}}$ and $J^{\mathrm{new}}$ is clear.
    \item  If there is no level-$m$ dam between $I,J$. Consider two subcases:
    \begin{itemize}
        \item  If $I'$ contains some thick level-$m$ dam $\widehat{\beta}^m$ (see (\ref{eq.fat.dam})), then take $I^{\mathrm{new}}\subset I'$ to be the largest sub-intervals grounded $\mathrm{rel}$ $\hZ^m_{\bfS}$. By (\ref{eq.central.condition}), $|I^{\mathrm{new}}|_{m,\bfS}\asymp |I'|_{m,\bfS}\asymp \frac{1}{\lambda}\min\{|I|_{m,\bfS},|J|_{m,\bfS}\}$. Clearly $|I^{\mathrm{new}}|\geq\l_{m+1}$ since it contains a level $m$ dam. Hence Assumption \ref{assump.7} is applicable and one gets
    $$\W^-_{\hZ^m_{\bfS}}(I\setminus I^\mathrm{new},J) = O_{\Bdelta}(\log\lambda).$$
    \item  If $I'$ does not contain any thick level-$m$ dam, then there exists a unique thick level-$m$ dam $\widehat{\beta}_I^m$ such that we have concatenation 
    $$I' = (\widehat{\beta}_I^m\cap I')\# I'_1.$$
    We allow that $\widehat{\beta}_I^m\cap I' = \emptyset$. Now if $I$ is already $\mathrm{rel}$ $\hZ^{m+1}$, then we go back to the beginning of the algorithm for $\hZ^{m+1}$.

    Otherwise, $\widehat{\beta}^m_I\subset I$ (since we assume that there is no level-$m$ dam between $I,J$). Let $x^m,y^m$ be the left/right end point of $\beta^m_I$ respectively. We set $I''\subset I'$ to be the right most sub-interval starting at $y^m$. Then $I''$ is grounded $\mathrm{rel}$ $\hZ^{m+1}$. Moreover, by (\ref{eq.central.condition}), $|I''|_{m,\bfS}\asymp|I''\cup\widehat{\beta}_I^m|$. Hence there is a uniform constant $0<c<1$ such that $|I''|_{m,\bfS}/|I|_{m,\bfS}\leq c$. Hence by Assumption \ref{assump.7},
    $$\W^-_{\hZ^m_\bfS}(I\setminus (I''\cup\widehat{\beta}_I^m),J) = O_{\Bdelta}(\log\frac{|I|_{m,\bfS}}{|I''\cup\widehat{\beta}_I^m|_{m,\bfS}}) = O_{\Bdelta}(\log\frac{|I|_{m,\bfS}}{|I''|_{m,\bfS}}).$$
    On the other hand, by Assumption \ref{assump.4}, curves starting from $\beta^m_I$ and landing at $J$ cross $A^{\mathrm{inn}}(\beta^m_I)$; by Assumption \ref{assump.3}, curves starting from $\widehat{\beta}_I^m\setminus(I''\cup\beta^m_I)$ landing at $J$ cross $A^{\mathrm{out}}(x^m)$. Therefore
    $\W^-_{\hZ^m_\bfS}(\widehat{\beta}_I^m\setminus I'',J) = O_{\Bdelta}(1)$. Hence we have
    $$\W^-_{\hZ^m_\bfS}(I\setminus I'',J) = O_{\Bdelta}(\log\frac{|I|_{m,\bfS}}{|I''|_{m,\bfS}}).$$
    \end{itemize}
    Replacing $I$ by $I''$, and we go back to the beginning of the algorithm for $\hZ^{m+1}$.

\end{itemize}
\end{proof}

\subsection{Outer geometry of $\hZ^m_{\bfS}$}\label{ss:outer.geometry}

Let $\mathfrak{X}^n_{i,j}$ be a rectangle given by Assumption \ref{assump.6} with $n\geq m$. Recall that there is a constant ${\BDelta}>0$ such that $\W(\mathfrak{X}^n_{i,j})\geq{\BDelta}$. Let $\accentset{\star}{\mathfrak{X}}^n_{i,j}$ be the rectangle obtained from $\mathfrak{X}^n_{i,j}$ by adding all bounded components of $\C\setminus(\overline{Z\cup\partial^h\mathfrak{X}^n_{i,j}})$. Similarly one can define $\accentset{\star}{\mathfrak{X}}^{n,\mathrm{inn}}_{i,j}$ and $\accentset{\star}{\mathfrak{X}}^{n,\mathrm{out}}_{i,j}$.

\begin{lem}\label{lem.Y-filledin}
    If ${\BDelta}\gg1$, then $\W(\accentset{\star}{\mathfrak{X}}^n_{i,j})\geq{\BDelta}/2$; $\W(\accentset{\star}{\mathfrak{X}}^{n,\mathrm{inn}}_{i,j}),\W(\accentset{\star}{\mathfrak{X}}^{n,\mathrm{out}}_{i,j})\geq{\BDelta}/6$.
\end{lem}
\begin{proof}
We only prove for $\accentset{\star}{\mathfrak{X}}^n_{i,j}$ and it will be the same for the others. We proceed the proof by induction on $n$. If $n$ is the largest integer such that $i(n)>0$ (recall $i(n)$ in Assumption \ref{assump.1}), then $\mathfrak{X}^n_{i,j} = \accentset{\star}{\mathfrak{X}}^n_{i,j}$. Suppose that the lemma is proved for the case $n\geq k+1$. Let us prove the case $n=k$. Let $D^{n_l}$ be any bounded connected component of $\C\setminus(\overline{Z\cup\partial^h\mathfrak{X}^n_{i,j}})$. Then $\partial D^{n_l}\setminus\partial Z = \alpha^{n_l}$ for some level-$n_l$ dam $\alpha^{n_l}$ with $n_l\geq k+1$. Notice that $D^{n_l}$ is protected by some rectangle $\Y^{n_l}$ given by Assumption \ref{assump.6}. By induction, $\W(\accentset{\star}{\mathfrak{X}}^{n_l})\geq{\BDelta}/2$. Moreover, we have $\mathfrak{X}^{n_l}\subset\mathfrak{X}^n_{i,j}$ from Assumption \ref{assump.6}. Therefore $\accentset{\star}{\mathfrak{X}}^{n_l}\subset\accentset{\star}{\mathfrak{X}}^n_{i,j}$ and (\ref{eq.ABh}) of Lemma \ref{lemappen.outer-protect} is applicable to $\R = \accentset{\star}{\mathfrak{X}}^n_{i,j}$, $\R' = \mathfrak{X}^n_{i,j}$. Hence $\W(\accentset{\star}{\mathfrak{X}}^n_{i,j})>(1-\epsilon)\cdot\W(\mathfrak{X}^n_{i,j})\geq {\BDelta}/2$.
\end{proof}

Let $\Z^{\mathrm{inn}}(\beta^n_{i,j})$ be the rectangle given by Assumption \ref{assump.4} with $n\geq m$. Recall that $\W(\Z^{\mathrm{inn}}(\beta^n_{i,j}))\asymp_{\Bdelta}1$  Similarly, let $\accentset{\star}{\Z}^{\mathrm{inn}}(\beta^n_{i,j})$ be the rectangle obtained from $\Z^{\mathrm{inn}}(\beta^n_{i,j})$ by adding all bounded components of $\C\setminus(\overline{Z\cup\partial^h \Z^{\mathrm{inn}}(\beta^n_{i,j})})$. Using the same argument as Lemma \ref{lem.Y-filledin}, we can prove:
\begin{lem}\label{lem.Z-filledin}
    If ${\BDelta}\gg1$, then $\W(\accentset{\star}{\Z}^{\mathrm{inn}}(\beta^n_{i,j}))\succeq_{\Bdelta}1$.
\end{lem}

For well-grounded intervals, the error on the outer width can be controlled independently of $\Bdelta$ when passing from $\partial Z$ to $\partial\hZ^m_{\bfS}$: 
\begin{thm}\label{thm.wellgrounded-outer-geometry}
    For every $\epsilon>0$ and ${\BDelta}\gg_\epsilon 1$ the following holds. For every pseudo-Siegel disk $\hZ^m_{\bfS}$ and every pair of well-grounded $\mathrm{rel}$ $\hZ^m_{\bfS}$ intervals $I,J\subset\partial Z$, we have
    $$1-\epsilon<\frac{\W^+_{\hZ^m_{\bfS}}(I^{m,\bfS},J^{m,\bfS})}{\W^+_{Z}(I,J)}<1+\epsilon,$$
    where $I^{m,\bfS},J^{m,\bfS}$ are the projections of $I,J$ to $\partial\hZ^m_{\bfS}$ (Definition \ref{def.project-intervals}).
\end{thm}
\begin{proof}
    Let us view $\widehat{\C}\setminus{Z}$ as a rectangle $\R$ with horizontal sides $I,J$ and $\widehat{\C}\setminus{\hZ^m_{\bfS}}$ as a rectangle
 $\R'$ with horizontal sides $I^{m,\bfS},J^{m,\bfS}$. Since $I,J$ are well-grounded (Definition \ref{defn.grnd-interval}), every connected component of $\R\setminus\R'$ is a Jordan disk $D_n$ protected by some $\accentset{\star}{\mathfrak{X}}^n_{i,j}$ with $n\geq m$. By Lemma \ref{lem.Y-filledin}, $\W(\accentset{\star}{\mathfrak{X}}^n_{i,j})\geq{\BDelta}/2$. Hence (\ref{eq.ABv}) of Lemma \ref{lemappen.outer-protect} is applicable, and the theorem follows.
\end{proof}

For grounded intervals, the error on the outer width can be controlled with respect to $\Bdelta$ when passing from $\partial Z$ to $\partial\hZ^m_{\bfS}$:
\begin{thm}\label{thm.grounded-outer-geometry}
 For every $\epsilon>0$ and ${\BDelta}\gg_\epsilon 1$ the following holds. For every pseudo-Siegel disk $\hZ^m_{\bfS}$ and every pair of grounded $\mathrm{rel}$ $\hZ^m_{\bfS}$ intervals $I,J\subset\partial Z$ such that $\dist(I,J)>3\min\{|I|,|J|\}$, then
    $$\W^+_{Z}(I,J)-O_{\Bdelta}(1)<\W^+_{\hZ^m_{\bfS}}(I^{m,\bfS},J^{m,\bfS})<(1+\epsilon)\cdot\W^+_{Z}(I,J)+O_{\Bdelta}(1),$$
    where $I^{m,\bfS},J^{m,\bfS}$ are projections of $I,J$ onto $\partial\hZ^m_{\bfS}$. 
\end{thm}
\begin{proof}
 We view $\widehat{\C}\setminus{Z}$ as a rectangle $\R$ with horizontal sides $I,J$ and we view $\widehat{\C}\setminus{\hZ^m_{\bfS}}$ as a rectangle
 $\R'$ with horizontal sides $I^{m,\bfS},J^{m,\bfS}$. Since $I,J$ are grounded (Definition \ref{defn.grnd-interval}), every connected component of $\R\setminus\R'$ is a Jordan disk $D_n$ attached to a vertical or horizontal side of $\R$, such that $\partial D_n\setminus(\partial D_n\cap\partial\R) = \alpha^n$, where $\alpha^n$ is some level $n$ dam with $n\geq m$. 

\paragraph{Step 1.}
Let $D_n$ be any Jordan disk as above attached to a vertical side of $\R$. Since $I,J$ are grounded, $\partial^h\accentset{\star}{\Z}^{\mathrm{inn}}(\alpha^n)$ is contained in the same vertical side of $\R$ as $D_n$. Let $\R_1\subset\R$ be the rectangle by adding all of these $D_n$ to $\R'$. On the other hand, $\W(\accentset{\star}{\Z}^{\mathrm{inn}}(\alpha^n))\succeq_{\Bdelta}1$ by Lemma \ref{lem.Z-filledin}. Therefore (\ref{eq.ABv}) of Lemma \ref{lemappen.outer-protect} is applicable to $\R_1,\R'$ and we get that
\begin{equation}\label{eq.R'-R_1}
    \W(\R_1)-O_{\Bdelta}(1) < \W(\R')\leq\W(\R_1).
\end{equation}

\paragraph{Step 2.}
Next consider any $D_n$ that is attached to a horizontal side of $\R$ (or equivalently, $\R_1$).
 We may assume that $|I|\leq |J|$. Let $k\geq-1$ be the unique integer such that $\l_k>|I|\geq\l_{k+1}$. 
    \begin{itemize}
        \item Assume that $\alpha^n\subset I^{m,\bfS}$. Recall that $\alpha_n$ is protected by $\mathfrak{X}^n$ given by Assumption \ref{assump.6}. Let $a^n<b^n$ be the unique pair of adjacent points in $\mathrm{CP}_n^{\mathrm{full}}$ that bounds $\alpha^n$ and $\mathfrak{X}^n$ (see Assumption \ref{assump.2}). Then by (\ref{eq.dam-reclaim-distance}) of Assumption \ref{assump.2}, 
        \begin{equation}\label{eq.dist(I,J)-Y^n}
            \dist(I,J)\geq 3|I|>3|\alpha^n|\geq3\cdot\frac{4}{5}|\lfloor a^n,b^n\rfloor|>|\lfloor\partial^h\mathfrak{X}^n\rfloor|.
        \end{equation}
       Hence, $\partial\mathfrak{X}^n$ is disjoint either from $I^{m,\bfS}$ or $J^{m,\bfS}$ and its filled-in $\partial\accentset{\star}{\mathfrak{X}}^n$ is disjoint either from $I$ or $J$. On the other hand, $\W(\partial\accentset{\star}{\mathfrak{X}}^n)\geq{\BDelta}/2$ by Lemma \ref{lem.Y-filledin}.
         \item Assume that $\alpha^n\subset J^{m,\bfS}$. There are three subcases:
         \begin{itemize}
             \item  If $n>k$, then $|I|\geq\l_{k+1}>|\alpha^n|$. Hence (\ref{eq.dist(I,J)-Y^n}) still holds and we are back to the above case. 
             \item  If $n=k$, then at most two $\mathfrak{X}^n$ intersect $I^m$. The two associated $D_n$ are separated from $I$ by $\accentset{\star}{\Z}^{\mathrm{inn}}(\alpha^n)$, whose width is $\succeq_{\Bdelta}1$ by Lemma \ref{lem.Z-filledin}. All the other $D_n$ are protected by $\accentset{\star}{\mathfrak{X}}^n$ with $\W(\accentset{\star}{\mathfrak{X}}^n)\geq{\BDelta}/2$ by Lemma \ref{lem.Y-filledin}. 
             \item If $n<k$, let's recall the rectangles $\mathfrak{X}^{n,\mathrm{inn}},,\mathfrak{X}^{n,\mathrm{out}}$ given by Assumption \ref{assump.6}. Since $n<k$, $$\partial^{h,0}\mathfrak{X}^{n,\mathrm{inn}},\partial^{h,1}\mathfrak{X}^{n,\mathrm{inn}},\partial^{h,0}\mathfrak{X}^{n,\mathrm{out}},\partial^{h,1}\mathfrak{X}^{n,\mathrm{out}}\geq\l_{n+1}\geq|I|.$$ 
             Hence either $I^m\subset \lfloor\partial^h\mathfrak{X}^{n,\mathrm{inn}}\rfloor$ or $I^m\cap\lfloor \partial^h\mathfrak{X}^{n,\mathrm{inn}}\rfloor = \emptyset$.  There are two further subcases:
             \begin{itemize}
             \item if $I^m\cap\lfloor \partial^h\mathfrak{X}^{n,\mathrm{inn}}\rfloor = \emptyset$, then $D_n$ is separated from $I$ by $\accentset{\star}{\mathfrak{X}}^{n,\mathrm{inn}}$ whose width is greater than ${\BDelta}/6$ by Lemma \ref{lem.Y-filledin}.
             \item if $I^m\subset\lfloor \partial^h\mathfrak{X}^{n,\mathrm{inn}}\rfloor$, then $n_0<k$ be the largest integer among these $n$. Clearly there is a unique $\mathfrak{X}^{n_0,\mathrm{inn}}$ such that $I^m\subset\lfloor \partial^h\mathfrak{X}^{n_0,\mathrm{inn}}\rfloor$. Therefore, the Jordan disk $D_{n_0}$ associated to $\mathfrak{X}^{n_0,\mathrm{inn}}$ is separated from $I$ by $\accentset{\star}{\Z}^{\mathrm{inn}}(\alpha^{n_0})$ whose width is $\succeq_{\Bdelta}1$ by Lemma \ref{lem.Z-filledin}. All the other $D_n$ are separated from $I$ by $\accentset{\star}{\mathfrak{X}}^{n_0,\mathrm{out}}$ with $\W(\accentset{\star}{\mathfrak{X}}^{n_0,\mathrm{out}})\geq {\BDelta}/6$ by Lemma \ref{lem.Y-filledin}.
         \end{itemize}
         \end{itemize}
    \end{itemize}
     Now $\R$ is the rectangle by filling in all the $D_n$ discussed in \textbf{Step 2} to $\R_1$ ($D_n$ is associated to $\alpha^n$). Hence (\ref{eq.AB'h}) of Lemma \ref{lemappen.outer-protect} is applicable to $\R,\R_1$ with $|S| = 3$:
        $$\W(\R)\leq \W(\R_1)\leq(1+\epsilon)\cdot\W(\R)+O_{\Bdelta}(1).$$
        Combining this with (\ref{eq.R'-R_1}), we get the desired estimate in the lemma.
\end{proof}

\subsection{From shortcut to combinatorial metric}\label{ss:ShortCut to Comb}

The following lemma says that if $I$ and $J$ are $\lambda$-separated with respect to the shortcut metric and $\cal{F}^+_{\hZ^m_{\bfS}}(I,J)$ is wide, then up to splitting $I$ into finitely many intervals, we can assume that $\Width^+_\lambda(I)$ is wide with respect to the combinatorial metric.

\begin{lem}[$\Width^+_{\lambda, m,\bfS}\leadsto \Width^+_{\lambda}$]\label{lem:ShortCut to Comb}
    Let $I,J\subset\partial\hZ^m_{\bfS}$ be grounded intervals and $\lambda >2$. Suppose that
    $$\dist_{m,\hZ^m_{\bfS}}(I,J)\geq \lambda |J|_{m,\bfS},\quad  \W^+_{\hZ^m_{\bfS}}(I,J)=: K\gg_{\Bdelta}1,$$
    then there exists $I^{\mathrm{new}}\subset J^\bullet\subset\partial Z$ grounded $\mathrm{rel}$ $\hZ^m_{\bfS}$ such that $\W^+_{\frac{\lambda}{3}}(I^\mathrm{new})\succeq K$. Here $J^\bullet$ is the projection of $J$ to $\partial Z$.
\end{lem}
\begin{proof}

    Let us consider two cases:
    \begin{itemize}
        \item if $|J|\geq\l_m$, then by Lemma \ref{lem.shortcut-projective-comparable}, we can find a grounded interval $J'$ $\mathrm{rel}$ $\hZ^m$ such that $$J\subset J',\quad \frac{\lambda}{3}J^\bullet\subset J'^{\bullet}.$$
        Therefore we have $\W^+_{\hZ^m}(J,\partial\hZ^m\setminus J')\geq\frac{1}{3}K$. By Theorem \ref{thm.grounded-outer-geometry}, we have $\W^+_{\frac{\lambda}{3}}(J^\bullet)\succeq K$. Therefore we can just take $I^{\mathrm{new}} = J^\bullet$.
        \item if $|J|<\l_m$, then let $m_1\geq m+1$ to the smallest integer such that $|J|\geq\l_{m_1}$. We can then write $J$ into a concatenation $$J = J_1\#\widehat{
        \beta}_1\#J_2\#...\#J_d\#\widehat{\beta}_d\# J_{d+1},$$ 
        where $\widehat{\beta}_i$ are level-$(m_1-1)$ thickened dams (recall (\ref{eq.fat.dam})) and $J_i$ do not intersect any level-$(m_1-1)$ thickened dam. Since $|J|<\l_{m_1-1}$, we have $d\leq s$ (where $s$ is the number of critical points of $f$). Since $\widehat{\beta}_i$ are protected by outer annulus with modulus $\succeq_{\Bdelta}1$ (Assumption \ref{assump.4}), the width in $\W^+_{\hZ^m}(I,J)$ is concentrate on $\W^+_{\hZ^m}(I,J_1\cup...\cup J_{d+1})$ up to $O_{\Bdelta}(1)$. Hence there exists some $J_i$ such that $\W^+_{\hZ^m}(I,J_i)\succeq K$. Since $J_i$ does not intersect any level-($m_1-1$) dam, we have $(\frac{\lambda}{3} J_i^\bullet)^{m,\bfS}\cap I = \emptyset$. Therefore, $\W^+_{\hZ^m}(J_i,\partial\hZ^m\setminus (\frac{\lambda}{3} J_i^\bullet)^{m,\bfS})\geq\frac{1}{3}K$. By Theorem \ref{thm.grounded-outer-geometry}, we have $\W^+_{\frac{\lambda}{3}}(J_i^\bullet)\succeq K$. Therefore we can just take $I^{\mathrm{new}} = J_i^\bullet$. 
    \end{itemize}
\end{proof}

\subsection{Bistability of $\wZ^m$}\label{sss:dfn:bistable} 
Given a pseudo-Siegel disk $\wZ^m$, its bistability $\text{bistab}(\wZ^m)$ is the largest number $k$ so that the induced image $\wZ^m_{*i} \coloneqq f^i_*(\wZ^m)$ (see Remark~\ref{rem.forward-backward-image} and \eqref{eqn:pushforward}) is also a pseudo-Siegel disk for all $|i| \leq k\q_{m+1}$, where all auxiliary objects for the pseudo-Siegel disks are also pushed forward to produce the auxiliary objects for $\wZ^m_{*i}$.

We remark that pseudo-Siegel disks $\wZ^m$ constructed in this paper have bistability $\ge e^{\sqrt \mathbf K}$, where $\mathbf K\gg 1$ is a big constant representing $O(1)$. In other words, pseudo-Siegel disks naturally come with arbitrarily large (by increasing $\mathbf K$) bistability; see \cite[\S 3.5]{DLL}.
We also remark that in the main application of the bistability (Amplification Lemma~\ref{lem.amplification}), only $2$-bistability is required.

 We remark that there is an alternative way of constructing the pseudo-Siegel structure for $f_*^i(\wZ)$, with $|i| \leq k \q_{m+1}$.
Indeed, by extending the concept of almost-invariance of the pseudo-Siegel disk $\wZ^m$ from $f^{\q_{m+1}}$ to $f^{k\q_{m+1}}$ for any fixed $k\in \mathbb{N}$ (see \cite[\S 3]{DL:wZ}), we might select the auxiliary objects of $\wZ^m$ so that they also give the auxiliary objects of the pseudo-Siegel structure for the image $f_*^i(\wZ)$ for all $|i| \leq k \q_{m+1}$.

\section{Snakes} \label{s:Snake+}

In this section, we review the framework of Snake Lemmas to make it compatible with the multicritical setting. Particularly, we refine the notion of \emph{groundedness}. The Snake framework analyzes and encodes the external geometry of pseudo-Siegel disks $\wZ^m$ relative to their inner geometry. Roughly:
\begin{enumerate}[label=\textnormal{(\Alph*)}]
    \item\label{item:A:s:snakes+} If a sufficiently wide lamination (``snake'') $\LL$ submerges into $\wZ^m$, then a $\Width^+_{\lambda}$-degeneration parameter can be detected of order $\succeq \Width(\LL)$. Moreover, the comparison ``$\succeq$'' is independent of $\lambda$ -- it affects only the degeneration threshold for being ``sufficiently wide''.
    \item\label{item:B:s:snakes+}  If, in addition, $\LL$ submerges $n$-times in $\wZ^m$, then the $\Width^+_{\lambda}$-degeneration will be of order $\succeq n \Width(\LL)$.
\end{enumerate}
We recall that the inner geometry of $\wZ^m$ is controlled by the Log-Rule with respect to the shortcut metric (labeled by ``$m,\bfS$'' in subindex). The shortcut and combinatorial metrics are equivalent on scales $\ge \l_m$. On a deeper scale, by splitting an interval, Lemma~\ref{lem:ShortCut to Comb} allows us to replace the shortcut metric with the usual combinatorial metric.

The framework of snakes is organized as follows. A \emph{grounded snake} is a sufficiently wide lamination representing ${\Width^\circ\setminus \Width^+}$ subject to the \emph{groundedness condition}, see the beginning of~\S\ref{ss:SL:IndStep} for a summary. In~\S\ref{ss:SL:IndStep}, we state Snake Induction Lemma~\ref{lem.snake}: given a snake, it produces an almost twice wider lamination representing $\Width^\circ_\lambda$ on a deeper scale and still satisfying the grounded condition. The new lamination either has a substantial outer part representing $\Width^+_\lambda$ or is represented by a new \emph{wider} grounded snake, and the process can be iterated. Since $\wZ^m$ is a non-uniformly qc disk, by inductively running Lemma~\ref{lem.snake}, any snake representing $\Width^\circ\setminus \Width^+$ induces $\Width^+_\lambda$-degeneration on a deeper scale. This inductive process is stated as Corollary~\ref{cor.snakelemma}.

Lemma~\ref{lem.trade-W-to-W+} ($\Width\leadsto \Width^+$) is a more convenient version of Corollary~\ref{cor.snakelemma}, where $\Width^\circ$ is replaced with a full $\Width$-degeneration subject to appropriate conditions. The input to Lemma~\ref{lem.trade-W-to-W+} is assumed to be on scale $\ge \l_m$. This scale condition is used to satisfy the initial grounded condition, then Corollary~\ref{cor.snakelemma} is applied. This concludes Item~\ref{item:A:s:snakes+}.

In~\S\ref{ss:SnakesWithBarriers}, we state the Snake Lemma with toll barriers and the Sneaking Lemma. Toll barriers form a notion formalizing multiple ``$n$-times'' submergence into $\wZ^m$, see Figure~\ref{fig:TollBarr}. Lemma~\ref{lem:SL:toll_barriers} ($\Width \ \overset{\times n}{\leadsto}\ \Width^+$) repeats the argument of Lemma~\ref{lem.trade-W-to-W+} but with initial amplification by $n$. Sneaking Lemma~\ref{lem:sneakI}, illustrated in Figure~\ref{Fg:snake thorugh WR}, either constructs toll barriers or squeezes a substantial part of a wide rectangle $\R$. In both cases, appropriate amplification by $n$ is achieved. This concludes Item~\ref{item:B:s:snakes+}.

\subsection{Snakes}\label{ss:GroundSnake} 

In this subsection, we formalize the notion of snakes. Topological snakes $\R$ as in \eqref{eq:dfn:Snake} arise naturally from the topological Series Decomposition (Appendix~\ref{appen.series-decomp}). Such a snake $\R$ is within a family $\Fam^\circ$ specified in~\S\ref{sss:Fam:Circ}. It is important to ground this topological snake $\R\leadsto \R^\grnd$, see~\eqref{eq:grnd:part:Snake}, so that Localization Lemma~\ref{lem.localization} can be applied to $\R^\grnd$. Property~\ref{property:X} in~\S\ref{sss:Fam:Circ} defines when the ``difference'' between $\R$ and $\R^\grnd$ is negligible. This property is stated so that it is stable under the overflowing operation,~\S\ref{sss:prop:X:overflow}.

In~\S\ref{sss:Cond:N}, we state Combinatorial Condition~\ref{item:cond:N} that implies Property~\ref{property:X} (Corollary~\ref{cor:Prop:X}). The implication relies on Lemma~\ref{lmm:N:meaning} -- this is the only place where outer annuli $A^{\mathrm{out}}(z)$ are used. Condition~\ref{item:cond:N} persists under considering small intervals; this will be used in Snake induction Lemma~\ref{lem.snake}.

\subsubsection{Family $\Fam^\circ$}\label{sss:Fam:Circ}
Consider a pseudo-Siegel disk $\wZ^m$ and two disjoint intervals $I,J\subset \partial \wZ^m$. Let 
\begin{itemize}
    \item $L$ be one of the two intervals in $\partial \wZ^m\setminus (I\cup J)$;
    \item $N = \partial \wZ^m\setminus (I\cup L\cup J)$ be the second interval;
    \item $I^+,J^+$ be the outer sides of $I,J$, i.e., $I^+,J^+$ are part of the Caratheodory boundary of $\wC\setminus \wZ^m$;
    \item  $\F^\circ_{L,\hZ^m}(I,J)$ be the family of all curves in $\wC\setminus N$ connecting $I^+$ and $J^+$. 
\end{itemize}

A (topological) \emph{snake in $\F^\circ_{L,\hZ^m}(I,J)$} is a lamination 
\begin{equation}
   \label{eq:dfn:Snake}
   \R\subset\F^\circ_{L,\hZ^m}(I,J).
\end{equation}

If for a snake $\R$ in~\eqref{eq:dfn:Snake}, the intervals  $I,J,L\subset \partial \wZ^m$ are grounded, then $\R$ is called a \emph{grounded snake}.

For general $I,J,L$, let 
\begin{itemize}
    \item $I^{\mathrm{grnd}}, J^{\mathrm{grnd}}$ be the biggest grounded interval within $I,J$,
    \item  $L^{\mathrm{GRND}}$ be the smallest grounded interval containing $L$.
\end{itemize}

For a snake $\R$ as in~\eqref{eq:dfn:Snake}, let $\R^{\mathrm{grnd}}$ be the set of curves in $\R$ that are disjoint from $I\setminus I^{\mathrm{grnd}}$ and from $J\setminus J^{\mathrm{grnd}}$. In other words, $\R^{\mathrm{grnd}}$ is the maximal sublamination of $\R$ such that
\begin{equation}
   \label{eq:grnd:part:Snake}\R^{\mathrm{grnd}}\subset\F^\circ_{L^{\mathrm{GRND}},\hZ^m}(I^{\mathrm{grnd}},J^{\mathrm{grnd}})
\end{equation}

We are interested in the following property (cf., a similar property in~\cite[\S 6.1]{DL:wZ}):

\vspace{0.2cm}

{\bf Property} {\setword{$\operatorname{(X)}$}{property:X}}. Assume that a lamination $\LL$ overflows a snake $\R \subset \F^\circ_{L,\hZ^m}(I,J)$ as in~\eqref{eq:dfn:Snake}. We say a pair $\big(\LL, \R\subset \F^\circ_{L,\hZ^m}(I,J)\big)$ has Property~\ref{property:X}, if after removing $O_{\Bdelta}(1)$ curves from $\LL$ in the new resulting pair $\big(\LL^{\mathrm{new}}, \R^{\mathrm{new}}\big)$, the $\R^{\mathrm{new}}$ is a sublamination of $\R^{\mathrm{grnd}}$.

\vspace{0.2cm}

In other words, Property~\ref{property:X} means that, by removing $O_{\Bdelta}(1)$ curves from a lamination $\LL$ overflowing a snake $\R$, the snake becomes grounded.

In applications, $\LL$ is either a snake constructed on the previous induction step (see Lemma~\ref{lem.snake}) or a lamination representing most of the width of a certain degeneration parameter. For the latter, $\LL$ will be on renormalization level $m$ or shallower; Property~\ref{property:X} will follow from the combinatorial separation. For the former, Property~\ref{property:X} will follow from the combinatorial Condition~\ref{item:cond:N}; see \S\ref{sss:Cond:N} and~\S\ref{sss:XfromN}.

\subsubsection{Stability of Property~\ref{property:X} under overflowing}\label{sss:prop:X:overflow} By design, Property~\ref{property:X} holds if $\RR$ is replaced by any lamination overflown by $\RR$.

\subsubsection{Combinatorial Condition~\ref{item:cond:N} on $(X,M,Y)$}\label{sss:Cond:N} Assume that $X,Y$ are grounded. Thus complementary intervals $M,N$ forming $\hZ^m\setminus(X\cup Y)$ are also grounded.  We assume that $N$ is combinatorially substantial in the following sense:

\begin{itemize}
    \item[{\setword{$\operatorname{(N)}$}{item:cond:N}}] The triple $(X,M,Y)$, where $X,Y$ are grounded, satisfies Condition~\ref{item:cond:N} if one of the following holds:
    \begin{itemize}[label = {-}]
        \item[{\setword{$\operatorname{(N1)}$}{item:cond:N1}}] $|X|\geq\l_m$ and  $|Y| \geq \l_m$;
        \item[{\setword{$\operatorname{(N2)}$}{item:cond:N2}}] $N$ is not contained in any $S^\mathrm{out}(\hZ^m)$ (recall Definition \ref{defn.grnd-interval}) of level $\ge m$.
    \end{itemize}
    For a lamination $\R\subset\F_M(X,Y)$, we simply say that $\R$ satisfies Condition~\ref{item:cond:N} if $(X,M,Y)$ does.
\end{itemize}

Combinatorial nature of Condition~\ref{item:cond:N} is shown in the following obvious lemma:

\begin{lem}
\label{lmm:N:meaning}
    Assume Condition~\ref{item:cond:N} holds for $(X,M,Y)$. Then for every dam $\alpha\subset M$ of level $\geq m$ and each of its endpoint $z\in \alpha$, the $S^{\mathrm{out}}(z)$ is disjoint from $X$ or from $Y$.

    Consequently, if a curve $\gamma\in \F(X,Y)$ intersects $S^{\mathrm{inn}}(z)$, then $\gamma$ crosses $A^{\mathrm{out}}(z)$. 
\end{lem}
\begin{proof}
The lemma is obvious under Condition~\ref{item:cond:N1}. So assume Condition~\ref{item:cond:N2}. If $S^{\mathrm{out}}(z)$ is not disjoint from $X$, then since $N\not\subset S^{\mathrm{out}}(z)$ and $Y$ is grounded, we get $S^{\mathrm{out}}(z)\cap Y = \emptyset$.
\end{proof}

Lemma~\ref{lmm:N:meaning} implies the following fact:
\begin{lem}
\label{lmm:N:applic}
    Assume Condition~\ref{item:cond:N} holds for $(X,M,Y)$. Then for every interval $S\subseteq M$, the width of curves in  $\F(X,Y)$ intersecting $S\setminus S^\mathrm{grnd}$ is at most $O_{\Bdelta}(1)$.
\end{lem}
\begin{proof}
    By definition of grounded intervals (Definition \ref{defn.grnd-interval}), the difference $(S\setminus S^{\mathrm{grnd}})$ has at most two connected components $S^1<S^2$ and each one of them is surrounded by an annulus $\Aout(\alpha^{m_1}_1),\Aout(\alpha^{m_2}_2)$ (recall Assumption \ref{assump.4}), where $\alpha^{m_1}_1,\alpha^{m_2}_2$ are level-$m_1$, $m_2$ dams respectively ($m_1,m_2\geq m$). If $X$ is separated from both of $S^1,S^2$ by $\Aout(\alpha^{m_1}_1),\Aout(\alpha^{m_2}_2)$ respectively, then the lemma follows since curves starting from $X$ and ending at $S^1\cup S^2$ must cross $\Aout(\alpha^{m_1}_1),\Aout(\alpha^{m_2}_2)$ and the modulus of these two annulus is $\succeq_{\Bdelta}1$. If not, then we can assume that $X\cap\Aout(\alpha^{m_1}_1)\not = \emptyset$. Let $x<y$ be the two end points of the dam $\alpha^{m_1}_1$. Let us write $S^1$ into concatenation $S^1 = S_x\cup\widehat{\alpha}^{m_1}_1\cup S_y$, where $S_x\subset S^{\mathrm{inn}}(x)$, $S_y\subset S^{\mathrm{inn}}(y)$ can be empty and $\widehat{\alpha}^{m_1}_1 \subseteq\alpha^{m_1}_1$. By Lemma \ref{lmm:N:meaning}, curves in $\F(X,Y)$ intersecting $S_x\cup S_y$ must cross annuli $A^{\mathrm{out}}(x)$ or $A^{\mathrm{out}}(y)$. On the other hand, curves in $\F(X,Y)$ intersecting $\alpha^{m_1}_1$ must cross $A^{\mathrm{inn}}(\alpha^{m_1}_1)$. All these three annuli have mudulus $\succeq_{\Bdelta}1$ by Assumption \ref{assump.4}. This proves the lemma.
\end{proof}

\subsubsection{Property~\ref{property:X} from Condition~\ref{item:cond:N}}\label{sss:XfromN} Assume now that 
\begin{enumerate}[label=(\arabic*)]
    \item\label{item:1:sss:XfromN} $\LL$ is a lamination in $\F(X,Y)$, where $X,Y$ are grounded;
    \item $\LL$ overflows a snake $\R$ in $\F^\circ_{\hZ^m}(I,J)$;
    \item either $I\subset M$ or $I\in \{X,Y\}$; similarly
    \item either $J\subset M$ or $J\in \{X,Y\}$; and
    \item\label{item:5:sss:XfromN}  Condition~\ref{item:cond:N} holds for $(X,M,Y)$.
\end{enumerate}
It easily follows from Lemma~\ref{lmm:N:applic}:
\begin{cor}\label{cor:Prop:X}
Under the above Assumptions~\ref{item:1:sss:XfromN}--\ref{item:5:sss:XfromN}, Property~\ref{property:X} holds for $\big(\LL, \R\subset \F^\circ_{L,\hZ^m}(I,J)\big)$. 
\end{cor}

\subsection{Snake Lemma: induction step}\label{ss:SL:IndStep} As outlined at the beginning of Section~\ref{s:Snake+}, the framework of snakes relies on the following lemma. Its proof relies on the topological Series Decomposition (Appendix~\ref{appen.series-decomp}), Localization Lemma~\ref{lem.localization}, and the fact that Condition~\ref{item:cond:N} persists under considering smaller intervals.

\begin{snakelmm}\label{lem.snake}
    Let $\LL\subset \F^\circ_{L,\hZ^m}\setminus \Fam^+$ be a submerging grounded snake as in~\eqref{eq:dfn:Snake} that satisfies Property~\ref{property:X}.
    
    If $K:=\Width(\LL)\gg_{\Bdelta}\log\lambda$ with $\lambda>2$, then there exist disjoint grounded intervals $\mathrm{rel}$ $\hZ^m$
    \begin{equation}\label{eq.lambda-separation}
        J_1,I_1\subset L,\quad |J_1|_{m}<\frac{\dist_{m}(I,J_1)}{\lambda},\quad |I_1|_{m}<\frac{\dist_{m}(I_1,J)}{\lambda},\quad I<J_1<I_1<J
    \end{equation}
    such that
\begin{equation}\label{eq.harmonicsum-submerge}
        \W^\circ_{L_I,\hZ^m}(I,J_1)\oplus\W^\circ_{L_J,\hZ^m}(I_1,J)\geq K-O_{\Bdelta}(\log\lambda),
    \end{equation}
    where $L_I,L_J\subset L$ are the complementary intervals between $I,J_1$ and $I_1,J$ respectively. Moreover, if $(I,L,J)$ satisfies Condition~\ref{item:cond:N}, then both of $(I,L_I,J_1)$ and $(I_1,L_J,J)$ satisfy Condition~\ref{item:cond:N2}.
\end{snakelmm}
\begin{proof}
  Applying Series Decomposition (Lemma \ref{lemappen.series.decomposition}) to $\LL$, we get that $\F(\LL)$ overflows the laminations
    
 \begin{equation}
    \label{eq:1:lem.snake:prf}
    \F_a\subset\F^\circ_{L_a,\hZ^m}(I,J_a),\quad \Gamma\subset\F^-_{\hZ^m}(J_a,I_b),\quad \F_b\subset\F^\circ_{L_b,\hZ^m}(I_b,J),
  \end{equation}   
    where $J_a,I_b\subset L$ and $L_a,L_b\subset L$ are the complementary interval of $I\cup J_a$, $I_b\cup J$ respectively such that $I<L_a<J_a$, $I_b<L_b<J$. Let $J_a^{\mathrm{grnd}}\subset J_a$, $ I_b^{\mathrm{grnd}}\subset I_b$ be the largest grounded sub-intervals $\mathrm{rel}$ $\hZ^m$.

By Property~\ref{property:X} (it follows from Corollary~\ref{cor:Prop:X}), if we set $\Lnew\subset\LL$ to be the subfamily of curves not intersecting $(J_a\setminus J_a^{\mathrm{grnd}})\cup(I_b\setminus I_b^{\mathrm{grnd}})$, then $\W(\Lnew) = K-O_{\Bdelta}(1)$.

By Localization Lemma~\ref{lem.localization}, there exists an innermost pair of intervals $J_1'\subset J_a^{\mathrm{grnd}}$, $I_1'\subset I_b^{\mathrm{grnd}}$ such that
\begin{equation}\label{eq.prelambda-separation}
    |J_1'|_{m}, |I_1'|_{m} \leq \frac{1}{\lambda}\min\{|J_a^{\mathrm{grnd}}|_{m},|I_b^{\mathrm{grnd}}|_{m}\}
\end{equation}
and
\begin{equation}\label{eq.throwaway-Rnew}
    \W^-_{\hZ^m}(J_a^{\mathrm{grnd}}\setminus J_1',I_b)+\W^-_{\hZ^m}(J_a^{\mathrm{grnd}},I_b\setminus I_1') = O_{\Bdelta}(\log\lambda).
\end{equation}

Let $\LNew\subset\Lnew$ be the subfamily of curves $\gamma$ such that both $\gamma_a^d,\gamma_b^d\in\F^-_{\hZ^m}(J_1',I_1')$ (recall Series Decomposition in Appendix \ref{appen.series-decomp}, see Figure~\ref{Fig:SeriesDecomp}). Therefore by (\ref{eq.throwaway-Rnew}), 
\begin{equation}\label{eq.WRNew}
    \W(\LNew) = K-O_{\Bdelta}(\log\lambda)
\end{equation}
 Apply Series Decomposition (Lemma \ref{lemappen.series.decomposition}) to $\LNew$, we get intervals $J_a^{\mathrm{New}}\subset J_1'$, $I_b^{\mathrm{New}}\subset I_1'$ such that $\LNew$ overflows the following sublaminations:
$$\F_a^{\mathrm{New}}\subset\F^\circ_{L_a^{\mathrm{New}},\hZ^m}(I,J_a^{\mathrm{New}}),\quad\F_b^{\mathrm{New}}\subset\F^\circ_{L_b^{\mathrm{New}},\hZ^m}(I_b^{\mathrm{New}},J),\quad\Gamma^{\mathrm{New}}\subset\F^-_{\hZ^m}(J_1',I_1'),$$
where $L_a^{\mathrm{New}}$and $L_b^{\mathrm{New}}$ are the complementary interval between $I,J_a^{\mathrm{New}}$ and $I_b^{\mathrm{New}},J$ respectively. Finally, let $$J_1 := (J_a^{\mathrm{New}})^\mathrm{grnd},\quad I_1 := (I_b^{\mathrm{New}})^\mathrm{grnd}.$$
Again by Property~\ref{property:X}, we have
\begin{equation}\label{eq.throwaway-RNew}
    \W^-_{\hZ^m}(J_a^{\mathrm{New}}\setminus J_1,I_b^{\mathrm{New}})+\W^-_{\hZ^m}(J_a^{\mathrm{New}},I_b^{\mathrm{New}}\setminus I_1) = O_{\Bdelta}(1).
\end{equation}
Let $\LNEW\subset\LNew$ be the subfamily of curves $\gamma$ such that both $\gamma_a^d,\gamma_b^d\in\F^-_{\hZ^m}(J_1,I_1)$. Hence by (\ref{eq.WRNew}) and (\ref{eq.throwaway-RNew}), $\W(\LNEW) = K-O_{\Bdelta}(\log\lambda)$. The lemma is proved by setting $L_I$ to be the complement between $I,J_1$ and $L_J$ to be the complement between $I_1,J$ (notice that (\ref{eq.lambda-separation}) is satisfied because of (\ref{eq.prelambda-separation})).

The last assertion about Condition~\ref{item:cond:N} is obvious, as $\widehat Z^m \setminus (I \cup I_1 \cup J_1)$ contains $J \cup N$.
\end{proof}

\begin{cor}\label{cor.snakelemma}
   Let $\LL\subset \F^\circ_{L,\hZ^m}\setminus \Fam^+$ be a grounded snake as in~\eqref{eq:dfn:Snake} satisfying $K:=\Width(\LL)\gg_{\Bdelta}\log\lambda$ with $\lambda>2$. Suppose furthermore that $\LL$ satisfies condition~\ref{item:cond:N}. Then there exists an interval $I^{\mathrm{new}}\subset L^\bullet\subset\partial Z$ grounded $\mathrm{rel}$ $\hZ^m$, such that $\W^+_{\frac{\lambda}{3}}(I^{\mathrm{new}})\succeq K$, where $L^\bullet$ is the projection of $L$ onto $\partial Z$.
\end{cor}
\begin{proof}
   By Corollary \ref{cor:Prop:X}, $\mathcal{L}$ satisfies Property ~\ref{property:X}. By Snake Induction Lemma \ref{lem.snake}, there exists $I_1,J_1$ such that (\ref{eq.harmonicsum-submerge}) holds. Hence we may assume that 
    $$\W^\circ_{L_I,\hZ^m}(I,J_1)\geq 2K-O_{\Bdelta}(\log\lambda).$$
   Moreover $(I,L_I,J_1)$ satisfies Condition~\ref{item:cond:N}. By Corollary \ref{cor:Prop:X}, the lamination $\F^\circ_{L_I,\hZ^m}(I,J_1)$ satisfies Property~\ref{property:X}. We set $I_1 := I$. Suppose that $\W^+_{\hZ^m}(I_1,J_1)\leq\frac{1}{3}K$. Then Snake Induction Lemma \ref{lem.snake} is applicable again to $(I_1,L_I,J_1)$. Hence one obtains $I_2,J_2$ such that (\ref{eq.lambda-separation}) and (\ref{eq.harmonicsum-submerge}) hold. Therefore we may assume that 
    $$\W^\circ_{L_{I_1},\hZ^m}(I_1,J_2)\geq \frac{10}{3}K-O_{\Bdelta}(\log\lambda).$$
  Notice that this process cannot go forever, otherwise we will construct a sequence of pairs of intervals $I_n,J_n$ such that $\W^\circ_{L_{I_{n-1}},\hZ^m}(I_{n-1},J_n)$ diverges to infinity, which contradicts the fact that $\hZ^m$ is a quasi-disk. Hence we must get some $I_n,J_n$ such that $\W^+_{\hZ^m}(I_n,J_n)\geq\frac{1}{3}K$. To fix the idea, we assume that $J_n\subset L$. To abuse the notation, we set $I_n = I, J_n=J$. Notice that by Snake Induction Lemma \ref{lem.snake}, we have $\dist_{m,\hZ^m}(I,J)\geq\lambda |J|_{m,\bfS}$. The result then follows from Lemma \ref{lem:ShortCut to Comb}.
\end{proof}

\subsection{From $\Width$ to $\Width^+_\lambda$}

\newcommand\new{{\operatorname{new}}}
\renewcommand{\sp}{{\ \ }}

\begin{figure}[t]
\centering 
\includegraphics[width=0.7\textwidth]{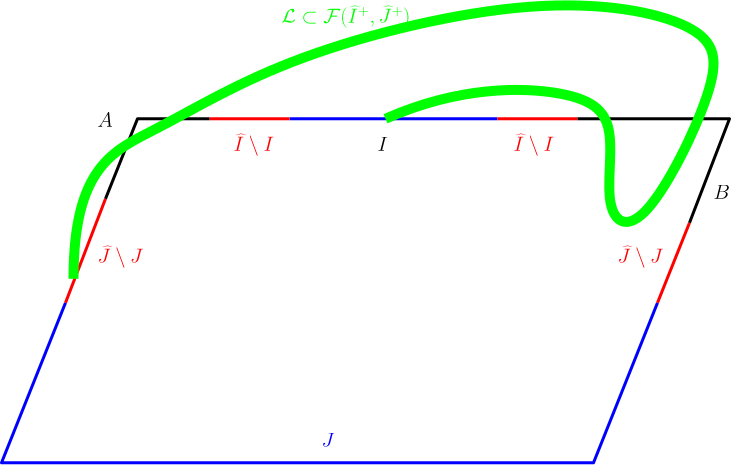}
\caption{Illustration of Lemma \ref{lem.trade-W-to-W+} and its proof. We remark that intervals in $\hat I\setminus I$ and in $\hat J\setminus J$ are much shorter than $I$. }
\label{fig.Snake} 
\end{figure}

The following lemma is similar to~\cite[Lemma 6.8]{DL:wZ}. It realizes Item~\ref{item:A:s:snakes+} stated at the beginning of this section.

\begin{lem}[$\Width\setminus \Width^+\leadsto \Width^+_\lambda$]\label{lem.trade-W-to-W+}
Let $I,  J\subset \partial \wZ$ be disjoint intervals such that (see Figure~\ref{fig.Snake}) \[|I|\ge \l_m, \sp \sp \dist( I, J)\ge \l_m, \sp\sp |J|\ge \frac 12, \sp\sp |I|\le |J^c|\le \lambda \l_m,\]
\[K\coloneqq \Width(I, J) -\Width^+(I,J) \gg_{\Bdelta,\lambda}1,\sp\sp \sp\text{ where  }\sp\sp \lambda \ge 3. \]  
Then there is an interval $I^\new\subset J^c\setminus I\subset \partial Z$ grounded rel $\wZ^m$ such that
\begin{equation}\label{eq.Inew-lambda}
    |I^\new|<\frac 1 \lambda |I|\sp\sp \text{ and }\sp\sp \Width_\lambda^+(I^\new)\succeq K.
\end{equation}
\end{lem}

In the lemma, condition $|I|\le |J^c|\le \lambda \l_m$ implies $\min\{|I|,|J|\}\asymp_\lambda \dist(I,J)$. This bounds the width of curves in $\Fam(I,J)$ that travel under $I$.

\begin{proof}
Pick thickenings $\tilde{I},\tilde{J}\subset\partial\hZ^m_{\bfS}$ of $I,J$, which are also grounded $\mathrm{rel}$ $\hZ^m_{\bfS}$, such that 
\begin{equation}\label{eq.Itilde-compare-I}
    |\tilde{I}\setminus I|\asymp\frac{1}{\lambda}|I|\asymp|\tilde{J}\setminus J|,\quad |\tilde{L}|\asymp_{\lambda} |I|,
\end{equation}
where $\tilde{L}\subset\partial\hZ^m$ is the shortest complementary interval of $\tilde{I},\tilde{J}$. Then by Assumption \ref{assump.7}, 
\begin{equation}\label{eq.I-tildeI}
    \W^-_{\hZ^m_{\bfS}}(I,\tilde{I}^c)\asymp_{\Bdelta} \log\lambda,\quad\W^-_{\hZ^m_{\bfS}}(J,\tilde{J}^c)\asymp_{\Bdelta} \log\lambda.
\end{equation}

Applying Lemma \ref{lemappen.trading-FtoF+}, we obtain intervals $\widehat{I},\widehat{J}$ with $I\subset\widehat{I}\subset\tilde{I}$, $J\subset\widehat{J}\subset\tilde{J}$ and a subfamily $\G^{\mathrm{New}}\subset\F(\widehat{I}^+,\widehat{J}^+)$ disjoint from the central arc of $\F^-_{\hZ^m_{\bfS}}(I,J)$. Combining (\ref{eq.Ihat-to-I}) of Lemma \ref{lemappen.trading-FtoF+} and (\ref{eq.I-tildeI}), we get 
\begin{equation}\label{eq.width-GNew}
\W(\G^{\mathrm{New}})\geq\W(I,J)-O_{\Bdelta}(\log\lambda).
\end{equation}
 Let $A,B$ be the complementary intervals to $\widehat{I}, \widehat{J}$, and let $I_a,I_b$ be the two complement intervals in $\widehat{I}\setminus I$. 
 Since $\W(I,J)-\W^+_{\hZ^m_{\bfS}}(I,J) = K$, there are two possibilities:
 \begin{itemize}
     \item $\G^{\mathrm{New}}$ contains a lamination $\H\subset \F^+_{\hZ^m_{\bfS}}(I_a, \widehat{J}) \cup\F^+_{\hZ^m_{\bfS}}(I_b,\widehat{J})$ with $\W(\H) \geq K/3$. Then at least one of $\F^+_{\hZ^m_{\bfS}}(I_a, \widehat{J}),\F^+_{\hZ^m_{\bfS}}(I_b,\widehat{J})$ has width $\geq K/6$. To fix the idea, let's say $\W^+_{\hZ^m_{\bfS}}(I_a, \widehat{J})\geq K/6$. Then we take $I^{\mathrm{new}}:=[(I_a)^\mathrm{grnd}]^\bullet$. Since the difference $(I_a\setminus I_a^{\mathrm{grnd}})$ is surrounded by at most two annuli $\Aout(\alpha^m_1),\Aout(\alpha^m_2)$ (where $\alpha^m_1,\alpha^m_2$ are level-$m$ dams), curves in $\F^+_{\hZ^m_{\bfS}}(I_a\setminus (I_a)^\mathrm{grnd}, \widehat{J})$ must cross one of these two annuli. Hence these curves have width $O_{\Bdelta}(1)$. Therefore $\W^+_{\hZ^m_{\bfS}}((I_a)^\mathrm{grnd}, \widehat{J})\geq K/6-O_{\Bdelta}(1)$ and $I^\mathrm{new}$ satisfies (\ref{eq.Inew-lambda}) by Theorem \ref{thm.grounded-outer-geometry}.
     
     \item $\G^{\mathrm{New}}$ contains a lamination $\H$ intersecting $A \cup B$ with $\W(\H) \geq K/3$. Let us denote by $\H_A\subset\H$ the family of curves intersecting $A$ before $B$; $\H_B\subset\H$ the family of curves intersecting $B$ before $A$. Hence at least one of $\H_A,\H_B$ has width $\geq K/6$. To fix the idea, let's say $\H_A\geq K/6$. Since $\G^{\mathrm{New}}$ does not intersect the central arc of $\F^-_{\hZ^m_{\bfS}}(I,J)$, $\H_A$ overflows a lamination  $$\tilde{\H}_A\subset\F^\circ_{A,\hZ^m_{\bfS}}(\widehat{I},\widehat{J}\cup B)\setminus\F^+_{\hZ^m_{\bfS}}(\widehat{I},\widehat{J}\cup B),\quad \W(\tilde{\H}_A)\geq K/6.$$ Hence Corollary \ref{cor.snakelemma} is applicable to $\tilde{\H}_A$ and produces $I^{\mathrm{new}}$ as required in (\ref{eq.Inew-lambda}).
 \end{itemize}
\end{proof}

\begin{rem}[Snake Lemma for pseudo-Siegel bubbles]\label{rem:Snakes.pseudo.bubble} In the final step of the proof of Calibration Lemma, see Figure~\ref{CL:Final}, we will be using the Snake Lemma for pseudo-Siegel bubbles. Its statement and the proof (modulo Lemma~\ref{lem.snake}) is the same as in the quadratic case; see \cite[Lemma 6.9]{DL:wZ}. In short, it states that if a wide lamination submerges in the pseudo-bubble, then the submergence can be localized and a wide lamination outside of the pseudo-bubble can be detected.
\end{rem}

\subsection{Snakes with toll barriers}\label{ss:SnakesWithBarriers} The notion of toll barriers (Definition~\ref{defn.barriers}) allows us to realize Item~\ref{item:B:s:snakes+} stated at the beginning of this section.

By Lemma~\ref{lem:RR:barrier}, barriers naturally exist within wide rectangles $\R$ blocking a lamination $\LL$ as illustrated in Figure~\ref{Fg:snake thorugh WR}. If barriers are toll barriers, then $\LL$ genuinely submerges many times and the Snake framework leads to the amplification $\Width \ \overset{\times n}{\leadsto}\ \Width^+_\lambda$ as stated in Lemma~\ref{lem:SL:toll_barriers}. Otherwise, assuming $\R$ is sufficiently wide, its substantial part is squeezed by $\LL$, see Figure~\ref{Fg:snake thorugh WR} again. This is the setting of Sneaking Lemma~\ref{lem:sneakI}.

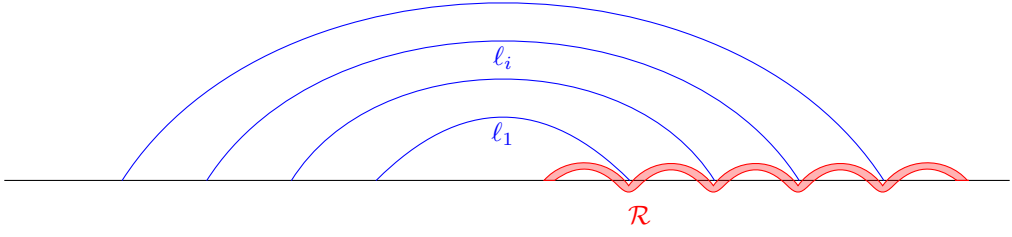
\begin{figure}
    \centering
\[\begin{tikzpicture}[scale=1.4]

\draw (-7,0) -- (2.5,0);

\node[red,below] at (-1,-0.15) {$\mathcal R$};

\begin{scope}[shift={(0,-0.05)} ,scale=0.4]

\draw[blue,  scale=4,shift={(0.82,+0.03)}] (-1.5,0) 
.. controls (-2, 0.5) and (-2.5,0.5) .. 
 (-3,0);

\begin{scope}[scale=4,shift={(0.82,+0.03)}]
\node[blue,below] at (-2.25,0.4) {$\ell_1$};

\node[blue,below] at (-2.25,0.85) {$\ell_i$};

\end{scope}

\draw[blue,  scale=4,shift={(0.82,+0.03)}] (-1,0) 
.. controls (-1.5, 0.8) and (-3,0.8) .. 
 (-3.5,0);

\draw[blue,  scale=4,shift={(0.82,+0.03)}] (-0.5,0) 
.. controls (-1.2, 1.1) and (-3.3,1.1) .. 
 (-4,0);

 \draw[blue,  scale=4,shift={(0.82,+0.03)}] (-0,0) 
.. controls (-0.9, 1.4) and (-3.6,1.4) .. 
 (-4.5,0);

\draw[red, fill=red, fill opacity=0.3]  (-4.5,0.125) 
.. controls (-4, 0.5) and (-3.5,0.5) .. 
(-3,0)
.. controls (-2.8, -0.2) and (-2.7,-0.2) .. 
(-2.5,0)
.. controls (2-4, 0.5) and (2-3.5,0.5) .. 
(2-3,0)
.. controls (2-2.8, -0.2) and (2-2.7,-0.2) .. 
(2-2.5,0)
.. controls (4-4, 0.5) and (4-3.5,0.5) .. 
(4-3,0)
.. controls (4-2.8, -0.2) and (4-2.7,-0.2) .. 
(4-2.5,0)
.. controls (6-4, 0.5) and (6-3.5,0.5) .. 
(6-3,0)
.. controls (6-2.8, -0.2) and (6-2.7,-0.2) .. 
(6-2.5,0)
.. controls (8-4, 0.5) and (8-3.5,0.5) .. 
(8-3,0.125)
.. controls (7.8-3,0.125) and (6.2-0.75,0.125 ) .. 
(6-0.75,0.125 )
.. controls   (6-1.25,0.7) and (6-2.25, 0.7) .. 
(4-0.75,0)
.. controls   (4-1.25,0.7) and (4-2.25, 0.7) .. 
(2-0.75,0)
.. controls (2-1.25,0.7) and (2-2.25, 0.7) .. 
(-0.75,0)
.. controls (-1.25,0.7) and (-2.25, 0.7).. 
(-2.75,0)
.. controls (-3.25,0.7) and  (-4.25, 0.7) .. 
(-4.75,0.125) 
.. controls (-4.75,0.125)  and  (-4.5,0.125) .. 
(-4.5,0.125); 

\end{scope}

\end{tikzpicture}\]

    \caption{Toll barriers for a lamination $\R$ are arcs $\ell_i$ in the complement of $\wZ^m$ so that every arc $\gamma\in \R$ submerges below every $\ell_i$. Moreover, $\gamma$ has no subcurve $\gamma'$ in $\gamma\cap \wZ^m$ such that $\gamma'$ skips under $\ell_i,\ell_{i+1}$}
    \label{fig:TollBarr}
\end{figure}

\begin{defn}[Barriers and Toll barriers]\label{defn.barriers}
Let $I,J\subset\partial\hZ^m$ be a pair of grounded intervals and $A,B$ be the complementary intervals of $I,J$ such that $|J|>1/2$. We orient the intervals clockwise so that $A<I<B<J$.

Let $\ell_1,...,\ell_n\subset\widehat{\C}\setminus\hZ^m$ be $n$ pairwise disjoint simple arcs such that each $\ell_i$ connects $a_i\in A$ to $b_i\in B$. We assume furthermore that $a_i<a_{i-1}<b_{i-1}<b_i$. We say that $\ell_1,...,\ell_n$ are barriers for a lamination $\R\subset\F(I,J)$ if no curves in $\R$ intersect $\bigcup_{i=1}^n\ell_i$ up to homotopy rel $\wZ^m$: for every $\gamma\in \R$, every $\ell_i$ can be homotoped into $\ell_i^\gamma$ rel $\wZ^m$ so that $\gamma$ is disjoint from $\ell_i^\gamma$. 

We say that a curve $\gamma\in\R$ skips under $[a_i,a_{i-1}]$ if $\gamma\cap\hZ^m$ contains a subcurve connecting two different components of $A\setminus[a_i,a_{i-1}]$; a curve $\gamma\in\R$ skips under $[b_{i-1},b_{i}]$ if $\gamma\cap\hZ^m$ contains a subcurve connecting two different components of $B\setminus[b_{i-1},b_{i}]$; a curve $\gamma\in\R$ skips under $[a_i,a_{i-1}]\cup[b_{i-1},b_i]$ if $\gamma$ skips under $[a_i,a_{i-1}]$ or $[b_{i-1},b_i]$.

We say that barriers $\ell_1,...,\ell_n$ are toll barriers for $\R$ if no curves in $\R$ skips under $[a_i,a_{i-1}]\cup[b_{i-1},b_i]$ for all $i=2,...,n$, see Figure~\ref{fig:TollBarr}.
\end{defn}

Now we are ready to state Snake Lemma with toll barriers below. Roughly speaking, $n$ toll barriers will force the lamination to submerge $n$ times into $\hZ^m$ and thus amplifies $n$ times the width. 

\begin{lem}[$\Width \ \overset{\times n}{\leadsto}\ \Width^+_\lambda$]
\label{lem:SL:toll_barriers}
Let $I,J$ and $A,B$ be intervals given by Definition \ref{defn.barriers}. Assume furthermore that $|A|,|B|\geq2\l_m$. Suppose that a lamination $\R\subset\F(I,J)$ has $n\geq 3$ toll barriers. Then for every $\lambda>1$, there exists a grounded interval $T\subset\partial Z$ such that if $\W(\R)\gg_{\Bdelta}\log\lambda$, then
$$\W_\lambda^+(T)\succeq n\cdot\W(\R).$$
\end{lem}
\begin{proof}
    The series decomposition with toll barriers can be done to $\mathcal{R}$ in the same way as in \cite[\S 6.4]{DL:wZ}: $\mathcal{R}$ overflows $n-1$ disjoint sublaminations  $\mathcal{F}_{i} := \mathcal{F}_{a_i}\cup\mathcal{F}_{b_i}$ ($2\leq i\leq n$), where \[\mathcal{F}_{a_i}\subset\mathcal{F}^{\circ}_{L_{a_i}}(J_{a_i},\lfloor I_{a_{i-1}},J_{b_i}\rfloor)\quad\text{ and }\quad\mathcal{F}_{b_i}\subset\mathcal{F}^{\circ}_{L_{b_i}}(\lfloor J_{a_i},I_{b_{i-1}}\rfloor,J_{b_i}).\] Here $I_{a_i}, I_{b_i}, J_{a_i}, J_{b_i}$ are some mutually disjoint subintervals in $A\cup B$, satisfying $I_{a_i}\cap J_{a_i} = \{a_i\}$, $I_{b_i}\cap J_{b_i} = \{b_i\}$ and $I_{a_i}<J_{a_i}, J_{b_i}<I_{b_i}$. By Grötzsch's inequality, there is some $i$ such that $\W(\mathcal{F}_i)\geq (n-1)\cdot\W(\R)$. 
    
    By Corollary \ref{cor:Prop:X}, Property~\ref{property:X} holds for $(\R, \R)$. Therefore, Property~\ref{property:X} is inherited for each $(\R_{a_i},\F_{a_i})$, $(\R_{b_i},\F_{b_i})$, see~\S\ref{sss:prop:X:overflow}, where $\R_{a_i}, \R_{b_i}$ are the associated parts of $\R$ overflowing $\F_{a_i}, \F_{b_i}$. Hence we may assume that $I_{a_i}, I_{b_i}, J_{a_i}, J_{b_i}$ are grounded intervals after removing $O_{\Bdelta}(1)$ curves in $\R$. We localize these intervals to make them $\lambda$-smaller (and remove $O(\lambda)$ curves from the families) and then again apply Property~\ref{property:X} to ground the intervals.

    In conclusion, at least one of $\F_{a_i},\F_{b_i}$, say $\F_{a_i}$, has width $\succeq n\cdot\W(\R)$. Notice that $\F_{a_i}$ satisfies Condition\ref{item:cond:N}, hence we can apply Corollary \ref{cor.snakelemma} to get the desired interval $T$.
\end{proof}

\begin{figure}[t!]
\[\begin{tikzpicture}[scale=1.4]

\draw (-7,0) -- (2.5,0);

\node[blue] at (-2.3,1) {$\R$};

\node[red,below] at (-1,-0.15) {$\LL$};
\node[blue,below] at (1,-0.0) {$B$};

\draw[blue, fill=blue, fill opacity=0.04] (-1.5,0) 
.. controls (-2, 0.5) and (-2.5,0.5) .. 
 (-3,0)
 .. controls (-3, 0) and (-6,0) .. 
 (-6.25,0)
.. controls (-3.25, 2) and (-1.25,2) .. 
(1.75,0)
.. controls (1.75,0) and  (-1.5,0)  .. 
 (-1.5,0) ;

 \node[blue, below] at (-5,0) {$A$};

\begin{scope}[shift={(0,-0.05)} ,scale=0.4]

\draw[red, fill=red, fill opacity=0.3]  (-4.5,0.125) 
.. controls (-4, 0.5) and (-3.5,0.5) .. 
(-3,0)
.. controls (-2.8, -0.2) and (-2.7,-0.2) .. 
(-2.5,0)
.. controls (2-4, 0.5) and (2-3.5,0.5) .. 
(2-3,0)
.. controls (2-2.8, -0.2) and (2-2.7,-0.2) .. 
(2-2.5,0)
.. controls (4-4, 0.5) and (4-3.5,0.5) .. 
(4-3,0)
.. controls (4-2.8, -0.2) and (4-2.7,-0.2) .. 
(4-2.5,0)
.. controls (6-4, 0.5) and (6-3.5,0.5) .. 
(6-3,0)
.. controls (6-2.8, -0.2) and (6-2.7,-0.2) .. 
(6-2.5,0)
.. controls (8-4, 0.5) and (8-3.5,0.5) .. 
(8-3,0.125)
.. controls (7.8-3,0.125) and (6.2-0.75,0.125 ) .. 
(6-0.75,0.125 )
.. controls   (6-1.25,0.7) and (6-2.25, 0.7) .. 
(4-0.75,0)
.. controls   (4-1.25,0.7) and (4-2.25, 0.7) .. 
(2-0.75,0)
.. controls (2-1.25,0.7) and (2-2.25, 0.7) .. 
(-0.75,0)
.. controls (-1.25,0.7) and (-2.25, 0.7).. 
(-2.75,0)
.. controls (-3.25,0.7) and  (-4.25, 0.7) .. 
(-4.75,0.125) 
.. controls (-4.75,0.125)  and  (-4.5,0.125) .. 
(-4.5,0.125); 

\end{scope}

\begin{scope}[shift={(0,-3.5)}]

\draw (-7,0) -- (2.5,0);

\node[blue] at (-2.3,1) {$\R$};

\node[red,below] at (-0.1,-0.15) {$\LL$};
\node[blue,below] at (1,-0.0) {$B$};

\draw[blue, fill=blue, fill opacity=0.04] (-1.5,0) 
.. controls (-2, 0.5) and (-2.5,0.5) .. 
 (-3,0)
 .. controls (-3, 0) and (-6,0) .. 
 (-6.25,0)
.. controls (-3.25, 2) and (-1.25,2) .. 
(1.75,0)
.. controls (1.75,0) and  (-1.5,0)  .. 
 (-1.5,0) ;

\begin{scope}[shift={(0,-0.05)} ,scale=0.4]

\draw[red, fill=red, fill opacity=0.3]  (-4.5,0.125) 
.. controls (-4, 0.5) and (2-3.5,0.5) .. 
(2-3,0)
.. controls (2-2.8, -0.2) and (2-2.7,-0.2) .. 
(2-2.5,0)
.. controls (4-4, 0.5) and (8-3.5,0.5) .. 
(8-3,0.125)
.. controls (7.8-3,0.125) and (6.2-0.75,0.125 ) .. 
(6-0.75,0.125 )
.. controls   (6-1.25,0.7)  and (2-2.25, 0.7) .. 
(-0.75,0)
.. controls (-1.25,0.7) and  (-4.25, 0.7) .. 
(-4.75,0.125) 
.. controls (-4.75,0.125)  and  (-4.5,0.125) .. 
(-4.5,0.125); 

\coordinate (l1) at  (-0.6,0.125);

\coordinate (l2) at  (-0.9,0.125);

\end{scope}

\draw[blue, fill=blue, opacity=0.3] (l2) 
.. controls (-1.5,0.7)  and  (-2.5,0.7) .. 
(-3.5,0)
.. controls (-3.5,0)  and  (-3.5,0) .. 
(-5,0)
.. controls (-3.5,1.7)  and  (-1.5,1.7) .. 
(l1)
.. controls (l1)  and  (l2) .. 
(l2);
 \node[blue, below] at (-5,0) {$A$};

\end{scope}

\end{tikzpicture}\]
\caption{Illustration to Sneaking Lemma~\ref{lem:sneakI}. If a lamination $\LL$ sneaks through a wide rectangle $\R$, then either $\LL$ submerges many times rel toll barriers inside $\R$, or substantial part of $\R$ (between two consecutive non-toll barriers) is squeezed by $\LL$.}
\label{Fg:snake thorugh WR}
\end{figure}
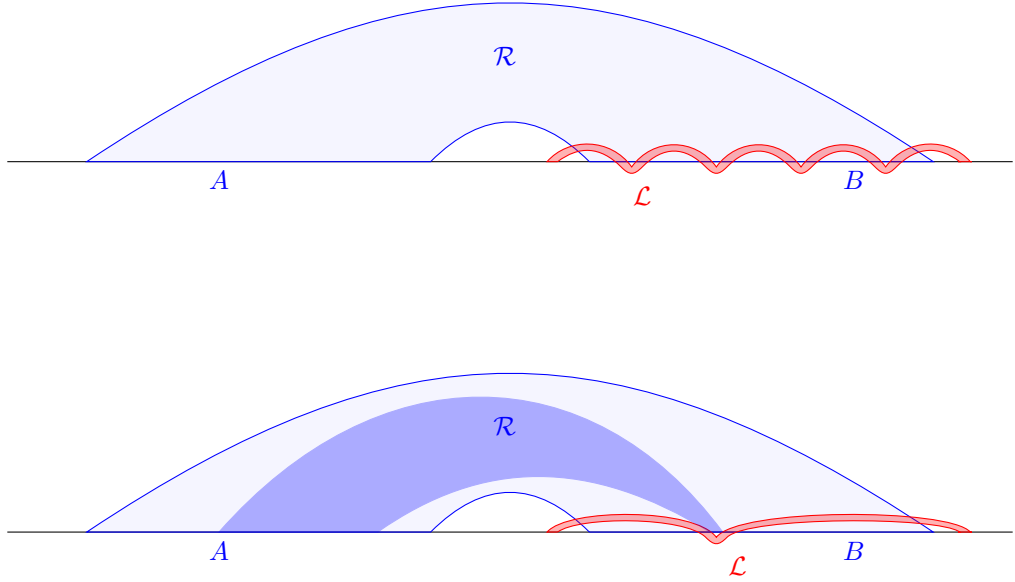

\begin{sneaklmm}[$\Width \ \overset{\times n}{\leadsto}\ \Width^+_\lambda$]
\label{lem:sneakI}
Let $I,J$ and $A,B$ be intervals given by Definition \ref{defn.barriers}. Assume furthermore that $|I|\geq\l_m$ and $|A|,|B|\geq 2\l_m$. For every $t,\chi,\lambda>2$, there exists $\kappa(t)$ such that the following holds: if 
\begin{itemize}[label ={-}]
    \item $1/\chi\leq|A|/|I|, |B|/|I|\leq\chi$;
    \item $\W(I,J)=: K\gg_{\chi,\Bdelta,\lambda} 1$;
    \item $\W^+_{\hZ^m}(A,B)\geq\kappa(t) K$,
\end{itemize}
then there exists an interval $T\subset\partial Z$ grounded $\mathrm{rel}$ $\hZ^m$ with $|T|<|I|$, such that $\W^+_\lambda(T)\geq tK$.
\end{sneaklmm}
\begin{proof}
    The argument is the same as \cite[Sneaking Lemma 6.13]{DL:wZ}. We sketch the proof here. Let us decompose $\F^+(A,B)$ into $m\approx\kappa(t)/t$ rectangles $\R_i$ with base and roof $[a_{i+1},a_{i}], [b_i,b_{i+1}]$ respectively, so that each one of them has width $tK$. Applying Lemma~\ref{lem:RR:barrier} to each $\R_i$, one creates $m$ barriers $\beta_i$ for $W(I,J)$. Now there are two cases (see Figure \ref{Fg:snake thorugh WR}): if each $[a_{i+1},a_{i}], [b_i,b_{i+1}]$ has combinatorial length $\geq |I|/\lambda$, then one can apply Localization Lemma \ref{lem.localization} to remove $O_{\chi,\Bdelta}(\log\lambda)$ width from $\W(I,J)$ so that $\beta_i$ are toll barriers for $W(I,J)$, and Lemma \ref{lem:SL:toll_barriers} is applicable to get the desired interval $T$; if some $[a_{i+1},a_{i}], [b_i,b_{i+1}]$ (say $[a_{i+1},a_{i}]$) has combinatorial length $< |I|/\lambda$, then one can just take $T = [a_{i+1},a_i]^{\mathrm{grnd}}$.  
\end{proof}

\section{Welding, Calibration, Amplification Lemmas}\label{s:wca}

In~\S\ref{ss:WeldingLmm} we state the Welding Lemma. It states that either Log-Rules in $\wZ^m$ and in a given parabolic fjord have a proper welding that produces an annulus with a definite modulus~\eqref{eq.strongwelding-modulus}, or there is a definite $\Width^+$- degeneration~\eqref{eq.strongwelding-expboost} on a scale deeper than $\l_m$ (where the Log Rule in parabolic fjords is not applicable). The argument repeats the unicritical case.

In~\S\ref{ss:CalLmm} we state the Calibration Lemma. Its proof follows the unicritical case. However, the preparation for the $1\le 1\oplus 1 =0.5$ argument is revised because there can be more than one orbit of preparabolic intervals on a given renormalization level, see Figure~\ref{CL:1+1=1} for illustration.

In~\S\ref{ss:AmplLmm} we state the Amplification Lemma relying on the Covering and Snake-Lair Lemma. The argument repeats the unicritical case.

We also remark that Snake Lemmas are used in the Calibration, Welding, and Amplification Lemmas to convert a submerging lamination into $\Width^+_\lambda$-degeneration, see Item~\ref{item:A:s:snakes+}. Snake Lemma with barriers and Sneaking Lemma (Item~\ref{item:B:s:snakes+}) are used in the Snake-Lair Lemma.

\subsection{Welding Lemma} \label{ss:WeldingLmm} Welding Lemma~\ref{ss:WeldingLmm} relies on the Log-Rule in $\wZ^{m+1}$ and on the Log-Rule in a parabolic fjord (Theorem~\ref{thm.fjord}). The argument is essentially the same as in the quadratic case. Roughly speaking, if the welding of geodesic rectangles inside and outside $\hZ^{m+1}$ is unsuccessful, then there will be an exponential boost for the outer width of $Z$ on some deeper level.

Recall that in Definition \ref{def.geodesic-annulus}, we defined for each $J\subset\partial\hZ^m_{\bfS}$ a topological disk $O_{J,\hZ^m_{\bfS}}$.
\begin{weldlmm}\label{lem.strongwelding}
    Consider a concatenation of intervals $J = N\#I\#M\subset\partial\hZ^{m+1}$ with 
    $$J^\bullet\subset T_{\mathrm{par}}\quad\text{ and }\quad |J^\bullet|<\frac{1}{2}|T_{\mathrm{par}}|,\quad(J^{\bullet}\text{ is the projection of }J \text{ to }\partial Z)$$ 
    where $T_{\mathrm{par}}$ is given by Theorem \ref{thm.fjord}, such that $N,I,M$ are well-grounded $\mathrm{rel}$ $\hZ^{m+1}$). Assume that
    \begin{equation}\label{eq.comparable-strongwelding}
        |N|\asymp|I|\asymp|M|.
    \end{equation}
 Then there exists a constant $\Bepsilon>0$ depending only on “$\asymp$” in $(\ref{eq.comparable-strongwelding})$, such that the following holds for all $\lambda>2$. If $\nu := |I|/\l_{m+1}\gg_{\Bdelta,\lambda} 1$ (“$\gg_{\Bdelta,\lambda}$” also depends on “$\asymp$” in $(\ref{eq.comparable-strongwelding})$), then either
 \begin{equation}\label{eq.strongwelding-modulus}
     \mathrm{mod}(O_{J,\hZ^{m+1}}\setminus I)\geq \Bepsilon,
 \end{equation}
 or there exists an interval $S\subset\partial Z$ with $|S|\leq\l_{m+1}$ such that $S$ is grounded $\mathrm{rel}$ $\hZ^{m+1}$ and
 \begin{equation}\label{eq.strongwelding-expboost}
     \log \W^+_{\lambda,Z}(S)\succeq\log\nu.
 \end{equation}
\end{weldlmm}

\begin{proof}
   The lemma will follow by the following two claims. Let $N_1$ $M_1$ be the middle  $\asymp 1/3$ subintervals of $N$ and $M$ respectively such that their end points also belong to $\mathrm{CP}_{m+1}$. Let $\R$ be the vertical lamination of the annulus $O_{J,\hZ^{m+1}}\setminus I$. Set $K := \W(\R)$.
    \begin{claim}\label{claim1.lem.welding}
        At least $0.99 K$ width of curves in $\R$ intersects $N_1\cup M_1$ before intersecting $\lfloor N_1, M_1\rfloor^c\cup \partial O_{J,\hZ^{m+1}}$.
    \end{claim}

Let $L_a\subset\partial\hZ^{m+1}$ be the shortest complementary interval between $N_1$ and $I$ and $L_b\subset\partial\hZ^{m+1}$ be the shortest complementary interval between $I$ and $M_1$. Then Claim \ref{claim1.lem.welding} and the Parallel Law imply that at least one of the following inequalities holds:
$$\W_{L_a,\hZ^{m+1}}(N_1,I)\geq 0.49K\quad\text{and}\quad \W_{L_b,\hZ^{m+1}}(I,M_1)\geq 0.49K.$$
Without loss of generality, we assume that $\W_{L_a}(N_1,I)\geq 0.49K$. Set $I_0 := I$, $I_1 := N_1$ and $L_1 := L_a$. Let us write
$$A \approx_{l} B \quad\text{ if both }\quad A\asymp B \quad\text{ and }\quad A =B+O(l)\quad\text{ hold.}$$
\begin{claim}\label{claim2.lem.welding}
    There exists a sequence of well-grounded intervals $I_1,I_2,...,I_k$ satisfying
    $$|I_t|\asymp\nu\l_{m+1}/3^t\text{ for all }t,\quad k\asymp\log\nu,\quad |I_k|\asymp\l_{m+1},$$
    such that for every $t$, there exists $j(t)\in\{0,1,...,t-1\}$ such that
    \begin{equation}\label{eq.WLt}
        \W_{L_t,\hZ^{m+1}}(I_t,I_{j(t)})-\W^+_{\hZ^{m+1}}(I_t,I_{j(t)})\geq 0.48\cdot1.9^{t-1} K,\quad |L_t|\approx_{\l_{m+1}}|I_t|\asymp\nu\l_{m+1}/3^t,
    \end{equation}
    where $L_t$ is the shortest complementary interval between $I_t$ and $I_{j(t)}$.
\end{claim}

Taking $t=k$ in Claim \ref{claim2.lem.welding}, then (\ref{eq.WLt}) allows us to apply Lemma \ref{lem.trade-W-to-W+}, which completes the proof of the Welding Lemma.

\end{proof}

\begin{proof}[Proof of Claim \ref{claim1.lem.welding}]
Let $\mathcal{R}^-\subset\hZ^{m+1}$, $\mathcal{R}^+\subset\C\setminus\overline{\hZ^{m+1}}$ be the geodesic rectangles with horizontal sides $N_1,M_1$. By Theorem \ref{thm.fjord} \ref{thm.fjord:third} and Theorem \ref{thm.wellgrounded-outer-geometry}, $\W(\R^+)\asymp1$; by Theorem \ref{thm.large-log-rule}, $\W(\R^-)\asymp1$ (since $|N_1|\asymp|M_1|\asymp\nu\cdot\l_{m+1}\gg_{\Bdelta}\l_{m+1}$). On the other hand, if a curve in $\R$ does not satisfy the property described in the claim, then it either crosses $\mathcal{R}^-$ or $\mathcal{R}^+$. Hence the width of such curves is at most $O(1)$.
\end{proof}
\begin{proof}[Proof of Claim \ref{claim2.lem.welding}] Assume that $I_t$ and $j(t)$ are constructed for $t \geq 1$. Set
 $$I_{t+1} \subset L_t \quad\text{ with }\quad |I_{t+1}|\approx_{\l_{m+1}}\dist(I_{t+1},I_t)\approx_{\l_{m+1}}\dist(I_{t+1},I_{j(t)})$$ 
to be a middle $\asymp1/3$-subinterval of $L_t$ such that both end points of $I_{t+1}$ belong to $\mathrm{CP}_{m+1}$. Let us show that either
\begin{equation}\label{eq.It+1-induction}
     \W_{L_a,\hZ^{m+1}}
 (I_{t+1},I_t) \geq 0.49\cdot1.9^{t}K\quad
 \text{ or }\quad
 \W_{L_b,\hZ^{m+1}}
 (I_{t+1},I_{j(t)}) \geq0.49\cdot1.9^{t}K,
\end{equation}
 where $L_a$ and $L_b$ are the shortest complementary intervals between $I_{t+1},I_t$ and $I_{t+1},I_{j(t)}$. This would finish the construction of $I_{t+1}$ with $j(t+1)\in\{t,j(t)\}$, since $\W^+(I_{t+1},I_{j(t+1)})$ is at most $O(1)$ by Theorem \ref{thm.fjord} \ref{thm.fjord:third} and Theorem \ref{thm.wellgrounded-outer-geometry}.

 In order to prove (\ref{eq.It+1-induction}), let us take a well-grounded interval $X\subset\partial\hZ^{m+1}$ attached to $I_t$ so that $|I_t| \approx_{\l_{m+1}} 
|X|$ and $I_t$ is between
 $X$ and $I_{t+1}$ in $T_{\mathrm{par}}$. As in Claim \ref{claim1.lem.welding}, consider the outer and inner geodesic rectangles (i.e., their vertical boundaries are hyperbolic geodesics in the corresponding domain)
 $$\R^+ \subset \widehat{\C}\setminus{Z},\quad \R^-\subset\hZ^{m+1}, \quad\partialhL\R^+=\partialhL\R^- =X, \quad\partialhR\R^+ =\partialhR\R^- =I_{t+1}.$$
Then by Theorem \ref{thm.fjord}\ref{thm.fjord:third} and Theorem \ref{thm.wellgrounded-outer-geometry}, $\W(\R^+) \asymp 1$. For $\R^-$, one has the following:
\begin{itemize}
    \item if $|I_{t+1}| \gg_{\Bdelta} 1$, then $\W(\R^-) \asymp 1$ by Theorem \ref{thm.large-log-rule};
 \item otherwise $\W(\R^-) \asymp_{\Bdelta} 1 $ by Assumption \ref{assump.7}.
\end{itemize}
Therefore, after removing at most $O_{\Bdelta}(1)$ curves from $\F_{L_t,\hZ^{m+1}}(I_t,I_{j(t)})$, the remaining family overflows $\F_{L_a,\hZ^{m+1}}(I_{t+1},I_{t})$ and $\F_{L_b,\hZ^{m+1}}(I_{t+1},I_{j(t)})$. Hence (\ref{eq.It+1-induction}) holds by Grötzsch inequality.
\end{proof}

\begin{cor}\label{cor.welding}
    Under the same assumption of Welding Lemma \ref{lem.strongwelding}, either (\ref{eq.strongwelding-expboost}) holds for some interval $S\subset \partial Z$ with $|S|\leq \l_{m+1}$ grounded $\mathrm{rel}$ $\hZ^{m+1}$ or $\mathrm{mod}(O_{J,\hZ^{m+1}}\setminus \overline{O_{I,\hZ^{m+1}}})\succeq1$. 
\end{cor}
\begin{proof}

    Let us split both of $N,M$ into three comparable pieces with end points in $\mathrm{CP}_{m+1}$: $N = N_1\#N_0\#N_{-1}$ and $M=M_1\#M_0\#M_{-1}$. Apply Lemma \ref{lem.strongwelding} to $N$ and $M$, then either we have (\ref{eq.strongwelding-expboost}), or both $\mathrm{mod}(O_{N,\hZ^{m+1}}\setminus N_0),\mathrm{mod}(O_{M,\hZ^{m+1}}\setminus M_0)\geq{\Bepsilon}$. Let $\R\subset\hZ^{m+1}$ be the hyperbolic rectangle based on $N_0,M_0$; $\R'\subset\widehat{\C}\setminus{Z}$ be the hyperbolic rectangle based on $N_0,M_0$. Then by Theorem \ref{thm.large-log-rule}, $\W(\R)\asymp 1$; by Theorem \ref{thm.fjord} \ref{thm.fjord:third} and Theorem \ref{thm.wellgrounded-outer-geometry}, $\W(\R')\asymp 1$. Let $\F$ be the family of vertical curves of $O_{J,\hZ^{m+1}}\setminus \overline{O_{I,\hZ^{m+1}}}$. Then for any $\gamma\in\F$, either it crosses one of $\R$ and $\R'$, or it crosses one of $O_{N,\hZ^{m+1}}\setminus N_0$ and $O_{M,\hZ^{m+1}}\setminus M_0$. This finishes the proof. 
\end{proof}

\subsection{Calibration Lemma}\label{ss:CalLmm}
In this section, we state the Calibration Lemma, which is a direct analogue of the Calibration lemma for quadratic polynomials; see \cite[Calibration Lemma 9.1]{DL:wZ}.

\begin{figure}
    \centering
    \includegraphics[width=0.7\linewidth]{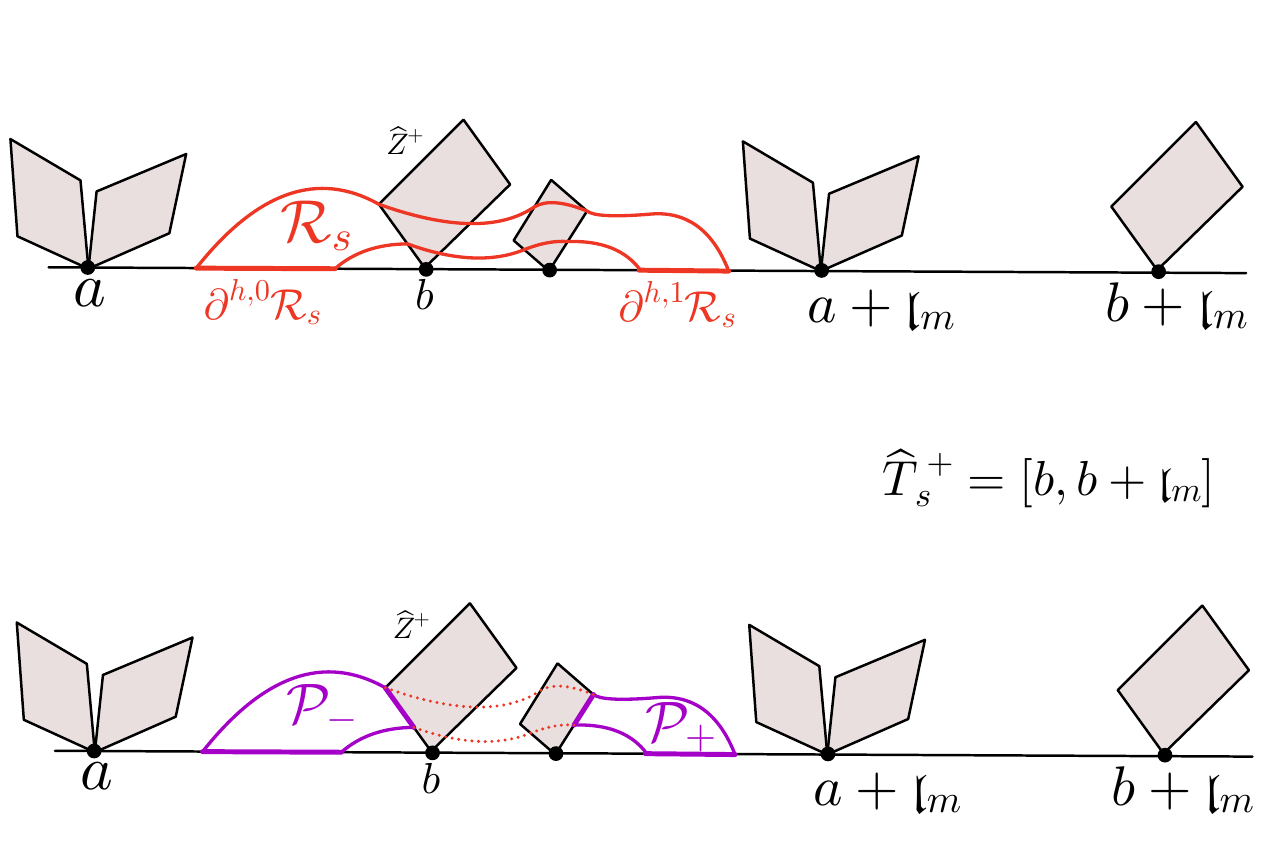}
    \caption{Illustration to the $1\le 1\oplus 1 =0.5$ argument in Calibration Lemma~\ref{lem.calibration}. Compare with \cite[Figure 26]{DL:wZ} in the unicritical case. Most of the rectangle $\mathcal R_s$ overflows two of its lifts $\mathcal P_\pm$; this is impossible.}
    \label{CL:1+1=1}
\end{figure}

\begin{caliblem}\label{lem.calibration}
    There is an absolute constant $\chi > 1$ such that the following holds for
 every $\lambda \geq 10$. Let $\hZ^{m+1}$ be the level $(m+1)$ pseudo-Siegel disk and consider a pre-parabolic interval $T \subset \partial Z$ in the
 diffeo-tiling $\D_m^{\mathrm{full}}$. If there are intervals
  $$I \subset T, L\subset\partial Z \text{ such that } \l_{m+1}\leq |I|\leq \l_m,\quad \dist(I,L^c)\geq \lambda\l_{m+1}$$
 $$\text{and }\quad
 \W^+_{\mathrm{div},m}(I,L^c)\geq \chi K \gg_\lambda 1,$$
 then
 \begin{enumerate}[label=\upshape(\roman*)]
 \item\label{lem.calibration.first} either there is a $[K,\lambda]^+$-wide level $(m + 1)$ combinatorial interval,
 \item\label{lem.calibration.second}  or there is a $[\chi^{1.5}K,\lambda]^+$-wide interval $I' \subset \partial Z$, $|I'| \leq \l_{m+1}$ grounded $\mathrm{rel}$ $\hZ^{m+1}$.
 \end{enumerate}
\end{caliblem}
\begin{proof}
    The proof is similar to the quadratic case in \cite[\S 9]{DL:wZ}; see also \cite[\S 6.1]{DLL}. We sketch the main arguments below.

By Lemma \ref{lem.extdiv-decompo}, we can take a rectangle $\mathcal{R}\subset\mathcal{F}^+_{\mathrm{div}}(I,L^c)$ such that $\mathcal{W}(\mathcal{R})\geq\chi K-O(1)$. We assume that item \ref{lem.calibration.first} does not hold. This implies that $\mathcal{R}$ is almost invariant under the first return map $f^{\q_{m+1}}$ with respect to its base: most of the $\mathcal{R}$ overflows its lift under $f^{\q_{m+1}}$ starting at the base. The almost invariance (i.e., ``most'') is quantified in terms of a big but fixed $\chi\gg 1$. 

Spreading around $\mathcal{R}$ by $f^s$ with $1\leq s<\mathfrak{q}_{m+1}$, we get $\cal{R}_s$. All $\mathcal{R}_s$ are again almost invariant under $f^{\q_{m+1}}$ with respect to their base. Let $T_s = [a,b]$ be the pre-parabolic interval containing $\partial^{h,0}\cal{R}$ and $\widehat{T}_s^+=[b, b+\l_m]$ and $\widehat{T}_s^-=[a-\l_m, a]$ the level-$m$ combinatorial interval to the right and left of $T_s$.

Since $\cal{R}_s$ are wide families, they cannot block each other. Therefore by removing $O(1)$ buffers from every $\cal{R}_s$, there is some $\cal{R}_s$ such that $\partial^{h,1}\cal{R}_s\subset (\widehat{T}_{s}^- \cup f^{\q_{m+1}}(\widehat{T}_{s}^-))\cup(\widehat{T}_{s}^+ \cup f^{\q_{m+1}}(\widehat{T}_{s}^+))$. Without loss of generality, let us assume that $\partial^{h,1}\cal{R}_s\subset\widehat{T}_{s}^+ \cup f^{\q_{m+1}}(\widehat{T}_{s}^+)$.

We claim that $\cal{R}$ has a subrectangle $\cal{R}^\text{new}_s$ with width $\chi^{0.9} K$ such that $\partial^{h,1}\R^\text{new}_s$ is small: $|\partial^{h,1}\R^\text{new}_s|\le \l_{m+1}/5$. Indeed, otherwise $\cal{R}_s$ is also almost invariant with respect to its roof  $\partial^{h,1}\R$, where the almost invariance is quantified in terms of $\chi^{0.1} \gg 1$. This is impossible by the ``$1\le 1\oplus 1 =0.5$  argument'' illustrated in Figure~\ref{CL:1+1=1}. More precisely, by replacing $\cal{R}_s$ with its subrectangle with width $K/\mathfrak{g} -O(1)$, we can assume that its roof $\partial^{h,1}\R_s$ is also within a preparabolic interval; we assume that this preparabolic interval is within $\widehat{T}_{s}^+ $ as shown in Figure~\ref{CL:1+1=1}. Let $\cal{P}_{-}$ and $\cal{P}_+$ be two components of $f^{-\q_{m+1}}(\cal{R}_s)$, such that $\partial^{h,0}\cal{P}_- = \partial^{h,0}\cal{R}_s\boxplus\l_{m+1}$ and $\partial^{h,1}\cal{P}_+ = \partial^{h,1}\cal{R}_s\boxplus\l_{m+1}$. We obtain that most of $\cal{R}_s$ overflows its lifts $\cal{P}_{-}$ and $\cal{P}_+$; this is impossible.

The final argument for the almost invariant rectangle $\cal{R}^\text{new}_s$ is illustrated in Figure~\ref{CL:Final}. The almost invariance is quantified by $\chi^{0.9}\gg 1$. After most of the $\cal{R}^\text{new}_s$ overflows its lift, $\cal{R}^\text{new}_s$ overflows a wide lamination with width $\chi^{1.7}K$ submerging in the pseudo-bubble $\wZ^+$. Since $|\partial^{h,1}\R^\text{new}_s|\le \l_{m+1}/5$, we apply the Snake Lemma for pseudo-bubbles (see Remark~\ref{rem:Snakes.pseudo.bubble}) and detect a wide lamination $\mathcal L$ that can be pushed forward to the pseudo-Siegel disk.

\end{proof}

\begin{figure}[t]
    \centering
    \includegraphics[width=0.7\linewidth]{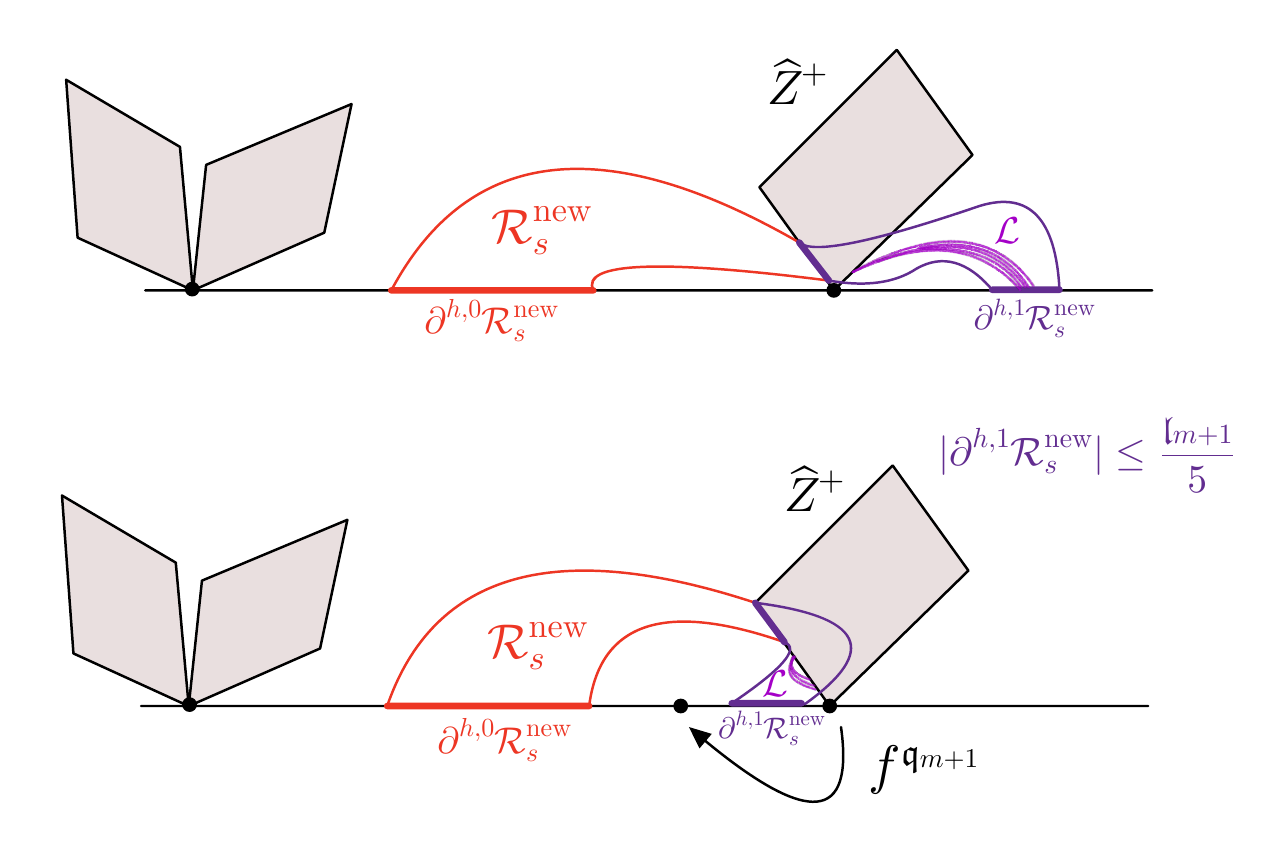}
    \caption{Illustration of the final argument in the proof of the Calibration Lemma~\ref{lem.calibration}. Since $\cal{R}_s^\new$ is almost invariant with respect to its base $\partial ^{h,0}\cal{R}_s^\new$ and since its roof $\partial ^{h,1}\cal{R}_s^\new$ is shorter than $\l_{m+1}$, we deduce the existence of an amplified lamination $\LL$.}
    \label{CL:Final}
\end{figure}

\subsection{Amplification Lemma}\label{ss:AmplLmm}
Once we have established the necessary ingredients in \S \ref{s:Snake+}, the results in this section can be derived in the same way as the quadratic case \cite[\S 8]{DL:wZ}. We aim to prove the following Amplification Lemma (Compare \cite[Amplification Theorem 8.1]{DL:wZ}):

\begin{amplmm}\label{lem.amplification}
     There are increasing functions
 $$\Blambda_{\Bt},\,\,\BK_{\Bt}\,\, \text{  for  } \Bt>1\,\, \text { with } \,\,\Blambda_{\Bt},\,\,\BK_{\Bt}\to+\infty\text{ as } \Bt\to+\infty,$$
such that the following holds. Suppose that there is a level $m$ combinatorial interval
\begin{equation}
    \label{eq:I:lem.amplification}
 I \subset \partial Z\,\,\text{ such that }\,\,\W^+
 _{\Blambda_{\Bt}}
 (I)=: \BK \geq\BK_{\Bt}\,\, \text{ and } \,\,|I|\leq\frac{|\theta_0|}{2\Blambda_{\Bt}}.
 \end{equation} 
 Consider the level-$m$ pseudo-Siegel disk $\hZ^m$. Then there is a grounded $\mathrm{rel}$
 $\hZ^m$ interval
 $J \subset\partial Z$ such that 
 $$\W^+
 _{\Blambda_{\Bt}}(J)\geq \Bt\BK \,\,\text{ and }\,\, |J|\leq|I|.$$
\end{amplmm}
\subsubsection{Applying the Covering Lemma}\label{sss:ApplCoverLmm}
\begin{lem}[Covering Lemma, \cite{KL}]
    Let $\Lambda\Subset\Lambda'\Subset U$ and $B\Subset B'\Subset V$ be two nests of Jordan disks\footnote{We allow $\Lambda$ and $B$ to degenerate into compact full continuum.}. Let $$f:(U,\Lambda',\Lambda)\longrightarrow(V,B',B),\quad\text{ with }\quad D := \mathrm{deg}(f|_{U}),\,\,\, d := \mathrm{deg}(f|_{\Lambda'})$$ 
    be a ramified holomorphic covering map between the respective domains. Then for any fixed $\kappa>1$, there exist positive constants $K = K(\kappa,D)>0$ and $C = C(\kappa,d)>0$, such that if $\W(U\setminus{\Lambda})>K$, then
    \begin{itemize}
        \item either $\W(B'\setminus{B})>\kappa K$;
        \item or $\W(V\setminus{B})>CK$.
    \end{itemize}
\end{lem}

Let $I = I_0\subset \partial Z$ be a combinatorial interval with at least one endpoint belonging to $\mathrm{CP}_{m}$. Let $I_s = f^{i_s}(I_0)$ ($s=0,1,...,q_{m+1}-1$) be the intervals obtained by spreading around $I_0$. We assume that
\begin{equation}\label{eq.Is-tiling}
    I_0,I_1,...,I_{q_{m+1}-1}
\end{equation}
 are enumerated in the clockwise direction. Let us fix $\lambda>10$ and two intervals $T,L\subset\partial Z$ whose both endpoints belong to $\mathrm{CP}_{m}$ such that 
$$10I\subset T\subset 12I,\quad \lambda I\subset L\subset (\lambda+2)I.$$
Spreading around $T$ and $L$ by $f^{i_s}$, we get 
$$10I_s\subset T_s\subset 12I_s,\quad \lambda I_s\subset L_s\subset (\lambda+2)I_s.$$ 

Applying the Cover Lemma, one can deduce the following (compare \cite[Lemma 8.2]{DL:wZ}):

\begin{lem}\label{lem.input.snake-lair}
    For every $\kappa>1$ and $\lambda>10$, there exists $K = K(\kappa,\lambda)$ (only depending on $\kappa,\lambda$) and $C_\kappa>0$ (only depending on $\kappa$), such that if $\W^+_{\lambda}(I)>K$, then for each $I_s$ in (\ref{eq.Is-tiling}), one of the following holds for its projection $I^m_s\subset \partial \wZ^m$ (as in Definition \ref{def.project-intervals}):
    \begin{itemize}
    \item either $\W_{10}(I_s^m)>\kappa K$;
        \item or $\W_{\lambda}(I_s^m)>C_\kappa K$.
    \end{itemize}
\end{lem}
\begin{proof}[Comments about Spreading Around] As in the quadratic case, to spread the degeneration around $\wZ^m$, we first project the outer degeneration $\Width^+$ witnessed by an interval $I\subset \partial Z$ as in~\eqref{eq:I:lem.amplification} from $Z$ to an appropriate pullback $\wZ^m_{*(-i)} = (f^i)^*\wZ^m$ of $\wZ^m$. Then the Covering Lemma moves degenerations from the pullbacks $(f^i)^*\wZ^m$ of  $\wZ^m$ to $\wZ^m$ so that the conclusion of Lemma~\ref{lem.input.snake-lair}.

The bistability of pseudo-Siegel disks allows us to consider such $(f^i)^*\wZ^m$, see~\S\ref{sss:dfn:bistable}. The projection of the degeneration between $Z$ and $(f^i)^*\wZ^m$ has a small effect on its width; see Theorem~\ref{thm.wellgrounded-outer-geometry}.    
\end{proof}

\subsubsection{Snake-Lair Lemma}
Amplification Lemma \ref{lem.amplification} will follow immediately from Lemma \ref{lem.input.snake-lair} and the following Snake-Lair Lemma: 
\begin{lem}\label{lem.snake-lair}
    For every $\Bt>1$, there exists $\Bkappa>1$, $\Blambda>10$ and $\bfK>0$ such that if each $I_s$ given by (\ref{eq.Is-tiling}) satisfies one of the following two types:
     \begin{enumerate}
     [label=\upshape(\roman*)]
    \item\label{type1} either $\W_{10}(I_s^m)>\Bkappa \bfK$;
        \item\label{type2} or $\W_{\Blambda}(I_s^m)>C_{\Bkappa} \bfK$
    \end{enumerate}
    then there exists a grounded $\mathrm{rel}$ $\hZ^m$ interval $J\subset\partial Z$ such that 
     $\W^+
 _{\Blambda}(J)\geq \Bt\BK$ and $|J|\leq|I|$.
\end{lem}
\begin{proof}[Sketch of the proof:]
    By Lemma \ref{lem.trade-W-to-W+}, we can assume that every interval of Type \ref{type1} satisfies (otherwise, the lemma will follow immediately)
    \begin{equation}\label{eq.W10}
        \W_{10}^+(I^m_s)\geq\Bkappa\bfK/2.
    \end{equation}
Another observation is that submergence is concentrated: there exists a subsequence of intervals of (\ref{eq.Is-tiling}) 
\begin{equation}\label{eq.subseq.concentrated}
    I_i,I_{i+1},...,I_{i+\Blambda/20},\quad \text{ s.t. }\quad \forall s\in\{i,i+1,...,i+\Blambda/20\},\quad \W_{\Blambda/4}^+(I_s^m)\leq 3. 
\end{equation} 
This is because wide families cannot cross intersect each other.

Next we can reduce the lemma to the case where Type \ref{type2} intervals appear at a certain frequency: there exists $\bfk>0$ depending on $\bft$ and $\Bkappa$ but not on $\Blambda$, such that for every two Type \ref{type2} intervals $I_s$ $I_{s'}$ appearing in (\ref{eq.subseq.concentrated}), $|s-s'|\leq\bfk$. This follows from the following argument: if there are $\bfk$ consecutive Type \ref{type1} intervals appearing in (\ref{eq.subseq.concentrated}), then using (\ref{eq.W10}) and (\ref{eq.subseq.concentrated}), one can show that these $\bfk$ Type \ref{type1} intervals will form a sufficiently wide outer family that skips over a Type \ref{type2} interval, so that Sneaking Lemma \ref{lem:sneakI} is applicable.

Finally, since Type \ref{type2} intervals in (\ref{eq.subseq.concentrated}) appears approximately $\Blambda/\bfk$ times (with $\bfk$ independent on $\Blambda$), we can take $\Blambda$ large enough so that (\ref{eq.subseq.concentrated}) creates $D_{\Blambda}$ toll barriers for some $I_s$ with $D_{\Blambda}\to \infty$ as $\Blambda\to\infty$. Then Snake Lemma with toll barriers \ref{lem:SL:toll_barriers} will amplify $\W_{\Blambda}(I_s)$ by $D_{\Blambda}$ times so that $D_{\Blambda}$ eventually beats $C_{\Bkappa}$.
\end{proof}

\section{Uniform multicritical a priori Bound}\label{s:wZ:bounds}
In this section, we will establish by induction on $m$ the uniform a priori bound:
\begin{thm}[Multicritical bounds]\label{thm.apriori-bound}
    There exist absolute constants $\bfK,\boldsymbol{\lambda},\chi>0$, such that for all $f\in \mathrm{Zak}_{\mathfrak{g}}(\theta)$, its sequence of pseudo-Siegel disks constructed in \S \ref{subsec.constructionZm} has the following properties:
    \begin{enumerate}[label=\upshape(\roman*)]
        \item \label{thm.apriori-bound:first}$\W^+_{\Blambda,\mathrm{ext},m}(J)\leq 2\sqrt{\bfK}$ for every interval $J\subset\partial Z$ grounded $\mathrm{rel}$ $\hZ^m$ with $\l_{m+1}\leq |J|\leq \l_m$.
         \item\label{thm.apriori-bound:second} For every interval $J\subset\partial Z$ grounded $\mathrm{rel}$ $\hZ^m$ with $\l_{m+1}\leq |J|\leq \l_m$, we have $$\W^+_{\Blambda}(J)\leq 2\chi\bfK.$$ 
          \item\label{thm.apriori-bound:third} For every level-$m$ combinatorial interval $I\subset\partial Z$, we have $$\W^+_{\Blambda}(I)\leq \bfK.$$ 
    \end{enumerate}
\end{thm}

The proof inductively follows by assembling all the ingredients as illustrated in Figure~\ref{fig:placeholder}.

We assume that Theorem \ref{thm.apriori-bound} has been proven for $m+1$. In particular, a pseudo-Siegel disk $\wZ^{m+1}$ has been constructed to satisfy Items~\ref{thm.apriori-bound:first}--\ref{thm.apriori-bound:third}. Note that the reference to $\wZ^{m+1}$ is indirect, through the groundedness condition. 

In~\S\ref{subsec.constructionZm}, we justify the existence of $\wZ^{m}$ so that Items~\ref{thm.apriori-bound:first}--\ref{thm.apriori-bound:third} and Assumptions~\ref{assump.1}--\ref{assump.6} of \S\ref{subsec.assump.basic} hold.

The last Assumption~\ref{assump.7}, the Log-Rule on intermediate scales~\S\ref{subsec.assump.logrule}, is verified in~\S\ref{subsec:verfication7}. In~\S\ref{ss:Tessellation}, we introduce a tessellation for $\hZ^m_s\equiv\hZ^m_{\bfS}$ controlling its inner geometry on intermediate scales. The existence of a required tessellation follows from the Log-Rule for the previous intermediate pseudo-Siegel disk $\hZ^m_{s-1}$ and the Welding Lemma~\ref{lem.strongwelding}. The tessellation yields Assumption~\ref{assump.7}.

\subsection{Construction of $\hZ^m$ from $\hZ^{m+1}$}\label{subsec.constructionZm}
For $m\gg1$, $\hZ^m={Z}$. We will construct the level-$m$ pseudo-Siegel disk $\hZ^{m}$ from the level-$(m+1)$ pseudo-Siegel disk $\hZ^{m+1}$ by the following inductively defined sequence of intermediate pseudo-Siegel disks which will be specified later: 
\begin{equation}\label{eq.intermediate-disks}
    \hZ^{m+1} =: \hZ^m_0\leadsto \hZ^m_1\leadsto \hZ^m_2\leadsto...\leadsto \hZ^m_{s} := \hZ^m. 
\end{equation}
Let $L\in\D_m$ such that $L = [c_1,c_1\boxminus\q_m\theta]$. 
Without loss of generality, we assume that 
\begin{equation}\label{eq.crit-partition}
    (c_1,c_1\boxminus\q_m\theta)\cap\left(\bigcup_{j=2}^s\mathrm{CP}_m^j\right) = \{a_2,a_3,...,a_{s}\}\text{ with } a_2\leq a_3\leq...\leq a_{s}\quad \text{(we allow }a_i \text{ to coincide.)}
\end{equation}

Set $a_1 := c_1$ and $a_{s+1} := c_1\boxminus\q_m\theta$. For $0 \leq i \leq k-1$, define 
\begin{equation}\label{eq.partition-Ti}
    T_{i,0} := [a_i,a_{i+1}],\quad T_{i,0}' = [a_i+\q_{m+1}\theta,a_{i+1}].
\end{equation}

We assume that $L$ is pre-parabolic as defined in \S~\ref{sss:dfn:preparab}. Otherwise, $|L|< 5\mathfrak{l}_{m+1}$, and no regularization is performed.

\paragraph{Existence of central rectangles.} If there does not exist any external parabolic rectangle $\R_{i,0}\subset\C\setminus{Z}$ based on $T_{i,0}'$ with $\W(\R_{i,0}) = \sqrt{\bfK}$, then we set $\hZ^m_{i+1} = \hZ^m_i$. Otherwise, by definition we have 
\begin{equation}
\label{eq:sqrt.bfK}
    \W^+_{\mathrm{ext},m}(\partialhL\R_{i,0},\partialhR\R_{i,0})=\sqrt{\bfK}.
\end{equation} 
We may assume that $\partialhL\R_{i,0}$ starts from $v'$ and $\partialhR\R_{i,0}$ ends at $w$ (recall that in \S \ref{sec.fjord} we set $T_{i,0} = [v,w]$ and $T_{i,0}' = [v',w]$, where $v' = v\boxplus\theta_{m+1}$).  By Lemma \ref{lem.central-exists}, if $\R_{i,0}$ does not contain a balanced, geodesic and central rectangle $\Rnew_{i,0}$ such that $\W(\Rnew_{i,0})\geq\sqrt{\bfK}/2$, then there exists $I\subset T'$ with $|I|>\l_{m+1}$ such that $\log\W^+_{\Blambda,\mathrm{div},m}(I)\succeq_{\Blambda}\W(\R_{i,0})$. Hence $\W^+_{\Blambda,\mathrm{div},m}(I)\gg2\chi{\bfK}$. By Calibration Lemma \ref{lem.calibration}, either there is a $[2\bfK,\Blambda]^+$-wide level-$(m + 1)$ combinatorial interval, or there is a $[2\chi^{1.5}\bfK,\Blambda]^+$-wide interval $I' \subset \partial Z$, $|I'| \leq \l_{m+1}$ grounded $\mathrm{rel}$ $\hZ^{m+1}$. The first case contradicts Theorem \ref{thm.apriori-bound} \ref{thm.apriori-bound:third}; the second case contradicts Theorem \ref{thm.apriori-bound} \ref{thm.apriori-bound:second}. Therefore such $\Rnew_{i,0}$ must exist. Let us spread $\Rnew_{i,0}$ to other $L\in\D_m$ under the conformal pullback by $f^j$ with $0\leq j<\q_{m+1}$ (this can be done since $\Rnew_{i,0}$ contains no critical value of $f^j$). Then for each $L\in\D_m$, we get a balanced and central rectangle $\Rnew_{i,j}$ based on $L$ such that $\W(\Rnew_{i,j})\geq\sqrt{\bfK}/2$.

\paragraph{Construction of dams $\beta^m_{i,j}$.}
Let us split $\Rnew_{i,0}$ into two balanced geodesic sub-rectangles $\mathcal{R}'_{i,0}$ and $\RNew_{i,0}$, where $\mathcal{R}'_{i,0}$ is the outermost one, such that $$\W(\mathcal{R}'_{i,0})\asymp\W(\RNew_{i,0})\asymp\sqrt{\bfK}.$$
Recall that in the proof of Lemma \ref{lem.central-exists}, we obtain $\Rnew_{i,0}$ by removing a sufficently wide buffer from $\R_{i,0}$. Hence Theorem \ref{thm.fjord} is applicable to intervals in $\lfloor\partialhL\Rnew_{i,0},\partialhR\Rnew_{i,0}\rfloor$. In particular, we may assume that the end points of $\partial^h\R'_{i,0}$ and $\partial^h\RNew_{i,0}$ belong to $\mathrm{CP}_{m+1}$. It is then easy to check that $\RNew_{i,0}$ is also central. Let us denote $I_{i,0}: =\partialhL\RNew_{i,0}$ and $J_{i,0} = \partialhR\RNew_{i,0}$. We construct 6 consecutive sub-intervals in $I_{i,0}$ such that
\begin{equation*}
    \begin{aligned}
        &X^{2,+}_{i,0}<X^{1,+}_{i,0}<X^{0,+}_{i,0}<X^{0,-}_{i,0}<X^{1,-}_{i,0}<X^{2,-}_{i,0}\text{ and }\\
        &X^{2,+}_{i,0}\text{shares the same left end point as }\partialhL\RNew_{i,0}.
    \end{aligned}
\end{equation*}
Similarly, We construct 6 consecutive sub-intervals in $J_{i,0}$ such that
\begin{equation*}
    \begin{aligned}
        &Y^{2,+}_{i,0}>Y^{1,+}_{i,0}>Y^{0,+}_{i,0}>Y^{0,-}_{i,0}>Y^{1,-}_{i,0}>Y^{2,-}_{i,0}\text{ and }\\
        &Y^{2,+}_{i,0}\text{shares the same right end point as }\partialhR\RNew_{i,0}.
    \end{aligned}
\end{equation*}
Moreover, we require that all the end points of these 12 intervals belong to $\mathrm{CP}_{m+1}$. We pick an absolute constant $\sqrt{\bfK}\gg S\gg1$ ($S$ to be determined later) and we require furthermore that
\begin{equation}\label{eq.interval-length-e^S}
|Y^{s,\epsilon}_{i,0}|=|X^{s,\epsilon}_{i,0}|\asymp |X^{s',\epsilon'}_{i,0}|\asymp e^{S}\cdot\dist(v,\partialhL\RNew_{i,0})\quad \text{for any }s,s'\in\{0,1,2\}\text{ and }\epsilon,\epsilon'\in\{+,-\}.
\end{equation}
Let $\H^{s,\epsilon}_{i,0}$ be the geodesic rectangle in $\widehat{\C}\setminus{Z}$ based on $X^{s,\epsilon}_{i,0},Y^{s,\epsilon}_{i,0}$. Then by Theorem \ref{thm.fjord} \ref{thm.fjord:second} and Lemma \ref{lemappen.geodesic}, $\W(\H^{s,\epsilon}_{i,0})\asymp S$. Since $\W(\R'_{i,0})\asymp\sqrt{\bfK}$, we have $\dist(v,\partialhL\RNew_{i,0})\geq|\partialhL\R'_{i,0}|\geq e^{c\sqrt{\bfK}}\cdot\l_{m+1}$ for some absolute constant $c>0$ by Theorem \ref{thm.fjord} \ref{thm.fjord:second}.
Hence by taking $S$ large enough so that $e^S$ beats "$\asymp$" in (\ref{eq.interval-length-e^S}), we have:

\begin{equation}\label{eq.interval-gg-K}
|Y^{s,\epsilon}_{i,0}|=|X^{s,\epsilon}_{i,0}|\geq e^{c\sqrt{\bfK}}\l_{m+1} \gg \bfK\l_{m+1}\gg\Blambda\l_{m+1}.
\end{equation}

Let us spread around $X^{s,\epsilon}_{i,0}$ and $Y^{s,\epsilon}_{i,0}$ along $\partial Z$ under the pullback of $f^j$ with $0<j<\q_{m+1}$. This gives $X^{s,\epsilon}_{i,j},Y^{s,\epsilon}_{i,j}$. Let $x^m_{i,j}$ be the common end point of $X^{0,+}_{i,j}$ and $X^{0,-}_{i,j}$; let $y^m_{i,j}$ be the common end point of $Y^{0,+}_{i,j}$ and $Y^{0,-}_{i,j}$. We set the dam $\beta^m_{i,j}$ to be the geodesic in $\widehat{\C}\setminus{Z}$ connecting $x^m_{i,j},y^m_{i,j}$. One can verify easily that Assumption \ref{assump.2} is satisfied. 

\paragraph{Construction of annuli $A^{\mathrm{inn}}(z)$ and $A^{\mathrm{out}}(z)$.}
For $k = 0,1,2$, let us denote $$X^k_{i,j} := \bigcup_{l=0}^kX^{l,-}_{i,j}\cup X^{l,+}_{i,j}\quad \text{and}\quad Y^k_{i,j} := \bigcup_{l=0}^kY^{l,-}_{i,j}\cup Y^{l,+}_{i,j}.$$ 
Recall that in Definition \ref{def.geodesic-annulus}, we construct a geodesic disk $O_{J,\hZ^{m+1}}$ for any interval $J\subset\partial\hZ^{m+1}$. For simplicity of notations, we temporarily omit the subscript $\hZ^{m+1}$ below (therefore $ O_{J} := O_{J,\hZ^{m+1}}$):
$$A^{\mathrm{inn}}(x^m_{i,j}) := O_{X^1_{i,j}}\setminus\overline{O_{X^0_{i,j}}}\quad\text{and}\quad A^{\mathrm{out}}(x^m_{i,j}) := O_{X^2_{i,j}}\setminus\overline{O_{X^1_{i,j}}}.$$
$$A^{\mathrm{inn}}(y^m_{i,j}) := O_{Y^1_{i,j}}\setminus\overline{O_{Y^0_{i,j}}}\quad\text{and}\quad A^{\mathrm{out}}(y^m_{i,j}) := O_{Y^2_{i,j}}\setminus\overline{O_{Y^1_{i,j}}}.$$
Now (\ref{eq.interval-gg-K}) makes Corollary \ref{cor.welding} applicable, hence either 
\begin{itemize}
    \item the modulus of all these annuli is $\succeq1$ (we denote by $\Bdelta$ this universal lower bound);
    \item or there exists an interval $I'\subset\partial Z$ with $|I'|\leq\l_{m+1}$ such that $I'$ is grounded $\mathrm{rel}$ $\hZ^{m+1}$ and for some absolute constant $c>0$, one has
 \begin{equation*}
     \W^+_{\boldsymbol{\lambda},Z}(I')\geq e^{c\sqrt{\bfK}}\gg\bfK.
 \end{equation*}
\end{itemize}
 However the second case contradicts Theorem \ref{thm.apriori-bound} \ref{thm.apriori-bound:third}. This finishes the construction of $A^{\mathrm{inn}}(z)$ and $A^{\mathrm{out}}(z)$. One can verify that Assumptions \ref{assump.3} and \ref{assump.4} and \ref{assump.6} are satisfied. Since $\X^{\mathrm{inn}}(\beta^m_{i,j})$ is based on $X^0_{i,j}\cup Y^0_{i,j}$ by definition, $\W(\X^{\mathrm{inn}}(\beta^m_{i,j}))\asymp S\gg 1$ by construction. Assumption \ref{assump.5} can be then verified by non-overlapping principle (Lemma \ref{lemappen.disjoint-boundaires}).

\subsection{Tessellation around the fjord}\label{ss:Tessellation}
In this part, we will construct a tessellation of rectangles, which gives a fine geometric control on the fjord. 

Let us fix any subset $\bfS\subset\bfI := \{0,1,...,i-1\}$ and consider the intermediate pseudo-Siegel disk $\hZ^m_{\bfS}\subset\hZ^m_i$.
Let us fix any three adjacent intervals $\widehat{L},L,\tilde{L}\in\D_m$ such that $\widehat{L}<L<\tilde{L}$ and let $p\in\{1,...,s\}\setminus\bfS$ be any integer. Let 
\begin{equation}\label{eq.Lp}
    \widehat{L}_p := [\widehat{l}_p,\widehat{l}_{p+1}]\subset\widehat{L},\quad L_p := [l_p,l_{p+1}]\subset L,\quad\tilde{L}_p := [\tilde{l}_p,\tilde{l}_{p+1}]\subset\tilde{L}
\end{equation}
be the pulled back intervals of $[a_p,a_{p+1}]$ by some $f^s$ with $0\leq s\leq\q_{m+1}-1$ (recall $[a_p,a_{p+1}]$ given by (\ref{eq.crit-partition}) and (\ref{eq.partition-Ti})). Denote by $\widehat{\beta}_p,\beta_p,\tilde{\beta}_p$ the dam on $\widehat{L}_p,L_p,\tilde{L}_p$ respectively. Let $x_p<y_p$ be the end points of $\beta_p$. Let us partition $L_p^{m,\bfS}$ into:
\begin{equation}\label{eq.partition-Li}
    L_p^{m,\bfS} = \left(\bigcup_{j=0}^{r} S_{j,+}\cup S_{j,-}\right)\cup S_\star
\end{equation}
such that 
\begin{itemize}
    \item $S_{0,-} = [l_p,x_p]^{m,\bfS}$, $S_{0,+} = [y_p,l_{p+1}]^{m,\bfS}$; $S_{j,\pm}$ is adjacent to $S_{j+1,\pm}$ for all $0\leq j\leq r-1$.
    \item $S_{j,\pm}$ are well-grounded $\mathrm{rel}$ $\hZ^{m+1}$.
    \item $|S_{j,-}|_{m,\bfS}\asymp|S_{j,+}|_{m,\bfS}\gg_{\lambda,\Bdelta}\l_{m+1}$, $|S_{j,\pm}|_{m,\bfS}\asymp\sum_{k=0}^{j-1}|S_{k,\pm}|_{m,\bfS}$, $|S_\star|\asymp|S_{r,\pm}|_{m,\bfS}$.
\end{itemize}

Similarly, we can partition $\widehat{L}_p^{m,\bfS},\tilde{L}_p^{m,\bfS}$ into
\begin{equation}\label{eq.partition-hatL-tilde-L}
    \widehat{L}_p^{m,\bfS} = \left(\bigcup_{j=0}^{r} \widehat{S}_{j,+}\cup \widehat{S}_{j,-}\right)\cup \widehat{S}_\star,\quad\tilde{L}_p^{m,\bfS} = \left(\bigcup_{j=0}^{r} \tilde{S}_{j,+}\cup \tilde{S}_{j,-}\right)\cup \tilde{S}_\star
\end{equation}
such that the partitions have the same properties as (\ref{eq.partition-Li}). Notice that we can keep the same number of pairs of intervals $r$ appearing in (\ref{eq.partition-Li}) and (\ref{eq.partition-hatL-tilde-L}) by adjusting the constant in “$\asymp$”, since we have
$|\widehat{L}_p^{m,\bfS}|_{m,\bfS}\asymp|L_i^{m,\bfS}|_{m,\bfS}\asymp|\tilde{L}_p^{m,\bfS}|_{m,\bfS}$ and $|\widehat{S}|_{0,\pm}\asymp|S_{0,\pm}|\asymp |\tilde{S}|_{0,\pm}$.

Let us denote by $\widehat{\mathcal{C}}_j,\mathcal{C}_j,\tilde{\mathcal{C}}_j\subset\widehat{\C}\setminus\widehat{Z}^{m}_{\bfS}$ ($0\leq j\leq r$) the rectangles based on $\widehat{S}_{j,\pm}$, $S_{j,\pm}$ and $\tilde{S}_{j,\pm}$ respectively, with vertical boundaries being geodesics in $\widehat{\C}\setminus Z$. Then by Theorem \ref{thm.fjord} \ref{thm.fjord:first} and Theorem \ref{thm.wellgrounded-outer-geometry}, we have 
\begin{equation}\label{eq.W(C-Ctilde)}
\W(\widehat{\mathcal{C}}_j),\W(\mathcal{C}_j),\W(\tilde{\mathcal{C}}_j)\asymp 1.
\end{equation}

Let us define two almost midpoints in the following way:

\begin{equation}\label{eq.midpoint-bi}
    \begin{aligned}  &b_{p,+}\in\mathrm{CP}_{m+1}\cap[l_{p+1},\tilde{l}_p] \quad\text{is the closest point to}\quad l_{p+1} \quad\text{such that}\quad\dist_{m,\bfS}(b_{p,+},l_{p+1})\geq\frac{1}{2}|[l_{p+1},\tilde{l}_p]^{m,\bfS}|_{m,\bfS};\\
    &b_{p,-}\in\mathrm{CP}_{m+1}\cap[\widehat{l}_{p+1},l_p] \quad\text{is the closest point to}\quad \widehat{l}_{p+1}\quad\text{such that}\quad \dist_{m,\bfS}(b_{p,-},\widehat{l}_{p+1})\geq\frac{1}{2}|[\widehat{l}_{p+1},l_p]^{m,\bfS}|_{m,\bfS}.
    \end{aligned}
\end{equation}
Denote $$\widehat{N}_{p,+} := [\widehat{l}_{p+1},b_{p,-}],\quad N_{p,-}:=[b_{p,-},l_p],\quad N_{p,+}:=[l_{p+1},b_{p,+}]\quad\text{and}\quad\tilde{N}_{p,-} := [b_{p,+},\tilde{l}_p].$$ Let us partition $N_{p,\pm}^{m,\bfS}$ into:
\begin{equation}\label{eq.partition-Ni}
    N_{p,\pm}^{m,\bfS} = \bigcup_{j=0}^{t} U_{j,\pm}
\end{equation}
such that 
\begin{itemize}
    \item the left end point of $U_{0,-}$ is $b_{p,-}$, the right end point of $U_{t,-}$ is $l_p$; the left end point of $U_{0,+}$ is $l_{p+1}$, the right end point of $U_{t,+}$ is $b_{p,+}$; $U_{j,\pm}$ is adjacent to $U_{j+1,\pm}$ for all $0\leq j\leq t-1$;
    \item $U_{j,\pm}$ are grounded $\mathrm{rel}$ $\hZ^m_{\bfS}$.
    \item $|U_{0,\pm}|_{m,\bfS}\asymp|S_{0,\pm}|_{m,\bfS}$, $|U_{j,\pm}|_{m,\bfS}\asymp\sum_{k=0}^{j-1}|U_{k,\pm}|_{m,\bfS}$.
\end{itemize}

Similarly we can partition $\widehat{N}^{m,\bfS}_{p,+},\tilde{N}^{m,\bfS}_{p,-}$ into
\begin{equation}\label{eq.partition-hatN-tildeN}
    \widehat{N}^{m,\bfS}_{p,+} = \bigcup_{j=0}^{t} \widehat{U}_{j,+},\quad\tilde{N}^{m,\bfS}_{p,-} = \bigcup_{j=0}^{t} \tilde{U}_{j,-}.
\end{equation}
such that the partitions have the same properties as (\ref{eq.partition-Ni}). Similarly, we can also keep the same $t$ for (\ref{eq.partition-Ni}) and (\ref{eq.partition-hatN-tildeN}).

\begin{rem}\label{rem.replace-shortcutlenght}
    We can replace the shortcut length $|\cdot|_{m,\bfS}$ of $\widehat{S}_{j,\pm},S_{j,\pm},\tilde{S}_{j,\pm}$ in (\ref{eq.partition-Li}) and (\ref{eq.partition-hatL-tilde-L}) by the projective length $|\cdot|$, since the intervals $\widehat{L}_p^{m,\bfS},L_p^{m,\bfS},\tilde{L}_p^{m,\bfS}$ do not intersect any level-$m$ dam; we can also replace the shortcut length $|\cdot|_{m,\bfS}$ of $\widehat{U}_{j,+},U_{j,\pm},\tilde{U}_{j,-}$ in (\ref{eq.partition-Ni}) and (\ref{eq.partition-hatN-tildeN}) by $|\cdot|_{m,\bfS\cup\{p\}}$, since $\widehat{N}^{m,\bfS}_{p,+},N^{m,\bfS}_{\pm},\tilde{N}^{m,\bfS}_{p,-}$ do not intersect any $p$-th level-$m$ dam $\beta^m_{p,j}$.
    
    In this section, we temporarily stick to $|\cdot|_{m,\bfS}$ for consistency in the notations.
\end{rem}
We can adjust the constant in “$\asymp$” appearing in partitions (\ref{eq.partition-Li}) and (\ref{eq.partition-Ni}), such that $r\not=t$ (this adjustment can be made uniform, i.e. independent of $m$ and $\bfS$).
Let us consider two cases in terms of the relation between $r$ and $t$:
\paragraph{Case $r> t$.} (see Figure \ref{fig.tess.r-geq-t}). 
\begin{itemize}
    \item For $0\leq j\leq t$, let 
    \begin{itemize}
        \item $\G_{j,\pm}\subset \hZ^m_{\bfS}$ be the geodesic rectangle based on $S_{j,\pm},U_{j,\pm}$; 
        \item $\widehat{\G}_{j,+}\subset \hZ^m_{\bfS}$ be the geodesic rectangle based on $\widehat{U}_{j,+},\widehat{S}_{j,+}$;
        \item $\tilde{\G}_{j,-}\subset \hZ^m_{\bfS}$ be the geodesic rectangle based on $\tilde{U}_{j,-},\tilde{S}_{j,-}$.
        \item $\widehat{\mathcal{D}}_{j,+} = \widehat{\G}_{j,+}\cup\widehat{\mathcal{C}}_j$, $\mathcal{D}_j = \G_{j,-}\cup\mathcal{C}_j\cup\G_{j,+}$, $\tilde{\mathcal{D}}_{j,-} = \tilde{\G}_{j,-}\cup\tilde{\mathcal{C}}_j$.
    \end{itemize}  
    
    \item For $t< j\leq r$, let 
    \begin{itemize}
        \item ${\H}_{j,-}\subset \hZ^m_{\bfS}$ be the geodesic rectangle based on $\widehat{S}_{j,+},{S}_{j,-}$;
        \item $\H_{j,+}\subset\hZ^m_{\bfS}$ be the geodesic rectangle based on ${S}_{j,+},\tilde{S}_{j,-}$;
         \item $\mathcal{D}_j = \H_{j,-}\cup\mathcal{C}_j\cup\H_{j,+}$.
    \end{itemize}
\end{itemize}

     By induction, the Log-Rule (Assumption \ref{assump.7}) on $\hZ^m_{\bfS}$ holds. Hence we have 
\begin{equation}\label{eq.W(G-Gtilde)}
\W(\widehat{\G}_{j,+}),\W(\G_{j,\pm}),\W(\tilde{\G}_{j,-}),\W(\H_{j,\pm})\asymp_{\Bdelta}1.
\end{equation}

Let $D_{j,\pm}\subset\partial\hZ^{m+1}$ be the common horizontal side of $\mathcal{C}_j$ and $\G_{j,\pm}$ or $\H_{j,\pm}$. Let us split $D_{j,\pm}$ into three sub-intervals with projective length $\asymp\frac{1}{3}|D_{j,\pm}|$ whose end points belong to $\mathrm{CP}_{m+1}$. Let $D^0_{\pm}$ be the middle sub-interval. By Welding Lemma \ref{lem.strongwelding}, we have
$$\mathrm{mod}(O_{D_{j,\pm},\hZ^{m+1}}\setminus D^0_{j,\pm})\succeq1$$ 
(recall the notation of geodesic disk $O_{J,\hZ^{m+1}}$ in Definition \ref{def.geodesic-annulus}). Notice that $O_{D_{j,\pm},\hZ^{m+1}}\subset O_{D_{{j,\pm},\hZ^m_{\bfS}}}$ because $\hZ^{m+1}\subset\hZ^m_{\bfS}$. Hence $\partial O_{D_{j,\pm},\hZ^{m+1}}$ intersects with $\partial\mathcal{D}_{j,\pm}$ only at the two end points of $D_{j,\pm}$. Combining (\ref{eq.W(C-Ctilde)}), (\ref{eq.W(G-Gtilde)}) and
using a similar argument as in Corollary \ref{cor.welding}, we get 
\begin{equation}\label{eq.W(Dpm)}
\W(\mathcal{D}_{j,\pm})\asymp_{\Bdelta}1.
\end{equation}
In the same way, we also have 
\begin{equation}\label{eq.W(hat-tilde-Dpm)}
\W(\widehat{\mathcal{D}}_{j,+}),\W(\tilde{\mathcal{D}}_{j,-})\asymp_{\Bdelta}1.
\end{equation}

\begin{figure}[t]
\centering 
\includegraphics[width=0.85\textwidth]{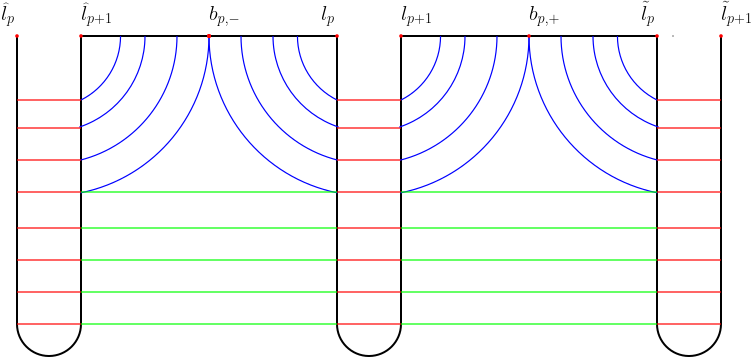} 
\caption{Tessellation in Case $r>t$.} 
\label{fig.tess.r-geq-t}
\end{figure}

\paragraph{Case $r<t$.}(see Figure \ref{fig.tess.r-leq-t}). 
\begin{itemize}
    \item For $0\leq j\leq r$, the rectangles ${\G}_{j,\pm}$, $\widehat{\G}_{j,+}$, $\tilde{\G}_{j,-}$ and $\mathcal{D}_j$ are defined in the same way as the first subcase in “Case $r\geq t$”. Similarly, (\ref{eq.W(G-Gtilde)}) and (\ref{eq.W(Dpm)}) hold.

\item For $j=r+1$, let $\mathcal{D}_r$ be the rectangle based on ${U}_{r,-},U_{r,+}$ such that $\partial^{v,0}\mathcal{D}_r= \partial^{v,1}\mathcal{D}_{r-1}$ and $\partial^{v,1}\mathcal{D}_r$ is the geodesic in $\hZ^m_{\bfS}$ connecting the left end point of $U_{r,-}$ and the right end point of $U_{r,+}$. Let $\mathcal{D}_r'\subset\hZ^m_{\bfS}$ be the geodesic rectangle based on $U_{r,-},U_{r,+}$. By construction, $\mathcal{D}_r'\subset\mathcal{D}_r$. Moreover, any vertical curve in $\mathcal{D}_r$ not contained in $\F^-_{\hZ^m_{\bfS}}(U_{r,-},U_{r,+})$ must intersect $S_\star$. However, by Welding Lemma \ref{lem.strongwelding} and Theorem \ref{thm.apriori-bound} \ref{thm.apriori-bound:third}, such curve must cross an annulus with modulus $\succeq1$. On the other hand, we have $\W(\mathcal{D}_r'),\W^-_{\hZ^m_{\bfS}}(U_{r,-},U_{r,+})\asymp_{\Bdelta}1$ by the Log-Rule (Assumption \ref{assump.7}) on $\hZ^m_{\bfS}$. Therefore we also have
\begin{equation}\label{eq.Dr}
    \W(\mathcal{D}_r)\asymp_{\Bdelta}1.
\end{equation}

\item For $r+1<j\leq t$, let $\mathcal{D}_j$ be the rectangle geodesic rectangle in $\hZ^m_{\bfS}$ based on $U_{j,-}$ and $U_{j,+}$. The Log-Rule (Assumption \ref{assump.7}) on $\hZ^m_{\bfS}$ holds by induction. Hence (\ref{eq.W(Dpm)}) also holds.
\end{itemize}

\begin{figure}[t]
\centering 
\includegraphics[width=0.95\textwidth]{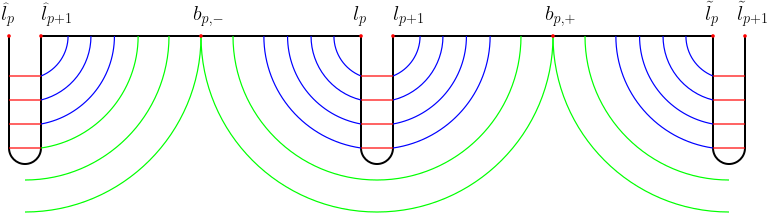} 
\caption{Tessellation in Case $r<t$.} 
\label{fig.tess.r-leq-t} 
\end{figure}

\subsection{Verification of Assumption \ref{assump.7}}\label{subsec:verfication7}

We will prove by induction on $i$ the Log-Rule on $\hZ^m_{\bfS}$.
\begin{thm}[Log-Rule on intermediate scale]\label{thm.inter-log-rule}
Let $I,J\subset \partial\hZ^m_{\bfS}$ be a grounded pair such that $\dist(I,J)\geq\l_{m+1}$ and $\min\{|I|,|J|\}\geq\l_{m+1}$.

 \begin{enumerate}[label=\upshape(\roman*)]
 \item If $\min\{|I|_{m,\bfS},|J|_{m,\bfS}\}\geq\dist_{m,\bfS}(I,J)$, then
 $$\W^-_{\hZ^m_{\bfS}}(I,J)\asymp_{\Bdelta}\log\frac{\min\{|I|_{m,\bfS},|J|_{m,\bfS}\}}{\dist_{m,\bfS}(I,J)}+1.$$

 \item If $\min\{|I|_{m,\bfS},|J|_{m,\bfS}\}\leq\dist_{m,\bfS}(I,J)$, then
 $$\W^-_{\hZ^m_{\bfS}}(I,J)\asymp_{\Bdelta}\left(\log\frac{\dist_{m,\bfS}(I,J)}{\min\{|I|_{m,\bfS},|J|_{m,\bfS}\}}+1\right)^{-1}$$
 \end{enumerate}
\end{thm}
\begin{proof} We proceed by induction. The induction base is Theorem~\ref{thm.large-log-rule}, which establishes uniform Log-Rule after the transition $\wZ^{m+1}\leadsto \wZ^m$; cf., Figure~\ref{fig:placeholder}. With each induction step, the constant representing ``$\asymp_{\Bdelta}$'' can increase. Since the induction has at most $\mathfrak g-1$ steps, Theorem~\ref{thm.inter-log-rule} will follow.

    Let us assume that the theorem is proved for any $\bfS'\subset\{0,1,...,i-1\}$ and we need to verify the theorem for any $\bfS\subset\{0,1,...,i\}$. Therefore, we can always assume that 
    $$\bfS = \bfS'\cup\{i\}\quad\text{for some}\quad\bfS'\subset\{0,1,...,i-1\}.$$
    By Splitting Argument (Appendix \ref{appen.log-rule}), it suffices to prove that $$\W^-_{\hZ^m_{\bfS}}(I,J)\asymp_{\Bdelta}1\quad \text{ if }\quad\min\{|I|_{m,\bfS},|J|_{m,\bfS}\} \asymp \dist_{m,\bfS}(|I|,|J|).$$ Let us consider two cases:
    \begin{itemize}
        \item Case $\min\{|I|,|J|\}\geq 10\l_{m}$. Then by Lemma \ref{lem.shortcut-projective-comparable}, we also have
        \begin{equation}\label{eq.min-asymp-dist}
            \min\{|I|,|J|\} \asymp \dist(I,J).
        \end{equation}
        To fix the idea, let us suppose that $|J|\geq|I|$. We can split $I$ into a concatenation $I = I_1\#I_2\#...\#I_n$ with $n$ only depending on the “$\asymp$” in (\ref{eq.min-asymp-dist}), such that for each $1\leq l\leq n$ 
        \begin{itemize}
            \item $I_l$ is grounded $\mathrm{rel}$ $\hZ^m_{\bfS}$ with $|I_l|\geq 10\l_m$;
            \item we associate $I_l$ with an interval $L_l\subset\partial\hZ^m_{\bfS}$ grounded $\mathrm{rel}$ $\hZ^m_{\bfS}$ such that $10|I_l|\leq\dist(I_l,L_l)\leq11|I_l|$;
            \item $\F^-_{\hZ^m_{\bfS}}(I,J)\subset\bigcup_{l=1}^n\F^-_{\hZ^m_{\bfS}}(I_l,L_l)$.
        \end{itemize}
       Then by Lemma \ref{lem.coarse-bound} and Parallel Law, we get $\W^-_{\hZ^m_{\bfS}}(I,J)\preceq_{\Bdelta}1$. Applying the same argument to the complement intervals of $I\cup J$, we get $\W^-_{\hZ^m_{\bfS}}(I,J)\succeq_{\Bdelta}1$. This finishes the proof.

       \item Case $\min\{|I|,|J|\}\leq10\l_{m}$. Again let us suppose that $|J|\geq|I|$. By Parallel Law, it suffices to prove the case where $|I|\leq\l_m$. We assume furthermore that $|J|_{m,\bfS}\geq|I|_{m,\bfS}$. Let us write uniquely $I$ into the concatenation 
       $$I = I_-\#\beta^m_{i,j}\#I_+,$$ where $I_-,I_+$ are regular intervals in $\hZ^m_{\bfS'}$ and $\beta^m_{i,j}$ is an $i$-th level $m$-dam. Recall that in \S \ref{ss:Tessellation}, there are two cases in the tessellation of rectangles around the fjord $\mathfrak{F}^m_{i,j}$ that we constructed (the tessellation relies on Log-Rule of $\hZ^m_{\bfS'}$, i.e. Assumption \ref{assump.7}). We only treat “Case $r>t$” below and it will be similar for “Case $r< t$”.

       Recall that there are three adjacent intervals $\widehat{L}<L<\tilde{L}\in\D_m$, such that $I_-\subset L^{m,\bfS}$ and $\beta^m_{i,j}\# I_+\subset\tilde{L}^{m,\bfS}$. Since $J$ is grounded $\mathrm{rel}$ $\hZ^m_{\bfS}$ and $\beta^m_{i,j}$ is protected by an annulus of modulus $\succeq_{\Bdelta}1$ ((\ref{eq.Ainn-out-dam}) in Assumption \ref{assump.4}, we have $\W^-_{\hZ^m_{\bfS}}(\beta^m_{i,j},J)=O_{\Bdelta}(1)$. 
           
           Next we verify that $\W^-_{\hZ^m_{\bfS}}(I_-,J)=O_{\Bdelta}(1)$ and it will be similar for $I_+$. Recall the construction of $L_i = [l_p,l_{p+1}]$ and $\tilde{L}_i=[\tilde{l}_p,\tilde{l}_{p+1}]$ given by (\ref{eq.Lp}). Recall the almost midpoints $\b_{p,\pm}$ given by (\ref{eq.midpoint-bi}). Let $I_{-,1}\subset I_-$ be the sub-interval on the left-hand side of $b_{p,+}$ and $I_{-,2}\subset I_-$ on the right-hand side of $b_{p,+}$. We only verify that $\W^-_{\hZ^m_{\bfS}}(I_{-,1},J) = O_{\Bdelta}(1)$ and it will be similar for $I_{-,1}$. For simplicity of notations, we assume that $I=I_{-,2}$. Recall that we set $N_{p,+} = [l_{p+1},b_{p,+}]$ and we have partition (\ref{eq.partition-Ni}): $N_{p,+}^{m,\bfS} = \bigcup_{j=0}^{t} U_{j,+}$. Let 
           $U_{s,+}$ ($s\leq t$) be the right most interval in the partition that intersects $I$. We can adjust the constant in “$\asymp$” appearing in partitions (\ref{eq.partition-Li}) and (\ref{eq.partition-Ni}), such that $|U_{s,+}|_{m,\bfS}\asymp|I|_{m,\bfS}$.
           \begin{itemize}
               \item We first estimate $\W^-_{\hZ^m_{\bfS}}(U_{s,+},J)$. Let $\F_1$ be the lamination of vertical curves in $\F^-_{\hZ^m_{\bfS}}(U_{s,+},J)$ that are contained in $\hZ^m_{\bfS'}$. Then by induction on the Log Rule of $\hZ^m_{\bfS'}$, we have $\W(\F_1) = O_{\Bdelta}(1)$. The remaining vertical curves must enter $i$-th level $m$ fjord. Hence they 
               \begin{itemize}
                   \item  either cross one of $\mathcal{D}_{s-1},\tilde{\mathcal{D}}_{t,-},\mathcal{D}_{s+1},\mathcal{D}_{t+1}$;
                   \item or overflow $\F^-_{\hZ^m_{\bfS'}}(U_{s,+},\bigcup_{l=s-1}^{t+1}S_{l,+}\cup \tilde{S}_{l,-})$.
               \end{itemize}
              The width of curves in the first case is at most $O_{\Bdelta}(1)$ by (\ref{eq.W(Dpm)}) and (\ref{eq.W(hat-tilde-Dpm)}); the width of curves in the second case is at most $O_{\Bdelta}(1)$ by induction on the Log Rule of $\hZ^m_{\bfS'}$.
             \item Next we estimate $\W^-_{\hZ^m_{\bfS}}(I\setminus U_{s,+},J)$. Let $\F_2$ be the lamination of vertical curves in $\F^-_{\hZ^m_{\bfS}}(I\setminus U_{s,+},J)$ that cross $\mathcal{D}_s$. Then $\W(\F_2) = O_{\Bdelta}(1)$ because of (\ref{eq.W(Dpm)}). The remaining vertical curves overflow one of the three families below: 
             \begin{itemize}
                 \item $\F^-_{\hZ^m_{\bfS'}}(I,J\cap\partial\hZ^m_{\bfS'})$;
                 \item $\F^-_{\hZ^m_{\bfS}}(I,\beta^m_{i,j})$;
                 \item $\mathcal{D}_{l}\cup\mathcal{D}_{l+1}\cup...\cup\mathcal{D}_{l+q}$.
             \end{itemize}
            In the first case, the width is $O_{\Bdelta}(1)$ by induction on the Log Rule of $\hZ^m_{\bfS'}$; the second case is also $O_{\Bdelta}(1)$ because of the protection of annulus around $\beta^m_{i,j}$ (\ref{eq.Ainn-out-dam}). In the last case, the number of rectangles $n$ is $\asymp1$ since $\min\{|I|_{m,\bfS},|J|_{m,\bfS}\} \asymp \dist_{m,\bfS}(|I|,|J|)$. Therefore by (\ref{eq.W(Dpm)}), the width is also $O_{\Bdelta}(1)$.
           \end{itemize}
    \end{itemize}
We have therefore verified that $\W^-_{\hZ^m_{\bfS}}(I,J)\preceq_{\Bdelta}1$. The same argument applied to the complement intervals of $I,J$ will give that $\W^-_{\hZ^m_{\bfS}}(I,J)\succeq_{\Bdelta}1$. 

Let $\hZ^m = \hZ^m_s$ be the full level-$m$ pseudo-Siegel disk. We have now established Log-Rule (Theorem \ref{thm.inter-log-rule}) on $\hZ^m_s$ but with the comparable constant "$\asymp_{\Bdelta_s}$" depending on both $\Bdelta$ and $s$. In order to improve $\Bdelta_s$ to be only $\Bdelta$-depending, we use Theorem \ref{thm.large-log-rule} on $\hZ^m$: if intervals of large scale are $T_{\Bdelta_s}$ separated, the Log-Rule is uniform. Therefore, the rectangles in the tesselation of $\hZ^m$ constructed in \S \ref{ss:Tessellation} all have width $\asymp1$, where the comparison constant is uniform. A similar argument as above shows that curves connecting $I,J\subset\partial\hZ^{m-1}_1$ will either cross/overflow some of these rectangles in $\hZ^{m}$ of width $\asymp 1$, or some rectangles in level-$(m-1)$ fjords, whose width is $\asymp_{\Bdelta} 1$, independent of $s$. This proves that $\W^-_{\hZ^{m-1}_1}(I,J) \asymp_{\Bdelta}1$ and finishes the induction. 
\end{proof}

\subsection{Nest of tiling on $\partial\hZ^{-1}$}\label{subsec.nest-tiling}

\begin{defn}\label{def.bdcombi}
Let $D\subset\widehat{\C}$ be a closed quasi-disk. We say that $\T = \{\T_{n}\}_{n\geq-1}$ is a nest of tiling if 
\begin{itemize}[label = {-}]
    \item every $\T_n$ is a finite partition of $\partial D$; $\T_{n+1}$ is a refinement of $\T_n$ and $\T_{-1} = \{\partial D\}$;
    \item the maximal diameter of intervals in $\T_n$ tends to zero as $n\to+\infty$;
    \item every interval in $\T_{n}$ is decomposed into at least two intervals in $\T_{n+2}$.
\end{itemize}
We say that $\T$ has {\it $M$-bounded combinatorics} if every interval in $\T_n$ contains at most $M$ intervals in $\T_{n+1}$.
\end{defn}

For an interval $I\in \T_n$, we set $I_\rho,I_\ell\in\T_n$ to be its two neighboring intervals; denote by $(3I)^c$ the interval $\partial D\setminus(\I_\ell\cup I\cup I_\rho)$. We define
\begin{itemize}[label = {-}]
    \item $\F^-_{3,\T}(I)$ to be the family of curves in $D$ connecting $I$ to $(3I)^c$;
    \item $\F^+_{3,\T}(I)$ to be the family of curves in $\widehat{\C}\setminus{D}$ connecting $I$ to $(3I)^c$;
    \item $\W^\pm_{3,\T}(I) := \W(\F^\pm_{3,\T}(I))$.
\end{itemize}
We say that $\T$ has {\it $C$-bounded inner geometry} if $\W^-_{3,\T}(I)\leq C$ for all $I\in\T$; similarly $\T$ has {\it $C$-bounded outer geometry} if $\W^+_{3,\T}(I)\leq C$ for all $I\in \T$. If moreover $\T$ has $M$-bounded combinatorics, then we say that $\T$ has {\it $(C,M)$-bounded inner/outer geometry}.

\begin{lem}[{\cite[Lemma 11.3]{DL:wZ}}]\label{lem.bdcombi-qcdisk}
    For every pair $(C,M)$, there is a $K_{C,M} > 1$ such that the following holds. Let $D$ be a
closed qc disk and $\mathcal{T}$ be a nest of tilings of $\partial D$. If $\mathcal{T}$ has $(C,M)$-bounded inner and outer geometries,
then $D$ is a $K_{C,M}$-qc disk.
\end{lem}

\subsection{A posteriori combinatorial threshold}\label{ss:a_posteriori:threshold} We are now ready to establish uniform threshold for regularization.

\begin{defn}\label{def.regularized.type}
    For $n\geq -1$, we say that $n$ is a regularized level, if $\hZ^{n+1}\subsetneq \hZ^n$; otherwise it is called a normal level. 
    We say a pseudo-Siegel disk is constructed with a combinatorial threshold $\bfN$ if $n$ is regularized if and only if $a_{n+1}\geq \bfN$. Here $a_{n+1}$ is the $(n+1)$-the entry in the continued fraction of $\theta$.
\end{defn}

\begin{thm}[Uniform regularization threshold]\label{thm.reg.threshold}
    There exists an absolute combinatorial regularization threshold $\bfN\gg_{\bfK}\  1$ such that for any $f\in\mathrm{Zak}_{\mathfrak{g}}(\theta)$, its sequence of pseudo-Siegel disks as in \S \ref{subsec.constructionZm} can be constructed with a combinatorial threshold $\bfN$ as in Definition~\ref{def.regularized.type}. 

    Consequently, there exists an absolute $\bfC$ such that for any $f\in\mathrm{Zak}_{\mathfrak{g}}(\theta)$, its pseudo-Siegel disk $\hZ^{n}$ ($n\geq -1$) has a nest of tiling $\T = \T({\hZ^{n}})$ with $\bfC$-bounded inner and outer geometries such that for each $T \in \T$,
    \begin{itemize}[label = {-}]
        \item  either $T = I^{\bullet}$, where $I\in\D_m$ for some $m\geq n$ and $I^{\bullet}$ is the projection of $I$ to $\partial\hZ^{n}$;
        \item or $T$ belongs to the subtiling $\T(\beta)\subset\T$ of some dam $\beta\subset\partial\hZ^{n}$.
    \end{itemize}
    Moreover, the top pseudo-Siegel disk $\wZ^{-1}$ is uniformly qc.

\end{thm}
\begin{proof}
The proof follows the quadratic argument in \cite[Theorem 10.1, Theorem 10.2, Theorem 11.1]{DL:wZ}.  Combining Theorem~\ref{thm.apriori-bound} and Calibration Lemma~\ref{lem.calibration} and the Theorem~\ref{thm.fjord}, the Log-Rule in parabolic fjords, we obtain the analogue of \cite[Lemma 10.5]{DL:wZ} for each full diffeo-tile. In particular, this implies that $T_{\mathrm{par}}$ in Theorem~\ref{thm.fjord} can be selected uniformly. The uniform regularization threshold then follows exactly as in the quadratic case.
The fact $\wZ^{-1}$ is uniformly qc follows from Lemma~\ref{lem.bdcombi-qcdisk} and the fact that the nest of tilings for $\wZ^{-1}$ has bounded combinatorics.
\end{proof}

\subsection{Non-structural complex bounds around $\wZ^{m}$}
It then follows immediately from Lemma \ref{lem.bdcombi-qcdisk} that
\begin{cor}\label{cor.comparable-interval-Z}
    There exists an absolute constant $\bfK = \bfK(\bfN)\gg1$, such that for all $f\in\mathrm{Zak}_{\mathfrak{g}}(\theta)$, $\hZ^{-1}$ is a $\bfK$-quasi-disk of definite diameter ($\mathrm{diam}_{\mathbb{C}}\hZ^{-1}\asymp1$). Moreover, every two neighboring intervals in $\T_{n+1}(\hZ^{n})$ have uniformly comparable (both in $n$ and $f$) Euclidean diameters.
\end{cor}

The following lemma can be justified in the same way as \cite{DL:sect.bounds} using shift argument.
\begin{lem}\label{lem.comparable-harmonic-meas}
    Let $\T(\hZ^{0}) = \{\T_n(\hZ^{0})\}$ be the nest of tiling given by Theorem~\ref{thm.reg.threshold}. Then all the intervals in $\T_0(\hZ^{0})$ have comparable inner harmonic measure on $\partial\hZ^{0}$ with respect to $z=0$.
\end{lem}

\begin{cor}\label{cor.definite.disk}
    There exists a universal constant ${\bfrho}>0$, such that for all $f\in\mathrm{Zak}_{\mathfrak{g}}(\theta)$, $\hZ^{-1}$ contains $\mathbb{D}(0,\bfrho)$.
    \end{cor}
\begin{proof}
       Suppose the contrary that $z=0$ is very close to some $I\in\T_{0}$. Since $z=0$ is not contained in any fjord of $\hZ^{-1}$, we may assume that $I$ is the projection of some interval in $\D_0$ by Theorem~\ref{thm.reg.threshold}. Hence $I$ is also an interval in the nest of tiling of $\hZ^0$: $I\in\T_0(\hZ^0)$. Moreover, by Corollary \ref{cor.comparable-interval-Z}, $\mathrm{diam}_{\mathbb{C}}I\asymp1$ and hence $I$ stays away from another $J\in\T_0(\hZ^0)$ with $\mathrm{diam}_{\mathbb{C}}J\asymp1$. This contradicts Lemma \ref{lem.comparable-harmonic-meas}, which states that $I,J$ should have comparable inner harmonic measure in $\hZ^0$ with respect to $z=0$.
\end{proof}

For a set $T\subset\mathbb{C}$, we denote by $U_{\epsilon}(T)$ the $\epsilon|T|_{\mathbb{C}}$-neighborhood of $T$, where $|\cdot|_\mathbb{C}$ is the Euclidean metric. The following result can be proved in the same way as \cite{DL:sect.bounds} using Koebe distortion theorem.

\begin{prop}[Complex bounds around $\hZ^{-1}$]\label{prop.complex.bd-for-fqn}
   For each fixed $\mathfrak{g}\geq2$, there exist absolute constants $1\gg \epsilon_3>\epsilon_2>\epsilon_1>0$ and $c(\mathfrak{g})>0$ such that the following holds for all $f\in\mathrm{Zak}_{\mathfrak{g}}(\theta)$.

   Let $I\subset\partial\hZ^{-1}$ be any well-grounded interval such that $|I| \leq \l_{n}$. For $1\leq j\leq \q_{n+1}$, let $O_j(I)$ be the pullback of $U_{\epsilon_2}(f^j(I))$ along $f^j: I\longrightarrow f^j(I)$. Then
   $$U_{\epsilon_1}(I)\subset O_j(I)\subset U_{\epsilon_3}(I),\,\quad \mathrm{deg}(f^j|_{O_j(I)})\leq c(\mathfrak{g}).$$
\end{prop}

\section{Sectorial renormalization towers}\label{s:SectorRenorm}

Let us give more details on Theorem \ref{mainthm:B} stated in the introduction.

\begin{thm}[Renorm-Attractor of $\Zak$]\label{thm:sector.bounds} For $f\in\Zak(\Theta_\egm)$, the sector renormalization~\eqref{eq:Rsec.egm} extends to a certain  neighborhood of $\wZ^{[n-1]}_{f}\supset \overline Z_f$ such that this neighborhood projects to a definite neighborhood of $\wZ^{-1}_{f_n}$. This constructs a sectorial renorm-tower 
\begin{equation}
\label{eq:1:thm.sector.bounds}
\Tow (f) =(f_n\colon U_{f_n}\to V_{f_n})_{n\ge 0}\ \equiv \ (f_n\colon \mathbb D\dashrightarrow  \mathbb D)_{n\ge 0}, 
\end{equation}
where
\[  U_{f_n}\cup V_{f_n} \subset W_{f_n} \ \simeq \ \mathbb D, \qquad 0=f_n(0)\text{ is a Siegel fixed point}\]
satisfies the following properties
\begin{enumerate}[label=\upshape(\Alph*)]
    \item\label{item:1:thm:sector.bounds} Under $f\leadsto f_n$, the pseudo-Siegel disk $\wZ^{[m]}_f$ becomes $\wZ^{[m-n]}_{f_n}$, where $m-n\ge -1$. 
    \item \label{item:2:thm:sector.bounds} The disks $\wZ^{-1}_{f_n}$ are uniformly qc independently of $f$ and $n$ with definite space around
    \[ \mod \big((U_{f_n}\cap V_{f_n})\setminus \wZ^{-1}_{f_n}\big)\ge \boldsymbol{\varepsilon} >0.\]
   \item \label{item:3:thm:sector.bounds} By identifying $W_{f_n} \simeq \mathbb D$ so that $0\in W_{f_n}$ is the center of rotation, the disk $\wZ^{-1}\subset \mathbb D$ contains a definite subdisk of $\mathbb D$:
    \begin{equation}
       \label{eq:bdelta.thm.sector.bounds}
       \wZ^{-1}_{f_n} \ \supset \ \overline {\mathbb D}(\boldsymbol \delta),\qquad \boldsymbol \delta>0 \text{ is universal}.
    \end{equation}
        
     \item \label{item:4:thm:sector.bounds} There is a uniform $k\ge1 $, so that the Uniform improvement of the domain holds for $f_n\leadsto f_{n+k}$. There is a uniform $\boldsymbol{\varepsilon} >0$ such that
     \[ \mod ( U_{f_n}\cap V_{f_n} \setminus \mathbb \Delta^k_{f_{n+k}})\ge \boldsymbol{\varepsilon}.\]
     Here $\mathbb{\Delta}^k_{f_{n+k}}$ is the tiling containing $\hZ^{-1}_{f_n}$ given by spreading around $W_{f_{n+k}}$ under anti-renormalization. 
\end{enumerate}

The set of towers 
\begin{equation}
\label{eq:Tow.egm}
\bTow(\Theta_\egm)  \coloneqq \{\Tow (f)\ : \ f\in \Zak(\Theta_\egm)\} 
\end{equation}
is precompact with respect to sequential limits and shifts $(f_n)_{n\ge 0} \mapsto (f_n)_{n\ge t}, t\ge 0$. The compactification
\begin{equation}
\label{eq:ovlTow}
  \big\{(f_n)_{n\ge t}\ :\ (f_n)_{n\ge 0}\in \bTow(\Theta_\egm)  \big\}\ \leadsto \ \ovlTow = \{(f_n\colon U_{f_n}\to V_{f_n})_{n\ge 0}\},
\end{equation}
produces a compact set of renormalization towers forward-invariant under the shift and satisfying Items~\ref{item:2:thm:sector.bounds},~\ref{item:3:thm:sector.bounds},~\ref{item:4:thm:sector.bounds}.

This constructs renormalization towers for all $f\in \Zak(\Theta)$
\begin{equation}
    \bTow(\Theta)=\{\Tow(f)\ : \ f\in \Zak(\Theta)\}\ \subset \ \ovlTow.
\end{equation}

All parabolic limits in $\ovlTow$ have multiplicity at most $\mathfrak{g}$.     
\end{thm}

\subsection{Combinatorics of sectorial renormalization}
Recall that for an irrational number $\theta\in\mathbb{S}^1\simeq[0,1]/\mathbb{Z}$, it has a unique infinite (classical) continued fraction expansion $\theta = [a_0,a_1,a_2,...]$, where $a_i\geq1$, $a_i\in\mathbb{N}$. The $n$-th closest return time is given inductively by
$$\q_{n+1} = a_n\q_n+\q_{n-1},\quad\q_{-1}=0,\,\,\q_{0}=1.$$
Let $\l_n$ be the length of level $n$-combinatorial interval, which is given by $$\l_n = \dist_{\mathbb{S}^1}(\q_n\theta,0),\,\,\,\forall n\geq 1\quad\text{ and }\quad \l_0 = \theta,\,\,\l_{-1} = 1.$$
Then clearly one has $\l_{n}>\l_{n+1}$ and $\l_{n+2}<\l_{n}/2$ for all $n\geq -1$. Now we remove inductively the terms in the sequence $\{\l_n\}_{n\geq -1}$ satisfying $\l_n>\l_{n-1}/2$. That is, we define inductively the following sequence
\begin{equation}
    {[n+1]} := \left\{\begin{aligned}
        &{[n]+1},\text{ if }\l_{[n]+1}< \frac{1}{2}\l_{[n]}\\
        &{[n]+2}, \text{ if }\l_{[n]+1}>\frac{1}2{\l_{[n]}}
    \end{aligned}\quad \text{ with }[-1] = -1\right.
\end{equation}
Therefore we have $\l_{[n+1]}<\l_{[n]}/2$. For $n\geq-1$, set $\epsilon_{n+1} = 1$ if $\l_{[n]+1}<\frac{1}{2}\l_{[n]}$; $\epsilon_{n+1} = -1$ if $\l_{[n]+1}>\frac{1}{2}\l_{[n]}$. Let $\overline{a}_n\geq 2$ be the unique integer such that 
$$\l_{[n-1]} = \overline{b}_n\l_{[n]}-\epsilon_n\epsilon_{n+1}\l_{[n+1]}, \quad\text{where}\quad b_n := \overline{a}_n+\frac{1+\epsilon_n\epsilon_{n+1}}{2}.$$ 
This identifies $\theta$ to the following {\it continued fraction with sector renormalization combinatorics}: 
\begin{equation}\label{eq.modified.frac}
    \overline{\theta} = \langle(\epsilon_0,\overline{a}_0),(\epsilon_1,\overline{a}_1),...\rangle := \cfrac{1}{\epsilon_0b_0-\cfrac{1}{\epsilon_1b_1-\cfrac{1}{\epsilon_2b_2-\cdots}}}\in(-\frac{1}{2},\frac{1}{2}).
\end{equation}

We say that $\overline{\theta}$ is {\it EGM (eventually golden mean)} if its entries stabilize at $(-,2)$:
\begin{equation}\label{eq.EGM.theta}
    \overline{\theta} = \langle(\epsilon_0,\overline{a}_0),...,(\epsilon_n,\overline{a}_n),(-,2),(-,2),...\rangle.
\end{equation}

One can verify directly that $\overline{\theta} = \theta\mod 1$. Let us set $\overline{\theta}_n = \langle(\epsilon_n,\overline{a}_n),(\epsilon_{n+1},\overline{a}_{n+1}),...\rangle$, $\overline{\theta} = \overline{\theta}_0$. It turns out that the sequence $\{\overline{\theta}_n\}$ encodes the combinatorics of the sector renormalization operator $\mathcal{R}_{\mathrm{rot}}$ acting on rigid rotation $\rho_{\overline{\theta}}$: that is, $\mathcal{R}_{\mathrm{rot}}(\rho_{\overline{\theta}_{n}}) = \rho_{\overline{\theta}_{n+1}}$. For a systematic study of sectorial renormalization combinatorics, see \cite{Lim26}.

For simplicity of notations, we set $\hat{\theta}_n := |\overline{\theta}_n|$. In \cite{Yoc}, Yoccoz introduced the following {\it Brjuno function}:
\begin{equation}\label{eq.Brjuno.function}
    \bfB(\theta) = \sum_{i=0}^\infty\beta_{i-1}\log\hat{\theta}_{i}^{-1},\quad\beta_i: =\prod_{j=0}^{i}\hat{\theta}_j,\,\,\beta_{-1} = 1.
\end{equation}

An irrational number $\theta$ is called {\it Brjuno} if $\bfB(\theta)<+\infty$.

\subsection{Construction of $\bTow(\Theta_\egm)$}
Let $f\in\mathrm{Zak}_{\mathfrak{g}}(\theta)$ be a Zakeri type polynomial with Siegel disk $Z$. Let $c_0,c_1,..c_d\in\partial Z$ be the critical points of $f$; let $v_0,v_1,...,v_d$ be the corresponding critical values. Recall that $\D_n$ is the level-$n$ diffeo-tiling generated by $\mathrm{CP}_n = \bigcup_{i=0}^{\q_{n+1}-1}f^{-i}(c_0)\cap\partial Z$.

Now we construct the $n$-th renormalization sector $S_n$ for $n\geq 0$. We select inductively a sequence of intervals $J_{[n]}\subset\partial\hZ^{-1}$ well-grounded rel $\partial\hZ^{-1}$ such that
\begin{itemize}[label = {-}]
    \item the two end points of $J_{[n]}$ are $x_n,f^{\q_{[n]}}(x_n)$ with $x_n$ not being a critical point of $f^{\q_{[n+1]}}$;
    \item $J_{[n+1]}\Subset J_{[n]}$ and $\min\{\dist(x_n,f^{\q_{[n+1]}}(x_{n+1})),\dist(f^{\q_{[n]}}(x_n),x_{n+1})\}\geq \frac{1}{2}\l_{[n+1]}$.
\end{itemize}
Let us denote $y_n := f^{\q_{[n]}}(x_n)$ and $I_{x_n},I_{y_n}\subset {Z}$ the internal ray toward $x_n,y_n$ respectively. Let $y_n' := f^{\q_{[n+1]}}(x_n)$. Fix an absolute small $\delta>0$. Let $\Sigma_{y_n'}$ be an arc such that
\begin{itemize}[label = {-}]
    \item $\Sigma_{y_n'}$ is a segment of the hyperbolic geodesic in $\widehat{\C}\setminus{\hZ^{-1}}$ joining $y_n'$ and $\infty$ that starts at $y_n'$;
    \item $\Sigma_{y_n'}$ has Euclidean diameter $\delta_n := \delta\cdot\dist_{\mathbb{C}}(x_n,y_n)$.
\end{itemize}
By Proposition \ref{prop.complex.bd-for-fqn}, we can pull back $B(y_n',\delta_n)$ by $f^{\q_{[n+1]}-\q_{[n]}}$ to a neighborhood $B(y_n)$ of $y_n$ and by $f^{\q_{[n+1]}}$ to a neighborhood $B(x_n)$ of $x_n$. Let 
\begin{itemize}[label = {-}]
    \item $\Sigma_{y_n}\subset B(y_n)$ be the connected component of $f^{-\q_{[n+1]}+\q_{[n]}}(\Sigma_{y_n'})$ containing $y_n$;
    \item $\Sigma_{x_n}\subset B(x_n)$ be the connected component of $f^{-\q_{[n+1]}}(\Sigma_{y_n'})$ containing $x_n$.
\end{itemize}

\begin{lem}\label{lem.well-defined-sector}
    For each $n\geq0$, $\Sigma_{x_n}\cap({\hZ^{-1}\cup f^{\q_{[n]}}(\hZ^{[n-1]})}) = \{x_n\}$ and $\Sigma_{y_n}\cap({\hZ^{-1}\cup f^{\q_{[n]}}(\hZ^{[n-1]})}) = \{y_n\}$. Moreover, $\dist_{\mathbb{C}}(x,\partial\hZ^{-1})\succeq \dist_{\mathbb{C}}(x,x_n)$ for all $x\in\Sigma_{x_n}$; $\dist_{\mathbb{C}}(y,\partial\hZ^{-1})\succeq \dist_{\mathbb{C}}(y,y_n)$ for all $y\in\Sigma_{y_n}$.
\end{lem}
\begin{proof}
We only prove the statements for $x_n$ and it will be similar for $y_n$. Suppose the contrary that there is another $x_n''\in\Sigma_{x_n}\cap{\hZ^{-1}}$. By construction of $\hZ^{-1}$, $x_n''$ must belong to some level $k$-fjord $\mathfrak{F}$ of $\hZ^{-1}$. 
    
    \begin{itemize}[label={-}]
        \item If $k\geq n$, then $f^{\q_{[n+1]}}(\mathfrak{F})\setminus\mathfrak{F}$ is contained in some level $k$-protection (see (\ref{eq.tilde-Ainn})). The protection is a topological disk surrounding the dam of $\mathfrak{F}$ and is disjoint from $\Sigma_{y_n'}$. On the other hand, we have $f^{\q_{[n+1]}}(\Sigma_{x_n}) = \Sigma_{y_n'}$ by construction and hence $f^{\q_{[n+1]}}(x_n'')\in \Sigma_{y_n'}$. This leads to a contradiction.
        \item If $k<n$, then since $J_{[n]}$ is well-grounded, we have $\dist_{\mathbb{C}}(J_{[n]},\mathfrak{F})\gg \delta\cdot |J_{[n]}|_{\mathbb{C}}$ by Corollary \ref{cor.comparable-interval-Z}. By Proposition \ref{prop.complex.bd-for-fqn}, $|\Sigma_{x_n}|_{\mathbb{C}}\asymp\delta\cdot |J_{[n]}|_{\mathbb{C}}\ll\dist_{\mathbb{C}}(J_{[n]},\mathfrak{F})$. Hence $\Sigma_{x_n}$ is disjoint from $\mathfrak{F}$.
    \end{itemize}
    Similarly one can prove that $\Sigma_{x_n}\cap\overline{f^{\q_{[n]}}(\hZ^{[n-1]})} = \{x_n\}$. The estimate $\dist_{\mathbb{C}}(x,\partial\hZ^{-1})\succeq \dist_{\mathbb{C}}(x,x_n)$ follows immediately from Proposition \ref{prop.complex.bd-for-fqn} and the fact that $\dist_{\mathbb{C}}(y',\partial\hZ^{-1})\succeq \dist_{\mathbb{C}}(y',y_n')$ for all $y'\in\Sigma_{y_n'}$ (since $\Sigma_{y'_n}$ is a geodesic in a uniform quasi-disk $\widehat{\C}\setminus\hZ^{-1}$).
\end{proof}

For $x,y\in\partial\hZ^{-1}$, we denote by $[x,y]_{\hZ^{-1}}$ the subinterval that has combinatorial length less than $\frac{1}{2}$ of $\partial\hZ^{-1}$ connecting $x,y$. By Corollary \ref{cor.comparable-interval-Z} $[x_n,y_n]_{\hZ^{-1}},\Sigma_{y_n}$ are $\bfK$-quasi-arc. By Proposition \ref{prop.complex.bd-for-fqn}, $\Sigma_{x_n}$ is also a $\bfK$-quasi-arc (enlargingn $\bfK$ if necessary). By Lemma \ref{lem.well-defined-sector}, $[x_n,y_n]_{\hZ^{-1}}\cup\Sigma_{x_n}\cup\Sigma_{y_n}$ is a $\bfK(\delta)$-quasi-arc. Let $x_n'$ (resp. $y_n'$) be the other end point of $\Sigma_{x_n}$ (resp. $\Sigma_{y_n}$). Then we can take a quasi-arc $J\subset\widehat{\C}\setminus\hZ^{-1}$ connecting $x_n'$ to $y_n'$ such that $J\cup[x_n,y_n]_{\hZ^{-1}}\cup\Sigma_{x_n}\cup\Sigma_{y_n}$ is a $\bfK(\delta)$-quasi-circle (enlarge $\bfK(\delta)$ if necessary). Let $\R_n$ be the corresponding rectangle with roof $\Sigma_{x_n}$ and base $\Sigma_{y_n}$. Then we have immediately the following:
\begin{lem}\label{lem.widthbound-on-Rn}
    The rectangle $\R_n$ is $\bfK(\delta)$-qc. The conformal modulus and width of $\R_n$ are bounded by some $K$ only depending on $\bfK$ and $\delta$.  
\end{lem}

Now we construct the renormalization sector. See Figure \ref{fig.renor_sector}. Let $S_n = S_n(\delta)$ be the open sector bounded by $\Sigma_{x_n}\cup\Sigma_{y_n}\cup I_{x_n}\cup I_{y_n}\cup J$. By the discussion above, $S_n$ is well-defined. Let $\varphi_n:\overline{S_n}\longrightarrow\overline{\mathbb{D}}$ be the dynamical conformal gluing map, that is $\varphi_n(z) = \varphi_n(f^{\q_{[n]}}(z))$ for all $z\in I_{x_n}\cup\Sigma_{x_n}$. Notice that $\varphi_n$ is unique up to post-composition of a M\"obius transformation. The $n$-th renormalized map $f_n:\mathbb{D}\dashrightarrow\mathbb{D}$ is defined by projecting the first return map $F_n: S_n\dashrightarrow S_n$ via $\varphi_n$. For all $k\geq -1$ and $n\geq 0$, let $\hZ^{[k]}_{f_n} := \varphi_n(\hZ^{[k+n]}\cap\overline{S_n})$ be the level-$k$ pseudo-Siegel disk of $f_n$. In particular, $\varphi_n$ projects $\hZ^{[n-1]}$ to $\hZ^{-1}_{f_n}$. A fjord of $\hZ^{[k]}_{f_n}$ is the projection of a fjord of $\hZ^{[k+n+1]}$ via $\varphi_n$. This proves Theorem \ref{thm:sector.bounds} \ref{item:1:thm:sector.bounds}.

\begin{figure}[t]
\centering 
\includegraphics[width=0.8\textwidth]{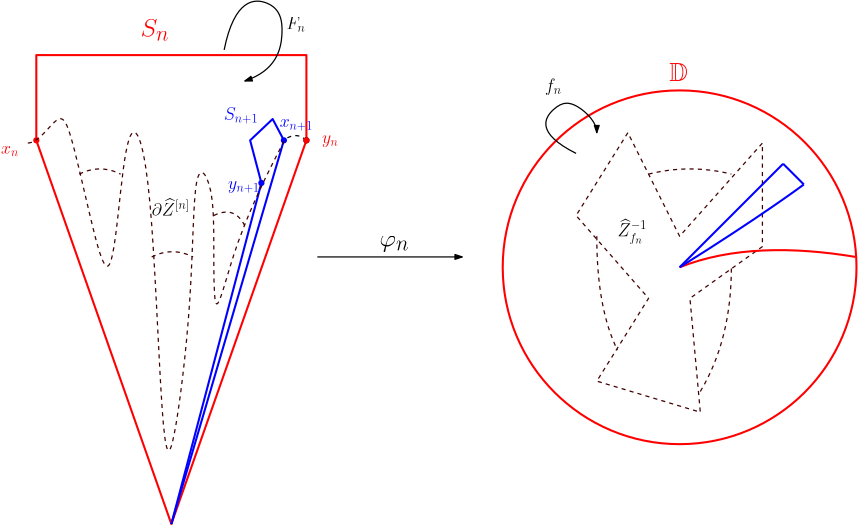} 
\caption{Sector renormalization projecting $F_n$ to $f_n$.} 
\label{fig.renor_sector} 
\end{figure}

\begin{lem}\label{lem.return.map.domain}
    There exists an absolute $\delta'<\delta$ with $\delta'\asymp\delta$, such that $F_n$ is at least defined on $S_n' := S_n(\delta')$.
\end{lem}
\begin{proof}
    Follows from Proposition \ref{prop.complex.bd-for-fqn}.
\end{proof}

The following lemma justifies Theorem \ref{thm:sector.bounds} \ref{item:2:thm:sector.bounds}. It can be derived immediately from Lemma \ref{lem.widthbound-on-Rn}.
\begin{lem}[Definite space around $\hZ^{-1}_{f_n}$]\label{lem.definite.space}

We have $\mathrm{mod}(\varphi_n(S'_n)\setminus{\hZ^{-1}_{f_n}})\asymp1$. Moreover, for $k\geq 0$, let $\R_{f_n,k} := \varphi_n(\R_{n+k+1})$. Then $\mathrm{mod}(\R_{f_n,k})\asymp 1$ and the roofs/bases of $\R_{f_n,k}$ have comparable Euclidean diameters.
\end{lem}

Following the lines in \S \ref{subsec.nest-tiling} and \S \ref{ss:a_posteriori:threshold}, we can also prove similarly the following:

\begin{lem}\label{lem.fjord.quasidisk}
    Slightly enlarging the constant $\mathbf{K}$ in Corollary \ref{cor.comparable-interval-Z}, the pseudo-Siegel disk $\widehat{Z}^{-1}_{f_n}$ and all its fjords are $\mathbf{K}$-quasi-disks.
\end{lem}

\subsection{Compactness of $\bTow(\Theta_\egm)$} Let us recall basic properties of qc compactness.

\begin{lem}\label{lem.comparable.dist}
    Let $D$ be a $K$-qc disk and $z\in D$ such that there are $a,b,c\in\partial D$ dividing $\partial D$ such that the three segments have comparable harmonic measure seen from $z$. Then 
    $$\dist_\C(z,\partial D)\asymp_K\min\{|a-b|,\,|b-c|,\,|a-c|\}.$$

    Moreover, if $|a-b| = \min\{|a-b|,\,|b-c|,\,|a-c|\}$, then $|z-a|\asymp_K|z-b|\asymp_K|a-b|$.
\end{lem}
\begin{proof}
    Let us rescale $D$ so that $z = 0$ and $\dist_\C(z,\partial D) = 1$. Consider the Riemann map $\phi:\mathbb{D}\longrightarrow D$ such that $\phi(0) = 0$. By hypothesis, the three segments in $\partial\mathbb{D}$ partitioned by $\phi^{-1}(a),\phi^{-1}(b),\phi^{-1}(c)$ have comparable Euclidean size. On the other hand, $\phi$ extends to a $K$-qc map on the Riemann sphere. Therefore the spherical distance between $z$ and $\partial D$ is $K$-comparable to the spherical distances between $a,b,c$. Hence at least two of $a,b,c$ stays away from infinity. While away from infinity, Euclidean distance is comparable to spherical distance. This finishes the proof.
\end{proof}

\begin{prop}\label{prop.sectorial.bounds}
For every fixed $\mathfrak{g}\geq 2$, there exists a universal constant $\bfrho>0$ such that for all $f\in\mathrm{Zak}_{\mathfrak{g}}(\theta)$, the following hold:
\begin{enumerate}[label=\upshape(\roman*)]
    \item \label{thm.sectorial.bounds.first}$f_n:\mathbb{D}\dashrightarrow\mathbb{D}$ is at least defined and injective on $\widehat{Z}^{-1}_{f_n}\supset\mathbb{D}(0,\bfrho)$.
    \item \label{thm.sectorial.bounds.second}If $f_n$ is regularized type, then $f_n$ has at most $d+1$ and at least 2 distinct fixed points in $\mathbb{D}(0,\bfrho)$, where $d\leq\mathfrak{g}-1$ is the number of critical points of $f$; moreover, all these fixed points except for $z=0$ are repelling and are contained in $\widehat{Z}^{-1}_{f_n}\setminus\overline{\widehat{Z}^0_{f_n}}$.
    \item \label{thm.sectorial.bounds.third} If $f_n$ is regularized type, then let $\sigma_{\star,n}$ be the non-zero fixed point of $f_n$ that is closest to $z=0$. We have $\log|\sigma_{\star,n}|\asymp\log|\overline{\theta}_n|$. Moreover $\hZ^{0}_{f_n}$ contains $\mathbb{D}(0,r)$ with $\log r\asymp \log |\overline{\theta}_n|$.
\end{enumerate}
\end{prop}

As a corollary of Item~\ref{thm.sectorial.bounds.second}, we deduce the last claim of Theorem~\ref{thm:sector.bounds}:
\begin{cor}\label{cor:parab_limits}
    All parabolic limits in $\ovlTow$ have multiplicity at most $\mathfrak{g}$.\qed
\end{cor}

\begin{proof}[Proof of Proposition~\ref{prop.sectorial.bounds}]
    \ref{thm.sectorial.bounds.first} is established in the same way as Corollary \ref{cor.definite.disk}. \ref{thm.sectorial.bounds.second} follows from the the facts that $S_n$ contains at most $d$ fjords and each fjord contains exactly one  repelling $\q_{[n]}$-periodic point (see \cite{DL:sect.bounds}).

    By \ref{thm.sectorial.bounds.second}, $f_n(z) - z = z(z-\sigma_1)...(z-\sigma_s)h_n(z)$, where $h_n$ does not vanish on $\mathbb{D}(0,\bfrho)$ and $s \leq d$. On the other hand, $f_n(z) - z = (e^{2\pi i\overline{\theta}_n}-1)z + O(z^2)$. Here $O(z^2)$ is uniform by \ref{thm.sectorial.bounds.first}. Therefore $|\sigma_1...\sigma_s|\asymp \overline{\theta}_n$. Therefore $|\sigma_{*,n}|\preceq |\overline{\theta}_n|^{1/s}$. On the other hand, since all $|\sigma_i|\preceq 1$, we have $|\sigma_{*,n}|\succeq|\overline{\theta}_n|$.  This proves $\log|\sigma_{*,n}|\asymp \log|\overline{\theta}_n|$. The last assertion of \ref{thm.sectorial.bounds.third} will follow directly from the following claim:
    \begin{claim*}
        $\dist_\C(0,\partial \hZ^0_{f_n})\asymp \dist_\C(\sigma_{*,n},\partial \hZ^0_{f_n})\asymp \dist_{\C}(0,\sigma_{*,n})$.
    \end{claim*}
    \begin{proof}[Proof of the claim]
        Let $\mathfrak{F}$ be the fjord containing $\sigma_{*,n}$. The boundary $\partial\mathfrak{F}$ has a tiling $\cal{T}(\mathfrak{F})$ inherited from $\hZ^0_{f_n}$. Similarly as Lemma \ref{lem.comparable-harmonic-meas}, one can show that the harmonic measure of each segments in the level zero tiling $\cal{T}_0(\mathfrak{F})$ with respect to $\sigma_{*,n}$ is uniformly comparable. Let us choose three points $a,b,c\in\partial\mathfrak{F}\cap\partial\hZ^0_{f_n}$ such that they partition $\partial\mathfrak{F}$ into segments with comparable harmonic measure seen from $\sigma_{*,n}$ in $\mathfrak{F}$ and also partition in the same time $\partial( \hZ^{-1}_{f_n}\setminus\mathfrak{F})$ into segments with comparable harmonic measure seen from $0$ in $\partial( \hZ^{-1}_{f_n}\setminus\mathfrak{F})$. The first assertion of Lemma \ref{lem.comparable.dist} implies that $\dist_\C(0,\partial \hZ^0_{f_n})\asymp \dist_\C(\sigma_{*,n},\partial \hZ^0_{f_n})$. The second assertion of Lemma \ref{lem.comparable.dist} implies that $|a|\asymp|\sigma_{*,n}-a|\asymp \dist_\C(\sigma_{*,n},\partial \hZ^0_{f_n})$ (after rearranging $a,b,c$). This proves the claim.
    \end{proof}
\end{proof}

\subsection{Consecutive renormalization for $\bTow(\Theta_\egm)$} In fact, one can obtain the same renormalization tower $\{f_n\}$ (after post-composing with a rigid rotation) by consecutive sector renormalizations. Let $\tilde{f}_0 := f_0$. Suppose by induction that $\tilde{f}_n:\mathbb{D}\dashrightarrow\mathbb{D}$ is defined and $\tilde{f}_n = f_n$. Then set $S_{f_n} = \varphi_n(S_{n+1})$. Let $\varphi_{f_n}: \overline{S_{f_n}}\longrightarrow\overline{\mathbb{D}}$ be the conformal gluing map by identifying $z$ and $f_n(z)$. Let $\tilde{f}_{n+1}:\mathbb{D}\dashrightarrow\mathbb{D}$ be the projection of the first return map $F_{f_n}:S_{f_n}\dashrightarrow S_{f_n}$ via $\varphi_{f_{n}}$. Now we verify that $\tilde{f}_{n+1} = f_{n+1}$. By construction, $\tilde{f}_{n+1} = \varphi_{f_n}\circ F_{f_n}\circ\varphi_{f_n}^{-1}$. On the other hand, since $\overline{S_n}$ and $\overline{S_{n+1}}$ only intersect at $z=0$ by construction, we have $F_{f_n} = \varphi_{n}\circ F_{n+1}\circ \varphi_n^{-1}$. We also notice that $\varphi_{f_n}\circ \varphi_n = \varphi_{n+1}$ because the identifications on the two sides of $S_{n+1}$ are the same and hence they are the same gluing map (up to post-composing with a rigid rotation of $\mathbb{D}$). We conclude that $\tilde{f}_n = f_n$ for all $n\geq 0$.\\

\section{Antirenormalization coordinates and Brjuno condition}\label{s:anti.renorm.coord} We now detail the statements outlined in~\S\ref{ss:outline.SmallPerPts}. Estimate~\eqref{eq:intro.beta_n}:
\begin{equation} \label{eq:sec8.beta_n} - \ln |\FPsec_n| \ \asymp \ - \ln \wtheta_n +O(1), 
\end{equation} 
has already been established in Proposition~\ref{prop.sectorial.bounds}. Lemma~\ref{MainLemma} relies on understanding the antirenormalization coordinate near $0$. Let us start by providing the heuristic for Lemma~\ref{MainLemma}.

For sectorial renormalization of rigid rotations, see Figure~\ref{Fig:sect.enorm}, 
\[\rot_n  \colon z\mapsto e^{2\pi i \theta_n} z \colon \overline {\mathbb D}\selfmap,\qquad\qquad \rot_{n+1}=\R_\sec \rot_{n} ,\]
the antrirenormalization coordinate from $\rot_{n+1}$ to $\rot_n$ is given by the power law 
\begin{equation}
\label{eq:anti.ren.rot}
   \Xi_{n,\mathrm{anti}} \equiv  \Xi_{\mathrm{anti}}\colon  z\mapsto z^{\wtheta_n}, \qquad\qquad \wtheta_n =\dist_{\RR/\ZZ}(\theta_n, 0)\in (0,1/2).
\end{equation}
It is convenient to lift $\Xi_{n,\mathrm{anti}}$ to the universal cover. Set 
\begin{equation}
\label{eq:ixp}
    \mathrm{ixp}(w)=e^{2\pi i w}=: z, \qquad\qquad \mathrm{ixp}\colon \mathbb H_{\ge 0}\to \overline {\mathbb D}. 
\end{equation}
Then the lift of $\Xi_{\mathrm{anti}}$ takes form
\begin{equation}
\label{eq:psi.anti.rotat}
    {\psi}_{\mathrm{anti}}\ \colon \ w\mapsto \wtheta_nw,\qquad\qquad  \Xi_{\mathrm{anti}}\circ \mathrm{ixp} =\mathrm{ixp}\circ{\psi}_{\mathrm{anti}} .
\end{equation}

Consider now a tower $\Tow(f)=(f_n\colon \mathbb D\dashrightarrow  \mathbb D)_{n\ge 0}$ as in~\eqref{eq:1:thm.sector.bounds}, where $f\equiv f_0\in \Zak(\Theta_\egm)$. The antirenormalization from $f_{n+1}\colon \mathbb D\dashrightarrow  \mathbb D$ to $f_n\colon \mathbb D\dashrightarrow  \mathbb D$ near $0$ can be decomposed as 
\[\overline {\mathbb D} \ \approx\ \wZ^{-1}_{f_{n+1}}\ \leadsto\ \wZ^{0}_{f_{n}} \ \hookrightarrow \ \overline {\mathbb D},\]
where $\overline {\mathbb D}\ \approx\ \wZ^{-1}_{f_{n+1}}$ represents the bound in~\eqref{eq:bdelta.thm.sector.bounds}. Therefore, substituting $\overline {\mathbb D} $ with $\wZ^{-1}_{f_{n+1}}$ leads to the $O(1)$-term in the universal cover coordinate $w$, see~\eqref{eq:ixp}. 

The antirenormalization $\wZ^{-1}_{f_{n+1}}\leadsto \wZ^{0}_{f_{n}}$ is governed by~\eqref{eq:psi.anti.rotat}, which describes how rigid rotations are antirenormalized, up to an error term:
\begin{equation}
\label{eq:sec8.m1}
    w\ \mapsto \  \frac{w}{\wtheta_n} +O(1)\qquad \text{ for }\quad \wZ^{-1}_{f_{n+1}}\leadsto \wZ^{0}_{f_{n}}.
\end{equation}

Finally, the inclusion $\wZ^{0}_{f_{n}} \ \hookrightarrow \ \overline {\mathbb D}$ is described by the conformal radius of $\wZ^{0}_{f_{n}}$ and is governed by~\eqref{eq:sec8.beta_n}. With respect to the inner conformal geometry of $\wZ^{0}_{f_{n}}$ and $\mathbb D$, this inclusion takes the following form in the $w$-plane:
\begin{equation}
\label{eq:sec8.m2}
w\mapsto w+ b (-\ln \wtheta_n)i + O(1), \qquad \text{where $b\asymp 1$, for $\wZ^{0}_{f_{n}} \ \hookrightarrow \ \overline {\mathbb D}$.}    
\end{equation}

Assembling~\eqref{eq:sec8.m1} and~\eqref{eq:sec8.m2}, we would expect that the antirenormalization coordinate near $0$ take the following form in the $w$-plane:
\begin{equation}
\label{eq:sec8.m3}
w\mapsto \wtheta_n w+ b (-\ln \wtheta_n)i + O(1), \qquad \text{where $b\asymp 1$, for $f_{n+1} \ \leadsto \ f_n$.}    
\end{equation}

This is the main estimate below (which will be proved in \S \ref{subsec.Yoccoz.estimate}). See Figure \ref{fig.anti_sector}.

\begin{lem}\label{lem.keyestimate}
There is an absolute $\mathbf L>1$ such that for every antirenormalization  $f_{n+1} \ \leadsto \ f_n$, and any 
\begin{equation}\label{eq.condition.Imw}
    w\in \mathbb H_{>0}\qquad\text{ with } 0\le \Re w\le 1,\qquad \Im w\ge \mathbf L/\wtheta_n
\end{equation}

we have:

\begin{equation}
\label{eq:sec8.m4} 
\Psi_{\mathrm{anti}}: w\mapsto \wtheta_n w+ b (-\ln \wtheta_n)i + O(1), \qquad \text{where $b\asymp 1$, $b\in \C$.}    
\end{equation}
    
\end{lem}

\subsubsection{Notations in Section~\ref{s:anti.renorm.coord}} In this section, we will be mostly following the notations of~\cite{Yoc}. Here are the key notations.
\begin{itemize}[label = {-}]
    \item $\theta\in (0,1)$ is the rotation number.
    \item $\bar \theta \in (-1/2, 1/2)$ is the rotation number rel a different fundamental interval. We can write $\bar \theta$ using the combinatorics of sector renormalization. Recall ~\eqref{eq.modified.frac} and Figure~\ref{Fig:sect.enorm}.
    \item We also write $\wtheta= |\bar \theta|\in (0,1/2)$. 
    \item For the level-$(n+1)$ rotation number $\overline{\theta}_{n+1}$, we set $\iota := \overline{\theta}_{n+1}$ and $\alpha := \overline{\theta}_{n+1} - \epsilon_{n}\overline{b}_{n} = -1/\overline{\theta}_n$.
\end{itemize}

\subsection{Anti-renormalization sector}\label{subsec.anti-renor} 
In this subsection, we consider $f_{n+m}$ from Theorem~\ref{thm:sector.bounds}, select a "dividing arc" for $f_{n+m}$, lift it to construct the renorm sector in the dynamical plane of $f_n$ for $f_{n}\leadsto f_{n+m}$.  
This will eventually allows us to use the following local Yoccoz estimates.

\begin{figure}[t]
\centering 
\includegraphics[width=0.8\textwidth]{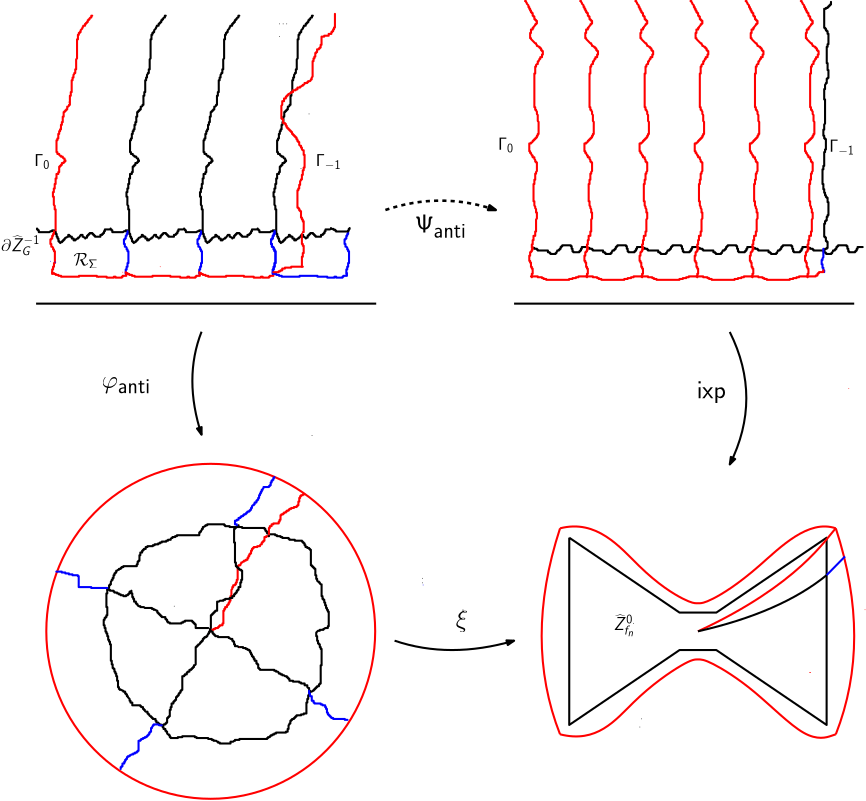} 
\caption{Anti-renormalization from $f_{n+1}$ to $f_n$.} 
\label{fig.anti_sector} 
\end{figure}

\begin{lem}[{\cite[Lemma 3.2]{Yoc}}]\label{lem.koebe-on-halfplane}
    Let $F:\mathbb{H}\longrightarrow\mathbb{C}$ be a univalent holomorphic map such that $F(z) = z+\alpha +\varphi(z)$, $\varphi$ is $\mathbb{Z}$-periodic and $\lim\limits_{\mathfrak{Im}\,z\to+\infty}\varphi (z) = 0$. Then $\varphi$ satisfies
    $$|\varphi(z)|\leq-\frac{1}{\pi}\log(1-\mathrm{exp}(-2\pi\mathfrak{Im}\,z)),\quad\text{for}\quad z\in\mathbb{H}.$$
\end{lem}

Let $f = f_0\in\mathrm{Zak}_{\mathfrak{g}}(\theta)$ with EGM rotation number $\overline{\theta}$ (expressed by continued fraction with sector renormalization combinatorics). Let us fix $g := f_{n+1}$ in the renormalization tower. Denote by $\iota := \overline{\theta}_{n+1}$ the rotation number of $g$. Recall that $\overline{\theta}_{n}$ is the rotation number of $f_{n}$ and it has continued fraction with sector renormalization combinatorics:
\begin{equation}
\label{eq:bar.theta}
    \overline{\theta}_n = \langle(\epsilon_n,\overline{a}_n),(\epsilon_{n+1},\overline{a}_{n+1}),...\rangle.
\end{equation}
We set $\overline{a} =  \overline{a_n}$. Let $G$ be the lift of $g$ via $\mathrm{ixp}$ such that $G(z) = z+\iota +o(1)$ near $+i\infty$. For $m\geq -1$, let $\widehat{Z}^{-1}_G$ be the lift of $\widehat{Z}^{-1}_g$ via $\mathrm{ixp}$; similarly, $Z_G$ is the lift of the Siegel disk $Z_g$. 

\begin{defn}
    We say that a simple arc $\ell\subset\mathbb{H}$ in the upper-half plane has $\bfS$-bounded spiral, if 
$$\sup_{z,w\in\ell} |\mathfrak{Re}(z-w)|\leq \bfS.$$
\end{defn}

\begin{lem}\label{lem.kquasiarc.hightype}
    For any $\epsilon>0$, there exists $\mathbf{N}'\geq\bfN$ ($\bfN$ is given by Definition \ref{def.regularized.type}), such that the following holds. Let $g = f_{n+1}$ satisfying $\overline{a}_{n+1} \geq \mathbf{N}'$. Then there exists a $\mathbf{k}$-quasiconformal arc $\gamma$ in the dynamical plane of $G$ with $\bfS$-bounded spiral such that
    \begin{enumerate}[label=\upshape(\roman*)]
     \item\label{lem.kquasiarc.hightype.first} $\gamma$ is a concatenation of $\alpha\# \beta$, such that $\alpha\subset\mathfrak{F}$, where $\mathfrak{F}$ is a fjord of $\widehat{Z}^{-1}_{G}$;
        \item\label{lem.kquasiarc.hightype.second} $|G(z) - z|<\epsilon$ for all $z\in\beta$ and $G(\gamma)\cap T_{\pm1}(\gamma) = \emptyset$, where $T_{\tau}:z\mapsto z+\tau$.
    \end{enumerate}
\end{lem}
\begin{proof}
Proposition \ref{prop.sectorial.bounds} \ref{thm.sectorial.bounds.first} together with
Lemma \ref{lem.koebe-on-halfplane} implies that for a uniform $\bfM\gg1$, $G$ is uniformly close to the translation: $G(z) = z +\iota + o(1)$ if $\mathfrak{Im}\,z\geq \bfM$. By Proposition \ref{prop.sectorial.bounds} \ref{thm.sectorial.bounds.third}, there exists a fjord $\mathfrak{F}$ of $\hZ^{-1}_G$ that enters $\{z\,|\,\mathfrak{Im}\,z\geq M\}$. By Lemma \ref{lem.comparable.dist}, every fjord of $\hZ^{-1}_G$ is a uniform qc-disk. Therefore, there exists a uniform $\bfk$-qc arc $\gamma$ that is a concatenation of $\alpha\subset\mathfrak{F}$ and a straight line $\beta\subset\{z\,|\,\mathfrak{Im}\,z\geq M\}$ (make a uniform qc adjustment at the concatenation point). The second assertion follows from the fact that $\mathfrak{F}$ is quasi-invariant under $G$: $G(\gamma)$ is protected by level $-1$ and level $0$ annuli surrounding $\partial\mathfrak{F}$.  
\end{proof}

\begin{defn}\label{def.high.type}
    We say that $g = f_{n+1}$ is high type, if $\overline{a}_{n+1} \geq \mathbf{N}'$; otherwise it is called bounded type.
\end{defn}
\begin{rem}
    Notice that if $g$ is high type, then it is also regularized type (Definition \ref{def.regularized.type}); a regularized type can also be bounded type. However, if $g$ is both regularized and bounded type, then $\hZ^{-1}_g$ still has the same properties as $\hZ^{-1}_{f_0}$, with the bounds depending implicitly on $\bfN'$. The proof follows the same lines in \S \ref{subsec.nest-tiling}.   
\end{rem}

\vspace{0.15cm}

Going back temporarily to the dynamical plane of $f=f_0$, we construct a rectangle $\tilde{\R}_n$ for the purpose of anti-renormalization:

\begin{defn}[Preparation for anti-renormalization]\label{defn.tildeRn}
    Let $\tilde{\Sigma}_{y_n}\subset \Sigma_{y_n}$ be a subarc starting from $y_n$, such that $\mathrm{diam}_{\C}(\tilde{\Sigma}_{y_n})\asymp\delta\cdot\dist_{\C}(x_n,y_n)$. By Lemma \ref{lem.return.map.domain}, we have $f^{\q_{[n+1]}}(\tilde{\Sigma}_{y_n})\subset S_n$, by taking the constant in "$\asymp$" sufficiently small (but still bounded away from 0). Moreover $f^{\q_{[n+1]}}|_{\tilde{\Sigma}_{y_n}}$ is injective.
   
    Similarly as in the construction of $\R_n$, we can select a uniform quasi-arc $\tilde{J}$ connecting the other end points of  $\tilde{\Sigma}_{y_n}$ and $f^{\q_{n+1}}(\tilde{\Sigma}_{y_n'})$ such that $f^{\q_{n+1}}(\tilde{\Sigma}_{y_n})\cup\tilde{\Sigma}_{y_n}\cup[f^{\q_{n+1}}(y_n),y_n]_{\partial\hZ^{-1}}\cup\tilde{J}$ bounds a rectangle $\tilde{\R}_n$. 
\end{defn}
Using the similar arguments for Lemmas \ref{lem.well-defined-sector} and \ref{lem.widthbound-on-Rn}, one can get the following

\begin{lem}\label{lem.tildeRn.for.antirenor}
    One can select $\tilde{J}$, such that the rectangle $\tilde{\R}_n$ is a uniform $\bfK(\delta)$-quasi-disk with uniformly $K$-bounded modulus and width. Moreover, $f^{\q_{n+1}}(\tilde{\Sigma}_{y_n})\cap\overline{f^{\q_{[n+1]}}(\hZ^{[n]})} = \{f^{\q_{n+1}}(y_n)\}$.
\end{lem}

\begin{defn}[Selection of wall near $\partial\hZ^{-1}_G$]\label{def.wall}
    We first select a qc-arc $\Sigma\subset\mathbb{H}\setminus\hZ^{-1}_G$ starting from $\partial \hZ^{-1}_G$ such that $\mathrm{ixp}(\Sigma)=\varphi_{n+1}(\tilde{\Sigma}_{y_{n+1}})$. Recall that $\tilde{\Sigma}_{y_{n+1}}$ is constructed in Definition \ref{defn.tildeRn}. We have that $G$ sends $\Sigma$ injectively to $G(\Sigma)$, $G(\Sigma)\cap\Sigma = \emptyset$. We also have $|\Sigma|_{\C}\asymp 1$ since $|\tilde{\Sigma}_{y_n}|_{\C}\asymp|\R_{n+1}|_{\C}$ by construction and $\varphi_{n+1}$ glues uniformly qc-disk $\R_{n+1}$ (Lemma \ref{lem.widthbound-on-Rn}) to uniformly qc-annulus $\mathbb{D}\setminus\hZ^{-1}_g$ (Lemma \ref{lem.fjord.quasidisk}).

    Next we select $\zeta$ to be a qc-arc $ \mathbb{H}\setminus\hZ^{-1}_G$ connecting the two endpoints of $\Sigma$ and $T_1(\Sigma)$ such that the rectangle $\R_{\Sigma}$ bounded by $\Sigma,T_1(\Sigma),\zeta$ and $\partial\hZ^{-1}_G$ is a uniform qc disk and has uniformly bounded modulus and width. Let $\tau\Sigma,\tau\R_{\Sigma}$ be the reflection of $\Sigma,\R_{\Sigma}$ with respect to $\partial\hZ^{-1}_G$.

    The set $\mathrm{Wall}(\partial\hZ^{-1}_G) := \bigcup_{m\in\mathbb{Z}}T_m(\R_\Sigma\cup\tau\R_\Sigma)$ is called the {\it wall} of $\partial\hZ^{-1}_G$.
\end{defn}

The following lemma follows from Proposition \ref{prop.complex.bd-for-fqn}.

\begin{lem}[Wall structure near $\partial\hZ^{-1}_G$]\label{lem.wall.Z_G}
    Let $z\in\hZ^{-1}_G$. Assume that $k\geq 1$ is the smallest integer such that $G^{k}(z)\notin\hZ^{-1}_G$ or $G^{-k}(z)\not\in\hZ^{-1}_G$. Without loss of generality, assume that $G$ acts as a right translation near $+i\infty$. Then there exists $m\in\mathbb{Z}$ such that $G^{\pm k}(z)\in T_m(\mathcal{R}_\Sigma)$ and $G^{\pm (k-1)}(z)\subset T_m(\tau\mathcal{R}_\Sigma)\cup T_{m\mp1}(\tau\mathcal{R}_\Sigma)$.\qed
\end{lem}
\subsection{Robustness of sectorial antirenormalization}

Next we construct anti-renormalization sector for $G$. We pick a qc-arc $\gamma\subset\widehat{Z}^{-1}_G$ connecting $\partial\widehat{Z}^{-1}_G\cap\Sigma$ to $+i\infty$ (\textbf{$\boldsymbol{\gamma}$ will be specified later in the proof of Proposition \ref{prop.bd.spiral}}). We define two Jordan arcs $\Gamma_0,\Gamma_{-1}\subset\widehat{Z}^{-1}_G$ and the deck transformation $\chi_{-1}$ of the anti-renormalization as follows:
\begin{equation}\label{eq.anti-renor.sector}
   \Gamma_0 := \Sigma\#\gamma,
    \quad\chi_{-1} :=\left\{\begin{aligned}
        & T_{\overline{a}_n}\circ G,\,\,\text{if }\epsilon_n<0 \text{ and }\epsilon_{n+1}>0  \\
        & T_{-\overline{a}_n-1}\circ G,\,\,\text{if }\epsilon_n>0 \text{ and }\epsilon_{n+1}>0 \\
         & T_{-\overline{a}_n}\circ G,\,\,\text{if }\epsilon_n>0 \text{ and }\epsilon_{n+1}<0  \\
         & T_{\overline{a}_n+1}\circ G,\,\,\text{if }\epsilon_n<0 \text{ and }\epsilon_{n+1}<0 
    \end{aligned}\right.,\quad\Gamma_{-1} :=\chi_{-1}(\Gamma_0).
\end{equation}

\begin{defn}[Admissible anti-renormalization sector]\label{defn.admissible.sector}
Suppose there is $\gamma\subset \hZ^{-1}_G$ such that the $\Gamma_0,\Gamma_{-1}$ constructed in (\ref{eq.anti-renor.sector}) satisfies $\Gamma_0\cap T_{\pm1}(\Gamma_{-1}) = \emptyset$ (in particular, $\Gamma_0\cap \Gamma_{-1} = \emptyset$) and $\Gamma_0\cap\mathrm{Wall}(\partial\hZ^{-1}_G) =\Sigma\cup\tau\Sigma$. We call such $\Gamma_0$ admissible.

Let $\eta\subset \mathbb{H}\setminus\hZ^{-1}_G$ be an arc connecting the two end points of $\Gamma_0,\Gamma_{-1}$. If $\Gamma_0$ is admissible and $\Gamma_0,\Gamma_{-1}$ together with $\eta$ bounds a sector $S_{\mathrm{anti}}$, we say that $S_{\mathrm{anti}}$ is an admissible anti-renormalization sector of $G$.
\end{defn}

\begin{prop}[Robustness of anti-renormalization]\label{prop.robustness.antirenor}
Let $S_{\mathrm{anti}}$ be an admissible anti-renormalization sector and $\varphi_{\mathrm{anti}}: S_{\mathrm{anti}}\longrightarrow\mathbb{D}$ be the conformal gluing map identifying $z\in \Gamma_{0}$ with $\chi_{-1}(z)\in\Gamma_{-1}$. Let $\widehat{f}_n:\mathbb{D}\longrightarrow\mathbb{C}$ be the map obtained by projecting $T$ via $\varphi_{\mathrm{anti}}$, where $T = T_1$ if $\epsilon_n>0$ and $T = T_{-1}$ if $\epsilon_n<0$.

    Then there exists an absolute small constant $r>0$ and a holomorphic embedding $\xi: \mathbb{D}\longrightarrow\mathbb{C}$, such that $\xi(\mathbb{D}(1-2r))\subset\hZ^0_{f_n}\subset \xi(\mathbb{D}(1-r))$ and $f_{n} = \xi\circ \widehat{f}_n\circ\xi^{-1}$ on $\xi(\mathbb{D}(1-r))$.
\end{prop}

\begin{proof}
    Let $S_{\mathrm{anti},I}$ be the sector associated to the internal rays of $G$: $S_{\mathrm{anti},I}$ coincides with $S_{\mathrm{anti}}$ on $\mathbb{H}\setminus\hZ^{-1}_G$ and $\partial S_{\mathrm{anti},I}\cap\hZ^{-1}_G = I_0\cup I_{-1}$, where $I_0,I_{-1}$ are two internal rays in the lifted Siegel disk $Z_G$ and $\chi_{-1}(I_{0}) = I_{-1}$. We make the following claim:
    \begin{claim*}
        For every $z\in S_{\mathrm{anti}}\cap\hZ^{-1}_{G}$, there exists $k\in\mathbb{Z}$, such that $\chi_{-1}^{k}(z)\in S_{\mathrm{anti},I}$.
    \end{claim*}
    \begin{proof}[Proof of the Claim:]
        Without loss of generality, we may assume that $\chi := \chi_{-1} = T_{\overline{a}_n}\circ G$ and that $z$ is on the left hand side component of $\hZ^{-1}_G\setminus I_{0}$. Consider two cases:
        \begin{itemize}[label = {-}]
            \item If $\chi^m(z) \in \hZ^{-1}_G$ for all $m\geq 0$. Suppose by contradiction that for any $k$, $\chi^k(z)\not\in S_{\mathrm{anti},I}$. Then there exists a smallest $k_0\geq 1$, such that $\chi^{k_0}(z)$ is on the right hand side component of $\hZ^{-1}_G\setminus I_{0}$ and $\chi^{k_0-1}(z)$ is on the left hand side component of $\hZ^{-1}_G\setminus I_{0}$. On the other hand, $\chi(\chi^{k_0-1}(z))$ is on the left hand side component of $\chi(\hZ^{-1}_G)\setminus I_{-1}$. Also notice that the left hand side component of $\chi(\partial\hZ^{-1}_G)$ does not intersect the right hand side component of $\partial\hZ^{-1}_G\setminus I_{0}$, since $I_{0}$ intersects $\partial\hZ^{-1}_G$ at a regular point (the point is also in $\partial Z_G$). This implies that $\chi^{k_0}(z)$ is on the left hand side component of $\hZ^{-1}_G\setminus I_{0}$, a contradiction.

            \item If there exists a smallest $m_0\geq 1$ such that $\chi^{m_0}\not\in\hZ^{-1}_G$ and $\chi^{m_0-1}(z)\in\hZ^{-1}_G$. Suppose by contradiction that for any $k\leq m_0$, $\chi^k(z)\not\in S_{\mathrm{anti},I}$. Then $\chi^{m_0-1}(z)$ must be in the left hand side component of $\hZ^{-1}_G\setminus I_{0}$, otherwise, it will be on the right hand side component of $\hZ^{-1}_G\setminus I_0$ and $\chi^{m_0-2}(z)$ is on the left hand side component of $\hZ^{-1}_G\setminus I_{0}$. Using the same argument in the above case, we can derive a contradiction. By Lemma \ref{lem.wall.Z_G}, there exists $l\in\mathbb{Z}$ such that $\chi^{m_0}(z)\in T_l(\R_{\Sigma})$. Now we claim that $l\geq 0$. If $l<0$, then notice that we have $\chi^{m_0-1}(z)\in\tau\R_{\Sigma}\cup T_{l-1}(\tau\R_{\Sigma})$ by Lemma \ref{lem.wall.Z_G}. Therefore $\chi^{m_0-1}(z)$ is in the left hand side component of $\hZ^{-1}_G\setminus\Gamma_{-1}$. But since $z\in S_{\mathrm{anti}}$ is in the right hand side component of $\hZ^{-1}_G\setminus\Gamma_{0}$, $\chi^{m_0-1}(z)$ is also in the right hand side component of $\hZ^{-1}_G\setminus\Gamma_{0}$.

            Now since $\chi^{m_0-1}(z)$ is in the left hand side component of $\hZ^{-1}_G\setminus I_{0}$, we have $\chi^{m_0}(z)$ in the left hand side component of $\chi(\hZ^{-1}_G)\setminus I_{-1}$. By Lemma \ref{lem.tildeRn.for.antirenor}, $\partial[\chi(\hZ^{-1}_G)]\cap\chi(\Sigma)$ only contains the endpoint of $\chi(\Sigma)$ on $\partial\hZ^{-1}_G$. However, $l>0$ implies that $\chi^{m_0}(z)\in S_{\mathrm{anti}}\setminus\hZ^{-1}_G = S_{\mathrm{anti},I}\setminus\hZ^{-1}_G$. This contradicts our assumption at the beginning that $\chi^{m_0}(z)\not\in S_{\mathrm{anti},I}$. 
        \end{itemize}
    \end{proof}

Now since $S_{\mathrm{anti}}\setminus\hZ^{-1}_G = S_{\mathrm{anti},I}\setminus\hZ^{-1}_G$, we have constructed a piecewise holomorphic map $\tilde{\xi}: S_{\mathrm{anti}}\longrightarrow S_{\mathrm{anti},I}$ such that $\tilde{\xi} = \mathrm{id}$ on $S_{\mathrm{anti}}\setminus\hZ^{-1}_G$. We define piecewisely an embedding $\zeta$ from $S_{\mathrm{anti},I}\cap\hZ^{-1}_G$ to the dynamical plane of $f_n$:
\begin{itemize}[label = {-}]
    \item For $0\leq j\leq \overline{a}_n-1$, let $U_j\subset\hZ^{-1}_G$ be the sector bounded by $T_{j}(I_{0}),T_{j+1}(I_{0})$ and $\partial\hZ^{-1}_G$; then $\zeta_j := f_n^{\overline{a}_n}\circ\varphi_{f_n}^{-1}\circ\mathrm{ixp}$ embeds $U_j$ into the dynamical plane of $f_n$.
    \item Let $U_{\overline{a}_n}$ be the sector bounded by $T_{\overline{a}_n}(I_{0}), I_{-1}$ and $\partial\hZ^{-1}_G$; then $\zeta_j := (f_n|_{S_{f_n}\cap\hZ^{-1}_{f_n}})^{-1}\circ\varphi_{f_n}^{-1}\circ\mathrm{ixp}$ embeds $U_{\overline{a}_n}$ into the dynamical plane of $f_n$.
\end{itemize}
By construction, $\zeta$ conjugates $T$ to $f_n$. Then $\xi := \zeta\circ\tilde{\xi}\circ\varphi_{\mathrm{anti}}^{-1}$ is the desired holomorphic embedding. The existence of $r>0$ is given by Proposition \ref{prop.complex.bd-for-fqn} and Koebe Distortion Theorem.
\end{proof}

\subsection{Yoccoz-type estimate on anti-renormalization}\label{subsec.Yoccoz.estimate}
We will keep the same notations as \S \ref{subsec.anti-renor}.  Let us fix $g := f_{n+1}$ in the renormalization tower and recall the definition of $\chi_{-1}$ in (\ref{eq.anti-renor.sector}). By Proposition \ref{prop.sectorial.bounds} \ref{thm.sectorial.bounds.first}, $\chi_{-1}$ is injective on $\widehat{Z}^{-1}_G$, which contains some uniform upper-half plane $\{w\,|\,\mathfrak{Im\,}w>\frac{1}{2\pi}\log^+\boldsymbol{\rho}\}$. Moreover, by (\ref{eq.anti-renor.sector}), there exists $\alpha = \overline{\theta}_{n+1}\pm\overline{a}_n\pm1$, such that $$\chi_{-1}(w) - w -\alpha = o(1),\quad\text{as }\mathfrak{Im}\,w\to+\infty.$$
By Lemma \ref{lem.koebe-on-halfplane}, we get immediately the following:
\begin{lem}\label{lem.estimate.translation}
    For any small constant $0<\mathbf{a}\ll1$, there exists $\mathbf{b} = \mathbf{b(\mathbf{a})}\gg1$, such that 
   $$|\chi^{-1}(w)-w-\alpha|\leq \mathbf{a}\cdot|\alpha|\quad\text{if}\quad\mathfrak{Im}\,w\geq \mathbf{b}.$$ 
\end{lem}

Essentially, Lemma \ref{lem.estimate.translation} together with bounded geometry on $S_{\mathrm{anti}}$ will imply Yoccoz's estimate: following the lines in \cite[\S 3.3-3.4]{Yoc}, we can derive the distortion estimate for $\psi_{\mathrm{anti}}$ (see \cite[\S 3.4, Proposition 2]{Yoc}). For the reader's convenience, we sketch the proof below.
\begin{prop}\label{prop.extension.psi}
Suppose that $G$ has an admissible anti-renormalization sector ${S_{\mathrm{anti}}}$. Let $\varphi_{\mathrm{anti}}:S_{\mathrm{anti}}\longrightarrow\mathbb{D}$ be the gluing map given by Proposition \ref{prop.robustness.antirenor}. Let ${\psi}_{\mathrm{anti}}:S_{\mathrm{anti}}\longrightarrow\mathbb{H}$ be the lift of $\varphi_{\mathrm{anti}}$ via $\mathrm{ixp}$. Suppose furthermore that $\Gamma_0$ is a straight line near $+i\infty$:
\begin{equation}\label{eq.almost.straight}
    \Gamma_0 \cap\{w\,|\,\mathfrak{Im}\,w\geq\bfb\} = \{i(\bfb+s)\,|\,s\geq 0\}.
\end{equation}
Then the following hold:  
\begin{enumerate}[label=\upshape(\roman*)]
    \item\label{prop.extension.psi.first} $\psi_{\mathrm{anti}}$ has an injective analytic extension to ${S_{\mathrm{anti}}}\cup\mathcal{V}$, where $$\mathcal{V} := \{w\,|\,\mathfrak{Im}\,w\geq\frac{\mathbf{a}}{2}\cdot|\mathfrak{Re}\,w|+\mathbf{b}\}.$$
    \item\label{prop.extension.psi.second} 
    
    $\psi_{\mathrm{anti}}(\chi_{-1}(w)) = \psi_{\mathrm{anti}}(w)+1$ as long as both sides of the equality are well-defined.
    
    \item\label{prop.extension.psi.third}  There exists a uniform $\mathbf{L}\gg1$, such that if $w$ satisfies 
    \begin{equation}\label{eq.condition.w}
        |\mathfrak{Re}\,w|\leq|\alpha|\quad\text{ and }\quad  \mathfrak{Im}\,w\geq \mathbf{L}\cdot|\alpha|,
    \end{equation}
    then    \begin{equation}\label{eq.estimate.psi.derivative}
        ||\alpha|\cdot\psi_{\mathrm{anti}}'(w)-1|\preceq{t^{-2}},\quad \text{where }t := \frac{\mathfrak{Im}\,w}{|\alpha|}.
    \end{equation}
\end{enumerate}
\end{prop}
\begin{proof}
Without loss of generality, let us assume that $\alpha >0$. Let $\mathcal{V}_- = \{w\in\mathcal{V}\,|\,\mathfrak{Re}\,w\leq 0\}$ and $\mathcal{V}_+ = \{w\in\mathcal{V}\,|\,\mathfrak{Re}\,w\geq 0\}$. Then by Lemma \ref{lem.estimate.translation}, $\chi_{-1}(\mathcal{V_-})\subset \mathcal{V_-}\cup S_{\mathrm{anti}}$ and $\mathcal{V}_+\subset S_{\mathrm{anti}}\cup\chi_{-1}(\mathcal{V}_+)$. Therefore for every $w\in\mathcal{V}_-$, there exists a smallest $n(w)$ such that $\chi_{-1}^{n(w)}(w)\in S_{\mathrm{anti}}$; for every $w\in\mathcal{V}_+$, there exists $n(w)$ such that $\chi_{-1}^{-n(w)}(w)\in S_{\mathrm{anti}}$. Hence we can extend $\psi_{\mathrm{anti}}$ to $\mathcal{V}$ by $\psi_{\mathrm{anti}}(w) = \psi_{\mathrm{anti}}(\chi_{-1}^{\pm n(w)}(w))\mp n(w)$. This proves \ref{prop.extension.psi.first} and \ref{prop.extension.psi.second}. 

To prove \ref{prop.extension.psi.third}, consider the map $\psi(w) := \psi_{\mathrm{anti}}(w_0+ r w)$ for $w_0$ satisfying \ref{eq.condition.w}, where $r\asymp \mathfrak{Im}\,w_0 - \mathbf{b}$. Notice that $\psi$ is univalent on the unit disk. Such maps satisfy (see \cite[\S 3.4, Lemma 1]{Yoc}) 
\begin{equation}\label{eq.Dpsi.estimate}
    |\psi(w)-\psi(-w)-2\psi'(0)w|\leq 2|\psi'(0)|\cdot|w|^3\cdot\frac{3-|z|^2}{(1-|z|^2)^2}.
\end{equation}
On the other hand, \cite[\S 3.4, Lemma 1]{Yoc} ensures that there is $w_1$ such that $\chi_{-1}(w_1)+ w_1 = 2w_0$ and $|w_1-w_0|\preceq\alpha$. Evaluate $\psi$ at $u = r^{-1}(w_1-w_0)$ and plug into (\ref{eq.Dpsi.estimate}). A straightforward simplification shows that 
$$|1-\alpha \psi_{\mathrm{anti}}'(w_0)|\preceq \frac{\alpha^3}{(\mathfrak{Im}\,w_0)^2}|\psi_{\mathrm{anti}}'(w_0)|.$$
Set $X = 1-\alpha\psi'(w_0)$. Then we have $$|X|\preceq \frac{\alpha^2}{(\mathfrak{Im}\,w_0)^2}|1-X|\leq \frac{\alpha^2}{(\mathfrak{Im}\,w_0)^2}(1+|X|).$$
By taking $\mathbf{L}$ large enough, we get (\ref{eq.estimate.psi.derivative}). 
\end{proof}

Yoccoz estimate above implies that the gluing map $\psi_{\mathrm{anti}}$ acts like $\frac{w}{|\alpha|}$ near infinity:
\begin{cor}\label{cor.yoccoz.estimate}
    Under the assumption of Proposition \ref{prop.extension.psi}, and suppose furthermore that $\Gamma_0$ is a $\bfk$-qc arc with $\bfS$-bounded spiral, then 
    \begin{equation}\label{eq.estimate.psi}
    \left|\psi_{\mathrm{anti}}(w) - \frac{w}{|\alpha|}\right|=O_{\bfS,\bfk}(1),\quad\text{ for all }w\text{ satisfying (\ref{eq.condition.w})}.
\end{equation}
    
    Moreover there exists $\bfS',\bfk'$ only depending on $(\bfS,\bfk)$, such that $\psi_{\mathrm{anti}}(\Gamma_0)$ is a $\bfk'$-qc arc with $\bfS'$-bounded spiral. 
\end{cor}
\begin{proof}
   
For every $w$ satisfying (\ref{eq.condition.w}), take $w_0$ such that 
\begin{equation}\label{eq.condition.w0}
    \mathfrak{Re}\,w_0 = \mathfrak{Re}\,w\quad\text{ and }\quad\mathfrak{Im}\,w_0 = \mathbf{L}\cdot|\alpha|.
\end{equation}
Integrating (\ref{eq.estimate.psi.derivative}) from $w_0$ to $w$, we get
\begin{equation}\label{eq.intermediate.estimate.psi}
||\alpha|\cdot[\psi_{\mathrm{anti}}(w)-\psi_{\mathrm{anti}}(w_0)]-(w-w_0)|\leq
\int_{w_0}^{w}||\alpha|\cdot\psi_{\mathrm{anti}}'(\zeta)-1|\,\mathrm{d}\zeta\preceq|\alpha|^2\int_{\mathbf{L}\cdot|\alpha|}^{\mathfrak{Im}\,w}t^{-2}\,\mathrm{d}t = \frac{|\alpha|}{\mathbf{L}}-|\alpha|^2\cdot(\mathfrak{Im}\,w)^{-1}.
\end{equation}
Dividing $|\alpha|$ on both sides, we get (notice that $|w_0|\asymp_{\mathbf{L}}|\alpha|$ and $\mathfrak{Im}\,w\succeq_{\mathbf{L}}|\alpha|$):
$$\left|\psi_{\mathrm{anti}}(w)- \frac{w}{|\alpha|} -\psi_{\mathrm{anti}}(w_0)\right|= O_{\mathbf{L}}(1).$$
On the other hand, $\psi_{\mathrm{anti}}$ sends the lowest subrectangle $R$ of $S_{\mathrm{anti}}$ with imaginary part $\leq\mathbf{L}\alpha$ to a rectangle with $O_{\mathbf{S},\bfk}(1)$ height, because (i). the shape of $R$ is $\bfS$-commensurable to a square with length $|\alpha|$; (ii). $\partial R$ is a $\bf k$-qc arc; (iii). $\mathfrak{Im}\,(G-\mathrm{id}) = O(1)$ by compactness of the renormalization tower (Proposition \ref{prop.sectorial.bounds} \ref{thm.sectorial.bounds.first}). In particular, this implies $\psi_{\mathrm{anti}}(w_0) = O_{\mathbf{S},\bfk}(1)$ and hence (\ref{eq.estimate.psi}) is proved. The assertion that  $\psi_{\mathrm{anti}}(\Gamma_0)$ is $\bfk'$-qc with $\bfS'$-bounded spiral follows from (\ref{eq.estimate.psi}) and the fact that $\psi_{\mathrm{anti}}$ extends to a qc map on a neighborhood of $\partial R$ with dilatation only depending on $\bfk$.
\end{proof}

In order to apply inductively Yoccoz's estimate, we need the following two lemmas:
\begin{lem}\label{lem.concatenation}
    Let $\Gamma\subset{\mathbb{H}}$ be a $\bfk$-qc arc starting from $z_0$ and ending at $+i\infty$. Let $z_{\bfb}$ be the first intersection point of $\Gamma\cap\{x+iy\,|\,y = \bfb\}$. Suppose that $\mathfrak{Im}\,z_0<\bfb$. Let $\Gamma'$ be the concatenation of the subarc in $\Gamma$ connecting $z_0$ to $z_\bfb$ and the straight half line from $z_\bfb$ to $+i\infty$. Then $\Gamma'$ is a $\bfk'$-qc arc with the same $\bfS$-bounded spiral, where $\bfk'$ only depends on $\bfk$.
\end{lem}
\begin{proof}
    We use the bounded turning condition. It suffices to verify the case where $u,v\in\Gamma'$ such that $u\in\Gamma$ and $v$ is on the straight half line. Notice that one has $|u-z_\bfb|,|v-z_\bfb|\leq |u-v|$. Hence 
    $$\mathrm{diam}(\Gamma'[u,v])\leq \mathrm{diam}(\Gamma'[u,z_\bfb])+|z_\bfb-v|\leq C(\bfk)|u-z_\bfb|+|z_\bfb-v|\leq (C(\bfk)+1)\cdot|u-v|.$$
\end{proof}

\begin{lem}\label{lem.modified.bounded-qc}
    Let $\Gamma_0$ be constructed in (\ref{eq.anti-renor.sector}) such that $\Gamma$ is a $\bfk$-qc arc. Suppose that the anti-renormalization sector associated with $\Gamma_0$ is admissible. Then there exists $\bfb= \bfb(\bfk)$, such that the modified arc $\Gamma'_0$ given by the operation in Lemma \ref{lem.concatenation} satisfies $\Gamma'_0\cap T_{\pm1}(\Gamma'_{-1}) = \emptyset$ (set $\Gamma'_{-1}= \chi_{-1}(\Gamma'_0)$), i.e., the anti-renormalization sector associated to $\Gamma'_0$ is also admissible. 
\end{lem}
\begin{proof}
    Let $\eta\subset \Gamma'_0$ be the straight line part and $\xi\subset \Gamma'$ be the remaining part. Let $\eta_{-1} = \chi_{-1}(\eta)$ and $\xi_{-1} = \chi_{-1}(\xi)$. Since the sector associated to $\Gamma_0$ is admissible, the only possible intersection between $\Gamma'_0,T_{\pm1}(\Gamma'_{-1})$ lies in $\eta\cap T_{\pm1}(\xi_{-1})$ and $\xi\cap T_{\pm1}(\eta_{-1})$. On the other hand, by Lemma \ref{lem.estimate.translation}, we can take $\bfb$ sufficiently large, such that if such intersection occurs, then $\Gamma'_0$ or $\Gamma'_{-1}$ will have sufficiently large turning, contradicting Lemma \ref{lem.concatenation}.
\end{proof}

Now consider the anti-renormalization cascade between two consecutive high type levels ("$\bfH$" and "$\bfB$" represent that the map is high-type or bounded-type):
\begin{equation}\label{eq.antirenor.cascade}
    (f_{n-m},\bfH)\dashleftarrow (f_{n-m+1},\bfB)\dashleftarrow...\dashleftarrow (f_n,\bfB)\dashleftarrow (f_{n+1},\bfH).
\end{equation}

\begin{prop}\label{prop.bd.spiral}
    There exists uniform constants $\bfk,\bfS>0$, such that for every $g=f_k$ in the cascade  (\ref{eq.antirenor.cascade}) the following hold:
    \begin{enumerate}[label=\upshape(\roman*)]
        \item $G$ has an admissible $S_{\mathrm{anti}}$;
        \item the associated $\Gamma_0$ is a $\bfk$-qc arc with $\bfS$-bounded spiral;
        \item $\Gamma_0$ is straight line near $+i\infty$, i.e. it satisfies (\ref{eq.almost.straight}).
    \end{enumerate}
\end{prop}
\begin{proof}
    For $k = n+1$, let us pick $\gamma'\subset\widehat{Z}^{-1}_G$ given by Lemma \ref{lem.kquasiarc.hightype}. Let $\gamma$ be an $\bfk$-qc arc with $\bfS$-bounded spiral contained in the sector bounded by $\gamma',T_1(\gamma')$ and $\partial\hZ^{-1}_G$ such that $\gamma$ joins with $\Sigma$ constructed in Definition \ref{def.wall}. Define $\Gamma_0$, $\Gamma_{-1}$ associated to $\gamma$ by (\ref{eq.anti-renor.sector}). By Lemma \ref{lem.kquasiarc.hightype} \ref{lem.kquasiarc.hightype.second}, we have $\Gamma_0\cap T_{\pm1}(\Gamma_{-1}) = \emptyset$ and that $\Gamma_0,\Gamma_{-1}$ have $\bfS$-bounded spiral and are $\bfk$-qc. In particular, the corresponding $S_{\mathrm{anti}}$ constructed in Definition \ref{defn.admissible.sector} is admissible. Let $\Gamma'_0$ be the modified arc of $\Gamma_0$ given by Lemma \ref{lem.concatenation}. By Lemma \ref{lem.modified.bounded-qc}, $\Gamma_0'$ is $\bfk'$-qc with $\bfS$-bounded spiral. Moreover, the associated anti-renormalization sector is admissible. Abusing the notation, let us still denote by $\Gamma_0$ this modified arc.

    Now fix a $\bfs\gg1$. For $n-\bfs +1 \leq k \leq n$, we construct $\Gamma^k_0$ inductively as follows: suppose that $\Gamma_0^{k+1}$ and $S_{\mathrm{anti}}^{k+1}$ have been constructed for level $k+1$. Define $$\Gamma_0^k = \frac{1}{2\pi i}\log(\xi\circ \mathrm{ixp}\circ\psi_{\mathrm{anti}}(\Gamma_0^{k+1})),$$
    where $\xi$ is the holomorphic embedding given by Proposition \ref{prop.robustness.antirenor}. By Corollary \ref{cor.yoccoz.estimate}, $\Gamma_0^k$ is a $\bfk'$-qc arc with $\bfS'$-bounded spiral, where $\bfk',\bfS'$ depend on the dilation and spiral of $\Gamma_0^{k+1}$. Notice that essentially, $(\bfk',\bfS')$ depends on $(\bfk,\bfS)$ and $\bfs$. Also notice that since $\Gamma_0^{k+1}\cap T_{\pm1}(\Gamma_0^{k+1}) = \emptyset$, we have $\Gamma_0^k\cap G(\Gamma_0^k) = \emptyset$. Therefore $S_{\mathrm{anti}}^{k}$ is also admissible. Use the operation in Lemma \ref{lem.concatenation} to get the modified arc. By Lemma \ref{lem.modified.bounded-qc}, the associated anti-renormalization sector is still admissible.

    For $n-m+1 \leq k \leq n-s$, we will use the fact that $g=f_k$ is very close to bounded regime by taking $\bfs$ sufficently large.
    More precisely, consider the collection of all renormalization towers generated by degree $\mathfrak{g}$ Zakeri type polynomials:
    $$\mathbf{Tow} := \{f_n\,|\,f_0\in\mathrm{Zak}_{\mathfrak{g}}(\Theta_{\mathrm{egm}}) \}.$$
    By Lemma \ref{lem.return.map.domain}, \ref{lem.definite.space} and Proposition \ref{prop.sectorial.bounds} \ref{thm.sectorial.bounds.first}, $f_n$ is defined at least on a uniformly qc domain $D_{f_n}$, such that $\mathbb{D}(0,\bfrho)\subset\hZ^{-1}_{f_n}\Subset D_{f_n}$ and $\mathrm{mod}(D_{f_n}\setminus\hZ^{-1}_{f_n})\asymp 1$. Therefore every sequence in $\mathbf{Tow}$ has a converging subsequence under compact-open topology. Consider the corresponding closure $\overline{\mathbf{Tow}}$ and denote by $\overline{\mathbf{Tow}}_{\bfN'}$ the collection of maps having an $\bfN'$-bounded type number at the zero fixed point (all the entries in the continued fraction are $\leq\bfN'$, where $\bfN'$ is given by Definition \ref{def.high.type}). We make the following claim:
    \begin{claim*}
        For any $\epsilon>0$, there exists $\bfs$, such that for all $f_k$ in the cascade (\ref{eq.antirenor.cascade}) with $n-m+1 \leq k \leq n-s$, $f_k$ is contained in the $\epsilon$-neighborhood of $\overline{\mathbf{Tow}}_{\bfN'}$. Moreover, there exists a uniform $\bfK'$ such that for all $f_\infty\in\overline{\mathbf{Tow}}_{\bfN'}$, the Siegel disk $Z_{f_{\infty}}$ is a $\bfK'$-qc disk containing a definite neighborhood of $0$.
    \end{claim*}
\begin{proof}[Proof of the claim]
   For the first statement, suppose by contradiction that such $\bfs$ does not exist. Then by pre-compactness, there exists a sequence of mappings $\{f_{k_n}\}$ (not necessarily in the same renormalization tower) converging to some $f_{\infty}\in \overline{\mathbf{Tow}}_{\bfN'}$, since the rotation number of $\{f_{k_n}\}$ converges to $\bfN'$-bounded type as $\bfs$ goes to infinity. This is a contradiction.

    Now we prove the second assertion. Assume that $f_k\in\mathbf{Tow}$ converges to $f_{\infty}\in\overline{\mathbf{Tow}}_{\bfN'}$. Let us locate each $f_k$ in its cascade (\ref{eq.antirenor.cascade}). Then $n-k$ tends to infinity as $k\to\infty$. Recall the regularization threshold $\bfN$ given by Definition \ref{def.regularized.type}. Denote by $\hZ^{-1}_{f_k}(\bfN)\subset\hZ^{-1}_{f_k}$ the sub-qc disk by deleting all level-$l$ fjords with $-1\leq l \leq n-k-1$. In particular, we have $\hZ^{[l]}_{f_k}(\bfN) = \hZ^{[l+1]}_{f_k}(\bfN)$ for $-1\leq l \leq n-k-1$. Since $\bfN'$ is a uniform constant, $\hZ^{-1}_{f_k}(\bfN)$ has uniformly bounded inner and outer geometry (Definition \ref{def.bdcombi}) inheriting from $\hZ^{-1}_{f_k}$. By Lemma \ref{lem.bdcombi-qcdisk}, $\hZ^{-1}_{f_k}(\bfN)$ is $\bfK'$-qc for some uniform constant $\bfK'$. On the other hand, by passing to a subsequence, the qc-disk $\lim_k \hZ^{-1}_{f_k}(\bfN)$ is invariant under $f_{\infty}$, because the non-invariance of $\hZ^{-1}_{f_k}(\bfN)$ under $f_k$ only lies in dams of level $\geq n-k$, whose diameter converges to 0 as $n-k\to\infty$. Also notice that $\lim_k \hZ^{-1}_{f_k}(\bfN)$ contains 0 and at least one critical point of $f_\infty$. Therefore $\lim_k \hZ^{-1}_{f_k}(\bfN) = Z_{f_\infty}$ by ergodicity of irrational rotation.  
\end{proof}

Now one can just take $\Gamma_0\cap\hZ^{-1}_G$ to be a perturbed internal ray of $Z_{f_\infty}$, which in particular is $\bfk$-qc and has $\bfS$-bounded spiral. Bounded regime will imply $\Gamma_0\cap G(\Gamma_0) = \emptyset$. Hence $S_{\mathrm{anti}}$ is admissible. Then repeating the operation given by Lemma \ref{lem.concatenation}, we will get the desired admissible sector. 
\end{proof}

\begin{proof}[Proof of Lemma \ref{lem.keyestimate}]
    Let us set $\Psi_{\mathrm{anti}} := \mathrm{ixp}^{-1}\circ \xi \circ \varphi_{\mathrm{anti}}$ (see Figure \ref{fig.anti_sector}). Here $\varphi_{\mathrm{anti}} = \mathrm{ixp}\circ \psi_{\mathrm{anti}}$ and $\xi$ is the embedding given by Proposition \ref{prop.robustness.antirenor}. Proposition \ref{prop.sectorial.bounds} makes Corollary \ref{cor.yoccoz.estimate} applicable. Therefore $|\varphi_{\mathrm{anti}}(w)|\asymp e^{-2\pi \mathfrak{Im}w/|\alpha|}$. By Koebe Distortion Theorem and Proposition \ref{prop.sectorial.bounds} \ref{thm.apriori-bound:third}, for some absolute $\tau>0$, we have
    $$|\xi\circ\varphi_{\mathrm{anti}}(w)|\asymp|\overline{\theta}_n|^\tau\cdot e^{-2\pi \mathfrak{Im}w/|\alpha|}.$$
    Finally, notice that $\hat{\theta}_n = |\overline{\theta}_n| = 1/|\alpha|$. Lifting via $\mathrm{ixp}$ will lead to the desired estimate (\ref{eq:sec8.m4}).
   
\end{proof}

\subsection{Optimality of Brjuno Condition}
Let us consider the space $\mathcal{P}_{\mathfrak{g}}(\theta)$ of degree $\mathfrak{g}$ monic polynomials $f$ with a fixed point of irrational rotation number $\theta$: 
$$f(z) = e^{2\pi i\theta}z+a_2z^2+...+a_{\mathfrak{g-1}}z^{\mathfrak{g}-1}+z^{\mathfrak{g}},\quad \theta\in\mathbb{R}\setminus\mathbb{Q}.$$

\begin{thm}[Theorem \ref{mainthm:C}]\label{thm.Brjuno.Zaktype}
   Let $f\in \mathrm{Zak}_{\mathfrak{g}}$ (see (\ref{eq:dfn:zak})). If the rotation number $\theta$ of $f$ is Brjuno and let $\rho(f)$ be the conformal radius of $f$ at $z=0$, then $$\log\rho(f)\asymp-\bfB(\theta).$$ 
\end{thm}
\begin{proof}
    First let us proof for EGM $\theta$ that $\log\rho(f)\preceq -\bfB(\theta)$.

    For each $n\geq 0$, if $f_n$ is regularized type, by Proposition \ref{prop.sectorial.bounds} \ref{thm.sectorial.bounds.second}, we can take $\sigma_{n}$ to be the closest non-zero fixed point to $z=0$ of $f_{n}$. Otherwise, $f_n$ is bounded type and we set $\sigma_n = \epsilon$ with some definite $0<\epsilon\ll1$. We set $\bfx_{n} := \frac{i}{2\pi}\log|\sigma_{n}|^{-1}$. By Proposition \ref{prop.sectorial.bounds} \ref{thm.sectorial.bounds.third}, for all $n\geq 0$, we have
\begin{equation}\label{eq.first.lift}
\mathfrak{Im}\,\bfx_{n}\asymp\log|\sigma_{n}|^{-1}\asymp \log\hat{\theta}_{n}^{-1}.
\end{equation}
Since $\theta$ is EGM, we can take $N = N(\theta)$ to be the largest integer such that Condition ``$\mathfrak{Im}\,w\ge \mathbf L/\wtheta_{N-1}$'' in \eqref{eq.condition.Imw} of Lemma \ref{lem.keyestimate} holds for $w=\bfx_{N}$. By the choice of $N$, the entries in the continued fraction with sector renormalization combinatorics of $\overline{\theta}$ after level $N$ satisfy $\log\overline{a}_{n+1}\preceq {\overline{a}_n}$. Let $\bfB_N(\theta)$ be the partial sum of $\bfB(\theta)$ as in~\eqref{eq:Br:function:v2}. Then we have
$$\bfB(\theta)-\bfB_{N}(\theta) = O(1).$$

We will now lift $\bfx_N$ along the renormalization tower as long as Lemma \ref{lem.keyestimate}  is applicable. Either this will produce a desired small periodic point in the dynamical plane of $f_0$, or the new smaller $N'$ can be selected due to~\eqref{eq:NvsN'}; i.e., the levels between $N'$ and $N$ can be omitted. By repeating the argument, a desired small periodic point will be constructed for $f_0$. 

Let $m\geq 1$ be the largest number of times where 
Lemma \ref{lem.keyestimate} can be applied consecutively starting from $\bfx_N$: define inductively $$\bfy_{N-k} = \Psi_{\mathrm{anti}}(\bfy_{N-k+1}),\quad \bfy_N = \bfx_N,\quad\text{for}\quad 1\leq k\leq m.$$
Then Lemma \ref{lem.keyestimate} yields 
$$\mathfrak{Im}\,\bfy_{N-m}\asymp \beta_{N-m-1}^{-1}\sum_{i=N-m}^N\beta_{i-1}\log\hat{\theta}_i^{-1},\quad\text{where}\quad \beta_i = \prod_{j=0}^{i}\hat{\theta}_j.$$
By definition of $m$, Condition (\ref{eq.condition.Imw}) of Lemma \ref{lem.keyestimate} does not hold for $\bfy_{N-m}$. Then we have $\hat{\theta}_{N-m-1}\mathfrak{Im}\,\bfy_{N-m} = O(1)$. Hence 
$$\beta_{N-m-2}^{-1}\sum_{i=N-m}^N\beta_{i-1}\log\hat{\theta}_i^{-1}\asymp O(1).$$
While $\beta_i = O(2^{-i})$ since $\hat{\theta}_n<1/2$ for all $n\geq 0$, we have 
\begin{equation}\label{eq.exponential.tail}
    \sum_{i=N-m}^N\beta_{i-1}\log\hat{\theta}_i^{-1}\asymp O(2^{-N+m}).
\end{equation}

Denote by $N'$ the largest integer less or equal than $N-m$ such that Condition (\ref{eq.condition.Imw}) holds for $\bfx_{N'}$. Then $\hat{\theta}_{n-1}\log\hat{\theta}_n^{-1} = O(1)$ for all $N'<n \leq N-m$. Combining with (\ref{eq.exponential.tail}) we get
\begin{equation}
    \label{eq:NvsN'}
\sum_{i=N'+1}^N\beta_{i-1}\log\hat{\theta}_i^{-1}\asymp O(\sum_{i = N'+1}^{N-m}2^{-i}) = O(2^{-N'}).\end{equation}
Then repeat the above argument for $\bfx_{N'}$. Eventually, We will either obtain $\bfB(\theta) = O(1)$, or some $\bfy$ which projects by $\mathrm{ixp}$ to a repelling periodic point of $f = f_0$, lifted via Lemma \ref{lem.keyestimate} consecutively from some $\bfx_n$, so that 
$$\mathfrak{Im}\,\bfy \asymp \bfB_n(\theta),\quad \bfB(\theta)-\bfB_n(\theta) = O(1).$$
Since a repelling periodic point cannot belong to the Siegel disk, we obtain $\log\rho(f)\preceq -\bfB(\theta)$. 

Next suppose that $\theta$ is Brjuno, then $f$ can be accumulated by a sequence of $f_n\in\mathrm{Zak}_{\mathfrak{g}}$ with EGM rotation number $\theta_n\to\theta$. By Stability of Siegel disk (\cite[Lemma 17]{ABC}), $\liminf_n\rho(f_n)\geq\rho(f)$. This proves $\log\rho(f)\preceq -\bfB(\theta)$. The $\succeq$ direction is covered by \cite[Theorem 2.6 Chapter I]{Yoc}.
\end{proof}

\appendix
\counterwithin{thm}{subsection}

\section*{Appendices}
\addcontentsline{toc}{section}{Appendices}
\renewcommand{\thesubsection}{\Alph{subsection}}
\numberwithin{equation}{subsection}

\subsection{Rectangles}
\renewcommand{\theequation}{A.\arabic{equation}}

\begin{defn}[Geodesic rectangle]
    Let $D$ be a Jordan disk and $\R\subset\overline{D}$ be a rectangle with horizontal sides $I,J$. We denote by $\R(I,J)$ the geodesic rectangle in $D$, that is, the horizontal sides are $I,J$ and the vertical sides are geodesics in $D$. 
\end{defn}

\begin{defn}[$\F$, $\F^\pm$ and $\F^\circ$ families]
    Let $D$ be a Jordan disk and $I,J\subset \partial D$ be disjoint intervals. Let $I^+,J^+$ be the outer sides of $I,J$ respectively. We denote by \begin{itemize}[label = {-}]
        \item $\F(I,J)$ the family of curves in $\widehat{\mathbb{C}}\setminus (I\cup J)$ connecting $I$ and $J$.
        \item $\mathcal{F}^+(I,J)$ (resp. $\mathcal{F}^-(I,J)$) the family of curves in $\widehat{\mathbb{C}}\setminus\overline{D}$ (resp. $D$) connecting $I$ and $J$
        \item $\mathcal{F}(I^+,J^+)$ the family of curves in $\widehat{\mathbb{C}}\setminus (I\cup J)$ starting from $I^+$ and $J^+$. Let $L\subset \partial D$ be a complementary interval of $I\cup J$. The family $\mathcal{F}^\circ_L(I,J)\subset \mathcal{F}(I^+,J^+)$ is the subfamily consisting of curves disjoint from $N$, where $N=\partial D\setminus (I\cup L\cup J)$.
    \end{itemize} 
\end{defn}

\begin{lem}\label{lemappen.geodesic}
   Let $D$ be a Jordan disk and $\R\subset\overline{D}$ with $\W(\R)>12$ be a rectangle such that $\partial^h\R\subset\partial D$. Set $\partialhL\R = I$ and $\partialhR\R = J$. Let $\Rnew$ be the rectangle obtained from $\R$ by removing two $6$-buffers on each side. Write $\partialhL\Rnew = I^{\mathrm{new}}$ and $\partialhR\Rnew = J^{\mathrm{new}}$. Then
$$\Rnew\subset\R(I,J),\quad \R(I^{\mathrm{new}},J^{\mathrm{new}})\subset\R,$$
$$\W^-(I,J) = \W(\R(I,J))+O(1),\quad \W(\R)= \W(\R(I^{\mathrm{new}},J^{\mathrm{new}}))+O(1).$$
\end{lem}
\begin{proof}
    We may assume that $\overline{D}$ is the standard rectangle $E = [0,x]\times [0,1]$ with $x = \W^-(I,J)$, $I = [0,x]$ and $J = [i,i+x]$. Let us first prove that 
    \begin{equation}\label{eq.WIJ-WRIJ}
        \W^-(I,J) = \W(\R(I,J))+O(1),\quad\Rnew\subset\R(I,J).
    \end{equation}
    Let $a$ be the $x$-coordinate of the most right point on $\partial^{v,0}\R$. Set $l :=\frac{1}{2}+a$. Then vertical curves in $\R$ starting from $[0,a-\frac{1}{2}]$ crosses $[0,a-\frac{1}{2}]\times[0,1]$; vertical curves starting from $[a-\frac{1}{2},l]$ will either crosses $[l,l+\frac{1}{2}]\times[0,1]$, or is contained in $[a-\frac{1}{2},l+\frac{1}{2}]\times[0,1]$. Hence the width of vertical curves or $\R$ starting from $[0,l]$ is less than $6$. Similarly, let $b$ be the $x$-coordinate of the most left point on $\partial^{v,1}\R$ and set $r :=b-\frac{1}{2}$. Then the width of vertical curves in $\R$ starting from $[x-r,x]$ is less than $6$. Therefore
    \begin{equation}\label{eq.Inew-Jnew}
        I^\mathrm{new}\subset[l,r],\quad J^{\mathrm{new}}\subset[l+i,r+i].
    \end{equation}
    It remains to show that the geodesic in $E$ connecting $i,0$ is contained in $[0,\frac{1}{2}]\times[0,1]$. In order to see this, we need the following property of hyperbolic geodesics:
\begin{claim*}
     If $\gamma\subset{D}$ is a hyperbolic geodesic connecting $c,d\in\partial D$ so that it cuts $\partial D$ into two intervals $S,T$, then for any $z\in \gamma$, the probability $P_{D}(z,S)$ of the Brownian motion starting from $z$ and ending at $S$ is equal to the probability $P_{D}(z,T)$ ending at $T$.
\end{claim*}
   \begin{proof}[Proof of the Claim]
      We may assume that $D = \mathbb{D}$. The claim is obvious by symmetry if $\gamma = (-1,1)$. Since the probability of the Brownian motion hitting the interval on $\partial\overline{\mathbb{D}}$ is invariant under Möbius transformation, the claim follows in the general case.
   \end{proof}
Let $L\subset E$ be the semi-circle of radius $\frac{1}{2}$ connecting $i$ and $0$. Notice that $L\subset [0,\frac{1}{2}]\times[0,1]$ is a hyperbolic geodesic in $\mathbb{H}$. By the above claim, for any $z\in L$, $P_{\mathbb{H}}(z,[0,i]) = P_{\mathbb{H}}(z,i\mathbb{R}\setminus[0,i])$. On the other hand, clearly we have $$P_{E}(z,\partial E\setminus[0,i])\geq P_{\mathbb{H}}(z,i\mathbb{R}\setminus[0,i])=P_{\mathbb{H}}(z,[0,i]) \geq P_{E}(z,[0,i]).$$
Therefore the hyperbolic geodesic in $E$ connecting $i,0$ is contained in $[0,\frac{1}{2}]\times[0,1]$. This finishes the proof of (\ref{eq.WIJ-WRIJ}). 

Next we prove 
\begin{equation}\label{eq.WInewJnew}
        \W(\R)= \W(\R(I^{\mathrm{new}},J^{\mathrm{new}}))+O(1).,\quad\R(I^{\mathrm{new}},J^{\mathrm{new}})\subset\R.
    \end{equation}
    Since we have established (\ref{eq.Inew-Jnew}), we may assume that  $I^\mathrm{new}=[l,r]$ and $J^{\mathrm{new}}=[l+i,r+i]$. We will show that 
$\R(I^{\mathrm{new}},J^{\mathrm{new}})\subset [a,b]\times[0,1]$, which immediately implies (\ref{eq.WInewJnew}). It suffices to show that the geodesic $\gamma$ in $E$ connecting $l,l+i$ is contained in $[a,x]\times[0,1]$ (a similar result holds for the geodesic connecting $r,r+i$). Let $L'\subset [a,l]\times[0,1]$ be the semi-circle of radius $\frac{1}{2}$ connecting $l,l+i$. By a similar argument as above (using the claim), one can show that $\gamma$ is on the left-hand side of $L'$. This finishes the proof of (\ref{eq.WInewJnew}). 
\end{proof}

\begin{lem}\label{lemappen.spliting}
    Consider a Jordan disk $D$ and let $I,J\subset\partial D$ be a pair of disjoint intervals. Consider an vertical arc $l$ in the canonical rectangle of $\F^-_{D}(I,J)$. Suppose that $l$ splits $I$ and $J$ into $I_1,I_2$ and $J_1,J_2$ so that the pairs of $I_1,J_1$ and $I_2,J_2$ are on  the same side of $l$. Denote by $D_1$ and $D_2$ the connected components of $D\setminus l$ containing $I_1,J_1$ and $I_2,J_2$ on their boundaries respectively. Then
    \begin{equation}
    \W^-_{D}(I_1,J_1)+ \W^-_{D}(I_2,J_2)-2\leq  \W^-_{D_1}(I_1,J_1)+\W^-_{D_2}(I_2,J_2) =  \W^-_{D}(I,J)\leq \W^-_{D}(I_1,J_1)+ \W^-_{D}(I_2,J_2).
    \end{equation}
    \end{lem}
\begin{proof}
    The last inequality follows immediately from Parallel Law. Let us set $\R$ to be the canonical rectangle of $\F^-_{D}(I,J)$; for $i=1,2$, set $\R_i$ to be the canonical rectangle of $\F^-_{D_i}(I_i,J_i)$ and $\tilde{\R}_i$ the canonical rectangle of $\F^-_{D}(I_i,J_i)$. Since $l$ splits $\R$ into $\R_1$ and $\R_2$, we have 
    $$\W^-_{D_1}(I_1,J_1)+\W^-_{D_2}(I_2,J_2) = \W(\R_1)+\W(\R_2) = \W(\R)= \W^-_{D}(I,J).$$
    It remains to prove the first inequality. For concreteness, we may assume that $D_1$ is on the left-hand side of $D_2$. Let $l_2$ be the right vertical boundary of the leftmost $1$-buffer of $\tilde{\R}_2$. Then the width of vertical curves in $\tilde{R}_1$ intersecting $l_2$ is less than 1. Moreover these curves form a rightmost $1$-buffer of $\tilde{R}_1$. Hence after removing these two $1$-buffers, the resulting new rectangles $\tilde{\R}^{\mathrm{new}}_1,\tilde{\R}^{\mathrm{new}}_2$ are disjoint. Since they are contained in $\R$, the inequality follows from Parallel Law.
\end{proof}

\begin{lem}[{\cite[Lemma A.6]{DL:wZ}, small overlapping principle}]\label{lemappen.disjoint-boundaires}
     For every $\epsilon > 0$ the following holds. Consider rectangles
 $$\G, \R_1,\R_2,...,\R_n \subset \C,
 \W(\G), \W(\R_i) > 8n+2\epsilon,$$
 such that the $\R_i$ are pairwise disjoint. Then after removing buffers of width at most $4n + \epsilon$, we can
 assume that the new rectangles $\G^{\mathrm{new}}$, $\Rnew_1$, $\Rnew_2$,...,$\Rnew_n$ have disjoint vertical boundaries.
  \end{lem}
  
 \begin{lem}[{\cite[Lemma A.8]{DL:wZ}}]\label{lemappen.annulus-buffer}
      Let $D \subset \C$ be a closed Jordan disk and $A$ be a closed topological annulus with $\mu = \mathrm{mod}(A)$
 such that the bounded component $O$ of $\C\setminus A$ intersects $\partial D$. Then for every rectangle
 $\R\subset D$ such that $\partial^h\R\subset D\setminus(A\cup O)$,
 after removing two $1/\mu$-buffers from $\R$, the new rectangle $\Rnew$ is disjoint from $O$.
\end{lem}

Let us consider the following construction which will be used in Assumption \ref{assump.4}. Suppose:
\begin{itemize}
    \item a rectangle $\mathcal{Y}$ is a union of its genuine subrectangles $\mathcal{Y}^+,\mathcal{X},\mathcal{Y}^-$ with disjoint interiors, where
$\mathcal{X}$ is between $\mathcal{Y}^+$ and $\mathcal{Y}^-$;
\item  closed disks $\mathcal{O}_{l},\mathcal{O}_{\rho}$ contain $\partial^{h,0}\mathcal{Y} = [a,b]$ and $\partial^{h,1}\mathcal{Y} = [c,d]$ respectively;
\item $\mathcal{O}_l\cup \mathcal{Y}\cup \mathcal{O}_\rho$ is a closed topological disk.
\end{itemize}

\begin{lem}[{\cite[Lemma A.4]{DL:wZ}}]\label{lem.combination.annulus}
    For $\mathcal{O}_l$, $\Y=\Y_+\cup\X\cup\Y_-$, $\mathcal{O}_{\rho}$ as above, if
 $$\W(\Y_+) \asymp \W(\Y_-)\asymp \mathrm{mod}(\mathcal{O}_l\setminus[a,b]) \asymp \mathrm{mod}(\mathcal{O}_\rho\setminus[c,d]) \asymp 1,$$
 then $\mathrm{mod}(\mathcal{O}_l\cup\Y\cup\mathcal{O}_\rho\setminus\X) \succeq 1$.
\end{lem}

 Consider a closed Jordan disk $D$ and a rectangle
$\mathcal{R}\subset \mathbb{C} \setminus D$ such that $\partial^h \mathcal{R}\subset \partial D$
that has a conformal pullback or push-forward
$$\mathcal{R}_1 := f^t(\mathcal{R})
\xleftarrow{ 1:1}\mathcal{R},\quad \mathcal{R}_1\subset \mathbb{C}\setminus D,\quad \partial^h \mathcal{R}_1\subset\partial D,\quad t\in\mathbb{Z}.$$
Assume next that there is an interval $T \subsetneq \partial D$ containing $\lfloor\partial^h\mathcal{R}\rfloor\cup\lfloor\partial^h\mathcal{R}_
1\rfloor$ such that
$$\partial^{h,0}\mathcal{R} < \partial^{h,1}\mathcal{R}, \quad\partial^{h,0}\mathcal{R}_1 < \partial^{h,1}\mathcal{R}_1 \quad\text{in}\quad T.$$
and $f^t$ maps $\partial^{h,0}\mathcal{R},\partial^{h,1}\mathcal{R},\lfloor\partial^h\mathcal{R}\rfloor$ onto $\partial^{h,0}\mathcal{R}_1,\partial^{h,1}\mathcal{R}_1,\lfloor\partial^h\mathcal{R}_1\rfloor$. We say that $\mathcal{R},\mathcal{R}_1$ are {\it linked} if one of the four conditions below holds:
\begin{equation}
    \begin{aligned}
        & \partial^{h,0}\mathcal{R}_1 < \partial^{h,0}\mathcal{R} < \partial^{h,1}\mathcal{R}_1 \quad\text{or}\quad \partial^{h,0}\mathcal{R} < \partial^{h,0}\mathcal{R}_1 < \partial^{h,1}\mathcal{R}\\
        & \partial^{h,0}\mathcal{R}_1 < \partial^{h,1}\mathcal{R} < \partial^{h,1}\mathcal{R}_1 \quad\text{or}\quad \partial^{h,0}\mathcal{R} < \partial^{h,1}\mathcal{R}_1 < \partial^{h,1}\mathcal{R}
    \end{aligned}
\end{equation}

\begin{lem}[{\cite[Lemma A.11]{DL:wZ}}]\label{lemappen.link}
    If $\R$ is linked to its conformal pullback or push-forward $\R_1$, then $\W(\R)\leq 2$.
\end{lem}

\subsection{Log-Rule and splitting argument}\label{appen.log-rule}
\renewcommand{\theequation}{B.\arabic{equation}}

\begin{lem}[{\cite[Lemma 2.5]{DL:wZ}}]\label{lemappen.logrule}
Let $\mathbb{D}$ be the unit disk. Let $I,J\subset\partial\mathbb{D}$ be intervals. If $\dist(I,J)\leq\min\{|I|,|J|\}$, then
$$\W^-(I,J)\asymp\log \frac{\min\{|I|,|J|\}}{\dist(I,J)}+1;$$
otherwise
$$\W^-(I,J)\asymp(\log \frac{\dist(I,J)}{\min\{|I|,|J|\}}+1)^{-1}.$$
\end{lem}

\begin{lem}[Trading $\F(I,J)$ into $\F(I^+,J^+)$]\label{lemappen.trading-FtoF+}
    Let $D\subset\mathbb{C}$ be a closed Jordan disk and $I,J\subset \partial D$ be disjoint intervals. Let $\tilde{I},\tilde{J}$ be thickenings of $I,J$ respectively such that $\tilde{I},\tilde{J}$ are disjoint intervals in $\partial D$. Set
    $$C := \W^-(I,\tilde{I}^c)+\W^-(J,\tilde{J}^c).$$
    Then there are intervals $\widehat{I},\widehat{J}\subset\partial D$ with $I\subset\widehat{I}\subset\tilde{I}$, $J\subset\widehat{J}\subset\tilde{J}$ and a subfamily $\G^{\mathrm{New}}\subset\F(\widehat{I}^+,\widehat{J}^+)$ whose curves are disjoint from the central arc $\beta$ of $\F^-(I,J)$,
    such that 
    \begin{equation}\label{eq.Ihat-to-I}
\W(\widehat{I}^+,\widehat{J}^+)\geq\W(\G^{\mathrm{New}})\geq\W(I,J)-O(C).
    \end{equation}

\end{lem}
\begin{proof}
    Let $\R$ be the vertical lamination of $\F({I},{J})$, $\G^{\mathrm{new}}\subset\R$ be the sublamination obtained by removing all $l\in\R$ satisfying one of the following properties:
    \begin{itemize}
        \item $l$ intersects the central arc $\beta$ of $\F^-(I,J);$
        \item $l$ contains a subarc in $\F^-(I,\tilde{I}^c);$
        \item $l$ contains a subarc in $\F^-(J,\tilde{J}^c).$
    \end{itemize}
Let $\X$ be the vertical subfamily of $\F^-(I,J)$ such that $\X$ is centered at $\beta$ and $\W(\X) = 2$. Let $\X_1,\X_2$ be the left and right hand side rectangle of $\X\setminus\beta$. Notice that if $l\in\R$ intersects $\beta$, then either $l$ cross intersects at least one of $\X_1,\X_2$, or $l$ is entirely contained in one of $\X_1,\X_2$. Therefore the width of curves intersecting $\beta$ has width $O(1)$. Hence we have $\W(\G^{\mathrm{new}})\geq \W(I,J)-O(C)$.

Since curves in $\G^{\mathrm{new}}$ does not intersect $\beta$, they possess a left-right order. Denote by $\alpha_-,\alpha_+$ the leftmost and rightmost arc in $\G^{\mathrm{new}}$ respectively.We assume that $\alpha_-<\alpha_+<\beta$ with respect to the clockwise order around $I$. Write $I = [a,b]$, $J = [c,d]$ with $a<b<c<d$ in the clockwise order. Define $\widehat{I} = [a_1,b_1]$, $\widehat{J} = [c_1,d_1]$ as follows:
\begin{itemize}
    \item If $\alpha_-$ starts in $I^+$, then $a_1 := a$; otherwise $a_1$ is the first intersection of $\alpha_-$ with $\partial D\setminus I$.
     \item If $\alpha_+$ starts in $I^+$, then $b_1 := b$; otherwise $b_1$ is the first intersection of $\alpha_+$ with $\partial D\setminus I$.
     \item If $\alpha_-$ ends in $J^+$, then $d_1 := d$; otherwise $d_1$ is the last intersection of $\alpha_-$ with $\partial D\setminus J$.
     \item If $\alpha_+$ ends in $J^+$, then $c_1 := c$; otherwise $c_1$ is the last intersection of $\alpha_+$ with $\partial D\setminus J$.
\end{itemize}
By construction, $I\subset\widehat{I}\subset\tilde{I}$, $J\subset\widehat{J}\subset\tilde{J}$. Restricting $\G^{\mathrm{new}}$ to $\F^+(\widehat{I},\widehat{J})$, we get the family of curves $\G^{\mathrm{New}}\subset\F(\widehat{I}^+,\widehat{J}^+)$ and
$$\W(\widehat{I}^+,\widehat{J}^+)\geq\W(\G^{\mathrm{New}})\geq\W(I,J)-O(C)$$
\end{proof}

\subsection{Series decomposition}\label{appen.series-decomp}
\renewcommand{\theequation}{C.\arabic{equation}}

 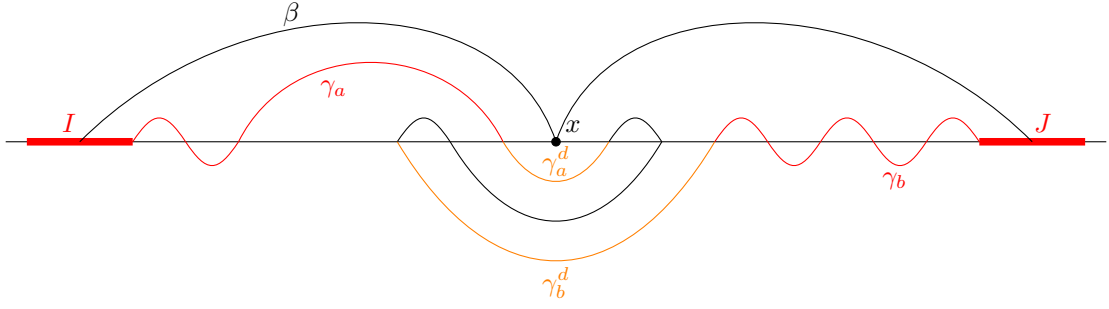
\begin{figure}[t!]
\[\begin{tikzpicture}[scale=1.4]

\node[red,above] at(4.6,0){$J$}; 
\node[red,above] at(-4.6,0){$I$};

\draw (-5.2,0) -- (5.2,0);

\draw[red,line width=1mm] (5,0)--(4,0)
(-5,0)--(-4,0);

\draw [] (-4.5,0)
 .. controls (-3, 1.5) and (-0.5,1.5) ..
 (0,0);

\draw [] (4.5,0)
 .. controls (3, 1.5) and (0.5,1.5) ..
 (0,0);
\node[above] at  (-2.5, 1.)  {$\beta$};

 \draw[red] (-3,0)
 .. controls (-2.5, 1) and (-1,1) ..
 (-0.5,0);
 
 \node[red] at (-2.1, 0.5){$\gamma_a$};
 
  \draw[red,shift={(-4,0)}]
(0.5,0)
.. controls (0.5+0.2, -0.3) and (1-0.2,-0.3) ..
(1,0); 
\draw[red,shift={(-4.5,0)}]
(0.5,0)
.. controls (0.5+0.2, 0.3) and (1-0.2,0.3) ..
(1,0); 
 
\draw[orange]
(0.5,0)
.. controls (0.2, -0.5) and (-0.2,-0.5) ..
(-0.5,0);

\draw[]
(0.5,0)
.. controls (0.5+0.2, 0.3) and (1-0.2,0.3) ..
(1,0);

\draw[scale =2]
(0.5,0)
.. controls (0.2, -0.5) and (-0.2,-0.5) ..
(-0.5,0);

\draw[shift={(-2,0)}]
(0.5,0)
.. controls (0.5+0.2, 0.3) and (1-0.2,0.3) ..
(1,0);

\draw[scale =3,orange]
(0.5,0)
.. controls (0.2, -0.5) and (-0.2,-0.5) ..
(-0.5,0);

\node[below,orange] at (0,-1.1) {$\gamma_b^d$};

\draw[red,shift={(1,0)}]
(0.5,0)
.. controls (0.5+0.2, 0.3) and (1-0.2,0.3) ..
(1,0); 
\draw[red,shift={(1.5,0)}]
(0.5,0)
.. controls (0.5+0.2, -0.3) and (1-0.2,-0.3) ..
(1,0); 
\draw[red,shift={(2,0)}]
(0.5,0)
.. controls (0.5+0.2, 0.3) and (1-0.2,0.3) ..
(1,0); 
 \draw[red,shift={(2.5,0)}]
(0.5,0)
.. controls (0.5+0.2, -0.3) and (1-0.2,-0.3) ..
(1,0); 
\draw[red,shift={(3,0)}]
(0.5,0)
.. controls (0.5+0.2, 0.3) and (1-0.2,0.3) ..
(1,0); 
 
 \node[below,red ] at (3.2,-0.2) {$\gamma_b$};
 
\filldraw (0,0) circle (0.04 cm);
\node[above right] at (0,0) {$x$};
\node[below,orange] at (0,0.05) {$\gamma_a^d$};

\end{tikzpicture}\]
\caption{The subcurves of $\gamma$: $\gamma_a$ (red), $\gamma^d_a$ (orange),  $\gamma^d_b$ (orange), $\gamma_b$ (red). Note that $\gamma_a$ and $\gamma_b$ are disjoint. }
\label{Fig:SeriesDecomp}
\end{figure}

Let $D$ be a Jordan disk, $I<J\subset\partial D$ be disjoint intervals and $L$ be the complement interval of $I\cup J$ such that $I<L<J$. Let $\R\subset\F^\circ_L(I,J)\setminus\F^+(I,J)$ be a lamination.

Let $\beta:[0,1]\longrightarrow\C$ ($\beta(0)\in I$, $\beta(1)\in J$) be the outermost curve in $\R$, i.e., all the other curves in $\R$ are between $\beta$ and $(L^c)^-$. Let $x = \beta(t)\in\beta\cap L$ be the first intersection point.

Consider any curve $\gamma:[0,1]\longrightarrow\C$ in $\R$ ($\gamma(0)\in I$, $\gamma(1)\in J$). Let us define:
\begin{itemize}
    \item $\gamma(a_2)$ to be the first intersection between $\gamma$ with $\lfloor x,J\rfloor$.
    \item $\gamma(a)$ to be the last before $a_2$ intersection between $\gamma$  with $\lfloor I,x\rfloor$.
    \item $\gamma_a$ to be the subcurve between $\gamma(0)$ and $\gamma(a)$.
    \item $\gamma_a^d$ to be the subcurve between $\gamma(a)$ and $\gamma(a_2)$.
    \item $\gamma(b_2)$ to be the last intersection between $\gamma$ and $\lfloor I,x\rfloor$.
    \item $\gamma(b)$ to be the first after $b_2$ intersection between $\gamma$ and $\lfloor x,J\rfloor$.
    \item $\gamma_b^d$ to be the subcurve between $\gamma(b_2)$ and $\gamma(b)$.
      \item $\gamma_b$ to be the subcurve between $\gamma(b)$ and $\gamma(1)$.
\end{itemize}
Notice that $\gamma_a,\gamma_b$ are disjoint and they are disjoint from $\gamma_a^d\cup\gamma_b^d$; $\gamma_a^d,\gamma_b^d$ may or may not coincide, but they are both contained in the interior of $D$. Moreover, if $\gamma_a^d\not=\gamma_b^d$, then $\gamma_b^d$ is beneath $\gamma_a^d$.

Let us specify the following laminations:
$$\tilde{\F}_a := \{\gamma_a\,|\,\gamma\in\R\},\quad \tilde{\F}_b := \{\gamma_b\,|\,\gamma\in\R\}.$$
$$\Gamma_a := \{\gamma_a^d\,|\,\gamma\in\R\},\quad \Gamma_b := \{\gamma_b^d\,|\,\gamma\in\R\},\quad \Gamma :=\Gamma_a\cup\Gamma_b.$$
Then $\R$ overflows $\tilde{\F}_a,\tilde{\F}_b,\Gamma$. Let $l$ be the lowest curve in $\Gamma$ with respect to $x$, i.e., any other curves in $\Gamma$ are above $l$. We denote by $J_a\subset L$ the interval between the left end point of $l$ and $x$; $I_b\subset L$ the interval between $x$ and the right end point of $l$. Define the restriction of laminations of $\tilde{F}_a,\tilde{F}_b$:
$$\F_a:=\{\gamma'\,|\,\gamma'\text{ is the first shortest subcurve of }\gamma\in\tilde{\F}_a\text{ connecting }I^+\text{ and }J_a^+\}.$$
$$\F_b:=\{\gamma'\,|\,\gamma'\text{ is the last shortest subcurve of }\gamma\in\tilde{\F}_b\text{ connecting }I_b^+\text{ and }J\}.$$
Consequently, one has:
\begin{lem}\label{lemappen.series.decomposition}
    The lamination $\R$ overflows ${\F}_a,{\F}_b,\Gamma$ with
    $$\F_a\subset\F^\circ_{L_a}(I,J_a),\quad\F_b\subset\F^\circ_{L_b}(I_b,J),\quad\Gamma\subset\F^-(J_a,I_b),$$ 
    where $L_a$ is the complementary interval between $I,J_a$ and $L_b$ is the complementary interval between $I_b,J$.
\end{lem}

\subsection{Removing small components from a rectangle}
\renewcommand{\theequation}{D.\arabic{equation}}

Fix a big ${\BDelta} \gg 1$ and some $\tau > 1$. Consider a rectangle $\R \subset \widehat{\C}$ with $\W(R)>1$. Let $D_i$, $i \in I$ be a finite set of open Jordan disks in $\R$ with pairwise disjoint closures satisfying the following properties:
\begin{itemize}
    \item[(A)] Every $D_i$ is a attached to a unique side $T_i$ of $\R$: the intersection $\partial D_i \cap \partial\R$ is a closed arc
 within $T_i$;
  \item[(B)] For each $i\in I$, there is a rectangle
 $$\Y_i\subset\R,\quad\partial^h\Y_i =\partial\Y_i\cap\partial\R\subset T_i,\quad \W(\Y_i)\geq{\BDelta}.$$
 protecting $D_i$ in the following sense: $\partial\R\setminus T_i$ and $D_i$ are in different components of $\R\setminus\Y_i$.
\end{itemize}
Let us denote by $T^{\mathrm{opp}}_i
 \subset\partial\R$ the opposite side to $T_i$. We will
 also consider the following weakening conditions of (B):
 \begin{itemize}
     \item[($\mathrm{B}_v$)] If $T_i$ is a vertical side of $\R$, then there is a rectangle
 $$\Y_i \subset \R,\quad \partial^h\Y_i =\partial\Y_i\cap\partial\R\subset T_i,\quad\W(\Y_i)\geq\tau.$$
 protecting $D_i$ in the following sense: $T^{\mathrm{opp}}_i$
 and $D_i$ are in different components of $\R\setminus\Y_i$.
  \item[($\mathrm{B}_h$)] If $T_i$ is a horizontal side of $\R$, then there is a rectangle
 $$\Y_i \subset \R,\quad \partial^h\Y_i =\partial\Y_i\cap\partial\R\subset\partial\R\setminus T^{\mathrm{opp}}_i,\quad\W(\Y_i)\geq{\BDelta}.$$
 protecting $D_i$ in the following sense: $T^{\mathrm{opp}}_i$
 and $D_i$ are in different components of $\R\setminus\Y_i$.
 \item[($\mathrm{B}'_h$)] If $T_i$ is a horizontal side of $\R$, then there is a rectangle
 $$\Y_i \subset \R,\quad \partial^h\Y_i =\partial\Y_i\cap\partial\R\subset\partial\R\setminus T^{\mathrm{opp}}_i,\quad\W(\Y_i)\geq{\BDelta}_i.$$
 protecting $D_i$ in the following sense: $T^{\mathrm{opp}}_i$
 and $D_i$ are in different components of $\R\setminus\Y_i$, such that for some index set $S\subset I$,
 $$ {\BDelta}_i \geq\tau,\quad\forall i\in S;\quad\quad{\BDelta}_i\geq{\BDelta},\quad \forall i\not\in S.$$
 \end{itemize}
Let us set 
$$\R' := \overline{\R\setminus\left(\bigcup_{i\in I}D_i\right)}$$
and view $\R'$ as a rectangle having the same vertices and labeling of vertical/horizontal sides as $\R$.
\begin{lem}\label{lemappen.outer-protect}
    For every $\epsilon > 0$, there is ${\BDelta}\gg_{\epsilon} 1$ such that the following holds.
    \begin{itemize}
        \item If $\R,\R'$ satisfy $\mathrm{(A)}$ and $(\mathrm{B}_v)$, then:
        \begin{equation}\label{eq.ABv}
            \W(\R)-O_\tau(1) < \W(\R')\leq\W(\R).
        \end{equation}
        Moreover, $O_\tau(1)\to0$ as $\tau\to+\infty$.
        
        \item If $\R,\R'$ satisfy $\mathrm{(A)}$ and $(\mathrm{B}_h)$, then:
        \begin{equation}\label{eq.ABh}
           \W(\R)\leq\W(\R') < (1+\epsilon)\cdot\W(\R).
        \end{equation}
        
        \item If $\R,\R'$ satisfy $\mathrm{(A)}$ and $(\mathrm{B}_h')$, then:
        \begin{equation}\label{eq.AB'h}
            \W(\R)\leq\W(\R') < (1+\epsilon)\cdot\W(\R)+O_\tau(|S|).
        \end{equation}
    \end{itemize}
\end{lem}
\begin{proof}
    It suffices to prove for the case when $\R = [0,x]\times[0,1]$ ($x>1$) is a Euclidean rectangle. 
    \begin{itemize}
        \item Suppose that $\R,\R'$ satisfy $\mathrm{(A)}$ and $(\mathrm{B}_v)$. The right inequality is obvious because $\W(\R')$ overflows $\W(\R)$. Since $D_i$ is protected by a rectangle $\Y_i$ based on $T_i$ with $\W(\Y_i)\geq\tau$, there exists $C_\tau>0$ such that the Euclidean diameter of $D_i$ is bounded by $C_\tau$ with $\lim\limits_{\tau\to+\infty}C_\tau=0$\footnote{To see this, let us assume that $\Y_i$ is attached to $\{0\}\times[0,1]$. Then $\Y_i$ is contained in $[0,2/\tau]\times[0,1]$: otherwise $\Y_i$ overflows the horizontal family of $[0,1/\tau]\times[0,1]$ and hence $\W(\Y)\leq\tau/2$, which contradicts to $\W(\Y)\geq\tau$.}. Hence $\R'$ is overflowed by $[C_\tau,x-C_\tau]\times[0,1]$ and (\ref{eq.ABv}) is proved. 
        
         \item Suppose that $\R,\R'$ satisfy $\mathrm{(A)}$ and $(\mathrm{B}_h)$. The left inequality is obvious because $\W(\R)$ overflows $\W(\R')$. For the same reason as above, the Euclidean diameter of $D_i$ is bounded by some $C_{\BDelta}>0$ with $\lim\limits_{{\BDelta}\to+\infty}C_{\BDelta}=0$. Since $D_i$ is based on $\partial\R\setminus T_i^{\mathrm{opp}}$, $\R'$ overflows $[0,x]\times[C_{\BDelta},1-C_{\BDelta}]$ and (\ref{eq.ABh}) is proved.

          \item Suppose that $\R,\R'$ satisfy $\mathrm{(A)}$ and $(\mathrm{B}'_h)$. Same as the above case, the left inequality is obvious. Let $R''$ be the rectangle by filling in $R'$ by $D_i$ with $i\in S$. Then $\R,\R''$ satisfy $\mathrm{(A)}$ and $(\mathrm{B}_h)$ and hence  $\W(\R'') < (1+\epsilon)\cdot\W(\R)$ by (\ref{eq.ABh}). On the other hand, any vertical curves $\gamma\in\F(\R')\setminus\F(\R'')$ must cross some protection $\Y_i$ with $i\in S$. Hence the width of those curves are bounded by $|S|\cdot\frac{1}{\tau}$ and (\ref{eq.AB'h}) is proved.
    \end{itemize}
\end{proof}

\subsection{Existence of barriers in rectangles} Assume that
\begin{enumerate}[label=\text{(\alph*)},font=\normalfont,leftmargin=*]
\item\label{item:a:ss:lam:Barriers} $\Fam\subset \wC\setminus (X\cup Y)$ is a lamination connecting compact sets $X, Y$; and
\item\label{item:b:ss:lam:Barriers}  $\R$ is a closed Jordan rectangle disjoint from $X\cup Y$.
\end{enumerate}

We say that a simple arc $\ell\in \R$ is a \emph{barrier} for $\Fam$ if $\ell\colon [0,1]\to \R$ connects  
\begin{equation*}
a=\ell(0)\in  \partial^{h,0}\R \qquad\text{ and }\qquad b=\ell(1)\in  \partial^{h,1}\R
\end{equation*}
and if for every leaf $\gamma \in \Fam$ the $\ell$ can be homotoped in $\R$ rel $\partial^{h}\R$ to a new arc $\ell^\gamma$ such that $\ell^\gamma$ is disjoint from $\gamma$. Such barrier $\ell$ connecting $a,b$ exists if and only if $\Fam$ is \emph{not linked} to $a,b\in \partial \RR$ rel $\RR$ in the following sense: for every $\gamma\in \Fam$, every component $\gamma'$ of $\gamma\cap \R$ starts and ends at the same arc of $\partial \RR\setminus \{a,b\}$.

Let $\Fam^{\# \R}$ be the subfamily of $\Fam$ consisting of curves $\gamma\in\Fam$ crossing $\R$ in the following sense: $\gamma$ has a component $\gamma'$ of $\R\cap \gamma$ connecting two vertical sides of $\R$.

\begin{lem}\label{lem:RR:barrier} For $\mathcal L$ and $\R$ as in Items~\ref{item:a:ss:lam:Barriers} and~\ref{item:b:ss:lam:Barriers}, if 
\begin{itemize}
    \item either $\Width(\R)\succeq 1$, 
    \item or $\Width(\Fam^{\# \R})=O( 1)$
\end{itemize}
 then, after removing $\Fam^{\# \R}$ from $\Fam$, the remaining new lamination \[\Fam^\new,\qquad\qquad \Width(\Fam^\new)=\Width(\Fam)-O(1)\] has a barrier $\ell \in \R$ as defined above. 
\end{lem}
\begin{proof}
    Clearly, $\Width(\RR)\succeq 1$ implies $\Width(\Fam^{\# \R})=O( 1)$; thus $\Width(\Fam^\new)=\Width(\Fam)-O(1)$. 

    Let $\overline \Fam^\new$ be the lamination of all $\gamma'$, where $\gamma'$ is a non-singleton component of $\gamma\cap \R$ for some $\gamma\in \Fam^\new$. 
    
    If $\overline \Fam^\new$ has a leaf $\beta$ connecting $\partial^{h,0}\R$ to  $\partial^{h,1}\R$, then we set $\ell\coloneqq \beta$ and remove the leaf $\gamma\in \Fam$ overflowing $\beta$.  
    
    Otherwise, consider two vertices $p\not= q$ of $\R$ on the same vertical side of $\R$. Let $\beta_p$ be the innermost leaf of $\overline \Fam^\new$ connecting two sides of $\R$ intersecting at $p$. Similarly, let $\beta_q$ be the innermost leaf of $\overline \Fam^\new$ connecting two sides of $\R$ intersecting at $q$. We define $a,b$ to be the intersections of $\beta_p,\beta_q$ with $\partial^{h}\R$. We set $\ell$ to be any arc in $\R$ connecting $a,b$. And we remove from $\Fam$ (at most two) leaves traveling through $a,b$.

\end{proof}

\end{document}